\documentclass[11pt]{amsart}

\newcommand{\preprint}[1]{}

\newcommand{\hide}[1]{}

\usepackage{amssymb}
\usepackage{amsbsy}
\usepackage{amscd}
\usepackage{amsmath}
\usepackage{amsthm}
\usepackage{bbold}
\usepackage{mathrsfs}
\input xy
\xyoption{all}

\numberwithin{equation}{section}

\theoremstyle{plain}
\newtheorem{thm}{Theorem}[section]
\newtheorem{prop}[thm]{Proposition}
\newtheorem{claim}[thm]{Claim}

\newtheorem{conj}[thm]{Conjecture}
\newtheorem{cor}[thm]{Corollary}
\newtheorem{lem}[thm]{Lemma}

\newtheorem{assumption}[thm]{Assumption}

\theoremstyle{definition}
\newtheorem{defi}[thm]{Definition}

\theoremstyle{remark}
\newtheorem{example}[thm]{Example}

\newtheorem{cons}[thm]{Construction}
\newtheorem{rem}[thm]{Remark}

\newcommand{\Def}{{\mathscr D}ef}

\newcommand{\fA}{\overline{\mathscr A}}
\newcommand{\fAz}{\overline{\Az}}
\newcommand{\A}{{\mathscr A}}

\newcommand{\Az}{\mathfrak{A}}

\newcommand{\fB}{\overline{\mathscr B}}
\newcommand{\B}{{\mathscr B}}
\newcommand{\catB}{{\mathcal B}}
\newcommand{\cov}{{\rm cov}}
\newcommand{\C}{{\mathscr C}}
\newcommand{\CC}{{\mathbb C}}
\newcommand{\catC}{{\mathcal C}}
\newcommand{\D}{{\mathscr D}}

\newcommand{\catD}{\mathcal D}
\newcommand{\dd}{{\delta}}
\newcommand{\gd}{{\Delta}}

\newcommand{\fE}{\overline{\mathscr E}}
\newcommand{\E}{{\mathscr E}}
\newcommand{\EE}{{\mathbb E}}
\newcommand{\ee}{{\epsilon}}
\newcommand{\eee}{{\varepsilon}}

\newcommand{\fF}{\overline{\mathscr F}}
\newcommand{\F}{{\mathscr F}}
\newcommand{\fG}{\overline{\mathscr G}}
\newcommand{\G}{{\mathscr G}}
\newcommand{\g}{{\gamma}}

\newcommand{\fH}{{\mathcal H}}
\renewcommand{\H}{{\mathscr H}}

\newcommand{\I}{{\mathscr I}}

\newcommand{\K}{{\mathscr K}}

\newcommand{\LL}{{\mathbb L}}
\newcommand{\LB}{{\mathscr L}}
\newcommand{\com}{{\mathbb{\Lambda}}}
\newcommand{\gL}{{\Lambda}}
\newcommand{\fM}{{\mathfrak M}}
\newcommand{\MM}{{\mathcal M}}
\newcommand{\M}{{\mathscr M}}
\newcommand{\N}{{\mathscr N}}
\newcommand{\ko}{{\mathscr O}}

\newcommand{\Go}{{\Omega}}
\newcommand{\go}{{\omega}}

\renewcommand{\P}{{\mathscr P}}

\newcommand{\PP}{{\mathbb P}}

\newcommand{\Q}{{\mathscr Q}}

\newcommand{\R}{{\mathscr R}}

\newcommand{\T}{{\mathscr T}}
\newcommand{\TT}{{\mathbb T}}

\newcommand{\fU}{\overline{\mathscr U}}
\newcommand{\U}{{\mathscr U}}
\newcommand{\fV}{\overline{\mathscr V}}
\newcommand{\V}{{\mathscr V}}
\newcommand{\W}{{\mathscr W}}

\newcommand{\X}{{\mathscr X}}
\newcommand{\Y}{{\mathscr Y}}
\newcommand{\YY}{{\mathbb Y}}

\newcommand{\Integers}{{\mathbb Z}}
\newcommand{\ZZ}{{\mathbb Z}}
\newcommand{\ComplexNumbers}{{\mathbb C}}
\newcommand{\RationalNumbers}{{\mathbb Q}}

\newcommand{\Sn}{{\mathfrak{S}}_n}
\newcommand{\SU}{D^b(X)_{\mathbb T}}
\newcommand{\Spi}{D^b(X)_{\widetilde{{\mathbb Y}}}}
\newcommand{\tr}{{\rm tr}}
\newcommand{\pt}{{\rm pt}}

\newcommand{\SHilb}{{\rm Hilb}^{\mathfrak{S}_n}}
\newcommand{\fgd}{\overline{\gd}}
\newcommand{\fdd}{\overline{\dd}}
\newcommand{\fee}{\overline{\ee}}
\newcommand{\fperp}{\overline{\LL}}

\newcommand{\dbp}[1]{D^b({#1})^\LL}
\newcommand{\fdbp}[1]{D^b({#1})^{\fperp}}
\newcommand{\ol}[1]{\overline{#1}}

\newcommand{\IsomRightArrow}{\widetilde{\to}}

\newcommand{\RightArrowOf}[1]{\stackrel{#1}{\rightarrow}}

\newcommand{\LongRightArrowOf}[1]{\stackrel{#1}{\longrightarrow}}

\newcommand{\StructureSheaf}[1]{{\mathscr O}_{#1}}

\newcommand{\rank}{{\rm rank}}

\newcommand{\Pic}{{\rm Pic}}

\newcommand{\Ext}{{\rm Ext}}
\newcommand{\Tor}{{\rm Tor}}
\newcommand{\Hom}{{\rm Hom}}

\newcommand{\End}{{\rm End}}

\newcommand{\Lotimes}{\stackrel{L}{\otimes}}
\newcommand{\SheafHom}{{\mathscr H}om}
\newcommand{\SheafEnd}{{\mathscr E}nd}
\newcommand{\SheafExt}{{\mathscr E}xt}
\newcommand{\SheafTor}{{\mathscr T}or}

\newcommand{\Ideal}[1]{{\mathscr I}_{#1}}
\newcommand{\id}{{\mathbb 1}}

\newcommand{\Wedge}[1]{\stackrel{#1}{\wedge}}

\newcommand{\DoubleTilde}[1]{\stackrel{\approx}{#1}}

\begin{document}
\title[Integral Transforms and Deformations]
{Integral Transforms and Deformations of K3 Surfaces}
\date{\today}
\author{E. Markman}
\address{Department of Mathematics and Statistics, 
University of Massachusetts, Amherst, MA 01003}
\email{markman@math.umass.edu}
\author{S. Mehrotra}
\address{Facultad de Matem\'aticas, PUC Chile, Av. Vicu\~na Mackenna 4860, Santiago, Chile; 
Chennai Mathematical Institute, H1, SIPCOT IT Park, Siruseri, Kelambakkam 603103, India}
\email{smehrotra@mat.uc.cl}

\begin{abstract}
Let $X$ be a $K3$ surface and $M$ a smooth and projective moduli space of 
stable sheaves on $X$ of Mukai vector $v$. A universal sheaf $\U$ over $X\times M$ induces an
integral transform $\Phi_\U:D^b(X)\rightarrow D^b(M)$ from the derived category of coherent sheaves on $X$
to that on $M$. 
\begin{enumerate}
\item
We prove that $\Phi_\U$ is faithful. $\Phi_\U$ is not full if the dimension of $M$ is $\geq 4$.
\item
We exhibit  the full subcategory of $D^b(M)$, consisting of objects in the image of $\Phi_\U$, 
as the quotient of a category, explicitly constructed from $D^b(X)$, 
by a natural congruence relation defined in terms of the Mukai vector $v$. 
\item
\label{abstract-item-deformability}
Let $\fM^0$ be a component of the moduli space of isomorphism classes of marked 
irreducible holomorphic symplectic manifolds deformation equivalent to the Hilbert scheme $X^{[n]}$ of $n$ points on a $K3$ surface $X$, $n\geq 2$. $\fM^0$ is $21$-dimensional, while the moduli of K\"{a}hler $K3$ surfaces is $20$-dimensional. 
We construct a geometric deformation of the derived categories of $K3$ surfaces over a Zariski dense open subset of
$\fM^0$, which coincides with $D^b(X)$ whenever the marked manifold 
is a moduli space of sheaves on $X$ satisfying a technical condition. 
\end{enumerate}
Statement (\ref{abstract-item-deformability}) assumes Conjecture \ref{vanishing}
asserting that the dimension, of the first sheaf cohomology of a reflexive hyperholomorphic sheaf with an isolated singularity, remains constant along twistor deformations. The conjecture is known to hold for hyperholomorphic vector \nolinebreak bundles.
\end{abstract}

\maketitle

\tableofcontents 
%
\section{Introduction}
\label{sec-introduction}
%
\subsection{A faithful functor}
\label{sec-faithful-functor}
Let $X$ be a projective $K3$ surface, $H$
an ample line bundle on $X$, and $M:=M_H(v)$ the
(coarse) moduli space of Gieseker-Simpson 
$H$-stable sheaves $\E$ on $X$ whose Chern classes are specified by the vector $v=v(\E):=ch(\E)\sqrt{td_X}$ in
the integral cohomology of $X$. 
$M$ is a smooth holomorphic-symplectic variety, by \cite{mukai-symplectic}. 
We shall assume here that there aren't any strict semi-stable sheaves 
of class $v$, so that $M$ is projective.  There is a numerical criterion
which guarantees this: it amounts to choosing $H$ to be $v${\em-generic}, that is,
in a region complementary to a countable union of suitable hyperplanes depending on $v$. 
While $X\times M$ need not carry a universal sheaf, there always exists a 
quasi-universal sheaf on it (see the appendix of \cite{mukai-tata}). In other words, there is a twisted sheaf
 $\U$ on $X\times M$ which is universal locally on $M$ in the etale or analytic topology
(see \cite{Cal-thesis}).
 
Write $\pi_{X}$ and $\pi_M$ for the two projections from $X\times M$, and let
$\Phi_\U: D^b(X)\to D^b(M,\theta)$ be the integral transform 
\begin{equation}
\label{eq-phi-U}
\Phi_\U(\_) \ \ \ := \ \ \ \pi_{M,*}(\pi_X^*(\_)\otimes \U). 
\end{equation}
Here $\theta$ is the class of $\U$ in 
the Brauer group of $M$, $D^b(M,\theta)$ stands for the bounded derived category 
of $\theta$-twisted coherent sheaves on $M$, and $\pi_{M,*}$ and $\pi_X^*$ are
derived functors appropriately defined in this context.
We refer the reader to \cite{Cal-thesis} for a careful discussion of this formalism.
Suffice it to say that most of the familiar properties of duality and the usual functorial calculus for 
derived categories of coherent sheaves continue to hold here. In particular, it follows by Serre
duality for twisted sheaves that $\Phi_\U$ has a right adjoint $\Psi_\U(\_):=
\pi_{X,*}(\pi_M^*(\_)\otimes \U^{\vee})[2]$.

The composition $\Psi_\U \circ \Phi_\U : D^b(X) \to D^b(X)$ is the integral transform
with kernel 
\begin{equation}
\label{eq-kernel-A}
\A:= {\pi_{13, *}}(\pi_1^*\omega_X\otimes \pi_{12}^*(\U)^\vee\otimes\pi_{23}^*(\U))[2] \in D^b(X\times X),
\end{equation}
where $\pi_{ij}$ is the projection from $X\times M \times X$ onto the product  
of the $i$-th and $j$-th factors.

The {\em Mukai lattice}  $\widetilde{H}(X,\Integers)$ of the $K3$ surface $X$ is the 
integral cohomology $H^*(X,\Integers)$ endowed with the {\em Mukai pairing}
$(u,v):=-\int_Su^\vee\cup v$, where the duality operator $u\mapsto u^\vee$ changes the sign of the direct summand in $H^2(X,\Integers)$.
If $u=ch(E)\sqrt{td_X}$ and $v=ch(F)\sqrt{td_X}$ are the Mukai vectors of two coherent sheaves $E$ and $F$ on $X$,
then $(u,v)=-\chi(E^*\otimes F)$, by Hirzebruch-Riemann-Roch, where the dual $E^*$ and the tensor product are 
taken in the derived category, and $\chi$ is the Euler characteristic. The dimension of the moduli space $M:=M_H(v)$
discussed above is $2+(v,v)$.
Denote by $\Delta:X\rightarrow X\times X$ the diagonal embedding.

\begin{thm}
\label{introduction-thm-A-is-direct-sum}
(Theorem \ref{thm-A-is-direct-sum}).
Let $v\in \widetilde{H}(X,\Integers)$ be a primitive class with $(v,v)=2n-2$,
$n\geq 2$, and  $H$ a $v$-generic polarization. One has a natural morphism 
\[
\alpha: \bigoplus_{i=0}^{n-1} \Delta_*\ko_X\otimes_\ComplexNumbers 
\Ext^{2i}\left(\StructureSheaf{M},\StructureSheaf{M}\right)[-2i] \rightarrow \A.
\]
\begin{enumerate}
\item 
When $M:=M_H(v)$ is the Hilbert scheme of $n$ points on $X$ and $\U$ is the universal ideal sheaf the morphism 
$\alpha$ is an isomorphism. 
In particular, a choice of a non-zero element of the one-dimensional vector space 
$\Ext^{2}\left(\StructureSheaf{M},\StructureSheaf{M}\right)$
yields an isomorphism $\A\cong \oplus_{i=0}^{n-1} \Delta_*\ko_X[-2i].$
\item 
In general, for $v$ arbitrary, the structure sheaf of the diagonal  $\Delta_*\ko_X$ is a 
direct summand of $\A$ in $D^b(X\times X)$. In particular, the integral transform 
$\Phi_\U : D^b(X) \to \nolinebreak D^b(M,\theta)$ is faithful.
\end{enumerate}
\end{thm}

Part (\ref{thm-item-A-is-a-direct-sum-Hilbert-scheme-case})
of Theorem \ref{thm-A-is-direct-sum} was independently proven by Nick Addington
\cite{addington}. 
The isomorphism $\A\cong \oplus_{i=0}^{n-1} \Delta_*\ko_X[-2i]$ was established for some moduli spaces of torsion sheaves 
on $K3$ surfaces in \cite{adm}.  
Further cases where $\alpha$ is an isomorphism
are provided in Lemmas \ref{lemma-a-sufficient-condition-to-be-totally-split} and 
\ref{lemma-a-sufficient-condition-to-belong-to-U}.
The proof of the first part of Theorem \ref{introduction-thm-A-is-direct-sum} 
relies heavily on Haiman's work \cite{Haiman-1, Haiman-2} on the 
$n!$-Conjecture  on one hand, and the derived McKay
correspondence of Bridgeland-King-Reid on the other \cite{BKR}. 
The second part makes use  of results of Mukai, O'Grady, and Yoshioka, 
which show that there is an
isomorphism of Hodge structures between $H^2(M_H(v),\Integers)$ and 
the orthogonal complement $v^\perp$ of the Mukai vector $v$ in
the Mukai lattice of  $X$ (\cite{mukai-sugaku,OG,yoshioka-abelian-surface}, see also Theorem \ref{Yoshioka-main} below).

%
\subsection{The full subcategory of $D^b(M,\theta)$ with objects coming from the surface}
Denote
by $\SU$ the full subcategory of $D^b(M,\theta)$ with objects of the form $\Phi_\U(x)$, for some object $x$ in $D^b(X)$. 
We provide next an explicit computation of $\SU$.
We will use the following standard construction in category theory.

\begin{defi}
\label{def-congruence-relation}
\cite[Section II.8]{working}.
Let $\C$ be a category.
\begin{enumerate}
\item
A {\em congruence relation} ${\mathfrak R}$ on $\C$ consists of an equivalence relation 
${\mathfrak R}_{x_1,x_2}$ on $\Hom(x_1,x_2)$, for every pair of objects $x_1$, $x_2$ in $\C$, satisfying the following property.
Given morphisms $f_1, f_2 : x_1\rightarrow x_2$ related by ${\mathfrak R}_{x_1,x_2}$,
objects $x_0$, $x_3$ of $\C$, and 
morphisms $e:x_0\rightarrow x_1$ and $g:x_2\rightarrow x_3$, the morphisms
$gf_1e$ and $gf_2e$ are related by ${\mathfrak R}_{x_0,x_3}$ in $\Hom(x_0,x_3)$.
\item
Let ${\mathfrak R}$ be a congruence relation on $\C$. The {\em quotient category} $\C/{\mathfrak R}$ 
is the category whose objects are those of $\C$ and such that 
$\Hom_{\C/{\mathfrak R}}(x_1,x_2):=\Hom_\C(x_1,x_2)/{\mathfrak R}_{x_1,x_2}$.
The natural functor $Q:\C\rightarrow \C/{\mathfrak R}$ is called the {\em quotient functor}.
\end{enumerate}
\end{defi}

An explicit computation of $\SU$ as a quotient category requires an explicit category $\C$ 
and an explicit relation ${\mathfrak R}$.
We take $\C$ to be the full subcategory $\Spi$  of $D^b(X\times \nolinebreak M)$
with objects of the form $\pi_X^*(x)$, for some object $x$ in $D^b(X)$. 
Set $\Hom^\bullet_{\Spi}(x,y):=\oplus_{i\in\Integers}\Hom_{\Spi}(x,y[i])$. The category $\Spi$ is explicit, since
\[
\Hom^\bullet_{\Spi}(x,y):=\Hom^\bullet_{D^b(X\times M)}(\pi_X^*x,\pi_X^*y)\cong 
\Hom^\bullet_{D^b(X)}(x,y)\otimes \Hom^\bullet_{D^b(M)}(\StructureSheaf{M},\StructureSheaf{M}).
\]

We describe next a congruence relation on $\Spi$.
Set ${\rm pt}:={\rm Spec}(\ComplexNumbers)$ and let $c:M\rightarrow {\rm pt}$ be the constant morphism. We get the object
$Y(\StructureSheaf{M}):=Rc_*\StructureSheaf{M}$ in $D^b({\rm pt})$. 
As a graded vector space $Y(\StructureSheaf{M})$ is 
$\oplus_{i=0}^n H^{2i}(M,\StructureSheaf{M})[-2i]$, where $n=\dim_\ComplexNumbers(M)/2$. 
Given a graded vector space $V$, 
let $\id_{D^b(X)}\otimes_\ComplexNumbers V$
be the endofunctor of $D^b(X)$ sending an object $x$ to
$x\otimes_\ComplexNumbers V$.
Set
\begin{eqnarray*}
\Upsilon & := & \id_{D^b(X)}\otimes_\ComplexNumbers Y(\StructureSheaf{M}),
\\
R & := & \id_{D^b(X)}\otimes_\ComplexNumbers H^{2n}(M,\StructureSheaf{M})[-2n].
\end{eqnarray*}
These two endofunctors are integral transforms with kernels
$\Y:=\Delta_*\StructureSheaf{X}\otimes _\ComplexNumbers Y(\StructureSheaf{M})$ and 
$\R:=\Delta_*\StructureSheaf{X}\otimes _\ComplexNumbers H^{2n}(M,\StructureSheaf{M})[-2n]$ in $D^b(X\times X)$.
Let $\pi_X:X\times M\rightarrow X$ be the projection. Note that the endo-functor $\Upsilon$ is naturally isomorphic to
$R\pi_{X_*}\circ\pi_X^*$. We thus have the adjunction isomorphism
\begin{equation}
\label{eq-adjunction-isomorphism-for-Upsilon}
\Hom_{\Spi}(x,y):=\Hom_{D^b(X\times M)}(\pi_X^*x,\pi_X^*y)\cong \Hom_{D^b(X)}(x,\Upsilon(y)).
\end{equation}
The congruence relation ${\mathfrak R}$ is defined in terms of a natural transformation
\[
h \ : \ R \rightarrow \Upsilon,
\]
which we define next. The natural transformation $h$ is 
induced by a morphism $h:\R\rightarrow \Y$.
Write $h=\sum_{i=0}^nh_{2i}$ according to the direct sum decomposition of $Y(\StructureSheaf{M})$,
so that 
\[
h_{2i} \ : \ \Delta_*(\StructureSheaf{X})\otimes_\ComplexNumbers H^{2n}(M,\StructureSheaf{M})[-2n]
\ \ \ \rightarrow \ \ \ 
\Delta_*(\StructureSheaf{X})\otimes_\ComplexNumbers H^{2n-2i}(M,\StructureSheaf{M})[2i-2n].
\]
Mukai's Hodge isometry (\ref{mukai-hom}) provides a canonical 
identification of $H^2(X,\StructureSheaf{X})$ with $H^2(M,\StructureSheaf{M})$.
It follows that the morphism $h_{2i}$ is naturally a class in
$\Hom_{X\times X}(\Delta_{X,*}(\StructureSheaf{X}),\Delta_{X,*}(\omega_X^{\otimes i})[2i])$, as we
carefully check in Section \ref{sec-Yoneda-Algebra}.
In particular, $h_2$ is a class in the Hochschild homology $HH_0(X)$.
The Hochschild-Kostant-Rosenberg isomorphism (reviewed in Section \ref{HH-coho}) maps
the Mukai vector $v$ to a class in $HH_0(X)$. We set $h_2$ to be the image of $v$, 
$h_0:=-1$, and $h_4:=-(h_2)^2$.  We necessarily have $h_{2i}=0$, for $i>2$, for dimension reasons. 

Given an object $x$ in $D^b(X)$, let 
\[
h_x:x\otimes_\ComplexNumbers H^{2n}(M,\StructureSheaf{M})[-2n]\rightarrow 
x\otimes_\ComplexNumbers Y(\StructureSheaf{M})
\] 
be the morphism induced by the natural transformation $h$. 
Consider the relation  ${\mathfrak R}$ on $\Spi$ given as follows. The morphisms $f_1$, $f_2$ in
$\Hom_{\Spi}(x_1,x_2)$ are related by ${\mathfrak R}_{x_1,x_2}$, if and only if $f_1-f_2$ belongs to the image of
the homomorphism
\begin{equation}
\label{eq-push-forward-by-h-x-2}
(h_{x_2})_*:\Hom_{D^b(X)}(x_1,x_2)\otimes H^{2n}(M,\StructureSheaf{M})[-2n]\rightarrow 
\Hom_{\Spi}(x_1,x_2)
\end{equation}
induced by composition with $h_{x_2}$. 
It is easy to check that ${\mathfrak R}$ is a congruence relation, as we do in the proof of Theorem
\ref{thm-quotient-category}.

We describe next the functor inducing the equivalence between $\SU$ and the quotient category $\Spi/{\mathfrak R}$.
Let $\Xi_\U:D^b(X\times M)\rightarrow D^b(M,\theta)$ be the composition of tensorization by $\U$ followed
by $R\pi_{M_*}$. Then $\Phi_\U=\Xi_\U\circ\pi_X^*$. 
Let $Q:\Spi\rightarrow \SU$ be the restriction of the functor $\Xi_\U$.

\begin{thm}
\label{thm-introduction-computation-of-the-monad-A}
(Theorem \ref{thm-quotient-category})
Assume that the morphism $\alpha$ in Theorem \ref{introduction-thm-A-is-direct-sum} is an isomorphism.
\begin{enumerate}
\item
${\mathfrak R}$ is a congurence relation.
\item
\label{thm-item-identification}
The kernel of 
$
Q:\Hom_{\Spi}(x,y)\rightarrow \Hom_{D^b(M,\theta)}(\Phi_\U(x),\Phi_\U(y))
$
is identified with the image of $(h_y)_*:\Hom_{D^b(X)}(x,R(y))\rightarrow \Hom_{D^b(X)}(x,\Upsilon(y))$
via the adjunction isomorphism (\ref{eq-adjunction-isomorphism-for-Upsilon}).
\item
\label{thm-item-Q-is-full-and}
The functor $Q$ is full. 
Consequently, $Q$ factors as the composition of the quotient functor $\Spi\rightarrow \Spi/{\mathfrak R}$
and an equivalence functor $\Spi/{\mathfrak R}\cong \SU$.
\end{enumerate}
\end{thm}

As a consequence we obtain the description of the Yoneda algebra of
$\Phi_\U(F)$, for a simple sheaf $F$ on $X$, as the quotient of 
the tensor product of the Yoneda algebras of $F$ and $\StructureSheaf{M}$ by 
the principal ideal generated by an explicit relation
(Theorem \ref{thm-the-yoneda-algebra-ofPhi-U-of-a-simple-sheaf}). 
Related results were obtained by A. Krug for the Hilbert scheme of points on a projective surface \cite{krug}.

Let us describe the key ingredient in the proof of parts 
(\ref{thm-item-identification}) and (\ref{thm-item-Q-is-full-and}) of Theorem 
\ref{thm-introduction-computation-of-the-monad-A}.
The equality $\Phi_\U=\Xi_\U\circ\pi_X^*$ gives rise to a natural transformation
$q:\Upsilon\rightarrow \Psi_\U\circ\Phi_\U$ given by $q:=R\pi_{X_*}\DoubleTilde{\eta}\pi_X^*$, where
$\DoubleTilde{\eta}:\id_{D^b(X\times M)}\rightarrow \Xi_\U^\dagger\circ \Xi_\U$ is the unit for the
adjunction $\Xi_\U \dashv \Xi_\U^\dagger$. The natural transformation $q$, in turn, is induced by a morphism 
of kernels $q:\Y\rightarrow \A$. We get an exact triangle in $D^b(X\times X)$, which admits a splitting
\begin{equation}
\label{eq-introduction-exact-triangle-of-h-q}
\R\RightArrowOf{h}\Y\RightArrowOf{q}\A \RightArrowOf{0} \R[1],
\end{equation}
by Proposition \ref{thm-annihilator-of-A-in-Y} and Theorem \ref{thm-universal-class-h-in-Hochschild-cohomology-of-X}. 
The  functor $Q$ of Theorem 
\ref{thm-introduction-computation-of-the-monad-A}
 is induced by the natural transformation $q:\Upsilon\rightarrow \Psi_\U\circ\Phi_\U$ in the sense
 that the following diagram commutes
 \[
 \xymatrix{
 \Hom_{D^b(X)}(x,\Upsilon(y)) \ar[r]^{q}  \ar[d]_{\cong} &
 \Hom_{D^b(X)}(x,\Psi_\U\Phi_\U(y)) \ar[d]^{\cong}
 \\
 \Hom_{D^b(X\times M)}(\pi_X^*x,\pi_X^*y) \ar[r]_{\Xi_\U} & 
 \Hom_{D^b(M,\theta)}(\Phi_\U(x),\Phi_\U(y)),
 }
 \]
where the vertical arrows are the adjunction isomorphisms (Theorem \ref{thm-determination-of-the-category-SU}).
Parts (\ref{thm-item-identification}) and (\ref{thm-item-Q-is-full-and}) of Theorem 
\ref{thm-introduction-computation-of-the-monad-A} follow immediately from the splitting of the exact 
triangle (\ref{eq-introduction-exact-triangle-of-h-q}). 

%
\subsection{A reconstruction of $D^b(X)$ as a category of comodules of a comonad}
\label{sec-introduction-reconstruction}
We shall be using the categorical device of (co)monads and (co)modules over them
in order to recover the category $D^b(X)$ in terms of data in $D^b(M\times M)$. 
The latter data will then be deformed with $M$ yielding non-commutative deformations of $D^b(X)$.
Let us briefly recall the necessary notions here. 
A detailed presentation can be found in Chapter VI of Mac Lane's text \cite{working}.
A {\em comonad} $\LL$ on a category $A$ is simply a comonoid
object in the functor category $End(A)$. Explicitly, $\LL$ is a triple 
$\langle L, \ee, \dd \rangle$ where 
$L: A \to A$ is an endofunctor, and the {\em counit} $\ee$, and {\em comultiplication} $\dd$ 
are natural transformations 
$$
\left.  \begin{array}{l} 
\epsilon: L \longrightarrow I,\\
\delta: L\longrightarrow L^2,
\end{array} \right.
$$
satisfying coassociativity:
\begin{equation}\label{comp-coass}
\xymatrix{
\ar[d]_{\delta}L \ar[r]^{\delta} & L^2 \ar[d]^{L\delta} \\
L^2 \ar[r]_{\delta L} & L^3}
\end{equation}
and the left and right counit laws:
\begin{equation}\label{comp-counit}
\xymatrix{
\ar@{=}[d] L \ar@{=}[r] & \ar[d]^{\delta} L \ar@{=}[r] & L \ar@{=}[d] \\
IL & \ar[l]^{\epsilon L} L^2  \ar[r]_{L \epsilon } & LI}
\end{equation}
Associated to any adjunction $F: X\to A$, $G:A\to X$, with $F\dashv G$, there is a natural
comonad with $L=FG: A \to A$, $\ee$ the counit of the adjunction, and 
$\delta=F\eta G: L\to L^2$, where $\eta$ is the unit of the adjunction.

A {\em comodule} for a comonad $\langle L, \eta, \delta \rangle$ is a
pair $(a, h)$ consisting of an object $a\in A$ and an arrow $h: a \to La$ such that
$\delta \circ h = Lh \circ h$ and $\epsilon \circ h = id$. A morphism $f: (a, h) \to (a', h')$
is an arrow $f\in \Hom_A(a,a')$ which renders commutative the diagram
$$\xymatrix{
\ar[d]_{f}a \ar[r]^{h} & La \ar[d]^{Lf} \\
a' \ar[r]_{h'} & La'.}
$$
The set of all $L$-comodules together with their morphisms form a category $A^L$. Finally,
{\em monads}, and {\em modules} over monads are simply the notions dual to those defined above.

Adjoint pairs of functor naturally give rise to (co)monads. In our situation,
let $L$ be the composition $\Phi_\U\circ \Psi_\U : D^b(M_H(v),\theta)\to D^b(M_H(v),\theta)$. 
The adjunction 
$\Phi_\U\dashv \nolinebreak \Psi_\U$ yields the unit $\eta:\Psi_\U\circ \Phi_\U {\rightarrow} id$, the 
counit $\ee: \Phi_\U\circ \Psi_\U {\rightarrow} id$, and the comultiplication
\begin{equation}
\label{eq-comultiplication}
\dd=\Phi_\U\eta\Psi_\U: L  {\rightarrow} L^2.
\end{equation}
These define a comonad
$\LL:=\langle L, \ee, \dd\rangle$ in $D^b(M_H(v),\theta)$. 
We get the category $\dbp{M_H(v),\theta}$ of comodules of $\LL$. 
Denote by $\overline{\Phi}_\U: \dbp{M_H(v),\theta} \to D^b(M_H(v),\theta)$ 
the forgetful functor $\overline{\Phi}_\U(\G,h)=\G$. For (co)monads given by
adjunctions, there is a natural {\em comparison functor} from the source category
to the category of (co)modules, which in our case will be
denoted as $\hat{\Phi}: D^b(X) {\longrightarrow} \dbp{M_H(v),\theta}$.
It takes $\H\in D^b(X)$ to the comodule
$(\Phi_\U(\H),\Phi_\U\eta_\H)$

The unit $ \eta: id {\rightarrow} \Psi_\U\circ\Phi_\U$ is split by 
Theorem \ref{introduction-thm-A-is-direct-sum}. Hence the triangulated version of the Barr-Beck Theorem 
\cite{elagin, bb} (see also \cite{bal2}) immediately yields:

\begin{prop} 
\label{prop-Phi-hat-is-an-equivalence}
The comparison functor 
$\hat{\Phi}: D^b(X)  \stackrel{\cong}{\longrightarrow} \dbp{M_H(v),\theta}$ 
is an equivalence of categories. 
In particular, it induces the structure of a triangulated
category on $\dbp{M_H(v),\theta}$ such that the forgetful functor  $\overline{\Phi}_\U$ is exact.
\end{prop}
\noindent This captures $D^b(X)$ in terms of structures defined only on $M_H(v)$.

%
\subsection{A deformation of  $K3$-categories}
The deformations of $M_H(v)$ arising from those of the underlying K3 surface
$X$ form a 20-dimensional locus in the 21-dimensional Kuranishi deformation space of $M_H(v)$.
Thus the generic deformation of $M_H(v)$ is {\em not} a moduli space of sheaves.
The next theorem interprets these as arising from
``non-commutative'' perturbations of $X$. More precisely, we construct
deformations of $D^b(X)$ which correspond to those of $M_H(v)$
away from the moduli space locus.

Let $\F\in D^b(M_H(v)\times M_H(v),\pi_1^*\theta^{-1}\pi_2^*\theta)$ be the kernel of the endofunctor
$L$.
The object $\F$ plays a prominent role in the study of the geometry and cohomology
of the moduli spaces $M_H(v)$ (see \cite{mukai-tata, KLS, markman-diagonal}). 
Let $\pi_{ij}$ be the projection from $M_H(v)\times X\times M_H(v)$ onto the product of the $i$-th and $j$-th factors.
$\F$ is a complex with cohomology concentrated in degrees $-1$ and $0$:
\begin{equation}\label{coho-of-F}
\H^i(\F)= \H om_{\pi_{13}}(\pi_{12}^*\U, \pi_{23}^*\U[i])
              \cong\left\{  \begin{array}{ll} \ko_{\gd_{M_H(v)}} & \mbox{if} \; i=0,\\
               \E & \mbox {if} \; i=-1, 
               \end{array} \right.
\end{equation}
where $\E$ is a twisted reflexive sheaf  of rank $(\dim M_H(v) -2)$. 
\begin{defi}
\label{def-modular}
We shall refer to 
$\F$ as the {\em modular complex} of $M_H(v)$, 
and to $\E$ as the {\em modular sheaf} of $M_H(v)$ to emphasize
their origin.  The sheaf of Azumaya algebras $\E nd (\E)$ will be referred to as the 
{\em modular Azumaya algebra} of $M_H(v)$. 
\end{defi}

Let $X$ be a $K3$ surface with a trivial Picard group and 
let $\F\in D^b(X^{[n]}\times X^{[n]})$ be the kernel of the endofunctor
$L$, so that its square $L^2$ has kernel the convolution $\F \circ \F$. 
By Proposition 5.1 of \cite{cal-mukai1}, the counit $\ee$ and the comultiplication 
$\dd$ of $L$
correspond to morphisms of objects $\F \to \ko_\gd$ and $\F \to \F \circ \F$, respectively;
denote these by $\ee$ and $\dd$ also. Thus we have the comonad object 
\begin{equation}
\label{eq-the-comonad-object-that-deforms}
\langle\F,\ee,\dd\rangle
\end{equation}
in $D^b(X^{[n]}\times X^{[n]})$ representing the comonad 
$\LL$.

We recall here some of the results obtained in \cite{torelli} which are required for the statement of 
the deformability of the comonad object given  in Equation (\ref{eq-the-comonad-object-that-deforms})
(Theorem \ref{deformability}). 
A holomorphic symplectic compact K\"{a}hler manifold $M$ is said to be of {\em $K3^{[n]}$-type},
if it is deformation equivalent to the Hilbert scheme $X^{[n]}$ of length $n$ subschemes of a $K3$ surface $X$.
The second integral cohomology of $M$ is endowed with a non-degenerate integral symmetric bilinear pairing 
of signature $(3,20)$ called the {\em Beauville-Bogomolov-Fujiki pairing} \cite{beauville-varieties-with-zero-c-1}. 
Fix a lattice $\Lambda$ isometric to $H^2(X^{[n]},\Integers)$.
A {\em marking} for such $M$ is an isometry 
$\eta:  H^2(M,\ZZ)\to \Lambda$.
Let $\Az$ be a reflexive sheaf of Azumaya algebras on $M\times M$ that is infinitesimally rigid,
slope-stable with respect to  some K\"ahler class $\omega$ on M, and having the same numerical 
invariants as the modular  Azumaya algebra $\E nd(\E)$ above (see \cite[Section 1]{torelli}).
In particular, $c_2(\Az)$ is monodromy invariant, under the diagonal action of the monodromy group of $M$
on $H^4(M\times M,\Integers)$, and hence remains of Hodge type under every smooth K\"{a}hler deformation
of $M$.
Associated to a K\"{a}hler class $\omega$ on $M$ is a twistor deformation $\pi:\mathcal{X}\to \PP^1_\omega$, 
where $\PP^1_\omega$ is the smooth conic defined in the complex projective plane
$\PP(H^{2,0}(M)\oplus H^{0,2}(M)\oplus \ComplexNumbers\omega)$ via the Beauville-Bogomolov-Fujiki pairing.
The $\omega$-slope-stability of $\Az$ and the invariance of $c_2(\Az)$ imply that the sheaf $\Az$ 
is $\omega$-stable hyperholomorphic in the sense of Verbitsky 
\cite{kaledin-verbitski-book}, which means that it 
deforms to a reflexive sheaf of Azumaya algebras $\fAz$ over the fiber square
$\mathcal{X} \times_{\PP^1_\omega} \mathcal{X}$ of the twistor family. 
We shall be only interested in reflexive sheaves $\Az$ which 
additionally satisfy a technical condition  on their singularities spelled out 
in \cite[Condition 1.6]{torelli}, but which we do not state in full here.

\begin{conj}\cite[Conj. 1.12]{torelli}\label{vanishing}
Let $M$ be an irreducible homolorphic symplectic manifold, $\omega$ a K\"{a}hler class on $M$, and
$E$ a reflexive $\omega$-slope-stable hyperholomorphic sheaf on $M$ with an isolated singularity. 
Assume that $H^1(X,E)=0$. 
Denote by $(\mathcal{X}_t,E_t)$, $t\in\PP^1_\omega$, the twistor deformation of $(M,E)$. 
Then $H^1(\mathcal{X}_t,E_t)=0$, for all $t\in\PP^1_\omega$.
\end{conj}

The above conjecture  is a theorem of Verbitsky when $E$ is locally free \cite[Cor. 8.1]{verbitsky-1996}. 

\begin{thm}
\label{thm-torelli}
Assume that Conjecture \ref{vanishing} holds.
\begin{enumerate}
\item 
\cite[Theorem 1.8]{torelli} 
There exists a coarse moduli space $\widetilde{\fM}_\Lambda$ 
of triples $(M,\eta,\Az)$ as above which is a non-Huasdorff complex manifold. The period map 
$$\widetilde{P}: \widetilde{\fM}_\Lambda \to 
\{x\in \PP[\Lambda\otimes \CC] : (x,x)=0, (x,\overline{x})>0\}$$
given by $(X,\eta,\Az)\mapsto \eta(H^{2,0}(M))$ is a surjective local analytic isomorphism.
\item
\cite[Theorem 1.9]{torelli} The restriction of the period map to each connected component $\widetilde{\fM}_\Lambda^0$
of $\widetilde{\fM}_\Lambda$  is generically injective in the following sense.
When the Picard group of $M$ is trivial, or cyclic generated by a class of non-negative Beauville-Bogomolov-Fujiki degree, 
then a point $(M,\eta,\Az)$ of  $\widetilde{\fM}_\Lambda^0$ is
the unique point of $\widetilde{\fM}_\Lambda^0$ in the corresponding fiber of $\widetilde{P}$.
\end{enumerate}
\end{thm}

Fix a component $\widetilde{\fM}_\Lambda^0$ of $\widetilde{\fM}_\Lambda$ 
containing a point of the form $(X^{[n]}_0, \eta_0, \Az_0)$, where $X_0$ is a $K3$ surface with a trivial Picard group 
and $\Az_0$ is the modular Azumaya algebra of the Hilbert scheme $X^{[n]}_0$. The modular Azumaya algebra $\Az_0$
is $\omega$-slope-stable, with respect to every K\"{a}hler class $\omega$ on $X_0^{[n]}\times X_0^{[n]}$, by 
\cite{markman-stability}.
Denote by 
\begin{equation}
\label{dense-subset-of-modular-Azumaya-algebras}
{\mathcal Hilb} \ \ \subset \ \ \widetilde{\fM}_\Lambda^0
\end{equation}
the subset consisting of triples $(X^{[n]}, \eta, \Az)$, where $X$ is a $K3$ surface with a trivial 
Picard group, $X^{[n]}$ is its Hilbert scheme, and $\Az$ is the modular Azumaya algebra.
By a {\em Zariski open subset} of an analytic space we mean the complement of a closed analytic subset.
In view of Proposition \ref{prop-Phi-hat-is-an-equivalence}, 
the following theorem is the deformation result mentioned above. 

\begin{thm}
\label{deformability}
Assume that Conjecture \ref{vanishing} holds.
There exists a Zariski dense open subset $U\subset \widetilde{\fM}_\Lambda^0$, containing ${\mathcal Hilb}$,
a universal family $\pi:\MM\rightarrow U$ of irreducible holomorphic symplectic manifolds, and a Brauer class $\Theta$ 
of order $2n-2$ over the fiber square 
$\MM^2:=\MM\times_U \MM$ with the following properties.
\begin{enumerate}
\item \label{deformability-of-monad} 
The triple $\langle \F, \ee, \dd \rangle$, given in Equation
(\ref{eq-the-comonad-object-that-deforms}), 
deforms to a 
comonad object\footnote{
By a comonad object we mean that the convolutions $\fF\circ \fF \circ \cdots \circ \fF$ are well defined and are all objects of
$D^b(\MM^2,\Theta)$. Furthermore, the counit $\fee:\fF\rightarrow \StructureSheaf{\Delta_\MM}$ 
and comultiplication $\fdd:\fF\rightarrow \fF\circ\fF$ satisfy the axioms of a comonad. 
} 
$\overline{\LL}:=\langle \fF, \fee, \fdd \rangle$ in $D^b(\MM^2,\Theta)$. 
\item
\label{thm-item-restriction-of-Brauer-class-over-contractible-V}
Given an open subset $V$ of $U$, denote by $\MM^2_V$ the restriction of $\MM^2$ to $V$ and by $\Theta_V$  the restriction of $\Theta$ to $\MM^2_V$. Define $\MM_V$ and $\fF_V$ similarly.
When $V$ is contractible and Stein, there exists a Brauer class $\theta$ over $\MM_V$,
such that $\Theta_V=\pi_1^*\theta^{-1}\pi_2^*\theta$, so that $\fF_V$ induces an endo-functor of  
$D^b(\MM_V,\theta)$.
\item 
\label{thm-item-triangulated-structure}
The category of comodules $\fdbp{\MM_V,\theta}$
carries a 2-triangulated structure\footnote{A 2-triangulated category is an additive category satisfying
all the axioms of a Verdier triangulated category, except the octahedral axiom.
See \S\ref{subsec-the-triangulated-structure-on-the-category-of-comodules}.}
such that the forgetful functor 
$\fdbp{\MM_V,\theta} \to D^b(\MM_V,\theta)$ is 2-exact. The same holds for the category $D^b(\MM_u,\theta_u)^{\overline{\LL}_u}$
of comodules in $D^b(\MM_u,\theta_u)$ for the fiber $\MM_u$ of $\pi$ over each point $u$ in $U$. 
\item
\label{thm-item-K3-category}
$D^b(\MM_u,\theta_u)^{\overline{\LL}_u}$ is a $K3$-category in the sense that the shift $[2]$ is a Serre functor.
\item
\label{thm-item-monodromy-invariance}
The open set $U$ is large in the following sense. 
Let $M$ be an irreducible holomorphic symplectic manifold of $K3^{[n]}$-type, whose Picard group has rank $\leq 20$. 
Then there exists 
an Azumaya algebra $\Az$ over $M\times M$, such that the triple $(M,\eta,\Az)$ belongs to $U$.
\end{enumerate}
\end{thm}

Part \ref{deformability-of-monad} of the theorem is proven in section \ref{subsec-deformability-of-the-monad},
part \ref{thm-item-restriction-of-Brauer-class-over-contractible-V} in Remark 
\ref{rem-restriction-of-Brauer-class-over-contractible-V},  
part \ref{thm-item-triangulated-structure} in section 
\ref{subsec-the-triangulated-structure-on-the-category-of-comodules},
part \ref{thm-item-K3-category} in section \ref{subsection-k3-category},
and part \ref{thm-item-monodromy-invariance} in section
\ref{subsec-monodromy-invariance}.

\hide{
\begin{rem}
\label{rem-M-rather-than-M-prime}
Part (\ref{thm-item-monodromy-invariance}) of the above theorem can probably be strengthened by stating that 
$\Az$ exists over $M\times M$, avoiding the passage to a bimeromorphic model $M'$. 
The proof of the stronger version follows from the density result \cite[Cor. 4.9]{verbitsky-ergodicity}  for monodromy orbits
in $\fM^0_\Lambda$, while the current weaker version uses a weaker density result
\cite[Theorem 4.7]{verbitsky-ergodicity} for monodromy orbits in the period domain. 
While we believe that  \cite[Cor. 4.9]{verbitsky-ergodicity} follows
from \cite[Theorem 4.7]{verbitsky-ergodicity}, we could not complete the argument missing 
in \cite{verbitsky-ergodicity}.
\end{rem}
}

Let $X$ be an algebraic $K3$-surface with Picard rank less than 20. In Section \ref{Toda-comp} we show that
the deformations of $X$ constructed in the above theorem via those of the Hilbert scheme $X^{[n]}$ 
may be interpreted infinitesimally as deformations of the category of coherent sheaves  $Coh(X)$ (see
\cite{toda-deformations}).
In fact, this family of deformations is 
the maximal family of generalized (non-commutative and gerby) deformations along which ideal sheaves of length $n$ subschemes
deform as objects of $Coh(X)$. Similar statements are true of deformations coming from those of 
$M_H(v)$ 
provided the triple $y=(M_H(v), \eta, \Az)$, with $\Az$ the modular Azumaya algebra of $M_H(v)$
(of Definition \ref{def-modular}), belongs to $\widetilde{\fM}_\Lambda^0$.

There are natural homomorphisms
$$
\xymatrix{
HH^2(X) \ar[r]^{\U\circ\_\hspace{1cm}} & \Hom_{X\times M}(\U,\U[2]) & 
\ar[l]_{\hspace{1cm}\_\circ \U} HH^2(M)
}
$$
from the Hochschild cohomologies of $X$ and $M$.
The left homomorphism is an injection, while the right is an isomorphism. Inverting
the right arrow and composing defines a homomorphism:
$$
\phi^{HH}: HH^2(X)\to HH^2(M).
$$
Set $HT^2(X):=H^0(\wedge^2TX)\oplus H^1(TX)\oplus H^2(\StructureSheaf{X})$ and similarly for $M$.
Conjugating $\phi^{HH}$ with the Hochschild-Kostant-Rosenberg isomorphisms yields a map
$\phi^T: HT^2(X) \to HT^2(M).$ 
For any class $t\in HT^2(X)$, let
$Coh(X,t)$ denote the first-order deformation of the category of coherent
sheaves of $X$ in the direction $t$ (see the general construction by Toda \cite{toda-deformations}). 
Given a tangent vector $\xi$ at a point $u$ in the open subset $U$ of Theorem \ref{deformability},
denote by $\MM_\xi$ the first order deformation of the fiber $\MM_u$ over the length $2$ subscheme of $U$ corresponding to 
$\xi$. Let $\overline{\LL}_\xi$ be the restriction of the comonad data $\overline{\LL}$ to 
$\MM_\xi$. Recall that Mukai vectors of objects of $D^b(X)$ are naturally elements of $HH_0(X)$ and that the
Hochschild homology $HH_*(X)$ is an $HH^*(X)$-module \cite{cal-mukai1}. 
Let $ann(v^\vee) \subset HT^2(X)$ be the image via the HKR-isomorphism of the subspace of $HH^2(X)$ annihilating  
the dual of the Mukai vector $v$. 

\begin{thm}\label{thm-comparison}
Keep the notation of Theorem \ref{introduction-thm-A-is-direct-sum} and 
assume that the morphism $\alpha$ in Theorem \ref{introduction-thm-A-is-direct-sum} is an isomorphism.
\begin{enumerate}
\item\label{directions} 
The image $\phi^{T}(ann(v^\vee))$  is the following subspace of $HT^2(M)$:
$$ \{(\xi,\theta) \ : \ \xi\in H^1(TM), \ \theta\in H^2(\StructureSheaf{M}), \ \mbox{and} \  \ 
\xi\cdot c_1(\alpha)+(2-2n)\theta=0\}.$$
\item\label{abelian-category}  
Let $\overline{\phi}^T : HT^2(X)\to H^1(TM)$ be the composition 
of $\phi^T$ with the projection $HT^2(M) \to  H^1(TM)$. Then $\overline{\phi}^T$ restricts to $ann(v^\vee)$
as an isomorphism onto $H^1(TM)$.
Fix a class $t \in ann(v^\vee)$ and set $\xi:=\overline{\phi}^T(t)$.
The comonad category 
$D^b(\MM_\xi)^{\overline{\LL}_{\xi}}$
of Theorem \ref{deformability} (\ref{thm-item-triangulated-structure}) 
is a triangulated category equivalent to the derived category $D^b(Coh(X,t))$. 
\end{enumerate}
\end{thm}

\hide{
%
\subsection{Old introduction}

Let $X$ be a $K3$ surface with an ample line bundle $H$,
and suppose $F$ is an $H$-stable coherent sheaf on $X$. 
First order generalized deformations of $X$ are parametrized 
by the $22$-dimensional Hochschild cohomology 
\[
HH^2(X)=H^2(X,\StructureSheaf{X})\oplus H^1(X,T_X)\oplus H^2(X,\Wedge{2}{T_X}).
\]
The subspace of this vector space parametrizing
directions along which $F$ also deforms is the kernel of 
the homomorphism
\[
a_F \ : \ HH^2(X) \ \longrightarrow \ \Ext^2(F,F)
\]
given by cupping with the Atiyah class of $F$ \cite[Theorem 1.1]{lowen}.
Note that $\Ext^2(F,F)\cong \Hom(F,F)^*$ is one dimensional, by 
the stability of $F$, Serre duality, and the triviality of 
the canonical line-bundle of $X$. Thus, one expects a $21$-dimensional space 
$\mathbb{D}$ of 
generalized deformations of the $K3$ surface $X$ over which the pair $(X,F)$ 
deforms. The aim of this work is to give a construction of this family of 
deformations of $X$.

There is a very natural candidate for the space $\mathbb{D}$.
Let $v\in K(X)$ be the class of $F$ in the topological Grothendieck ring, and
$M:=M_H(v)$ the moduli space
of $H$-stable sheaves on $X$ of class $v$. 
In the rank $1$ case of ideal
sheaves of length $n$ zero dimensional subschemes, this
is nothing but the Hilbert scheme $X^{[n]}$. More generally, if
$\dim(M))=2n$, $M$ is known to be a smooth and projective
holomorphic-symplectic variety (for a suitable choice of the polarization $H$), 
deformation equivalent to the Hilbert scheme $X^{[n]}$.  
The Kuranishi deformation space $\Def(M)$ of K\"{a}hler deformations of $M$ is smooth,
by the Bogomolov-Tian-Todorov Theorem.
If $\dim(M)\geq 4$, then $\dim[Def(M)]=21$, while $\dim[Def(X)]=20$. 

Consider the universal family $\fM$ of (commutative) deformations of $M$.
\[
\xymatrix{
M \ar[d] & \subset& \fM \ar[d]_{\pi}
\\
0 & \in & \Def(M).
}
\]
We describe below a universal deformation of the derived category $D^b(X)$, as a triangulated category, over 
the $21$-dimensional space $\Def(M)$, and construct a universal object $\F$ over the $(21+2n)$-dimensional
$\fM$, such that $\F$ is the universal generalized-deformation of $F$.

\medskip

{\bf Reconstruction of $D^b(X)$ from an endo-functor of $D^b(M)$.}
By a {\em deformation of the derived category} 
$D^b(X)$ we mean a holomorphic deformation of the data, needed to define this category,
in such a way that the deformed data defines a triangulated category. 
We need to replace the data $(S,\StructureSheaf{S})$ by data which involve only the
moduli space $M$. The main ingredient is an endo-functor $L$ of $D^b(M)$, which we now define.
Let $\pi_S$ and $\pi_M$ be the projections from $S\times M$.
Assume, for simplicity, that a universal sheaf $\U$ exists over $S\times M$.
We get the {\em Fourier-Mukai functor  
$\Phi_\U:D^b(S)\rightarrow D^b(M)$ with kernel $\U$}, given by 
\[
\Phi_\U(F) \ \ := \ \ R\pi_{M_*}(\pi_S^*(F)\Lotimes \U).
\]
Denote by $\Psi_{\U}:D^b(M)\rightarrow D^b(S)$ its right adjoint 
and consider the composition
\[
L:=\Phi_\U\circ \Psi_\U:D^b(M)\rightarrow D^b(M).
\]

The Fourier-Mukai kernel $\F\in D^b(M\times M)$, of the endo-functor $L$, 
fits in an exact triangle 
$E[-1]\rightarrow \F\rightarrow \StructureSheaf{\Delta} \rightarrow E,$
where $\Delta\subset M\times M$ is the diagonal and 
$E$ is the relative extension sheaf 
\begin{equation}
\label{eq-relative-Ext-1-over-M-times-M}
E:=\SheafExt^1_{\pi_{13}}\left(\pi_{12}^*\U,\pi_{23}^*\U\right).
\end{equation}
$E$ is a reflexive sheaf of rank $\dim(M)-2$. 
The sheaf $E$ deforms to a sheaf\footnote{More precisely, the deformation is 
as a twisted sheaf, or equivalently, a deformation 
of the Azumaya algebra $\SheafEnd(E)$.} 
over the self product of every $X$ deformation equivalent 
to  $S^{[n]}$, $n:=\dim(M)/2$,
by \cite{markman-hodge}.
We show that 
the kernel $\F$ of $L$ similarly deforms to an object in the 
derived category of twisted sheaves over $\M_t\times \M_t$ of every fiber $\M_t$, $t\in Def(M)$, of the Kuranishi family.
We get a deformation of the endo-functor $L$
to endo-functors of bounded derived categories of twisted sheaves
$D^b(\M_t,\theta)$, for a natural torsion Brauer class $\theta$. We will omit the technicality involved with the Brauer class 
in this introduction.

We show that the functor $\Phi_\U$ is {\em faithful}. Using this faithfulness we prove that  
the category $D^b(S)$ can be reconstructed from the endo-functor $L$
as follows. Much of the information, encoded by the adjoint pair $\Phi_\U\dashv \Psi_\U$, 
is captured by the data 
$\LL:=(L,\epsilon,\delta)$, of the endo-functor $L$ and two natural transformation: the counit
$\epsilon:L\rightarrow id$, and the comultiplication
$\delta:L\rightarrow L^2$.   
The data $\LL$ 
defines a {\em comonad} in $D^b(M)$, in the sense that it 
satisfies the coassociativity and the left and right counit laws.
Associated to this comonad is the category 
$D^b(M)^\LL$ of coalgebras for the comonad
\cite[section VI]{working}. 
The reconstruction of $D^b(S)$ is an extension to the case of triangulated categories of 
Beck's Theorem   \cite[Section VI.7]{working},
which states that when $\Phi_\U$ is faithful, then its natural lift
$\widetilde{\Phi}_\U:D^b(S)\rightarrow D^b(M)^\LL$
is an equivalence of categories.

Denote  the Kuranishi deformation space $Def(M)$ by ${\mathfrak D}$ for brevity.
The comonad data $\LL:=(L,\epsilon,\delta)$ deforms to a universal monad
$\widetilde{\LL}:=(\widetilde{L},\tilde{\epsilon},\tilde{\delta})$ over ${\mathfrak D}$
(including deformations  
along which the $K3$ surface $S$ does not deform). In other words, we extend the Fourier-Mukai kernel 
$\F$ of the endo-functor $L$ to a sheaf 
$\overline{\F}$ over $\M\times_{\mathfrak D}\M$, and similarly extend the counit and comultiplications
as morphisms of Fourier-Mukai kernels.
We deform the triangulated structure of $D^b(S)$ to a triangulated structure on 
the comodules categories, using recent results of P. Balmer.
This is the sense, in which the category $D^b(M)^\LL$, which is equivalent to $D^b(S)$, deforms along 
every  deformation of $M$. 

{\bf The universal object:} 
The pair $(\overline{\F},\bar{\delta})$ may be viewed as the
universal object over $\M$, deforming the pair $(S,F)$, in the following sense. 
Let $F_t$ be a point  in the fiber  $\M_t$ of $\M$ over $t\in {\mathfrak D}$. 
Then the restriction of $(\overline{\F},\bar{\delta})$ to
$\M_t\times \{F_t\}$  is an object of $D^b(\M_t)^{\LL_t}$.
At a point $F\in M$ of the central fiber,
the restriction of the kernel $\F\in D^b(M\times M)$ to 
$M\times \{F\}$ is isomorphic to $\Phi_\U(F^\vee)$ (up to a shift). Restriction of the comultiplication $\delta$
provides the lift $\widetilde{\Phi}_\U(F^\vee)$ to an object of the category $D^b(M)^\LL$ of comodules.
Thus, the pair $(\widetilde{\F},\tilde{\delta})$ represents the desired  universal object.

}

%
\subsection{Notational conventions}  
\label{sec-notation-and-terminology}
We shall be working throughout over the complex numbers. 
The spaces that we deal with will be denoted by roman letters, while their deformations  
will be denoted by the corresponding calligraphic letters. For instance, if $X$ is a K3 surface, then
$\mathcal{X}$ will stand for a flat family with $X$ as its central fiber. 
Azumaya algebras will be denoted by fraktur letters, such as $\Az$.
Sheaves, and more generally, 
complexes of sheaves in the derived category will be denoted by script letters, while their 
deformations will be denoted by the same letter decorated with an over-line: for example, $\E$ 
and $\overline{\E}$. 
The same notational convention for deformations will be followed for Azumaya algebras and for 
morphisms between complexes.

Given schemes or analytic spaces $X$ and $Y$, and a morphism $f:X \to Y$, we 
denote by $f^*:D^b(Y) \to D^b(X)$ the left derived functor of the pullback 
functor from $Coh(Y)$ to $Coh(X)$. When $f$ is proper $f_*:D^b(X) \to D^b(Y)$  will 
denote the right derived functor of the direct image functor. Occasionally, we will use 
the notation $Lf^*$ and $Rf_*$ for the same functors to emphasize their derived nature.


{\bf Acknowledgments:} We thank Ivan Mirkovic for pointing out to us the relevance of the categorical
concept of monads. The work of Eyal Markman was partially supported by a grant  from the Simons Foundation (\#245840), and by NSA grant H98230-13-1-0239. Sukhendu Mehrotra was partially supported by a grant from the
Infosys Foundation, and FONDECYT grant 1150404. He thanks Daniel Huybrechts for a useful conversation,
to Emanuele Macr\`i for kindly inviting him to the Hausdorff Institute, and the institute for its hospitality.
He is grateful to Paul Balmer and Matthias K\"unzer, and especially to Moritz Groth for very useful 
email correspondence about $N$-triangulations.

\newpage

%
\section{A universal monad in $D^b(X\times X)$}

%
\subsection{The monad associated to a morphism of varieties}
\label{sec-introduction-to-monads}

The following basic construction will be used repeatedly in the proof of the main result of this section. 

\begin{cons}\label{cons-yoneda-monad} Let $f: T \to S$ be a morphism of smooth and projective 
varieties. We get the endofunctor $f_*f^*$ of $D^b(S)$, where the pullback and push forward are in the derived sense.
Denote by $u: \mathbb{1}_S \rightarrow f_*f^*$ the unit for the adjunction,
and let $\epsilon : f^*f_* \rightarrow \mathbb{1}_T$ be the counit. 
Set $\mu:= f_*\epsilon f^*: f_*f^*f_*f^* \rightarrow  f_*f^* $  and consider the 
monad ${\mathbb Y}:=(f_*f^*,u,\mu)$ in $D^b(S)$.
Any object, which is isomorphic to $f_*(\G)$ for some $\G\in D^b(T)$, admits an action 
\begin{equation}\label{action}
f_*f^*f_*\G \stackrel{f_*\ee}{\longrightarrow} f_*\G,
\end{equation}
so that the pair $(f_*\G,f_*\ee)$ is an object of the category 
$D^b(S)^{\mathbb Y}$ of modules for the monad.
We get the functor
\[
\widetilde{f_*} \ : \ D^b(T) \ \ \ \longrightarrow \ \ \ D^b(S)^{\mathbb Y},
\]
sending an object $\G$ of $D^b(T)$ to $\widetilde{f_*}(\G):=(f_*\G,f_*\epsilon)$.
What is more, under the isomorphism $f_*f^*f_*\G \to f_*\G \otimes f_*\ko_T$, this 
monadic action can be seen as an action of the algebra object $f_*\ko_T$. Indeed, recall
that the product on $f_*\ko_X$ is given by the composition
$$
f_*\ko_T\otimes f_*\ko_T \stackrel{\cong}{\longrightarrow}f_*f^*f_*f^*\ko_S
\stackrel{\mu_{\ko_S}}{\longrightarrow} f_*f^*\ko_S \cong f_*\ko_T.
$$
The desired compatability between the the product on $f_*\ko_X$ and its action on 
an object $\widetilde{f_*}(\G):=(f_*\G,f_*\epsilon)$
now follows from the axioms for a module for the monad
${\mathbb Y}$. 
\end{cons}


%
\subsection{A splitting of the monad}
Keep the notation of Theorem \ref{introduction-thm-A-is-direct-sum}.
Set $\pt:=Spec(\ComplexNumbers)$.
Let $c:M\rightarrow \pt$ be the constant map and set
$Y(\StructureSheaf{M}):=Rc_*(\StructureSheaf{M})$, as an object in $D^b(\pt)$. 
Then $Y(\StructureSheaf{M})$ is naturally isomorphic to
$\oplus_{i=0}^n\Ext^{2i}(\StructureSheaf{M},\StructureSheaf{M})[-2i]$,
thought of as the Yoneda algebra of $\StructureSheaf{M}$.

Let $u:\pi_{13_*}\pi_{13}^*\rightarrow \mathbb{1}_{X\times X}$ be the unit for the adjunction
$\pi_{13}^*\vdash \pi_{13_*}$, 
$\epsilon:\pi_{13}^*\pi_{13_*}\rightarrow \mathbb{1}_{X\times M\times X}$ the counit,
and $\mu:=\pi_{13_*}\epsilon:(\pi_{13_*}\pi_{13}^*)^2\rightarrow \pi_{13_*}\pi_{13}^*$ the multiplication 
natural transformation.
Denote by
\begin{equation}
\label{eq-Monad-YY}
\YY \ \ := \ \ (\pi_{13_*}\pi_{13}^*,u,\mu)
\end{equation}
the monad in $D^b(X\times X)$. 
We get the category $D^b(X\times X)^\YY$ of modules for the monad $\YY$ and
the functor 
\[
\widetilde{\pi_{13_*}}:D^b(X\times M \times X)\rightarrow D^b(X\times X)^\YY,
\]
as a special case of the construction in section 
\ref{sec-introduction-to-monads}.

$\A$ is the push-forward of the object 
$\widetilde{\A}:=\pi_1^*\omega_X\otimes \pi_{12}^*(\U)^\vee\otimes\pi_{23}^*(\U)[2]$
of $D^b(X\times M\times X)$ by $\pi_{13}$.
We get a natural morphism
\begin{equation}
\label{eq-multiplication-of-A-by-Y-M}
m:\A\otimes_{\ComplexNumbers}Y(\StructureSheaf{M}) \ \ \rightarrow \ \ \A,
\end{equation}
so that the object $(\A,m)$ of $D^b(X\times X)^\YY$
is the $\YY$-module corresponding to the object $\widetilde{\A}$ 
via the functor $\widetilde{\pi_{13_*}}$.
Note that the algebra structure on $\pi_{13_*}\ko_{X\times M \times X}$ is now identified with 
cup product on $H^*(\ko_M)$, or the composition product on 
$Y(\StructureSheaf{M})=\oplus_{i=0}^n\Ext^{2i}(\ko_M,\ko_M)[-2i]$. 

Let $\eta:\Delta_*\StructureSheaf{X}\rightarrow \A$ 
be the morphism corresponding to the unit of the adjunction $\Phi_\U\dashv\Psi_\U$. 
We get the composite morphism 
\begin{equation}
\label{eq-q-original}
\xymatrix{
\Delta_*\ko_X\otimes_\ComplexNumbers Y(\StructureSheaf{M}) 
\ar[r]^{\eta\otimes id} \ar@/^2pc/[rr]^{q} 
&
\A\otimes_\ComplexNumbers Y(\StructureSheaf{M}) \ar[r]^{\hspace{4ex} m} 
&
\A.
}
\end{equation}

Let $\lambda_n$ be the object 
$\oplus_{i=0}^{n-1}\Ext^{2i}(\StructureSheaf{M},\StructureSheaf{M})[-2i]$ in $D^b(\pt)$.
We have a natural morphism $\iota:\lambda_n\rightarrow Y(\StructureSheaf{M})$.
The object $Y(\StructureSheaf{M})$ is naturally the direct sum of $\lambda_n$ and 
$\Ext^{2n}(\StructureSheaf{M},\StructureSheaf{M})[-2n]$. 
Set $\alpha:=q\circ\iota:\Delta_*\ko_X\otimes_\ComplexNumbers \lambda_n\rightarrow \A$. 
So $\alpha$ is the composition
\begin{equation}
\label{eq-composition-of-three-morphisms}
\Delta_*\ko_X\otimes_\ComplexNumbers \lambda_n \LongRightArrowOf{\iota}
\Delta_*\ko_X\otimes_\ComplexNumbers Y(\StructureSheaf{M}) 
\LongRightArrowOf{\eta\otimes id} 
\A\otimes_\ComplexNumbers Y(\StructureSheaf{M}) 
\LongRightArrowOf{m}
\A.
\end{equation}
The main result of this section is 

\begin{thm}
\label{thm-A-is-direct-sum}
Let $v\in K(X)$ be a primitive class with $(v,v)=2n-2$,
$n\geq 2$, and  $H$ a $v$-generic polarization.
\begin{enumerate}
\item 
\label{thm-item-A-is-a-direct-sum-Hilbert-scheme-case}
When $v=(1,0,1-n)$, that is when $M:=M_H(v)$ is the Hilbert scheme of $n$ points on $X$,
then the morphism $\alpha$, 
displayed in equation (\ref{eq-composition-of-three-morphisms}),
is an isomorphism.
In particular, a choice of a non-zero element $t_M$ of \ 
$\Ext^2\left(\StructureSheaf{M},\StructureSheaf{M}\right)$ 
determines an isomorphism 
$$\A \cong \oplus_{i=0}^{n-1} \Delta_*\ko_X[-2i].$$
\item 
\label{thm-item-structure-sheaf-is-a-direct-summand}
In general, for $v$ arbitrary, the structure sheaf of the diagonal  $\Delta_*\ko_X$ is a 
direct summand of $\A$ in $D^b(X\times X)$. In particular, the integral transform 
$\Phi_\U : D^b(X) \to \nolinebreak D^b(M,\theta)$ is faithful.
\end{enumerate}
\end{thm}

Part (\ref{thm-item-A-is-a-direct-sum-Hilbert-scheme-case}) of the theorem is proven in Section 
\ref{sec-hilbert-scheme-case-of-thm-A-is-direct-sum}
and part (\ref{thm-item-structure-sheaf-is-a-direct-summand}) in Section 
\ref{general-moduli}.

The splitting of the monad object $\A$ in Theorem 
\ref{thm-A-is-direct-sum} (\ref{thm-item-A-is-a-direct-sum-Hilbert-scheme-case}) extends over
a Zariski open subset of the base of a family in the following sense.
Let $\pi : \X\rightarrow B$ be a smooth and proper family of $K3$ surfaces over an analytic space $B$
and $v$ a continuous primitive section of the local system $R\pi_*\Integers$ of Mukai lattices.
Let $p:\M\rightarrow B$ be a smooth and proper family of irreducible holomorphic symplectic manifolds,
such that each fiber $\M_b$ of $p$ is isomorphic to the moduli space $M_{H_b}(v_b)$ of $H_b$-stable sheaves with Mukai vector $v_b$ over the fiber $\X_b$
of $\pi$ for some polarization $H_b$ over $\X_b$. We do not assume that $H_b$ varies continuously (see for example
\cite[Prop. 5.1]{yoshioka-abelian-surface}). There exists
a twisted sheaf $\U$ over $\X\times_B\M$, flat over $B$,
such that its restriction $\U_b$ to $\X_b\times \M_b$ is a universal sheaf for the coarse moduli space 
$M_{H_b}(v_b)$, by the appendix in \cite{mukai-hodge}. 
Denote by $\A_b$ the monad object over $\X_b\times \X_b$ associated to $\U_b$.
The construction of the morphism $\alpha$, given in Equation (\ref{eq-composition-of-three-morphisms}) above,
goes through in this relative setting to yield a family of morphisms 
$
\alpha_b:\Delta_*\StructureSheaf{\X_b}\otimes_{\ComplexNumbers} \lambda_n\rightarrow \A_b,
$ 
$b\in B$, corresponding to a global morphism 
\[
\alpha:\Delta_*\StructureSheaf{\X}\otimes_{\StructureSheaf{B}} 
\left(\oplus_{i=0}^{n-1}R^{2i}p_*\StructureSheaf{\M}[-2i]\right)\rightarrow \overline{\A}
\]
of monad objects over $\X\times_B\X$. 

\begin{lem}
\label{lem-A-splits-over-an-open-subset}
The locus $B_0$ in $B$, where $\alpha_b$  is an isomorphism, is a {\em Zariski open subset} of $B$ in the analytic topology.
The open subset $B_0$ is non-empty,
whenever there exists a point $b_0\in B$, a $K3$ surface $X$, and an equivalence of derived categories 
$D^b(\X_b)\rightarrow D^b(X)$, which maps the Mukai vector $v_b$ to that of the ideal sheaf of a
length $n$ subscheme of $X$, and which maps $H_b$-stable sheaves on $\X_b$ to
ideal sheaves on $X$. 
\end{lem}

\begin{proof}
$B_0$ is the complement of the image in $B$ of the support of the object $C_\alpha$ in $D^b(\X\times_B\X)$
representing a cone of the morphism $\alpha$. The support of $C_\alpha$ is the union of the support of the 
sheaf cohomologies of $C_\alpha$. The support of $C_\alpha$ intersects the fiber $\X_b\times\X_b$ if
and only if $\alpha_b$ is not an isomorphism, since a point $x\in \X_b\times\X_b$ belongs to the support of
$C_\alpha$ if and only if one of the the cohomologies of
$C_\alpha\otimes \StructureSheaf{x}$ is non-zero, 
by \cite[Lemma 5.3]{BM}. The non-emptiness statement follows from Theorem 
\ref{thm-A-is-direct-sum} (\ref{thm-item-A-is-a-direct-sum-Hilbert-scheme-case}).
\end{proof}

A more explicit extension of the splitting Lemma \ref{lem-A-splits-over-an-open-subset} is given in 
Lemmas \ref{lemma-a-sufficient-condition-to-be-totally-split} and \ref{lemma-a-sufficient-condition-to-belong-to-U}.
%
\subsection{Splitting of the monad in the Hilbert scheme case}
\label{sec-hilbert-scheme-case-of-thm-A-is-direct-sum}
Write $X^{[n]}$ for the Hilbert scheme of $n$ points on $X$ and $S^nX$ for the
$n$-th symmetric product of $X$. Denote by $\mu$ the Hilbert-Chow morphism from 
$X^{[n]}$ to $S^nX$, and let $\pi: X^n \to S^nX$ be the quotient by $\mathfrak{S}_n$.
The following result is the main point of the proof of Theorem \ref{thm-A-is-direct-sum} for
$M=X^{[n]}$; it allows us to transport calculations from the derived category
of $X^{[n]}$ to the more combinatorial ${\mathfrak{S}_n}$-equivariant derived category 
of $X^n$. The latter is denoted by $D^b_{\mathfrak{S}_n}(X^n)$ below.

\begin{thm}(\cite{Haiman-2}, Corollary 5.1) 
\label{Deep}
Let $X$ be a smooth quasi-projective surface, and denote by $B_n$
the reduced fiber-product of $X^{[n]}$ and $X^n$ over $S^nX$:
$$\xymatrix{
 & \ar[dl]_q B_n \ar[dr]^p & \\
X^{[n]} \ar[dr]_{\mu} & & X^n \ar[dl]^{\pi}\\
 & S^nX & }$$
Then, the map $q$ is flat, and
$Rp_*q^* : D^b(X^{[n]}) \stackrel{\cong}{\longrightarrow} D^b_{\mathfrak{S}_n}(X^n)$
is an equivalence of derived categories.
\end{thm}

Let us give a quick word of explanation here. Denote by $\SHilb(X^n)$ the 
${\mathfrak{S}_n}$-{\em Hilbert scheme} of $X^n$, the
fine moduli space whose closed points parametrize the ${\mathfrak{S}_n}$ orbits in $X^n$ 
with structure sheaves isomorphic to the regular representation of ${\mathfrak{S}_n}$.
The flatness of the map $q$ above says that $B_n\subset X^n \times X^{[n]}$ is a family of such
${\mathfrak{S}_n}$ orbits parametrized by $X^{[n]}$. This yields
a morphism from  $X^{[n]}$ to $\SHilb(X^n)$, which is in fact seen to be an isomorphism, 
identifying $B_n \to X^{[n]}$ with the universal family of $\SHilb(X^n)$. 

On the other hand, given a finite group $G$ acting nicely\footnote{Nicely here means that
the  canonical bundle of $M$ is locally trivial as a $G$-sheaf.}
on a smooth projective variety $M$,
the derived McKay correspondence of \cite{BKR} relates the derived category of the $G$-Hilbert scheme $D^b({\rm Hilb}^G(M))$ and the $G$-equivariant derived category
$D^b_G(M)$: whenever the map 
${\rm Hilb}^G(M) \to M/G$ is {\em semismall}, it says that the structure sheaf of the universal family
gives a Fourier-Mukai type equivalence between these two categories. 
The map $\mu: X^{[n]} \to S^nX$ satisfies the
semismallness hypothesis. This establishes the second statement of Theorem \ref{Deep}.

The following is a special case of a vanishing theorem proved in \cite{Haiman-2}.
Let $Z_n \subset X\times B_n$ be the pullback of the universal family 
$U_n\subset X\times X^{[n]}$ of $n$ points on $X$, and $D_n\subset X\times X^n$ the union of
graphs of the $n$ projections to the $i$th factor, $\pi_i : X^n \to X$.
Consider the following diagram
$$
\xymatrix{
&& Z_n\ar[d]_{\cap} \ar[ddll] \ar[ddrr] \ar[r]^{t} & B_n \ar[ddrr]^{p}
\\
& & \ar[dl]_{id\times q = \hat{q}} X\times B_n \ar[dr]^{id \times p=\hat{p}} & \\
U_n \ar[r]^{\subset}
& X\times X^{[n]} & & X\times X^n  
& D_n \ar[l]_{\supset} \ar[r]^{s} & \hspace{2Ex} X^n \hspace{2Ex}
}
$$

\begin{thm}
[\cite{Haiman-2}, Proposition 5.1]
$R\hat{p}_* \ko_{Z_n} 
\cong \ko_{D_n}$.
\end{thm}

\begin{prop} \label{ideal}
$R\hat{p}_*\hat{q}^*(\I_{U_n}) \cong \I_{D_n}$, where $\I_{U_n}$ is the ideal sheaf of $U_n\subset X\times X^{[n]}$,and $\I_{D_n}$ that of ${D_n}\subset X\times X^n$.
\end{prop}
\begin{proof}
Applying $R\hat{p}_*\hat{q}^*$ to the sequence
$$
0 \to \I_{U_n} \to \ko_{X\times X^{[n]}} \to \ko_{U_n} \to 0,
$$
and using the previous result, we obtain the exact triangle
$$
R\hat{p}_*\I_{Z_n} \to R\hat{p}_*\ko_{X\times B_n} \to \ko_{D_n}.
$$
Thus, it suffices to show that $R\hat{p}_*\ko_{X\times B_n} \cong \ko_{X\times X^n}$, or even that
$Rp_*\ko_{B_n} \cong \ko_{X^n}$.
Now, $Z_n \stackrel{t}{\to} B_n$ is flat and finite, being the 
pullback of $U_n \to X^{[n]}$. Therefore, $t_*\ko_{Z_n}$ contains $\ko_{B_n}$ as a direct factor. 
Further, $D_n \stackrel{s}{\to} X^n$ is finite, and $s\circ \hat{p} = p\circ t$, so that
$s_* R\hat{p}\ko_{Z_n} \cong s_*\ko_{D_n} \cong Rp_*t_*\ko_{Z_n}$ is concentrated 
in degree 0. Consequently $Rp_*\ko_{B_n}$ is concentrated in degree 0 also, and because 
$p$ has connected fibers, we are done. (See also Prop. 1.3.3, \cite{Scala}.)
\end{proof}

\begin{lem} \label{fiber}
\begin{itemize}
\item[(i)]Suppose the schemes and morphisms in the commutative diagram 
$$\xymatrix{
 & \ar[dl]_q X \ar[dr]^p & \\
Y \ar[dr]_{\mu} & & Z \ar[dl]^{\pi}\\
 & W & }$$
are such that $Lq^*$ and $p^!$ are defined between bounded derived categories
(for example, if $Y,Z$ are smooth and projective, and $X$ is a closed subscheme of $Y\times Z$). 
If $\Phi:=Rp_*Lq^* : D^b(Y)  \stackrel{\cong}{\longrightarrow} D^b(Z)$ is an equivalence, given 
$\M,\N \in D^b(Y)$, we have a bifunctorial isomorphism 
$$\mu_* R\H om_{D^b(Y)}(\M,\N) \cong \pi_*R\H om_{D^b(Z)} (\Phi(\M), \Phi(\N)).$$
\item[(ii)] Suppose $G$ is a finite group, and the morphism $p: X \to Z$ in part (i) is $G$-equivariant.
Further, let $Y=X/G$, $W=Z/G$, $q$ and $\pi$ the quotient morphisms, and $\mu: X/G\to Z/G$
the morphism induced by $p$. Denote by $D^b_G(Z)$ the derived category of equivariant coherent
sheaves. Then, if $\Phi:=Rp_*Lq^* : D^b(Y)  \stackrel{\cong}{\longrightarrow} D^b_G(Z)$ is
an equivalence, there is a bifunctorial isomorphism
$$\mu_* R\H om_{D^b(Y)}(\M,\N) \cong [\pi_*R\H om_{D^b_G(Z)} (\Phi(\M), \Phi(\N))]^{G}.$$   
\end{itemize}
\end{lem}
\begin{proof}
Both parts follow from essentially the same formal calculation using Grothendieck Duality. 
We provide a proof of part (ii). Denote by $\pi_*^G$ the composition $[-]^G\circ \pi_*$; the
symbols $q_*^G$, $\mu_*^G$ are defined similarly. Then,
\begin{align*}
\pi_*^G R\H om_{D^b_G(Z)}(p_*q^*\M, p_*q^*\N) & \cong \pi_*^G p_* R\H om_{D^b_G(X)}(q^*\M, p^!p_*q^*\N)\\
 & \cong \mu_*q_*^G R\H om_{D^b_G(X)} (q^*\M, p^!p_*q^*\N)\\
 & \cong \mu_* [R\H om_{D^b(Y)}(\M,q_*p^!p_*q^*\N)]^G\\
 & \cong \mu_* R\H om_{D^b(Y)}(\M,(q_*^Gp^!)(p_*q^*)\N) \\
 & \cong \mu_* R\H om_{D^b(Y)}(\M,\N).
\end{align*}
The first isomorphism is Grothendieck Duality, the second follows from the commutativity
of the diagram above and the $G$-invariance of $\mu$, the third is adjunction, and the
fourth follows from the $G$-invariance of $\M$ (see \cite{Scala}, Prop. 1.3.2). The last 
isomorphism follows from the fact that $q_*^Gp^!$ is the right adjoint of $\Phi=p_*q^*$,
and so also its quasi-inverse.
\end{proof}

$$\xymatrix{
 & \ar[dl]_q X\times B_n\times X \ar[dr]^p & \\
\ar[ddr]_{\pi_{13}} X\times X^{[n]}\times X \ar[dr]_{\mu} & & 
\ar[ddl]^{p_{13}} X\times X^n\times X \ar[dl]^{\pi}\\
 & \ar[d]_f X\times S^nX \times X & \\
& X\times X
}$$

We shall now apply this lemma to the diagram above:
Let $p_{ij}$ stand for the projection from the product $X\times X^n \times X$ to the $(i,j)$-th 
factor. Consider the object $\A$, given in equation
(\ref{eq-kernel-A}), in the case $M=X^{[n]}$ and $\U = \I_{U_n}$.
In view of Proposition \ref{ideal}, we immediately have the isomorphism
$$\A \cong \{Rp_{13*} (p_{12}^*\I_{D_n}^\vee \otimes p_{23}^*\I_{D_n})[2]\}^{\Sn}.$$
Furthermore, the above is an isomorphism of $Y(\StructureSheaf{X^{[n]}})$-modules, under the 
natural identification of $Y(\StructureSheaf{X^{[n]}})$ with 
$Y(\StructureSheaf{X^n})^{\Sn}$.
\medskip
 
Denote by $\gd_i \subset X\times X^n$, the graph of the $i$-th projection $\pi_i:X^n \to X$,
and by $\gd_I$, for $I\subset {1,...,n}$, the partial diagonal $\cap_{i\in I}\gd_i$.  By Corollary A.4
of \cite{Scala}, there is a \v{C}ech-type $\Sn$-equivariant resolution of $\I_{D_n}$ as follows:
$$
0 \to \I_{\D_n} \to \ko_{X\times X^n} \to \oplus_{i=1}^n \ko_{\gd_i} \to \cdots \to \oplus_{|I|=k} \ko_{\gd_I} \to
\cdots \to \ko_{\gd_{\{1,...,n\}}} \to 0.
$$ 
As the diagonals $\gd_i$ intersect transversally, it is easy to see that, in fact, this resolution
is the tensor product of complexes 
$$
\otimes_{i=1}^n \{ \ko_{X\times X^n} \to \ko_{\gd_i} \},
$$
or alternatively, 
$$\I_{D_n}\cong \I_{\gd_1}\otimes \I_{\gd_2} \otimes \cdots \otimes \I_{\gd_n}$$ in
$D^b_{\Sn}(X\times X^n)$, 
where the $\I_{\gd_i}$ are the ideal sheaves of the indicated diagonals.

\begin{rem}\label{linearization} The $\Sn$-linearization of the component $\oplus_{|I|=k} \ko_{\gd_I}$ in the
complex above consists simply of permuting the factors $\ko_{\gd_I}$ according to the action 
of $\Sn$ on the indexing sets $I$. Tracing through the above calculation, it is easily seen
that the $\Sn$-linearization of 
$(\I_{\gd_1}\otimes \I_{\gd_2} \otimes \cdots \otimes \I_{\gd_n})$ is also the expected
one, namely, permutation of factors.
\end{rem}

The following calculation is due to Mukai. We present the details for the convenience of the reader
\begin{lem} 
\label{lemma-Mukai-Prop-4-10}
(\cite{mukai-hodge}, Prop. 4.10).
Let $p_{ij}$ be the $(i,j)$-th projection from $X\times X\times X$. Then, there is a
natural isomorphism 
$\B:=p_{13,*}(p_{12}^*\I_{\gd}^\vee \otimes p_{23}^*\I_{\gd}) \cong 
H^2(X,\StructureSheaf{X})\otimes_{\ComplexNumbers}\ko_{\gd}[-2]$,
where $\gd$ is the diagonal in $X\times X$. 
Equivalently, the relative extension sheaves $\SheafExt^j_{p_{13}}\left(p_{12}^*\I_{\gd},p_{23}^*\I_{\gd}\right)$
vanish, for $j\neq 2$, and 
$\SheafExt^2_{p_{13}}\left(p_{12}^*\I_{\gd},p_{23}^*\I_{\gd}\right)\cong 
H^2(X,\StructureSheaf{X})\otimes_\ComplexNumbers\StructureSheaf{\Delta}$.
\end{lem} 
\begin{proof}
The convolution of  the exact triangle
$$
\ko_{\gd}^\vee \to \ko_{X\times X} \to \I_{\gd}^\vee
$$
with $\I_{\gd}$ yields the following exact triangle on $X\times X$:
$$
p_{13*}R\H om (p_{12}^* \ko_{\gd}, p_{23}^*\I_{\gd})\to p_{13,*}p_{23}^*\I_{\gd} \to \B.
$$
By the use of flat base-change for the Cartesian diagram
$$\xymatrix{
 & \ar[dl]_{p_{12}} X\times X \times X \ar[dr]^{p_{13}} & \\
X\times X \ar[dr]_{p_1} & & X\times X \ar[dl]^{p_1}\\
& X & }$$
the second term in the triangle is isomorphic to $p_1^*p_{1*}\I_{\gd}$. 
It is then simple to conclude from the short exact sequence 
\[0\to \I_{\gd} \to \ko_{X\times X} \to \ko_{\gd}\to 0\]
that $p_{13,*}p_{23}^*\I_{\gd} \cong p_1^*p_{1*}\I_{\gd} \cong 
H^2(X,\StructureSheaf{X})\otimes_\ComplexNumbers \ko_{X\times X}[-2]$. 
The first term in the triangle is computed as follows:
\begin{align*}
p_{13*}R\H om (p_{12}^* \ko_{\gd}, p_{23}^*\I_{\gd}) &\cong  p_{13*}R\H om (\gd_{12*} \ko_{X\times X}, p_{23}^*\I_{\gd})\\
&\cong p_{13*}\gd_{12*}R\H om (\ko_{X\times X}, \gd_{12}^! \; p_{23}^*\I_{\gd})\\
& \cong  p_{13*}\gd_{12*}R\H om (\ko_{X\times X}, p_3^*\omega_X^{-1}\otimes\gd_{12}^*\; p_{23}^*\I_{\gd}[-2])\\
& \cong p_2^*\omega_X^{-1}\otimes \I_{\gd}[-2].
\end{align*}
The first isomorphism is flat base-change, while the second is Grothendieck duality. As 
$Hom(\I_{\gd}, \ko_{X\times X}) \cong \mathbb{C} $ and $\B$ is supported on the diagonal, it 
follows that $\B$ is isomorphic to $H^2(X,\StructureSheaf{X})\otimes_{\ComplexNumbers}\ko_{\gd}[-2]$. 
\end{proof}

Let us now calculate $ Rp_{13*} (p_{12}^*\I_{D_n}^\vee \otimes p_{23}^*\I_{D_n})$ 
when $n=2$; the answer for general $n$ will be apparent from this. Set $X_1 = \cdots X_4 =X$,
and for $I \subset \{1,...,4\}$, let $X_I := \times_{i\in I} X_i$. Denote by $u_I: X_{\{1,..,4\}} \to X_I$ 
the projection to the $I$-th factor; similarly,  for $I \subset J  \subset \{1,..,4\}$, $|I|=2, |J|=3$, let $v^J_I:X_J\to X_I$ be the obvious projection. 

\[
\xymatrix{
 & \ar[dl]_{u_{124}} X_1\times X_2\times X_3\times X_4 \ar[dr]^{u_{134}} \ar[dd]_{u_{14}}& \\
 X_1\times X_2\times X_4 \ar[dr]_{v^{124}_{14}} & & X_1\times X_3\times X_4 \ar[dl]^{v^{134}_{14}}\\
 &  X_1\times X_4 & 
}
\]

We then have
\begin{align*}
Rp_{13*} (p_{12}^*\I_{D_2}^\vee \otimes p_{23}^*\I_{D_2}) & \cong 
u_{14*} \{u_{12}^*\I_{\gd}^\vee \otimes  u_{13}^*\I_{\gd}^\vee  \otimes u_{24}^*\I_\gd \otimes u_{34}^*\I_\gd \} \\ 
& \cong v^{134}_{14*}u_{134*} \{(u_{134}^*v^{134*}_{13}\I_{\gd}^\vee \otimes  u_{134}^*v^{134*}_{34}\I_\gd ) \otimes (u_{12}^*\I_{\gd}^\vee  \otimes  u_{24}^*\I_\gd)\} \\
& \cong v^{134}_{14*}\{(v^{134*}_{13}\I_{\gd}^\vee \otimes v^{134*}_{34}\I_\gd ) \otimes 
u_{134*}( u_{12}^*\I_{\gd}^\vee  \otimes  u_{24}^*\I_\gd)\} \\
& \cong v^{134}_{14*}\{(v^{134*}_{13}\I_{\gd}^\vee \otimes v^{134*}_{34}\I_\gd ) \otimes 
u_{134*}u_{124}^*( v_{12}^{124*}\I_{\gd}^\vee  \otimes  v_{24}^{124*}\I_\gd)\} \\
& \cong v^{134}_{14*}\{(v^{134*}_{13}\I_{\gd}^\vee \otimes v^{134*}_{34}\I_\gd ) \otimes 
v^{134*}_{14}v^{124}_{14*}( v_{12}^{124*}\I_{\gd}^\vee  \otimes  v_{24}^{124*}\I_\gd)\} \\
& \cong v^{134}_{14*}(v^{134*}_{13}\I_{\gd}^\vee \otimes v^{134*}_{34}\I_\gd ) \otimes 
v^{124}_{14*}( v_{12}^{124*}\I_{\gd}^\vee  \otimes  v_{24}^{124*}\I_\gd)\\
& \cong \left(H^2(X,\StructureSheaf{X})\otimes_\ComplexNumbers\ko_{\gd}[-2]\right) 
\otimes \left(H^2(X,\StructureSheaf{X})\otimes_\ComplexNumbers\ko_{\gd}[-2]\right),
\end{align*}
where the third isomorphism is the projection formula, the fifth is flat base-change and the 
last one is the isomorphism of the previous lemma. Clearly, the same working proves the

\begin{prop}
$Rp_{13*} (p_{12}^*\I_{D_n}^\vee \otimes p_{23}^*\I_{D_n}) \cong 
\left(H^2(X,\StructureSheaf{X})\otimes_\ComplexNumbers\ko_{\gd}[-2]\right)^{n\otimes}$
as objects in the category $D^b_{\Sn}(X\times X)$ (where $\Sn$-acts trivially on $X\times X$). 
Moreover, the $\Sn$-linearization of  the tensor product on the
right hand side is simply permutation of factors.
\end{prop}
\begin{proof}
The linearization is clear from the calculation above and Remark \ref{linearization}.
\end{proof}
 
\begin{lem}
\label{lemma-triviality-of-an-object}
Let $S$ be a scheme and $F$ an object in $D^b(S)$.
Assume that the cohomology sheaves $\fH^j(F)$ satisfy the condition:
$\Ext^{k+1}(\fH^j(F),\fH^{j-k}(F))  =  0$,
$\forall j$ and $\forall k > 0$.
Then $F$ is isomorphic to $\oplus_{j}\fH^j(F)[-j].$
\end{lem}

\begin{proof}
Set $J_F:=\{j \ : \ \fH^j(F)\neq 0\}$. 
The proof is by induction on the cardinality 
$\sharp(J_F)$ of $J_F$. Set $b:=\max(J_F)$ and $a:=\min(J_F)$.
Then $F$ is represented by a complex of coherent 
$\StructureSheaf{S}$-modules 
\[
F_a\rightarrow F_{a+1} \rightarrow \dots \rightarrow F_b,
\]
with $F_i$ in degree $i$.
If $\sharp(J_F)=1$, then $a=b$ and the statement holds trivially.
Assume that the statement holds for every object $F'$
satisfying $\sharp(J_{F'})< n$, where $n:=\sharp(J_F)$. 
Set $A:=\fH^a(F)$.
Let $F'$ be the cone of the natural morphism $\fH^a(F)[-a]\rightarrow F$. 
Then $\sharp(J_{F'})=n-1$, $\fH^a(F')=0$, and $\fH^i(F')=\fH^i(F)$,
for $i\neq a$. The equivalence
$F'=\oplus_{j}\fH^j(F')[-j]$ follows, by the induction hypothesis. 
We get the exact triangle
\[
A[-a]\rightarrow F \rightarrow F' \RightArrowOf{\delta} A[1-a].
\]
The morphism $\delta$ decomposes as the sum of the morphisms
\[
\delta_j\in \Hom(\fH^j(F')[-j],A[1-a])=\Ext^{1+j-a}(\fH^j(F),\fH^a(F)),
\]
for $j> a$.
These groups vanish, by assumption. Hence, $\delta=0$, and $F$ is 
isomorphic to the direct sum of $A[-a]$ and $F'$.
\end{proof}

\noindent {\it Proof of Theorem \ref{thm-A-is-direct-sum} part
\ref{thm-item-A-is-a-direct-sum-Hilbert-scheme-case}.}

\underline{Step 1:}
To compute the $\Sn$-invariants in $Rp_{13*} (p_{12}^*\I_{D_n}^\vee \otimes p_{23}^*\I_{D_n})$, 
we first calculate the $\Sn$-invariant parts of the the $n$-fold multi-tors
$Tor^n_q(\ko_{\gd}):=Tor_q(\ko_{\gd}, \ko_{\gd}, \cdots ,\ko_{\gd})$. This is precisely the 
sort of calculation that is carried out in Corollary B.7 in \cite{Scala}. 
In our particular situation this says 
that $Tor^n_q(\ko_{\gd})$ has nonzero $\Sn$-invariants if and only if $q=2h$, $0\leq h \leq n-1$,
in which case 
$$
Tor^n_{2h}(\ko_{\gd})^{\Sn} \simeq (\Wedge{2}N^*_{\Delta/X\times X})^{\otimes h}\simeq
H^0(X,\omega_X)^{\otimes h}\otimes_\ComplexNumbers \ko_{\gd}.
$$
Equivalently, the relative extension sheaves 
$\SheafExt^j(\pi_{12}^*\I_U,\pi_{23}^*\I_U)$ are isomorphic to $\Delta_*\StructureSheaf{X}$,
for $j$ even in the range $2\leq j \leq 2n$, and vanish for all other values of $j$.

The hypotheses of Lemma \ref{lemma-triviality-of-an-object},
applied to the object $R\pi_{13*} (\pi_{12}^*\I_U^\vee \otimes \pi_{23}^*\I_U)$, are satisfied, 
since the odd cohomologies of  this object
vanish,  and 
$\Ext^{i}_{X\times X}(\ko_{\gd},\ko_{\gd})  =  0$ for odd $i$. 
Lemma \ref{lemma-triviality-of-an-object} 
implies the existence of an isomorphism
\begin{equation}
\label{eq-object-isomorphic-to-direct-sum-of-relative-ext-sheaves}
R\pi_{13*} (\pi_{12}^*\I_U^\vee \otimes \pi_{23}^*\I_U)\cong
\oplus_{i=1}^nH^2(X,\StructureSheaf{X})^{\otimes i}\otimes_\ComplexNumbers
\Delta_*\StructureSheaf{X}[-2i].
\end{equation}
Note that the $i$-th summand on the right hand side is naturally isomorphic to 
the sheaf cohomology of degree $2i$ of the left hand side,
namely the relative extension sheaf $\SheafExt^{2i}_{\pi_{13}}\left(\pi_{12}^*\I_U, \pi_{23}^*\I_U\right)$.
The isomorphism
 $\A\cong \oplus_{i=0}^{n-1} H^2(X,\StructureSheaf{X})^{\otimes i}\otimes_\ComplexNumbers
\Delta_*\StructureSheaf{X}[-2i]$
follows immediately from Equation
(\ref{eq-object-isomorphic-to-direct-sum-of-relative-ext-sheaves}).
In particular, the functor $\Phi_\U$ is faithful.

\underline{Step 2:} We prove in this step the following claim. 
\begin{claim}
\label{claim-exact-triangle-g-e-f}
For $i$ even, we have the canonical short exact sequence
\[
0\rightarrow \H^i(\A) \RightArrowOf{e^i} 
\Delta_{X_*}\left[\omega_X\otimes \SheafExt^{i+2}_{\pi_X}\left(\I_U,\I_U\right)\right]
\RightArrowOf{f^i}
\H^{i+2}(\A)\otimes\pi_1^*\omega_X
\rightarrow 0.
\]
For $i$ odd, we have the following natural isomorphism:
\[
\Delta_{X_*}\left[
\omega_X\otimes\SheafExt^{i+2}_{\pi_{X_*}}(\I_U,\I_U)
\right]
\LongRightArrowOf{f^i}
\H^{i+1}(\A)\otimes \pi_1^*(T^*X).
\]
\end{claim}
\begin{proof}
Let $\Delta_{13}:X\times M \rightarrow X\times M \times X$ be the diagonal embedding.
Consider the short exact sequence:
\[
0\rightarrow \pi_{13}^*\I_{\Delta_X}\rightarrow \StructureSheaf{X\times M\times X}\rightarrow 
\Delta_{13_*}\StructureSheaf{X\times M}\rightarrow 0.
\]

Tensoring with $\widetilde{\A}:=\pi_1^*\omega_X[2]\otimes \pi_{12}^*\I_U^\vee\otimes \pi_{23}^*\I_U$
we get the exact triangle:
\begin{equation}
\label{eq-exact-triangle-g-f-e}
\widetilde{\A}\otimes \pi_{13}^*\I_{\Delta_X} \RightArrowOf{\tilde{g}}
\widetilde{\A} \RightArrowOf{\tilde{e}} 
\pi_1^*\omega_X[2]\otimes \Delta_{13_*}\left(\I_U^\vee\otimes\I_U\right) \RightArrowOf{\tilde{f}} 
\widetilde{\A}\otimes \pi_{13}^*\I_{\Delta_X}[1]
\end{equation}
Applying the derived push-forward $\pi_{13_*}$ to the exact triangle 
(\ref{eq-exact-triangle-g-f-e}) we get the following exact triangle in $D^b(X\times X)$.
\begin{equation}\label{exact-triangle-20}
\A\otimes\I_{\Delta_X} \RightArrowOf{g} \A \RightArrowOf{e} 
\Delta_{X_*}\left[\omega_X[2]\otimes \pi_{X_*}\left(\I_U^\vee\otimes \I_U\right)
\right] \RightArrowOf{f} \A\otimes\I_{\Delta_X}[1].
\end{equation}
Let $\iota:\I_{\Delta_X}\rightarrow \StructureSheaf{X\times X}$ be the natural inclusion.
The object $\A$ has been shown to be a direct sum of sheaves supported on the diagonal,
and the morphism $g$ is the composition
$\A\otimes \I_{\Delta_X}\RightArrowOf{\id_\A\otimes \iota}\A\otimes\StructureSheaf{X\times X}\cong \A$.
Hence, the morphism $g$ vanishes and we get a short exact sequence, for every integer $i$.
\[
0\rightarrow \H^i(\A) \RightArrowOf{e^i} 
\Delta_{X_*}\H^i\left[\omega_X[2]\otimes \pi_{X_*}\left(\I_U^\vee\otimes \I_U\right)\right] 
\RightArrowOf{f^i} \H^{i+1}\left(\A\otimes\I_{\Delta_X}\right)\rightarrow 0.
\]

It remains to calculate the sheaf cohomologies $\H^{i+1}\left(\A\otimes\I_{\Delta_X}[1]\right)$.
Note first the isomorphisms 
\begin{eqnarray*}
\Tor_0^{X\times X}(\StructureSheaf{\Delta_X},\I_{\Delta_X}) & \cong & \Delta_{X_*}T^*X,
\\
\Tor_{-1}^{X\times X}(\StructureSheaf{\Delta_X},\I_{\Delta_X})
& \cong & \Delta_{X_*}\omega_X,
\end{eqnarray*}
and $\Tor_i^{X\times X}(\StructureSheaf{\Delta_X},\I_{\Delta_X})=0,$
if $i$ is not equal to $0$ or $-1$. We have already shown that 
$\A$ is isomorphic to $\oplus_{i=0}^{n-1}\H^{2i}(\A)$.
For even $i$ we get the isomorphism 
\[
\H^{i+1}(\A\otimes\I_{\Delta_X}) \ \ \cong \ \ 
\Tor_{-1}^{X\times X}\left(\H^{i+2}(\A),\I_{\Delta_X}\right)\ \ \cong \ \ 
\H^{i+2}(\A)\otimes \pi_1^*\omega_X.
\]
For odd $i$ we get:
\[
\H^{i+1}(\A\otimes\I_{\Delta_X}) \ \ \cong \ \ 
\Tor_{0}^{X\times X}\left(\H^{i+1}(\A),\I_{\Delta_X}\right)\ \ \cong \ \ 
\H^{i+1}(\A)\otimes \pi_1^*(T^*X).
\]
\end{proof}

\underline{Step 3:}  We have the isomorphisms
$\H^{2i}(\A)\cong \pi_1^*\omega_X\otimes \SheafExt^{2i+2}_{\pi_{13}}(\pi_{12}^*\I_U,\pi_{23}^*\I_U)$, by definition of $\A$.
The homomorphisms 
\[
e^{j-2} \ : \ \SheafExt^{j}_{\pi_{13}}\left(\pi_{12}^*\I_U, \pi_{23}^*\I_U\right) \ \ \rightarrow \ \ 
\Delta_*\left(\SheafExt^j_{\pi_X}(\I_U,\I_U)\right)
\]
are injective, for all $j$, and $e^{2n-2}$ is an isomorphism, by Claim 
\ref{claim-exact-triangle-g-e-f}.
Let $t_M$ be a non-zero element of $H^2(M,\StructureSheaf{M})$, where $M=X^{[n]}$. 
Then $t_M$ induces a morphism from $\pi_{13_*}\left(\pi_{12}^*\I_U^\vee\otimes\pi_{23}^*\I_U\right)$
to $\pi_{13_*}\left(\pi_{12}^*\I_U^\vee\otimes\pi_{23}^*\I_U\right)[2]$. We get an induces homomorphism
of relative extension sheaves
\[
t_M:\SheafExt^j_{\pi_{13}}\left(\pi_{12}^*\I_U,\pi_{23}^*\I_U\right) 
\rightarrow \SheafExt^{j+2}_{\pi_{13}}\left(\pi_{12}^*\I_U,\pi_{23}^*\I_U\right),
\]
via the morphism (\ref{eq-multiplication-of-A-by-Y-M}).
We prove in this step the following statement. 

\begin{claim}
\label{claim-t-M-to-the-power-n-1-is-an-isomorphism}
The sheaf homomorphism 
\begin{equation}
\label{eq-t-M-to-the-power-n-1-is-an-isomorphism}
t_M^{i} \ : \ \SheafExt^2_{\pi_{13_*}}\left(\pi_{12}^*\I_U,\pi_{23}^*\I_U\right)
\ \ \ \rightarrow \ \ \ 
\SheafExt^{2i+2}_{\pi_{13_*}}\left(\pi_{12}^*\I_U,\pi_{23}^*\I_U\right)
\end{equation}
is an isomorphism, for $0\leq i \leq n-1$.
\end{claim}

\begin{proof}
We have the homomorphism 
\[
\mu_{\I_U}(t_M): \SheafExt^j_{\pi_{X_*}}\left(\I_U,\I_U\right)\rightarrow \SheafExt^{j+2}_{\pi_{X_*}}\left(\I_U,\I_U\right).
\]
The composition 
$
\StructureSheaf{X\times M}\LongRightArrowOf{\mu_{\I_U}}
\I_U^\vee\otimes \I_U \LongRightArrowOf{\tr}
\StructureSheaf{X\times M}
$
is the identity, since $\rank(\I_U)=1$. 
Thus, the composition 
\[
R^0_{\pi_{X_*}}\StructureSheaf{X\times M} \LongRightArrowOf{\mu_{\I_U}}
\SheafExt^0_{\pi_X}\left(\I_U,\I_U\right) \LongRightArrowOf{\mu_{\I_U}(t_M)^n}
\SheafExt^{2n}_{\pi_X}\left(\I_U,\I_U\right)
\LongRightArrowOf{\tr} R^{2n}_{\pi_{X_*}}\StructureSheaf{X\times M}
\]
is multiplication by $t_M^n$.
Hence, the following is an isomorphism
\[
\mu_{\I_U}(t_M)^n \ : \ \SheafExt^0_{\pi_{X_*}}\left(\I_U,\I_U\right) \ \ \ \longrightarrow \ \ \ 
\SheafExt^{2n}_{\pi_{X_*}}\left(\I_U,\I_U\right).
\]

The exact triangle (\ref{exact-triangle-20}) is the image of one in $D^b(X\times M\times X)$
and lifts to an exact triangle in the category $D^b(X\times X)^\YY$ for the monad $\YY$
given in (\ref{eq-Monad-YY}). Taking cohomology we get 
the following commutative diagram: 
\[
\xymatrix{
\hspace{3ex}
\SheafExt^{2i}_{\pi_{13_*}}\left(\pi_{12}^*\I_U,\pi_{23}^*\I_U\right)
\hspace{3ex}
\ar[r]^{t_M} \ar[d]_{e^{2i-2}}
&
\hspace{3ex}
\SheafExt^{2i+2}_{\pi_{13_*}}\left(\pi_{12}^*\I_U,\pi_{23}^*\I_U\right)
\hspace{3ex}
\ar[d]^{e^{2i}}
\\
\hspace{4ex}  \Delta_{X_*}\left[\SheafExt^{2i}_{\pi_X}(\I_U,\I_U)\right] \hspace{4ex} 
\ar[r]^{\Delta_*(\mu_{\I_U}(t_M))}
\ar[d]_{f^{2i-2}}
&
\hspace{4ex} 
\Delta_{X_*}\left[\SheafExt^{2i+2}_{\pi_X}(\I_U,\I_U)\right]
\hspace{4ex} 
\ar[d]^{f^{2i}}
\\
\hspace{4ex} 
\pi_1^*\omega_X\otimes \SheafExt^{2i+2}_{\pi_{13_*}}\left(\pi_{12}^*\I_U,\pi_{23}^*\I_U\right)
\hspace{4ex} 
\ar[r]^{\pi_1^*1_{\omega_X}\otimes t_M}
&
\hspace{4ex} 
\pi_1^*\omega_X\otimes \SheafExt^{2i+4}_{\pi_{13_*}}\left(\pi_{12}^*\I_U,\pi_{23}^*\I_U\right).
\hspace{4ex} 
}
\]
We are ready to prove that the homomorphism (\ref{eq-t-M-to-the-power-n-1-is-an-isomorphism})
is an isomorphism. The proof is by contradiction. Assume otherwise and 
let $i$ be the minimal integer in the range $1\leq i \leq n-1$, such that 
\begin{equation}
\label{eq-t-M-power-i-assumed-to-vanish}
t_M^{i} \ : \ \SheafExt^2_{\pi_{13_*}}\left(\pi_{12}^*\I_U,\pi_{23}^*\I_U\right)
\ \ \ \rightarrow \ \ \ 
\SheafExt^{2i+2}_{\pi_{13_*}}\left(\pi_{12}^*\I_U,\pi_{23}^*\I_U\right)
\end{equation}
is not an isomorphism. Then the above homomorphism must vanish. 
Hence,
\begin{equation}
\label{eq-t-M-power-i-tensor-1-vanishes}
t_M^i\otimes \pi_1^*1_{\omega_X} \ : \ 
\SheafExt^2_{\pi_{13_*}}\left(\pi_{12}^*\I_U,\pi_{23}^*\I_U\right) \otimes\pi_1^*\omega_X
\ \ \ \rightarrow \ \ \ 
\SheafExt^{2i+2}_{\pi_{13_*}}\left(\pi_{12}^*\I_U,\pi_{23}^*\I_U\right) \otimes\pi_1^*\omega_X
\end{equation}
vanishes as well. Consider the following (abreviated) commutative diagram with short exact columns.
\[
\xymatrix{
0 \ar[r] \ar[d]_{e^{-2}} 
& 
\SheafExt^2_{\pi_{13}} \ar[r]^{t_M}\ar[d]_{e^0}
&
\cdots \ar[r]
&
\SheafExt^{2i}_{\pi_{13}} \ar[r]^{t_M}\ar[d]^{e^{2i-2}}
& \SheafExt^{2i+2}_{\pi_{13}}\ar[d]^{e^{2i}}
\\
\Delta_{X_*}\SheafExt^0_{\pi_X}\ar[r]^{\mu_{\I_U}(t_M)}\ar[d]_{f^{-2}}^{\cong}
&
\Delta_{X_*}\SheafExt^2_{\pi_X}\ar[r]\ar[d]_{f^{0}}
&
\cdots\ar[r]
&
\Delta_{X_*}\SheafExt^{2i}_{\pi_X}\ar[r]^{\mu_{\I_U}(t_M)}\ar[d]^{f^{2i-2}}
&
\Delta_{X_*}\SheafExt^{2i+2}_{\pi_X}
\\
\hspace{3ex} 
\SheafExt^2_{\pi_{13}}\otimes\pi_1^*\omega_X 
\hspace{3ex} 
\ar[r]^{t_M\otimes\pi_1^*1_{\omega_X}}
&
\SheafExt^4_{\pi_{13}}\otimes\pi_1^*\omega_X \ar[r]
&
\cdots\ar[r]
&
\SheafExt^{2i+2}_{\pi_{13}}\otimes\pi_1^*\omega_X&&
}
\]
The vanishing of (\ref{eq-t-M-power-i-tensor-1-vanishes}), the fact that $f^{-2}$ is an isomorphism,
and the exactness of the columns in the above diagram, combine to
imply that the image of 
\[
\Delta_{X_*}\left(\mu_{\I_U}(t_M)\right)^i \  : \ 
\Delta_{X_*}\SheafExt^0_{\pi_X}\left(\I_U,\I_U\right)
\rightarrow 
\Delta_{X_*}\SheafExt^{2i}_{\pi_X}\left(\I_U,\I_U\right)
\]
is contained in the image of the injective homomorphism
\[
e^{2i-2}:\SheafExt^{2i}_{\pi_{13}}\left(\pi_{12}^*\I_U,\pi_{23}^*\I_U\right)\rightarrow 
\Delta_{X_*}\SheafExt^{2i}_{\pi_X}\left(\I_U,\I_U\right).
\]
Hence, the injective homomorphism
\[
\Delta_{X_*}\left(\mu_{\I_U}(t_M)\right)^{i+1}:
\Delta_{X_*}\SheafExt^0_{\pi_X}\left(\I_U,\I_U\right)
\rightarrow 
\Delta_{X_*}\SheafExt^{2i+2}_{\pi_X}\left(\I_U,\I_U\right)
\]
factors through 
\[
t_M:\SheafExt^{2i}_{\pi_{13}}\left(\pi_{12}^*\I_U,\pi_{23}^*\I_U\right)\rightarrow
\SheafExt^{2i+2}_{\pi_{13}}\left(\pi_{12}^*\I_U,\pi_{23}^*\I_U\right).
\]
In particular, the above homomorphism does not vanish. Hence, the above homomorphism 
is an isomorphism. The minimality of $i$ implies that 
\[
t_M^{i-1}:\SheafExt^{2}_{\pi_{13}}\left(\pi_{12}^*\I_U,\pi_{23}^*\I_U\right)
\rightarrow 
\SheafExt^{2i}_{\pi_{13}}\left(\pi_{12}^*\I_U,\pi_{23}^*\I_U\right)
\]
is an isomorphism. 
It follows that the homomorphism (\ref{eq-t-M-power-i-assumed-to-vanish})
is an isomorphism. A contradiction. 
This complete the proof of Claim
\ref{claim-t-M-to-the-power-n-1-is-an-isomorphism}.
\end{proof}
\hide{
\underline{Step 4:}  (???) Do we need the following (???)
\[
\xymatrix{
& 
\SheafExt^{2}_{\pi_{13}}\left(\pi_{12}^*\I_U, \pi_{23}^*\I_U\right)
\ar[r]^{t_M^{n-1}}
\ar[d]_{e^0} &
\SheafExt^{2n}_{\pi_{13}}\left(\pi_{12}^*\I_U, \pi_{23}^*\I_U\right)
\ar[d]_{e^{2n-2}}^{\cong}
\\
\Delta_*\left(\SheafExt^2_{p_1}(\ko_\Delta,\ko_\Delta)\right)
\ar[r]_{\Delta_*(\Phi_\U)} &
\hspace{4ex} \Delta_*\left(\SheafExt^{2}_{\pi_X}(\I_U,\I_U)\right) \hspace{4ex}
\ar[r]_{\Delta_*(\mu_{\I_U}(t_M)^{n-1})} &
\hspace{4ex} \Delta_*\left(\SheafExt^{2n}_{\pi_X}(\I_U,\I_U)\right). \hspace{4ex}
}
\]
We have already proven that the integral transform $\Phi_\U$ is faithful. 
Hence, the homomorphism $\Delta_*(\Phi_\U)$ is injective. 
Considering $\ko_\Delta$ as a family of structure sheaves of length one subschemes of $X$,
we see that the composition  $\Delta_*(t_M^{n-1})\circ \Delta_*(\Phi_\U)$ vanishes, by Lemma 
\ref{lemma-equality-in-Ext-2n}
(??? change the order of the presentation avoiding forward references).
The homomorphism $e^0$ is injective, by Claim \ref{claim-exact-triangle-g-e-f}.
}

\underline{Step 4:}  
Let $\iota_2:\SheafExt^{2}_{\pi_{13}}\left(\pi_{12}^*\I_U, \pi_{23}^*\I_U\right)\rightarrow
R\pi_{13}\left(\pi_{12}^*\I_U^\vee\otimes \pi_{23}^*\I_U\right)$ be the natural morphism
from the first non-vanishing sheaf cohomology of the object to the object itself.
Let 
\[
m:R\pi_{13}\left(\pi_{12}^*\I_U^\vee\otimes \pi_{23}^*\I_U\right)\otimes_\ComplexNumbers Y(\StructureSheaf{M})
\rightarrow R\pi_{13}\left(\pi_{12}^*\I_U^\vee\otimes \pi_{23}^*\I_U\right)
\] 
be the natural morphism.
The morphism of $Y(\StructureSheaf{M})$-modules
\[
m\circ (\iota_2\otimes id) : 
\SheafExt^{2}_{\pi_{13}}\left(\pi_{12}^*\I_U, \pi_{23}^*\I_U\right)\otimes_\ComplexNumbers Y(\StructureSheaf{M})
\rightarrow
R\pi_{13}\left(\pi_{12}^*\I_U^\vee\otimes \pi_{23}^*\I_U\right)
\]
induces an isomorphism of the sheaf cohomologies for degrees between $2$ and $2n$,
by Claim \ref{claim-t-M-to-the-power-n-1-is-an-isomorphism}.
Thus, the composite morphism $\alpha$ 
in equation (\ref{eq-composition-of-three-morphisms}) induces an isomorphism of all
sheaf cohomologies, i.e., for degrees in the range from $0$ to $2n-2$. 
\qed

%
\subsection{Hochschild (co)homology}
\label{HH-coho}

We present a brief review here of some concepts and definitions on this topic which will be 
used in the proof of part \ref{thm-item-structure-sheaf-is-a-direct-summand} 
of Theorem \ref{thm-A-is-direct-sum}, in the proof of Proposition \ref{prop-a-topological-formula-for-a-trace},
and also later
in Section \ref{Toda-comp}. We follow the presentation in \cite{cal-mukai1,cal-wil-mukai1}.
Those familiar with the Hochschild (co)homology of varieties 
should skip to Section \ref{general-moduli}.

Let $T$ be a smooth, projective variety over $\CC$. Write 
${\gd_T}: T\to T\times T$ for the diagonal, and $\pi_i:T\times T \to T, \; i=1,2$, 
for the projections; denote by $S_{\gd_T}$ the kernel of the Serre functor 
$\gd_{T,*}\go_T[\dim T]$, and by $S_{\gd_T}^{-1}$ the object 
$\gd_{T,*}\go_T^{-1}[-\dim T]$.
Let $\gd_{T,!}:D^b(T)\rightarrow D^b(T\times T)$ be the left adjoint of
$\gd_T^*$. 

\begin{defi} (i) The {\em Hochschild structure} of $T$ consists of a graded  
ring $HH^*(T)$,
$$HH^i(T):=\Hom_{T\times T}(\ko_{\gd_T},\ko_{\gd_T}[i]),$$
and a graded left module $HH_*(T)$ over $HH^*(T)$,
$$HH_i(T):=\Hom_{T\times T}(\gd_{T,!}\ko_{T}[i],\ko_{\gd_T})=
\Hom_{T\times T}(S_{\gd_T}^{-1}[i],\ko_{\gd_T}).$$ 
Both the ring and module structures are defined by Yoneda composition in $D^b(T\times T)$.

\noindent (ii) The {\em harmonic structure} of $T$ consists of a graded ring $HT^*(T)$,
$$
HT^i(T):= \bigoplus_{p+q=i}H^p(T, \wedge^q \T_T),
$$
and a graded module $H\Go_*(T)$ over $HT^*(T)$,
$$H\Go_i(T):= \bigoplus_{q-p=i}H^p(T, \Go_T^q).$$
Exterior product and contraction define the ring and module structures, respectively.
\end{defi}

These two structures are related by the Hochschild-Kostant-Rosenberg isomorphism
\begin{equation}\label{HKR-isom}
I_T: \gd_T^*\ko_{\gd_T} \stackrel{\sim}{\longrightarrow}\bigoplus_i\Go_T^i[i], 
\end{equation}
which yields isomorphisms of graded vector spaces:
\begin{align*}
&I^*: HT^*(T)\stackrel{\sim}{\longrightarrow}HH^*(T),\\
&I_*: HH_*(T)\stackrel{\sim}{\longrightarrow}H\Go_*(T).
\end{align*}
We shall also have occasion to work with the following modified isomorphisms 
\begin{equation}
\label{eq-normalized-HKR-isomorphisms}
\widetilde{I_T}^*=I^*\circ (\frac{1}{\sqrt{td_T}}\lrcorner(\_)),
\mbox{\;\;\;and\;\;\;} \widetilde{I^T}_*=(\sqrt{td_T}\wedge(\_))\circ I_*.
\end{equation}

Hochschild homology behaves functorially under {\em integral transforms}: 
associated to an integral functor $\Phi: D^b(T) \to D^b(Y)$, there is a 
natural map of graded vector spaces
\begin{equation}
\label{eq-push-forward-on-Hochschild-homology}
\Phi_*:HH_*(T) \to HH_*(Y),
\end{equation}
which is constructed as follows. Let $\F$ be the kernel of $\Phi$, $\R$ the
kernel of its right adjoint, and $\LB$ the kernel of its left adjoint. 
Denote by $p_{ij}$ the projection from $T\times Y \times T$ to its
$ij$-th factor. Any object $\Q\in D^b(Y\times T)$ gives rise to a functor 
\begin{eqnarray}
\label{defi-convolution}
\Q\circ\_  : & D^b(T\times Y) &  \to \;\;\;\;\;  D^b(T \times T) \\
  & \T & \mapsto \; p_{13,*}(p_{12}^*(\T)\otimes  p_{23}^*(\Q)) \nonumber
\end{eqnarray}
via convolution; define $\_\circ \Q:D^b(Y\times T)\rightarrow D^b(Y\times Y)$ similarly.   
We first define a natural map 
$$\Phi^{\dagger}: \Hom_{Y\times Y}(\ko_{\gd_Y}, S_{\gd_Y}[i])
\to \Hom_{T\times T}(\ko_{\gd_T}, S_{\gd_T}[i]).$$
Given a morphism $\nu: \ko_{\gd_Y} \to S_{\gd_Y}[i]$, let
$\Phi^{\dagger}\nu$ be the composite morphism
\begin{equation}\label{HH-functoriality}
\ko_{\gd_T}\stackrel{\eta}{\to} \R\circ\F=\R\circ\ko_{\gd_Y}\circ\F \stackrel{\nu}{\to}
\R\circ S_{\gd_Y}[i]\circ\F= S_{\gd_T}[i]\circ\LB \circ \F \stackrel{\epsilon}{\to}S_{\gd_T}[i],
\end{equation}
where $\eta$ and $\epsilon$ are the natural unit and the counit, respectively. Notice
that we have the isomorphism 
\begin{equation}\label{HH-dual}
\Hom_{T\times T}(\ko_{\gd_T}, S_{\gd_T}[i])\cong HH_{i}(T)^\vee
\end{equation}
by Serre duality on $T\times T$, and that $\Hom_{Y\times Y}(\ko_{\gd_Y}, S_{\gd_Y}[i])\cong 
HH_{i}(Y)^\vee$ in the same way. 
The above construction defines a homomorphism
\begin{eqnarray*}
\Phi^{\dagger}:HH_{i}(Y)^\vee&\rightarrow & HH_{i}(T)^\vee,
\\
\nu&\mapsto & \Phi^{\dagger}\nu.
\end{eqnarray*}
Then the desired map $\Phi_*$ is the transpose of 
$\Phi^{\dagger}$ under these identifications.

Recall that any integral transform $\Phi$ induces a map $\varphi: H^*(T,\CC)\to H^*(Y,\CC)$ on singular cohomology. We have the following result stating the compatibility of this map
with $\Phi_*$ under the HKR isomorphism:

\begin{thm}[\cite{mms}]\label{HH-H-commute}
Let $\Phi: D^b(T)\to D^b(Y)$ be an integral transform. Then,
the following diagram commutes.
 \[
\xymatrix{
\ar[d]_{\widetilde{I}_*^T}HH_*(T)\ar[r]^{\Phi_*} & HH_*(Y)\ar[d]^{\widetilde{I}_*^Y}\\
H\Go_*(T) \ar[r]_{\varphi} & H\Go_*(Y)
}
\]
\end{thm}

%
\subsection{Splitting of the monad for a general moduli space} 
\label{general-moduli}
Let $v\in H^*(X)$ be a primitive and effective class with $\langle v,v\rangle=2n-2$,
$n\geq 2$, and, as above, let $H$ be a $v$-generic polarization. We recall the
fundamental result on the second cohomology of the moduli space $M_H(v)$.

Write $v^{\perp}\subset H^*(X,\Integers)$ for the sublattice orthogonal to $v$. 
Mukai introduced the natural homomorphism 
\begin{align}\label{mukai-hom}
\theta_v : & \; v^{\perp}  \to  H^2(M_H(v),\ZZ) \\
          & \; \theta_v(x)  \mapsto  
\frac{1}{\rho}\left[ \pi_{M,*}\left( v(\E^\vee)\cdot\pi^*_X(x)\right)\right]_2 \nonumber
\end{align}
where $\E$ is a quasi-universal family of similitude $\rho$ (see \cite{mukai-tata, mukai-sugaku}).

\begin{thm}[\cite{huybrechts-basic-results,OG, yoshioka-abelian-surface, yoshioka-note-on-fourier-mukai}]
\label{Yoshioka-main} 
\begin{enumerate}
\item[(1)] The moduli space $M_H(v)$ is an irreducible pojective holomorphic symplectic
manifold deformation equivalent to the Hilbert scheme of $n$ points on $X$.
\item[(2)] The homomorphism {\rm (\ref{mukai-hom})} is a Hodge isometry between $v^{\perp}$ 
and $H^2(M_H(v),\ZZ)$, where the lattice structure on the latter is given by the 
Beauville-Bogomolov form. 
\end{enumerate}
\end{thm}

\begin{proof}[Proof of Theorem \ref{thm-A-is-direct-sum} part 
\ref{thm-item-structure-sheaf-is-a-direct-summand}]
We treat first the case where an untwisted universal sheaf $\U$ exists over $X\times M_H(v)$. Let 
$\eta: \ko_{\gd_X} \to \A=\U^{\vee}[2]\circ\U$ 
be the morphism in $D^b(X\times X)$ corresponding to the unit of the adjunction
$\Phi_{\U} \dashv \Psi_{\U}$ (see \cite{cal-mukai1}, Prop. 5.1). It suffices to show that 
pre-composition induces a surjection:
$$
\Hom_{X\times X}(\U^{\vee}[2]\circ\U, \ko_{\gd_X}) \stackrel{\circ\eta}{\longrightarrow} 
\Hom_{X\times X}(\ko_{\gd_X},\ko_{\gd_X}).
$$
Note that since $S_{\gd_X}=\ko_{\gd_X}[2]$, by (\ref{HH-dual}), we may interpret this as
a map $$\Hom(\U^{\vee}[2]\circ\U, S_{\gd_X}[-2]) \to HH_{-2}(X)^{\vee}.$$ Thus the construction
(\ref{HH-functoriality}) gives homomorphisms
$$
HH_{-2}(M_H(v))^{\vee}=\Hom(\ko_{\gd_{M_H(v)}},S_{\gd_{M_H(v)}}[-2])\to \Hom(\U^{\vee}[2]\circ\U, S_{\gd_X}[-2])
\stackrel{\circ\eta}{\to} HH_{-2}(X)^{\vee} 
$$ 
whose transpose is the natural map in Hochschild homology induced by $\Phi_{\U}$. So it is
enough to show that $\Phi_{\U,*}:HH_{-2}(X) \to HH_{-2}(M_H(v))$ is injective.
Now, note that $\widetilde{I}_*^X(HH_{-2}(X))\subset H\Omega_{-2}(X)=H^2(\ko_X)$ as 
$\widetilde{I}_*^X$ is a graded map. Therefore, by Theorem \ref{HH-H-commute}, 
$$\Phi_{\U,*}|_{HH_{-2}(X)}=(\widetilde{I}_*^M)^{-1}\circ [\varphi_\U]_2\circ\widetilde{I}_*^X,$$
where $[\varphi_\U]_2$ is the degree 2 part of the map on singular cohomology induced by 
$\Phi_\U$.
But observe that by the formula (\ref{mukai-hom}), $\theta_v=-[\varphi_\U]_2$,
whence, by Theorem \ref{Yoshioka-main}, $[\varphi_\U]_2$ is injective. This proves the result for fine moduli spaces.

We sketch next the proof in the case where the universal sheaf $\U$ is twisted with respect to a Brauer class 
$\alpha\in H^2_{an}(M_H(v),\StructureSheaf{}^*)$. Let $F$ be an $H$-stable sheaf of class $v$ and denote by $[F]$ the 
corresponding point of $M_H(v)$. Let $\beta:\hat{M}\rightarrow M_H(v)$ be the blow-up centered at $[F]$.
The sheaf cohomology $\H^1(\Phi_\U(F^\vee))$
is an $\alpha$-twisted reflexive sheaf of rank $2n-2$ over $\M_H(v)$, which is locally free away from the point $[F]$,
and the quotient $V$ of $\beta^*\H^1(\Phi_\U(F^\vee))$ by its torsion subsheaf is a locally free $\beta^*\alpha$-twisted sheaf
\cite[Prop. 4.5]{markman-hodge}.
Let $p:\PP(V)\rightarrow \hat{M}$ be the associated projective bundle.
Set $\widetilde{\U}:=(1\times \beta p)^*\U$, and let $\tilde{\pi}_{ij}$ be the projection from $X\times \PP{V}\times X$ to the
product of the $i$-th and $j$-th factors. Set
\[
\B:=R\tilde{\pi}_{13_*}\left[\tilde{\pi}_{12}^*(\widetilde{\U})^\vee\otimes \tilde{\pi}_{23}^*\widetilde{\U}\otimes \tilde{\pi}_1^*\omega_X[2]
\right]
\]
Then we have the natural isomorphism
$\A:= R\pi_{13_*}\left[\pi_{12}^*\U^\vee\otimes \pi_{23}^*\U\otimes \pi_1^*\omega_X[2]\right]\cong \B$, where the latter
isomorphism follows from the projection formula and the isomorphism $R(\beta p)_*\StructureSheaf{\PP{V}}\cong 
\StructureSheaf{M_H(v)}$.

The Brauer class $p^*\beta^*\alpha$ is trivial, by \cite[Lemma 29(4)]{markman-poisson}.
We thus have an equivalence of triangulated categories $D^b(X\times\PP{V},p^*\beta^*\alpha)\cong D^b(X\times\PP{V})$
and the image of $\widetilde{\U}$ is represented by an untwisted coherent sheaf $\G$. We get an induced isomorphism
\[
\B\cong R\tilde{\pi}_{13_*}\left[\tilde{\pi}_{12}^*\G^\vee\otimes \tilde{\pi}_{23}^*\G\otimes \tilde{\pi}_1^*\omega_X[2]
\right].
\]
Replacing $M_H(v)$ by $\PP{V}$ and $\E$ by $\G$ in Equation (\ref{mukai-hom}) 
we get an analogue $\tilde{\theta}_v:v^\perp\rightarrow H^2(\PP{V},\Integers)$
of the Mukai homomorphism. The homomorphism  $\tilde{\theta}_v$ is the composition of 
$p^*\beta^*:H^2(M_H(v),\Integers)\rightarrow H^2(\PP{V},\Integers)$ with the Mukai homomorphism 
(\ref{mukai-hom}). Indeed, the argument which shows that the Mukai homomorphism is independent of the choice of a quasi-universal sheaf shows also that using either $(1\times \beta p)^*\E$ or the direct sum $\G^{\oplus \rho}$ 
of $\rho$ copies of $\G$ in Equation (\ref{mukai-hom}) results in the same homomorphism.

Theorem \ref{HH-H-commute} applies now to the integral transform $\Phi_{\G}:D^b(X)\rightarrow D^b(\PP{V})$
with kernel $\G$ and the argument in the case of untwisted universal sheaf goes through to show that 
$\StructureSheaf{\Delta_X}$ is a direct summand of
$\B$, and hence of $\A$ as well.
\end{proof}

%
\section{Yoneda algebras}
\label{sec-Yoneda-Algebra}
We prove Theorem \ref{thm-introduction-computation-of-the-monad-A} in this section.
The reader interested only in the results on generalized deformations (Theorems \ref{deformability} and \ref{thm-comparison}) may skip this section.
Let $X$ be a projective $K3$ surface, $M:=M_H(v)$ a moduli space 
of $H$-stable sheaves with Mukai vector $v$ satisfying the hypothesis of Theorem \ref{thm-A-is-direct-sum},
and $\Phi_\U:D^b(X)\rightarrow D^b(M,\theta)$ the faithful functor in Theorem
\ref{thm-A-is-direct-sum}. Assume that $(v,v)\geq 2$.

\begin{defi}
\label{def-totally-split-monad}
We say  that 
the monad object $\A$ given in Equation (\ref{eq-kernel-A})
is {\em totally split}, if the composition $\alpha$ given in 
Equation (\ref{eq-composition-of-three-morphisms}) is an isomorphism. 
\end{defi}

Assume for the rest of section 
\ref{sec-Yoneda-Algebra}  that the monad object $\A$ is totally split. 
This is the case if $M$ is the Hilbert scheme $X^{[n]}$ and $\U$ is the universal ideal sheaf, 
by Theoem \ref{thm-A-is-direct-sum}. For a more general sufficient condition for
$\alpha$ to be an isomorphism, see Lemmas
\ref{lemma-a-sufficient-condition-to-be-totally-split} and  \ref{lemma-a-sufficient-condition-to-belong-to-U}.

Set ${\rm pt}:={\rm Spec}(\ComplexNumbers)$ and let $c:M\rightarrow {\rm pt}$ be the constant morphism. We get the object
$Y(\StructureSheaf{M}):=Rc_*\StructureSheaf{M}$ in $D^b({\rm pt})$. 
As a graded vector space $Y(\StructureSheaf{M})$ is 
$\oplus_{i=0}^n H^{2i}(M,\StructureSheaf{M})[-2i]$, where $n=\dim_\ComplexNumbers(M)/2$. 
Given a graded vector space $V$, 
let $\id_{D^b(X)}\otimes_\ComplexNumbers V$
be the endofunctor of $D^b(X)$ sending an object $x$ to
$x\otimes_\ComplexNumbers V$.
Set
\begin{eqnarray*}
\Upsilon & := & \id_{D^b(X)}\otimes_\ComplexNumbers Y(\StructureSheaf{M}),
\\
R & := & \id_{D^b(X)}\otimes_\ComplexNumbers H^{2n}(M,\StructureSheaf{M})[-2n].
\end{eqnarray*}
We define next a natural transformation
\[
h \ : \ R \rightarrow \Upsilon.
\]
Write $h=\sum_{i=0}^nh_{2i}$ according to the direct sum decomposition of $Y(\StructureSheaf{M})$,
so that 
\[
h_{2i} \ : \ \id_{D^b(X)}\otimes_\ComplexNumbers H^{2n}(M,\StructureSheaf{M})[-2n]
\ \ \ \rightarrow \ \ \ 
\id_{D^b(X)}\otimes_\ComplexNumbers H^{2n-2i}(M,\StructureSheaf{M})[2i-2n].
\]

Let $t_X$ be a non-zero element of $H^2(X,\StructureSheaf{X})$, considered as the subspace $H^{0,2}(X)$ 
of the complexified Mukai lattice, and let $t_M$ be its image in $H^2(M,\StructureSheaf{M})$ via 
Mukai's Hodge isometry (\ref{mukai-hom}).
Denote by $t_M^*:Y(\StructureSheaf{M})\rightarrow Y(\StructureSheaf{M})$ the homomorphism, which sends 
$t_M^j$ to $t_M^{j-1}$, $1\leq j\leq n$, and sends $1$ to $0$.
The choice of $t_M$ identifies $h_{2i}$ as an element of 
the Hochschild cohomology $HH^{2i}(X)$. 
Explicitly, $h_{2i}=\tilde{h}_{2i}\otimes (t^*_M)^i$,
where $\tilde{h}_{2i}$ belongs to $\Ext^{2i}_{X\times X}(\StructureSheaf{\Delta_X},\StructureSheaf{\Delta_X})$.
Let $\sigma_X$ be the class in $H^0(X,\omega_X)$ dual to the class $t_X$ with respect to  Serre's duality.
Given $h_2$, the class $\tilde{h}_2\otimes \sigma_X$ in 
$\Hom_{X\times X}(\Delta_{X,*}(\StructureSheaf{X}),\Delta_{X,*}(\omega_X)[2])$ is a class in $HH_0(X)$, 
which depends canonically on $h_2$ and is
independent of the choice of the class $t_X$, since $t_M$ depends on $t_X$ linearly. 
Hence, the choice of $h_2$ corresponds to a choice of a class in $HH_0(X)$, which we denote by $h_2$ as well.
Let $h_2$ be the class in $HH_0(X)$ which is mapped to the Chern character $ch(v)$ 
in $H\Omega_0(X)$ of sheaves with Mukai vector $v$
via the Hochschild-Kostant-Rosenberg isomorphism $I^X_*:HH_0(X)\rightarrow H\Omega_0(X)$.
\[
I^X_*(h_2) := ch(v).
\]
Set $h_{2i}:=(-1)^{i+1}(h_2)^i$. Explicitly,
\begin{eqnarray}
\label{eq-h-4}
h_0 &=& -1,
\\
\nonumber
h_4&=&-(\tilde{h}_2)^2\otimes (t^*_M)^2,
\\
\nonumber
h_{k}&=&0, \ \mbox{for} \ k>4.
\end{eqnarray} 

Given an object $x$ in $D^b(X)$, let 
\[
h_x:x\otimes_\ComplexNumbers H^{2n}(M,\StructureSheaf{M})[-2n]\rightarrow 
x\otimes_\ComplexNumbers Y(\StructureSheaf{M})
\] 
be the morphism induced by the natural transformation $h$. 
Let $\pi_X:X\times M\rightarrow X$ be the projection. Note that the endo-functor $\Upsilon$ is naturally isomorphic to
$R\pi_{X_*}\circ\pi_X^*$. Denote\footnote{$\SU$ is the {\em Kleisli category} associated to the adjunction 
$\Phi_\U\dashv \Psi_\U$. The subscript ${\mathbb T}$ will later denote the monad associated to this adjunction, so that our notation is the standard one \cite[Sec. VI.5]{working}.}
by $\SU$
the full subcategory of $D^b(M)$ with objects of the form $\Phi_\U(x)$,
for some object $x$ in $D^b(X)$. 
Let $\Xi_\U:D^b(X\times M)\rightarrow D^b(M)$ be the composition of tensorization by $\U$ followed
by $R\pi_{M_*}$. Then $\Phi_\U=\Xi_\U\circ\pi_X^*$. Denote by $\Spi$ the full subcategory of $D^b(X\times M)$
with objects of the form $\pi_X^*(x)$, for some object $x$ in $D^b(X)$. 
Let $Q:\Spi\rightarrow \SU$ be the restriction of the functor $\Xi_\U$.

\begin{thm}
\label{thm-determination-of-the-category-SU}
The natural transformation $q:\Upsilon \rightarrow \Psi_\U\Phi_\U$, given in Equation (\ref{eq-q-original}) above, 
has the following properties.
\begin{enumerate}
\item
For every pair of objects $x_1$ and $x_2$ in $D^b(X)$, the first row below is a
short exact sequence. 
\[
\xymatrix{
0\ar[r] &
\Hom_{}(x_1,Rx_2)
\ar[r]^-{(h_{x_2})_*} &
\Hom_{}(x_1,\Upsilon x_2)
\ar[r]^-{(q_{x_2})_*}\ar[d]^{\cong}&
\Hom(x_1,\Psi_\U\Phi_\U (x_2))
\rightarrow 0
\ar[d]_{\cong}
\\
& 
& \Hom_{D^b(X\times M)}(\pi_X^*(x_1),\pi_X^*(x_2))
\ar[r]_-{Q} &
\Hom_{D^b(M)}(\Phi_\U(x_1),\Phi_\U(x_2)).
}
\]
\item
The vertical adjunction isomorphisms above conjugate $(q_{x_2})_*$ to the homomorphism induced by the
functor $Q:\Spi\rightarrow \SU$. The functor $Q$ is full.
\end{enumerate}
\end{thm}

The theorem is proven in section \ref{sec-proof-of-thm-determination-of-the-category-SU}.
Given objects $x$ and $y$ in a bounded triangulated category, set $\Hom^\bullet(x,y):=\oplus_{k\in\Integers}\Hom(x,y[k])[-k]$.
Note that $\Hom^\bullet_{D^b(X)}(x_1,\Upsilon x_2)$ is simply 
$\Hom^\bullet_{}(x_1,x_2)\otimes_\ComplexNumbers Y(\StructureSheaf{M})$.
An explicit calculation of the algebra $\Hom^\bullet_{}(\Phi_\U(x),\Phi_\U(x))$
as a quotient of $\Hom^\bullet_{}(x,x)\otimes_\ComplexNumbers Y(\StructureSheaf{M})$  is carried out in 
Theorem \ref{thm-the-yoneda-algebra-ofPhi-U-of-a-simple-sheaf} for any object $x$ 
represented by a simple sheaf.

Section \ref{sec-Yoneda-Algebra} is organized as follows.
In subsection \ref{sec-congruence-relation} we interpret Theorem \ref{thm-determination-of-the-category-SU}
in terms of the standard construction of a {\em quotient category} by a {\em congruence relation} (Definition
\ref{def-congruence-relation}).
The natural transformation $h$ gives rise to such a relation and 
Theorem \ref{thm-determination-of-the-category-SU} expresses the full subcategory $\SU$ of $D^b(M)$ as the quotient category
of the category $\Spi$, whose objects are the same as those of $D^b(X)$,
and such that 
\[
\Hom^\bullet_{\Spi}(x_1,x_2)=\Hom^\bullet_{D^b(X)}(x_1,x_2)\otimes_\ComplexNumbers \ComplexNumbers[t]/(t^{n+1}),
\] 
where $t$ has degree $2$.
%
In subsection \ref{sec-monad-A-is-a-quotient-of-a-constant-monad}
we show that the natural transformation $q:\Upsilon\rightarrow T$
in Theorem \ref{thm-determination-of-the-category-SU} induces a monad map between two monads in $D^b(X)$. 
Under the analogy between rings and monads, the statement that $q$ is a monad map says that $q$ is analogous to a ring homomorphism.
The functor $Q$ in Theorem \ref{thm-determination-of-the-category-SU} is an example of the general construction of a 
{\em Kleisli lifting} of a monad map \cite[Theorem 2.2.2 and Def. 2.2.3]{manes-mulry}.
In subsection
\ref{sec-a-universal-relation-ideal}
the natural transformation $h$ is defined formally in terms of a cone of $q$. In subsection
\ref{sec-proof-of-thm-determination-of-the-category-SU}
we reduce the proof of Theorem \ref{thm-determination-of-the-category-SU}
to the computation of the component $h_2$ of the natural transformation $h$. 

Subsections \ref{sec-traces} to \ref{subsection-structure-of-yoneda-algebra}
are dedicated to the proof that $h_2$ is the inverse image of $ch(v)$ via the Hochschild-Kostant-Rosenberg
isomorphism. The class $t_X$ in $H^2(X,\StructureSheaf{X})$ gives rise to a natural transformation $\id_{D^b(X)}\rightarrow \id_{D^b(X)}[2]$ and hence a morphism $x\rightarrow x[2]$, for every object $x$ of $D^b(X)$.
We get {\em two} morphisms from $\Phi_\U(x)$ to $\Phi_\U(x)[2]$ in $D^b(M)$, one is the image of the former
via the functor $\Phi_\U$. The second is induced by the natural transformation $\id_{D^b(M)}\rightarrow \id_{D^b(M)}[2]$
associated to the image $t_M$ in $H^2(M,\StructureSheaf{M})$ of $t_X$ via Mukai's Hodge isometry. 
These two morphisms in $\Hom_{D^b(M)}(\Phi_\U(x),\Phi_\U(x)[2])$ are linearly independent in general, but 
their traces belong to the one-dimensional space $H^2(M,\StructureSheaf{M})$. 
In subsection \ref{sec-traces} we calculate the ratio of the two traces using Hochschild cohomology techniques.
When $x$ is a simple object, the compositions of each of the two morphisms in 
$\Hom_{D^b(M)}(\Phi_\U(x),\Phi_\U(x)[2])$  with $t_M^{n-1}$
are linearly dependent in the one-dimensional space $\Hom_{D^b(M)}(\Phi_\U(x),\Phi_\U(x)[2n])$.
In subsection \ref{sec-a-relation-in-the-yoneda-algebra} we relate the ratio of the latter pair 
to the ratio of traces computed earlier. We get a relation in the Yoneda algebra of $\Phi_\U(x)$,
for every simple object $x$. 
In subsection \ref{subsection-structure-of-yoneda-algebra} 
we use that relation to determine the component $h_2$ of the natural transformation $h$ and complete 
the proof of Theorem \ref{thm-determination-of-the-category-SU}.

In subsection \ref{sec-moduli-spaces-of-sheaves-over-a-moduli-space} we relate moduli spaces of 
stable sheaves
on $X$ to certain moduli spaces of sheaves over $M$.

%
\subsection{A congruence relation associated to the natural transformation $h$}
\label{sec-congruence-relation}
We describe in this subsection how 
Theorem \ref{thm-determination-of-the-category-SU} reconstructs the category $\SU$ as a quotient category
in the sense of Definition \ref{def-congruence-relation}.

Let $\Spi$ be the category whose objects are the same as those of $D^b(X)$, and such that
\[
\Hom_{\Spi}(x_1,x_2) := \Hom_{D^b(X)}(x_1,\Upsilon x_2) =
\bigoplus_{k=0}^{2n} \Hom_{D^b(X)}(x_1,x_2[-2k])\otimes_\ComplexNumbers H^{2k}(M,\StructureSheaf{M}).
\]
Note that 
$\Hom^\bullet_{\Spi}(x_1,x_2)=\Hom_{D^b(X)}^\bullet(x_1,x_2)\otimes_{\ComplexNumbers} Y(\StructureSheaf{M})$.
Given morphisms $g$ in $\Hom_{D^b(X)}(x_1,x_2)$ and $f$ in $\Hom_{D^b(X)}(x_2,x_3)$, and
elements $a,b$ in $Y(\StructureSheaf{M})$,
the composition $(f\otimes a)\circ (g\otimes b)$ is $fg\otimes ab$, and is extended by linearity to all morphisms.
Note that $\Spi$ is equivalent to the full subcategory of $D^b(X\times M)$ whose objects are of the form $\pi_X^*(x)$,
for some object $x$ in $D^b(X)$. 
The above composition rule corresponds to composition
in $D^b(X\times M)$ via the adjunction isomorphism
$\Hom_{D^b(X)}(x_1,\Upsilon x_2)\cong \Hom_{D^b(X\times M)}(\pi_X^*x_1,\pi_X^*x_2)$. 
Let $\pi_X^*:D^b(X)\rightarrow \Spi$ be the functor sending each object to itself and inducing
the natural inclusion $\Hom_{D^b(X)}(x_1,x_2)\rightarrow \Hom_{\Spi}(x_1,x_2)$.

Definition \ref{def-congruence-relation} recalls the notion of a congruence relation on a category.
Consider the relation  ${\mathfrak R}$ on $\Spi$ given in Equation
(\ref{eq-push-forward-by-h-x-2}).
Following is a restatement of Theorem \ref{thm-determination-of-the-category-SU}
in the  language of quotient categories. 

\begin{thm}
\label{thm-quotient-category}
\begin{enumerate}
\item
\label{thm-item-congruence-relation}
${\mathfrak R}$ is a congruence relation.
\item
The natural transformation $q$ induces a fully faithful functor $\Sigma:\Spi/{\mathfrak R}\rightarrow D^b(M)$.
\item
The functor $\Phi_\U$ factors through the quotient functor $Q:\Spi\rightarrow \Spi/{\mathfrak R}$ as the composition
$\Phi_\U=\Sigma Q\pi_X^*:D^b(X)\rightarrow D^b(M)$.
\end{enumerate}
\end{thm}

\begin{proof}
The statement follows from Theorem \ref{thm-determination-of-the-category-SU}.
It is however instructive to see how the defining Equations (\ref{eq-h-4}) of $h$
formally imply that ${\mathfrak R}$ is a congruence relation, so we will prove part \ref{thm-item-congruence-relation}
of Theorem \ref{thm-quotient-category} independently.

The image of $(h_{x_2})_*$, given in (\ref{eq-push-forward-by-h-x-2}), is mapped under post-composition with
elements $g$ of $\pi_X^*[\Hom_{D^b(X)}(x_2,x_3)]$ to the image of $(h_{x_3})_*:\Hom_{D^b(X)}(x_1,Rx_3)\rightarrow \Hom(x_1,\Upsilon x_3)$, 
by the naturality of the transformation $h$. Indeed, if $g=\pi_X^*(\tilde{g})$ and $f$ belongs to 
$\Hom_{D^b(X)}(x_1,\Upsilon x_2)$, then $g\circ f=\Upsilon(\tilde{g})\circ f$ and naturality of $h$ yields
the following commutative diagram:
\[
\xymatrix{
\Hom_{D^b(X)}(x_1,Rx_2)\ar[r]^{(h_{x_2})_*}\ar[d]_{R(\tilde{g})_*} & 
\Hom_{D^b(X)}(x_1,\Upsilon x_2)\ar[d]_{\Upsilon(\tilde{g})_*}
\\
\Hom_{D^b(X)}(x_1,Rx_3)\ar[r]_{(h_{x_3})_*} & \Hom_{D^b(X)}(x_1,\Upsilon x_3).
}
\]
The algebra $Y(\StructureSheaf{M})$
is generated by the element $t_M$. Let
\[
\tau : \Upsilon \rightarrow \Upsilon[2]
\]
be the natural transformation corresponding to multiplication by $t_M$. 
We may regard the natural transformation 
$\tilde{h}_2:\id_{D^b(X)}\rightarrow \id_{D^b(X)}[2]$ as a natural transformation $\tilde{h}_2: R\rightarrow R[2]$ 
as well. Then we have the equality
\[
\tau\circ h = - h\circ \tilde{h}_2,
\]
which follows from the defining Equations (\ref{eq-h-4}) of $h$.
Indeed,
\[
h=\left(\begin{array}{c}
0
\\
\vdots
\\
0
\\
-\tilde{h}_2^2\otimes (t_M^*)^2
\\
\tilde{h}_2\otimes t_M^*
\\
-1
\end{array}
\right),
\ \ \ 
\tau \circ h = \left(\begin{array}{c}
0
\\
\vdots
\\
0
\\
0
\\
-\tilde{h}_2^2\otimes t_M^*
\\
\tilde{h}_2
\end{array}
\right)=
-h\circ \tilde{h}_2.
\]
It follows that the image of $(h_{x_2})_*$, given in (\ref{eq-push-forward-by-h-x-2}),  
is mapped under post-composition with $1\nolinebreak\otimes\nolinebreak t_M\in \Hom_{\Spi}(x_2,x_2[2])$ to the image of 
$(h_{x_2[2]})_*:\Hom_{D^b(X)}(x_1,Rx_2[2])\rightarrow \Hom(x_1,\Upsilon x_2[2])$.
The analogous statement holds for powers of $1\otimes t_M$, by induction. 
Hence, post-composition with every element $g\in\Hom_{\Spi}(x_2,x_3)$
maps the image of $(h_{x_2})_*$ to the image of $(h_{x_3})_*$.

Let $e$ be an element of $\Hom_{\Spi}(x_0,x_1)$. We need to show that pre-composition with
$e$ maps the image of $(h_{x_2})_*:\Hom_{D^b(X)}(x_1,Rx_2)\rightarrow \Hom_{D^b(X)}(x_1,\Upsilon x_2)$ to the image
of $(h_{x_2})_*:\Hom_{D^b(X)}(x_0,Rx_2)\rightarrow \Hom_{D^b(X)}(x_0,\Upsilon x_2)$.
If $e=\pi_X^*(\tilde{e})$, where $\tilde{e}$ belongs to $\Hom_{D^b(X)}(x_0,x_1)$,  
then for $a\in\Hom_{D^b(X)}(x_1,Rx_2)$ we have
\[
(h_{x_2}\circ a)\circ e=h_{x_2}\circ(a\circ \tilde{e}).
\]
The right hand side above belongs to the image of $(h_{x_2})_*$.
It remains to prove the statement
in case $x_0=x_1[-2]$ and $e=1\otimes t_M$. 
Now $\pi_M^*(t_M):\id_{D^b(X\times M)}\rightarrow \id_{D^b(X\times M)}[2]$ is a natural transformation.
It follows that for every pair of objects $y_1$, $y_2$ in $D^b(X\times M)$ and for every morphism $f:y_1\rightarrow y_2$
we have the commutative diagram
\[
\xymatrix{
y_1\ar[r]^{f} \ar[d]_{\pi_M^*(t_M)} & y_2 \ar[d]^{\pi_M^*(t_M)}
\\
y_1[2] \ar[r]_{[2](f)} & y_2[2].
}
\]
This holds in particular for objects $y_i$ of the form $\pi_X^*(x_i)$ and for $f:=h_{x_2}\circ a$. 
We get the equality 
\[
(h_{x_2}\circ a)\circ (1\otimes t_M)
=
(1\otimes t_M)\circ [2](h_{x_2}\circ a),
\]
for all $a\in \Hom_{D^b(X)}(x_1,Rx_2)$.
Post-compositions were shown already to preserve the relation ${\mathfrak R}$.
\end{proof}

%
\subsection{The monad $\A$ is a quotient of a constant monad}
\label{sec-monad-A-is-a-quotient-of-a-constant-monad}

Set 
\[
T:=\Psi_\U\Phi_\U:D^b(X)\rightarrow D^b(X).
\]
Denote by $\eta:\id\rightarrow T$ the unit for the adjunction $\Phi_\U\dashv \Psi_\U$, by 
$\epsilon:\Phi_\U\Psi_\U\rightarrow \id$ the counit, and
set $m:=\Psi_\U\epsilon\Phi_\U:T^2\rightarrow T$  the multiplication natural transformation.
We get the monad
\begin{equation}
\label{eq-monad-TT}
\TT:=(T,\eta,m).
\end{equation}

\hide{
%
Let $D^b(X)^\TT$ be the {\em category of modules for the monad} 
\begin{equation}
\label{eq-monad-TT}
\TT:=(T,\eta,m)
\end{equation}
(see \cite{working}, Ch. VI section 2).
Recall that an object $(x,h)$ of $D^b(X)^\TT$ consists of an object $x$ of $D^b(X)$
and a morphism $h:T(x)\rightarrow x$, making the following diagrams
commutative.
\[
\xymatrix{
T^2x\ar[r]^{Th} \ar[d]_m 
& Tx \ar[d]_{h} & \hspace{10ex} & x \ar[r]^{\eta_x} \ar[dr]_1 & Tx \ar[d]_h
\\
Tx \ar[r]_h & x & & & x
}
\]
Objects in $D^b(X)^\TT$ will be called {\em $T$-modules} for short.
A morphism $f:(x,h)\rightarrow (x',h')$ is a morphism $f$ in $\Hom(x,x')$
satisfying the equality $h'T(f)=fh$ in $\Hom(Tx,x')$.
We get the natural functor
\[
\widetilde{\Psi}_\U \ : \ D^b(M) \ \ \ \longrightarrow \ \ \ D^b(X)^\TT,
\]
sending an object $y$ in $D^b(M)$ to the module $(\Psi_\U(y),\Psi_\U(\epsilon_y))$.
The space
\begin{equation}
\label{eq-hom-space-in-monad-category}
\Hom_{D^b(X)^\TT}\left((\Psi_\U(y_1),\Psi_\U(\epsilon_{y_1})), \ (\Psi_\U(y_2),\Psi_\U(\epsilon_{y_2}))\right)
\end{equation}
is a subspace of $\Hom_{D^b(X)}\left(\Psi_\U(y_1),\Psi_\U(y_2)\right)$, by definition.
The functor $\widetilde{\Psi}_\U$ sends a morphism $f\in \Hom(y_1,y_2)$ to $\Psi_\U(f)$. Hence, 
the subspace (\ref{eq-hom-space-in-monad-category}) contains the image of
\[
\Psi_\U:\Hom(y_1,y_2)\rightarrow \Hom(\Psi_\U(y_1),\Psi_\U(y_2)).
\]

Let 
$
\Sigma \ : \ \SU \ \ \ \rightarrow \ \ \ D^b(M) 
$
be the natural full and faithful functor.

\begin{prop} 
\label{prop-fully-faithful-embedding-of-S-U-in-category-of-modules-for-monad-T}
(\cite{working}, Theorem VI.5.3 and Exercise VI.5.2, or \cite{bal1}, Prop. 2.8).
The composite functor
\[
\widetilde{\Psi}_\U\circ\Sigma \ : \ \SU \ \ \ \rightarrow \ \ \ D^b(X)^\TT
\] 
is fully faithful.
\end{prop}

Given two objects $x_1$ and $x_2$ in $D^b(X)$, the proposition implies that 
\begin{equation}
\label{eq-lift-of-Psi-U-to-the-T-module-category}
\widetilde{\Psi}_\U \ : \ \Hom_{D^b(M)}(\Phi_\U(x_1),\Phi_\U(x_2)) \ \ \ \longrightarrow \ \ \ 
\Hom_{D^b(X)^\TT}\left((T(x_1),h_1),(T(x_2),h_2)\right)
\end{equation}
is an isomorphism, where $h_i:=\Psi_\U(\epsilon_{\Phi_\U(x_i)})$. 
}

Let $\A$ be the object in $D^b(X\times X)$ given in Equation (\ref{eq-kernel-A}).
Recall that $\A$ is the kernel of the integral transform $T$. 
We have the natural morphism 
\[
m:\A\otimes_{\ComplexNumbers}Y(\StructureSheaf{M}) \ \ \rightarrow \ \ \A,
\]
given in Equation (\ref{eq-multiplication-of-A-by-Y-M}).
Let $\eta:\StructureSheaf{\Delta_X}\rightarrow \A$ 
be the morphism corresponding to the unit of the adjunction $\Psi_\U\dashv\Phi_\U$.
Set 
\begin{equation}
\label{eq-q}
q\ :=\ m\circ(\eta\otimes id) \ : \ \StructureSheaf{\Delta_X}\otimes_\ComplexNumbers Y(\StructureSheaf{M})
\ \ \rightarrow \ \ \A.
\end{equation}
The Yoneda algebra $Y(\StructureSheaf{M})$, considered as an object of $D^b(\pt)$, is a monad
in an obvious way. Denote by $\tilde{\eta}:\id\rightarrow Y(\StructureSheaf{M})$ its unit and by
$\tilde{m}:Y(\StructureSheaf{M})\otimes_\CC Y(\StructureSheaf{M})\rightarrow Y(\StructureSheaf{M})$
its multiplication. The endo-functor 
$Y(\StructureSheaf{M})\otimes_\CC(\bullet):D^b(X)\rightarrow D^b(X)$, 
of tensorization by $Y(\StructureSheaf{M})$ over $\CC$, has kernel
$\Y:=\StructureSheaf{\Delta_X}\otimes_\ComplexNumbers Y(\StructureSheaf{M})$.
We get  a monad 
\begin{equation}
\label{eq-monad-YY-on-one-factor-of-X}
\widetilde{\YY}:=(\Upsilon,\tilde{\eta},\tilde{m})
\end{equation}
in $D^b(X)$ as follows. Denote by $\pi_X:X\times M\rightarrow X$ the projection.
We have the adjunction $\pi_X^*\dashv R\pi_{X_*}$, $\Y$ is the kernel of the functor
$\Upsilon:=R\pi_{X_*}\circ \pi_X^*$, and $\widetilde{\YY}$ is the monad for that adjoint pair.
We denote again by $q$ the natural transformation from $\Upsilon$
to $T$ induced by the homomorphism of kernels (\ref{eq-q}).

\begin{rem}
\label{rem-q}
The above definition of $q$ is the one we used earlier in Equation (\ref{eq-q-original}).
The natural transformation $q$ admits a second functorial expression, which will be needed below (in 
Lemma \ref{lemma-monad-map}).
Let $\Xi_\U:D^b(X\times M)\rightarrow D^b(M)$ be the composition of tensorization by $\U$ followed
by $R\pi_{M_*}$. Then $\Phi_\U=\Xi_\U\circ\pi_X^*$. Denote by $G$ the right adjoint of $\Xi_\U$. 
Set $F:=R\pi_{X_*}$, so that $\Upsilon=F\pi_X^*$ and $T=FG\Xi_\U\pi_X^*$. 
Let $\DoubleTilde{\eta}:\id_{D^b(X\times M)}\rightarrow G\Xi_\U$ be the unit for the adjunction.
Then the natural transformation $q:\Upsilon\rightarrow T$ is equal to $F\DoubleTilde{\eta}\pi_X^*$.
Similarly, the homomorphism $q$ in (\ref{eq-q}) admits an analogous description
provided below.

Let $\Delta_{2,3}:X\times M\rightarrow (X\times M)\times M$ be the diagonal map.
The kernel of the integral functor $\Xi_\U$ is $\Delta_{2,3_*}(\U)$. 
The kernel of $G$ is $\Delta_{2,3_*}(\U^\vee\otimes \pi_X^*\omega_X[2])$. 
The kernel $\mathcal{K}$ of $G\Xi_\U$ is their convolution and is identified as an object in 
$D^b((X\times M)\times (X\times M))$ as follows. 
Let $\Delta_{2,4}:X\times M\times X \rightarrow (X\times M)\times (X\times M)$ be the diagonal map.
Let $\pi_{i,j}$ be the projection from $X\times M\times X$ onto the product of the $i$-th and $j$-th factors.
Then 
\[
{\mathcal K}=\Delta_{2,4_*}\left(\pi_{1,2}^*(\U^\vee\otimes\pi_X^*\omega_X[2])\otimes\pi_{2,3}^*\U\right).
\]
Let $p_{i_1\cdots i_k}$ be the projection from $X\times M\times X\times M$ onto the product 
of the $i_1$, \dots, $i_k$ factors. Let 
$\DoubleTilde{\eta}:\Delta_{X\times M_*}\StructureSheaf{X\times M}\rightarrow {\mathcal K}$
be the unit morphism for the adjunction $\Xi_\U\dashv G$.
Then 
\[
Rp_{13_*}:D^b((X\times M)\times (X\times M))\rightarrow D^b(X\times X)
\]
maps $\Delta_{X\times M_*}\StructureSheaf{X\times M}$ to the object $\Y$, maps
${\mathcal K}$ to the object $\A$, and maps the morphism $\DoubleTilde{\eta}$
to the morphism $q$ given in Equation (\ref{eq-q}). The latter equality is proven in the following Lemma.
\end{rem}

\begin{lem}
\label{lemma-q-equal-q-prime}
The equality $q=Rp_{13_*}(\DoubleTilde{\eta})$ holds.
\end{lem}

\begin{proof}
Set $q':=Rp_{13_*}(\DoubleTilde{\eta})$. 
Let $\gamma:\pi_{13}^*R\pi_{13_*}\rightarrow \id_{D^b(X\times M\times X)}$ be the counit for the adjunction
$\pi_{13}^*\dashv R\pi_{13_*}$. Let $u:\id_{D^b(X\times X)}\rightarrow R\pi_{13_*}\pi_{13}^*$ 
be the unit natural transformation for this adjunction.
Let $\tilde{\eta}:\StructureSheaf{\Delta_X}\rightarrow R\pi_{13_*}\pi_{13}^*\StructureSheaf{\Delta_X}$
be the morphism associated by $u$ to the object $\StructureSheaf{\Delta_X}$. 
The morphism $\tilde{\eta}$ is itself the unit morphism for the adjunction $\pi_X^*\dashv R\pi_{X_*}$.
Set $\widetilde{\A}:=Rp_{123_*}(\K)$, so that $R\pi_{13_*}\widetilde{\A}\cong \A$. We have
\[
q:=m\circ (\eta\otimes id)= 
\left(R\pi_{13_*}(\gamma_{\widetilde{\A}})\right)\circ \left(R\pi_{13_*}\pi_{13}^*(\eta)\right)=
R\pi_{13_*}\left(\gamma_{\widetilde{\A}}\circ \pi_{13}^*\eta\right).
\]
On the other hand, $q'=Rp_{13_*}(\DoubleTilde{\eta})=R\pi_{13_*}\left(Rp_{123_*}(\DoubleTilde{\eta})\right).$
Hence, it suffices to prove the equality
\begin{equation}
\label{eq-STP}
Rp_{123_*}(\DoubleTilde{\eta})=\gamma_{\widetilde{\A}}\circ \pi_{13}^*\eta.
\end{equation}
We have the equality $\eta=Rp_{13_*}(\DoubleTilde{\eta})\circ \tilde{\eta}$, since 
$\eta$ is the unit morphism for the adjunction $\Phi_\U\dashv \Psi_\U$, where $\Phi_\U=\Xi_\U\circ \pi_X^*$. 
We get 
\[
\pi_{13}^*\eta=\pi_{13}^*R\pi_{13_*}\left(Rp_{123_*}(\DoubleTilde{\eta})\right)\circ \pi_{13}^*(\tilde{\eta}),
\]
and Equation (\ref{eq-STP}) becomes
\[
Rp_{123_*}(\DoubleTilde{\eta})=\gamma_{\widetilde{\A}}\circ 
\pi_{13}^*R\pi_{13_*}\left(Rp_{123_*}(\DoubleTilde{\eta})\right)\circ \pi_{13}^*(\tilde{\eta}).
\]
The latter is a special case of Lemma \ref{lemma-gamma-Ff-u-equal-f}
applied with $F=\pi_{13}^*$, $G=R\pi_{13_*}$, $f=Rp_{123_*}(\DoubleTilde{\eta})$, 
$A=\StructureSheaf{\Delta_X}$, and $B=\widetilde{\A}$.
\end{proof}

Let $F:\C\rightarrow \D$ be a functor, $F\dashv G$ an adjunction, $u:\id_\C \rightarrow GF$ the unit,
and $\gamma:FG\rightarrow \id_\D$ the counit for the adjunction.
Let $A$ be an object of $\C$, let  $B$ be an object of $\D$, and let $f:F(A)\rightarrow B$ be a morphism.

\begin{lem}
\label{lemma-gamma-Ff-u-equal-f}
$f=\gamma_B\circ FG(f)\circ F(u_{A})$.
\end{lem}

\begin{proof}
The composition $F(A)\RightArrowOf{F(u_A)} F((GF)(A))=((FG)F)(A)\RightArrowOf{\gamma_{F(A)}}F(A)$
is the identity, by \cite[Theorem 1]{working}.
Hence, it suffices to prove the equality $\gamma_B\circ FG(f)=f\circ \gamma_{F(A)}$.
The latter equality follows from the commutativity of the following diagram. Set $A':=F(A)$. 
\[
\xymatrix{
\Hom(A',B) \ar[r]^{G} \ar[dr]_{\gamma_{A'}^*} &
\Hom(G(A'),G(B)) \ar[r]^{F} \ar[d]^{adj}_{\cong} &
\Hom(FG(A'),FG(B)) \ar[dl]^{(\gamma_B)_*}
\\
& \Hom(FG(A'),B).
}
\]
The left triangle above commutes, by \cite[Lemma 1.21]{huybrechts-book}.
The proof of the commutativity of the right triangle is similar.
\end{proof}

\begin{lem}
\label{lemma-monad-map}
The natural transformation $q$ is a {\em monad map}\footnote{We use the the term {\em monad map} following
\cite[Def. 2.2.3]{manes-mulry}. A monad map is a special case of a {\em monad functor} between two monads in different categories \cite{street}.} 
in the sense that $q\tilde{\eta}=\eta$ and
the following diagram commutes:
\[
\xymatrix{
&
\Y\circ\A \ar[r]^{q\A} & \A\circ\A \ar[dd]^{m}
\\
\Y\circ \Y \ar[ru]^{\Y q} \ar[rd]_{\tilde{m}} 
\\
& 
\Y \ar[r]_q & \A.
}
\]
\end{lem}

\begin{proof}
The equality $q\tilde{\eta}=\eta$ is clear. We prove only the commutativity of the above diagram.
Let $\catB$, $\catC$, $\catD$ be categories, let $G:\catB\rightarrow \catC$ and $F:\catC\rightarrow \catD$ be functors.
Assume given adjunctions $G^*\dashv G$ and $F^*\dashv F$. 
Let $\eta$ and $\epsilon$  be the unit and counit for $G^*\dashv G$.
Define $\tilde{\eta}$, $\tilde{\epsilon}$ similarly for $F^*\dashv F$, and let $\tilde{m}:=F\tilde{\epsilon}F^*$ 
be the multiplication for the corresponding monad.
Set $\Psi:=FG$, $\Phi:=G^*F^*$, $T:=\Psi\Phi=FGG^*F^*$, and $Y:=FF^*$. Let the natural transformation 
$q:Y\rightarrow T$ be given by 
$q:=F\eta F^*:FF^*\rightarrow FGG^*F^*$. We claim that $q$ is a monad map, in the sense that the following diagram commutes
\[
\xymatrix{
&
YT \ar[r]^{qT} & TT \ar[dd]^{m}
\\
YY \ar[ru]^{Y q} \ar[rd]_{\tilde{m}} 
\\
& 
Y \ar[r]_q & T.
}
\]

The above diagram is obtained by applying $F$ on the left and $F^*$ on the right to the 
circumference of the following diagram
\[
\xymatrix{
&
F^*FGG^* \ar[r]^{\eta F^*FGG^*} \ar[d]^{\tilde{\epsilon}GG^*} & GG^*F^*FGG^* \ar[d]^{GG^*\tilde{\epsilon}GG^*}
\\
F^*F \ar[ru]^{F^*F\eta} \ar[rd]_{\tilde{\epsilon}} & GG^* \ar[r]^{\eta GG^*} & GG^*GG^* \ar[d]^{G\epsilon G^*}
\\
& \id_{\catC} \ar[r]_{\eta}  \ar[u]_{\eta} & GG^*.
}
\]
The left triangle and the two right squares in the above diagram evidently commute.

Apply the above argument with $F:=R\pi_{X_*}:D^b(X\times M)\rightarrow D^b(X)$
and with the functor $G:D^b(M)\rightarrow D^b(X\times M)$ given by the composition of 
$\pi_M^*$ with tensorization by the object $\U^\vee[2]$ in $D^b(X\times M)$. ($G$ is the right adjoint of $\Xi_\U$).
This establishes that $q:R\pi_{X_*}\pi_X^*\rightarrow T$ is a monad map. Every step in the above argument admits
an evident translation to the case of integral functors. Note that we used above the description of $q$ given in Lemma 
\ref{lemma-q-equal-q-prime}.
\end{proof}

\begin{rem}
The monad map $q$, induced by the morphism $q:\Y\rightarrow \A$ of Fourier-Mukai kernels,
induces a functor  
\[
P:D^b(X)^\TT\rightarrow D^b(X)^{\widetilde{\YY}}
\] 
between the categories of modules 
for the monads $\TT:=(T,\eta,m)$ and $\widetilde{\YY}:=(\Upsilon,\tilde{\eta},\tilde{m})$ in $D^b(X)$  \cite[Lemma 1]{johnstone}.
The functor $P$ takes the $\TT$-module $(x,a)$ to 
\[
P(x,a) = (x,a\circ q_{x}) \ \ \ \in \ \ \ D^b(X)^{\widetilde{\YY}}.
\]
The functor $P$ is faithful, as the homomorphisms spaces are both subspaces of those of $D^b(X)$. 
The functor $P$ is an example of an {\em Eilenberg-Moore lifting} of a monad functor, where the monad functor 
in our case is $(\id_{D^b(X)},q)$ \cite[Def. 2.2.1]{manes-mulry}. Under the analogy between the monad map $q$ and 
an algebra homomorphism, the functor $P$ corresponds to the change of scalars functor, or to push-forward. 
The functor $Q$ in Theorem \ref{thm-determination-of-the-category-SU} is analogous to a pull-back functor and goes in the opposite direction. For that reason the functor $P$ will not play a role below.
\end{rem}


%
\subsection{A universal relation ``ideal''}
\label{sec-a-universal-relation-ideal}

The following proposition introduces a  ``universal relation ideal'' $\R$.
Consider the object 
$\R:=\StructureSheaf{\Delta_X}[-2n]\otimes_\ComplexNumbers \Ext^{2n}(\StructureSheaf{M},\StructureSheaf{M})$
in $D^b(X\times X)$.

\begin{prop}
\label{thm-annihilator-of-A-in-Y}
Assume that 
the monad $\A$ is totally split (Definition \ref{def-totally-split-monad}).
There exists a morphism 
$h:\R \rightarrow \Y$, unique up to a scalar factor, 
such that the following 
is an exact triangle, which admits a splitting.
\begin{equation}
\label{eq-split-exact-triangle}
\R
\LongRightArrowOf{h}
\Y
\LongRightArrowOf{q}
\A \RightArrowOf{0} \R[1].
\end{equation}
\end{prop}

\begin{proof}
There exists an object $\R'$ in $D^b(X\times X)$ and a morphism $h:\R'\rightarrow \Y$
such that $\R'
\RightArrowOf{h}
\Y
\RightArrowOf{q}
\A \rightarrow \R'[1]$ is an exact triangle, by the axioms of a triangulated category.
The following composition $\alpha$
$$\ko_{\gd_X}\otimes_\CC \lambda_n\stackrel{\iota}{\longrightarrow} 
\Y \stackrel{q}{\longrightarrow} \A,$$
given in Equation (\ref{eq-composition-of-three-morphisms}), 
is an isomorphism by the assumption that the monad is totally split.
Using the long exact sequence in sheaf cohomology coming from the
exact triangle $\R' \stackrel{h}{\to} \Y \stackrel{q}{\to} \A$, one immediately obtains 
$\R' \cong \ko_{\gd_X}[-2n]\otimes_\CC  \Ext^{2n}(\ko_M, \ko_M)$. The triangle is split as there
are no odd-degree self-extensions of $\ko_{\gd_X}$. Finally, $h$ is determined up to
scalars as the automorphism group of the object $\R'\cong \ko_{\gd_X}[-2n]$ is $\CC^*$.
\end{proof}

\hide{
\begin{enumerate}
\item
The exact triangle (\ref{eq-split-exact-triangle}) is the image via the functor 
\[
\pi_{13_*}:D^b(X\times M\times X)\rightarrow D^b(X\times X)
\] 
of the exact triangle
\[
C(\tilde{\eta})[-1]\rightarrow \widetilde{\Delta}_*\StructureSheaf{X\times M}
\RightArrowOf{\tilde{\eta}} \widetilde{\A} \rightarrow C(\tilde{\eta})
\]
associated to the cone $C(\tilde{\eta})$ of the morphism 
$\tilde{\eta}$ given in Equation 
(\ref{eq-the-lift-tilde-eta-of-the-unit}). 
In particular, the exact triangle (\ref{eq-split-exact-triangle})
is the image of an exact triangle 
\[
(\R,a)\rightarrow (\Y,\tilde{m})\rightarrow (\A,m)\RightArrowOf{\gamma} (\R[1],a)
\]
in the category $D(X\times X)^{\widetilde{\YY}}$
of modules for the monad $\widetilde{\YY}$ via the forgetful functor
$F:D(X\times X)^{\widetilde{\YY}}\rightarrow D(X\times X)$. 
Here we use Balmer's result, that $D(X\times X)^{\widetilde{\YY}}$ is triangulated (??? reference ???).
The functor $F$ is faithful and $F(\gamma)=0$. Hence, $\gamma=0$ and the above triangle splits in the category 
$D(X\times X)^{\widetilde{\YY}}$. The latter splitting implies the non-triviality of the 
morphism $t_M^n:\A\rightarrow \A[2n]$,
induced by the action $m:\A\otimes_\ComplexNumbers Y(\StructureSheaf{M})\rightarrow \A$,
when $n:=\dim(M)/2>2$.
Indeed, in that case $t_M^n$ acts trivially on $\R$ but non-trivially on $\Y$.
\item
}
The morphism $h:\R\rightarrow\StructureSheaf{\Delta_X}\otimes Y(\StructureSheaf{M})$ is naturally an element of 
\begin{equation}
\label{eq-Hochschild-cohomolog-of-X-twisted-by-powers-of-a-two-form}
\oplus_{j=0}^2\Ext^{2j}(\StructureSheaf{\Delta_X},\StructureSheaf{\Delta_X})\otimes_\CC
\Hom\left(H^{2n}(M,\StructureSheaf{M}),H^{2n-2j}(M,\StructureSheaf{M})\right).
\end{equation}
Let $t_X$ be a non-zero element of $H^2(X,\StructureSheaf{X})$, considered as a subspace 
of the complexified Mukai lattice, and let $t_M$ be its image in $H^2(M,\StructureSheaf{M})$ via 
Mukai's Hodge isometry (\ref{mukai-hom}).
Denote by $t_M^*:Y(\StructureSheaf{M})\rightarrow Y(\StructureSheaf{M})$ the homomorphism, which sends 
$t_M^j$ to $t_M^{j-1}$, $1\leq j\leq n$,  and sends $1$ to $0$.
The choice of $t_M$ identifies $h$ as an element of 
the Hochschild cohomology $HH^*(X)$. 
Explicitly, $h=\tilde{h}_0\otimes 1+\tilde{h}_2\otimes t^*_M+ \tilde{h}_4\otimes (t^*_M)^2$,
where $\tilde{h}_{2j}$ belongs to $\Ext^{2j}(\StructureSheaf{\Delta_X},\StructureSheaf{\Delta_X})$.
Let $\sigma_X$ be the class in $H^0(X,\omega_X)$ dual to the class $t_X$ with respect to  Serre's duality.
The class $h_2:=\tilde{h}_2\otimes \sigma_X$ in 
$\Hom_{X\times X}(\Delta_{X,*}(\StructureSheaf{X}),\Delta_{X,*}(\omega_X)[2])$ is a class in $HH_0(X)$, 
independent of the choice of the class $t_X$, since $t_M$ depends on $t_X$ linearly. 

\begin{thm}
\label{thm-universal-class-h-in-Hochschild-cohomology-of-X}
\begin{enumerate}
\item
\label{thm-item-h-4}
The class $\tilde{h}_0$ does not vanish and $\tilde{h}_0\tilde{h}_4=(\tilde{h}_2)^2$. 
\item
\label{thm-item-h-2}
Rescale the morphism $h$, so that $\tilde{h}_0=-1$.
Then 
the class $I^X_*(h_2)$ in $H\Omega_0(X)$ is equal to the Chern character $ch(v)$ of the Mukai vector $v$ of sheaves parametrized by $M$. 
\end{enumerate}
\end{thm}

\begin{proof}
Part \ref{thm-item-h-2} of the theorem is proven in section \ref{subsection-structure-of-yoneda-algebra}. 
We include here  the proof of part \ref{thm-item-h-4},
which follows formally from Lemma \ref{lemma-monad-map}. 
The class $\tilde{h}_0$ does not vanish, since the sheaf cohomology ${\mathcal{H}}^{2n}(\Y)$ does not vanish, while
${\mathcal{H}}^{2n}(\A)$ vanishes. It remains to compute $\tilde{h}_4$.

Identify $\A$ with $\Delta_*\StructureSheaf{X}\otimes_\ComplexNumbers\lambda_n$ via the isomorphism
$\alpha$ given in Equation (\ref{eq-composition-of-three-morphisms}). The morphism $q:\Y\rightarrow\A$ 
decomposes $q=(q_{i,j})$, $0\leq i\leq n-1$, $0\leq j\leq n$, where $q_{i,j}$ is a morphism
\[
q_{i,j}:\Delta_*\StructureSheaf{X}\otimes_\ComplexNumbers H^{2j}(M,\StructureSheaf{M})[-2j]\rightarrow
\Delta_*\StructureSheaf{X}\otimes_\ComplexNumbers H^{2i}(M,\StructureSheaf{M})[-2i].
\]
Then $q_{i,i}$ is the identity, for $0\leq i\leq n-1$,
and $q_{i,j}=0$ for $i\neq j$ and $0\leq j\leq n-1$, by construction of $\alpha$. Note that we use here the equality 
of the two descriptions of $q$
in Lemma \ref{lemma-q-equal-q-prime}, as $\alpha$ was constructed in terms of the earlier description, while
Lemma \ref{lemma-monad-map}, soon to be applied, uses the second description. 
The morphism $h:\R\rightarrow \Y$ decomposes
as a column $h=(h_{i,n})$, where $n$ is fixed, $0\leq i\leq n$,  and $h_{i,n}$ is a morphism
\[
h_{i,n}:\Delta_*\StructureSheaf{X}\otimes_\ComplexNumbers H^{2n}(M,\StructureSheaf{M})[-2n]\rightarrow
\Delta_*\StructureSheaf{X}\otimes_\ComplexNumbers H^{2i}(M,\StructureSheaf{M})[-2i].
\]
Clearly, $h_{i,n}=
\left\{\begin{array}{ccl}
0 & \mbox{if} & 0\leq i\leq n-3
\\
\tilde{h}_{2(n-i)}\otimes (t_M^*)^{n-i} & \mbox{if} & n-2\leq i \leq n.
\end{array}
\right.$
\[
(q_{i,j})=\left(
\begin{array}{ccc}
& & q_{0,n}
\\
& I_{\lambda_n} & \vdots
\\
& & q_{n-1,n}
\end{array}
\right),
\hspace{5ex}
h=
\left(
\begin{array}{c}
0
\\
\vdots
\\
0
\\
h_{n-2,n}
\\
h_{n-1,n}
\\
h_{n,n}
\end{array}
\right)
=
\left(
\begin{array}{c}
0
\\
\vdots
\\
0
\\
\tilde{h}_{4}\otimes (t_M^*)^2
\\
\tilde{h}_{2}\otimes t_M^*
\\
\tilde{h}_0
\end{array}
\right).
\]
We have the equality 
\[
0=qh=\left(
\begin{array}{c}
q_{0,n}h_{n,n}
\\
\vdots
\\
q_{n-3,n}h_{n,n}
\\
q_{n-2,n-2}h_{n-2,n}+q_{n-2,n}h_{n,n}
\\
q_{n-1,n-1}h_{n-1,n}+q_{n-1,n}h_{n,n}
\end{array}
\right)=
\left(
\begin{array}{c}
q_{0,n}\tilde{h}_0
\\
\vdots
\\
q_{n-3,n}\tilde{h}_0
\\
\tilde{h}_{4}\otimes(t_M^*)^2+q_{n-2,n}\tilde{h}_0
\\
\tilde{h}_{2}\otimes t_M^*+q_{n-1,n}\tilde{h}_0
\end{array}
\right).
\]
We get the equalities:
\begin{eqnarray}
\label{eq-first-q-i-n-vanish}
q_{i,n}&=& 0, \ \mbox{for} \ 0\leq i \leq n-3, 
\\
\label{eq-q-n-minus-2-comma-n}
q_{n-2,n}& = & -\frac{\tilde{h}_4}{\tilde{h}_0}\otimes (t_M^*)^2, 
\\
\label{eq-q-n-minus-1-comma-n}
q_{n-1,n}& = & -\frac{\tilde{h}_2}{\tilde{h}_0}\otimes t_M^*.
\end{eqnarray}

The class $t_M$ yields a morphism $t:\Delta_*\StructureSheaf{X}[-2]\rightarrow \Y$,
which is an embedding of $\Delta_*\StructureSheaf{X}[-2]$ as a direct summand of $\Y$.
We get the commutative diagram
\[
\xymatrix{
\Y[-2] \ar[r]^q \ar[d]_{t\Y} &
\A[-2] \ar[d]^{t\A}
\\
\Y\circ\Y \ar[r]_{\Y q} & \Y\circ\A.
}
\]
The morphism $\Y q$ in the above diagram appears also in the commutative diagram in the statement of Lemma 
\ref{lemma-monad-map}. Glue the two diagrams along the arrow $\Y_q$.
Set $\tau:=m(q\A)(t\A):\A[-2]\rightarrow \A$ and $\tilde{\tau}:= \tilde{m} (t\Y):\Y[-2]\rightarrow \Y$.
The commutativity of these two diagrams yields the equality
\[
\tau q = q \tilde{\tau} : \Y[-2]\rightarrow \A.
\]
The matrix of $\tilde{\tau}$ is $(\tilde{\tau}_{i,j})=
\left(
\begin{array}{cc}
\vec{0} & 0
\\
I_n\otimes t_M & \vec{0}
\end{array}
\right),
$
where $I_n$ is the $n\times n$ identity matrix and 
$t_M:H^{2j}(M,\StructureSheaf{M})\IsomRightArrow H^{2j+2}(M,\StructureSheaf{M})$ is the isomorphism 
obtained by multiplication by the class $t_M$, for $0\leq j\leq n-1$.
Hence, 
\[
\tilde{\tau}h=\left(
\begin{array}{c}
0 \\
\vdots
\\
0
\\
\tilde{h}_4\otimes t_M^*
\\
\tilde{h}_2
\end{array}
\right)
\ \mbox{and we get} \ 
0 = \tau(qh)=q(\tilde{\tau}h)=
\left(
\begin{array}{c}
0 \\
\vdots
\\
0
\\
-\frac{\tilde{h}_4\tilde{h}_2}{\tilde{h}_0}\otimes (t^*_M)^2
\\
\tilde{h}_4\otimes t_M^*-\frac{(\tilde{h}_2)^2}{\tilde{h}_0}\otimes t_M^*
\end{array}
\right),
\]
where the last equality follows from equations
(\ref{eq-first-q-i-n-vanish}), (\ref{eq-q-n-minus-2-comma-n}), and (\ref{eq-q-n-minus-1-comma-n}).
The equality $\tilde{h}_0\tilde{h}_4=(\tilde{h}_2)^2$ follows.
\end{proof}

%
\subsection{Computation of the full subcategory $\SU$ of $D^b(M)$}
\label{sec-proof-of-thm-determination-of-the-category-SU}
Let $\Spi$ be the full subcategory of $D^b(X\times M)$ consisting of objects of the form
$\pi_X^*(x)$, for some object $x$ in $D^b(X)$. We get a natural  full and faithful functor
$\widetilde{\Sigma}:\Spi\rightarrow D^b(X\times M)$. 
Let $\Xi_\U:D^b(X\times M)\rightarrow D^b(M)$ be the composition of tensorization by $\U$ followed
by $R\pi_{M_*}$. Then $\Phi_\U=\Xi_\U\circ\pi_X^*$.
We get the following commutative diagram, where the functor $Q$ is the restriction of $\Xi_\U$.
\begin{equation}
\label{eq-diagram-with-Theta}
\xymatrix{
& \Spi \ar[r]^-{\widetilde{\Sigma}} \ar[dd]_{Q} &
D^b(X\times M) \ar[dd]_{\Xi_\U} 
\\
D^b(X) \ar[ur]_{\pi^*_X} \ar[dr]^{\Phi_\U} 
\\
& \SU \ar[r]_{\Sigma} & 
D^b(M) 
}
\end{equation}


Let $q_{x_i}:\Upsilon(x_i)\rightarrow T(x_i)$ be the morphism induced by the natural transformation 
$q$, which in turn is induced by the morphism of kernels given in Equation (\ref{eq-q}).
We get the homomorphism
\begin{eqnarray*}
(q_{x_2})_*:\Hom_{D^b(X)}(x_1,\Upsilon x_2) & \rightarrow &
\Hom_{D^b(X)}(x_1,T x_2)
\\
f & \mapsto & q_{x_2}\circ f,
\end{eqnarray*}
and the diagram:
\begin{equation}
\label{eq-commutative-diagram-relating-theta-to-q}
\xymatrix{
\Hom_{D^b(X\times M)}(\pi_X^*(x_1),\pi_X^*(x_2)) \ar[r]^-{\cong} \ar[d]^{\Xi_\U}
& \Hom_{D^b(X)}(x_1,\Upsilon x_2) \ar[d]^{(q_{x_2})_*}
\\
\Hom_{D^b(X)}(\Phi_\U(x_1),\Phi_\U(x_2)) \ar[r]^-{\cong}&
\Hom_{D^b(X)}(x_1,Tx_2).
}
\end{equation}
where the horizontal isomorphisms are due to the adjunctions. 

\begin{lem}
\label{lemma-push-forward-by-q-x-2-is-Xi-U}
The above diagram is commutative. 
\end{lem}

\begin{proof}
The commutativity is a special case 
of the following Lemma applied with $F:=R\pi_{X_*}$, $G^*:=\Xi_\U$, $\B:=D^b(M)$, $\C:=D^b(X\times M)$, and $\D:=D^b(X)$.
\end{proof}

Let $\B$, $\C$, and $\D$ be categories, let $G:\B\rightarrow \C$ and $F:\C\rightarrow \D$ be functors,
and let $G^*\dashv G$ and $F^*\dashv F$ be adjunctions.
Let $\eta:\id_{\C}\rightarrow GG^*$ be the unit for the adjunction.
Set $q:=F\eta F^* : FF^*\rightarrow FGG^*F^*$. 
\begin{lem}
The following diagram is commutative for every pair of objects $x_1$, $x_2$ in $\D$.
\[
\xymatrix{
\Hom(F^*x_1,F^*x_2)  \ar[dd]_{G^*} \ar[dr]^{(\eta_{F^*x_2})_*}
& & \Hom(x_1,FF^*x_2) \ar[ll]_-{\cong} \ar[dd]^{(q_{x_2})_*}
\\
& \Hom(F^*x_1,GG^*F^*x_2) \ar[dr]_-{\cong} \ar[dl]^-{\cong}
\\
\Hom(G^*F^*x_1,G^*F^*x_2) & &
\Hom(x_1,FGG^*F^*x_2). \ar[ll]^-{\cong} 
}
\]
\end{lem}

\begin{proof}
All the arrows labeled as isomorphisms correspond to the adjunction isomorphisms. 
Hence, the lower middle triangle commutes. The left triangle commutes, by 
\cite[Lemma 1.21]{huybrechts-book}. The upper right triangle commutes, by definition of $q$ and the naturality of 
the adjunction isomorphisms.
\end{proof}



\begin{proof}[Proof of Theorem \ref{thm-determination-of-the-category-SU}]
Theorem \ref{thm-universal-class-h-in-Hochschild-cohomology-of-X}
yields the exact sequence with the natural transformation $h$ satisfying Equations (\ref{eq-h-4}).
The sequence in the statement of Theorem \ref{thm-determination-of-the-category-SU}
is short exact, by the the splitting of the exact triangle (\ref{eq-split-exact-triangle}). 
%
The diagram in Theorem \ref{thm-determination-of-the-category-SU} is commutative, 
by Lemma \ref{lemma-push-forward-by-q-x-2-is-Xi-U}.
The functor $Q$ is full, by the surjectivity of $(q_{x_2})_*$ in the diagram in 
Theorem \ref{thm-determination-of-the-category-SU}
and the commutativity of that diagram.
\end{proof}

%
\subsection{Traces}
\label{sec-traces}
Subsections \ref{sec-traces} to \ref{subsection-structure-of-yoneda-algebra}
are dedicated to the proof of part \ref{thm-item-h-2} of 
Theorem \ref{thm-universal-class-h-in-Hochschild-cohomology-of-X}. 
Given two smooth projective varieties $X$ and $M$, an integral functor $\Phi:D^b(X)\rightarrow D^b(M)$,
and an object $x\in D^b(X)$,
we have a natural composite homomorphism
\[
H^i(X,\StructureSheaf{X}) \LongRightArrowOf{\mu_x} \Hom(x,x[i])
\LongRightArrowOf{\Phi}
\Hom(\Phi(x),\Phi(x)[i])
\LongRightArrowOf{\tr} H^i(Y,\StructureSheaf{Y}).
\]
The definition of the natural transformations $\mu$ and $\tr$ are recalled below. In this subsection 
we use known results about the functoriality of Hochschild homology in order to 
provide a topological formula for the homomorphism displayed above in a special case
(see Proposition \ref{prop-a-topological-formula-for-a-trace}).

The following lemma will be needed in the proof of Proposition
\ref{prop-a-topological-formula-for-a-trace} below.
Let $X$ and $Y$ be smooth projective varieties, $f:X\rightarrow Y$ a morphism,
$x$ an object of $D^b(X)$ and $y$ an object of $D^b(Y)$. We will use the notation $f_*$ and $f^*$
for the right and left derived functors $Rf_*$ and $Lf^*$ for brevity.
Assume given morphisms $t:x\rightarrow f^*y$ and
$\phi:f^*y\rightarrow x\otimes \omega_X[\dim X]$.
Let $\eta:\id\rightarrow f_*f^*$ be the unit for the adjunction $f^*\dashv f_*$.
Let $f_!$ be the left adjoint of $f^*$ and 
let $\eta_y(t)\in \Hom\left(f_!(x),f_*f^*(y)\right)$ be the image of $t$ via the composition
\begin{equation}
\label{append1}
\Hom(x,f^*(y)) \LongRightArrowOf{f^*(\eta_y)}
\Hom(x,f^*f_*f^*(y)) \cong
\Hom(f_!(x),f_*f^*(y)).
\end{equation}
Note that $\phi\circ t$ belongs to $\Hom(x,x\otimes\omega_X[\dim X]))$.
Using the isomorphism
$f_*(x\otimes\omega_X[\dim X])\cong f_!(x)\otimes \omega_Y[\dim Y]$
we see that $f_*(\phi)\circ \eta_y(t)$ belongs to
$\Hom(f_!(x), \ f_!(x)\otimes\omega_Y[\dim Y])$. Let 
\begin{equation}
\label{eq-Tr-X}
Tr_X \ : \ \Hom(x,x\otimes\omega_X[\dim X]) 
\ \ \ \longrightarrow \ \ \ \ComplexNumbers
\end{equation}
be the composition of the isomorphism 
$\Hom(x,x\otimes\omega_X[\dim X])\cong \Hom(x,x)^*$
induced by Serre duality, followed by evaluation $\Hom(x,x)^*\rightarrow \ComplexNumbers$
on the identity morphism in $\Hom(x,x)$. 

\begin{lem}
\label{lemma-an-equality-of-traces}
The following equality holds.
\[
Tr_X(\phi\circ t) \ \ \ = \ \ \  Tr_Y(f_*(\phi)\circ \eta_y(t)).
\]
\end{lem}

\begin{proof}
Given $z\in D^b(Y)$, we get the adjunction isomorphism
\begin{equation}\label{adjoint}
\Hom(x,f^*z) \stackrel{\cong}{\longrightarrow} \Hom(f_!x, z).
\end{equation}
Serre duality yields the dual  isomorphism:
$$
\Hom(f^*z, S_Xx) \stackrel{\cong}{\longleftarrow} \Hom(z,S_Yf_!x).
$$
Thus given $t'\in \Hom(x,f^*z)$ and $\phi' \in \Hom(f^*z, S_Xx)$, we have an equality
\begin{equation}\label{traces}
Tr_X(\phi'\circ t')=Tr_Y(\ol{\phi'}\circ \ee \circ f_!(t'))
\end{equation}
where $\ol{\phi'}\in \Hom(z, S_Yf_!x)$ is the preimage of $\phi'$, while 
$\ee(f_!(t'))\in\Hom(f_!x, z)$ is the adjoint map to $t'$, where $\epsilon:f_!f^*\rightarrow \id$ is the counit.

Replace $z$ by $f_*f^*y$ in (\ref{adjoint}), and pre-compose with the unit to get the
diagram (\ref{append1})
whose Serre dual is
\begin{equation}\label{append2}
\Hom(f^*y, S_Xx) \stackrel{\circ f^*\eta}{\longleftarrow} \Hom(f^*f_*f^*y, S_Xx)
\stackrel{\cong}{\longleftarrow} \Hom(f_*f^*y,S_Yf_!x)=\Hom(f_*f^*y,f_*S_Xx).
\end{equation}
Given $t\in \Hom(x,f^*y)$, its image in $\Hom(f_!x, f_*f^*y)$ via (\ref{append1}) is
nothing but $\eta_y(t)$. Similarly, for $\phi \in \Hom(f^*y, S_Xx)$, the preimage
of $\phi\circ\ee_{f^*y} \in \Hom(f^*f_*f^*y, S_Xx)$ in $\Hom(f_*f^*y,f_*S_Xx)$ via the middle isomorphism in 
(\ref{append2}) is easily verified to be $f_*(\phi)$. Thus, equation (\ref{traces})
implies that 
$$
Tr_X((\phi\circ\ee_{f^*y} )\circ(f^*\eta_y \circ t))=Tr_Y(f_*(\phi)\circ\eta_y(t)).
$$ 
Finally, since $\ee_{f^*y}\circ f^*\eta_y = id_{f^*y}$, we have that
$(\phi\circ\ee_{f^*y} )\circ(f^*\eta_y \circ t) = \phi\circ t$, and the result follows.
\end{proof}

Let $\nu$ be a class in $HH_0(X)\cong \Hom(\Delta_{X,*}\StructureSheaf{X},\Delta_{X,*}\omega_X[\dim X])$
and let $x$ be an object of $D^b(X)$. Regarding $\nu$ as a natural transformation from
$\id_{D^b(X)}$ to the Serre functor we get a morphism
$\nu_x:x\rightarrow x\otimes\omega_x[\dim X]$. Given a class $c$ in $H\Omega_*(X)$, denote by
$c_{p,q}$ the direct summand in $H^p(X,\Omega^q_X)$.
Set $d:=\dim(X)$.

\begin{lem}
\label{lemma-chern-character-as-a-linear-functional-on-HH}
The following equality holds:
${\displaystyle
Tr_X(\nu_x)=Tr_X\left(\left[\widetilde{I}^X_*(\nu)\sqrt{td_X}ch(x^\vee)\right]_{d,d}\right).
}$
\end{lem}
Note that the right hand side is the Mukai pairing of $\nu$ and the class $(\widetilde{I}^X_*)^{-1}(ch(x))$ in $HH_0(X)$
as defined in \cite[Def. 6.1]{cal-mukai1}. Mukai's sign convention, which we will follow, is different and 
we would regard the right hand side as {\em minus} the Mukai pairing.
\begin{proof}
The statement is essentially the definition of the Chern character as a class in $HH_0(X)$ (see
\cite[Sec. 6.2]{cal-mukai1} and \cite[Theorem 4.5]{cal-mukai2}). 
\end{proof}

Given a scheme $S$ and an object 
$x\in D^b(S)$,  denote by $x^\vee:= RHom(x,\StructureSheaf{S})  \in D^b(S)$ its dual object. Let
\[
\mu_x : \StructureSheaf{S}\rightarrow x^\vee\otimes x
\]
be the natural morphism and 
\[
\tr:x^\vee\otimes x \rightarrow \StructureSheaf{S}
\]
the trace morphism (\cite{mukai-symplectic}, page 114). 
The following identity holds. 
\begin{equation}
\label{eq-tr-composed-with-mu-is-scaling-by-rank}
\tr\circ \mu_x=\rank(x)\cdot 1.
\end{equation} 

Assume next that $S$ is smooth and projective. 
Consider the trace pairing 
\[
\Hom(x,x[i])\otimes \Hom(x[i],x\otimes\omega_S[\dim S])\RightArrowOf{\circ}
\Hom(x,x\otimes\omega_s[\dim S]) \RightArrowOf{tr} H^{\dim S}(S,\omega_S)\cong \ComplexNumbers,
\]
where the left arrow is composition. Mukai shows that the above pairing is a perfect pairing, for
$0\leq i \leq \dim S$
(\cite{mukai-symplectic}, page 114). Mukai's trace pairing is induced by Serre's duality
as follows. Set $y:=x^\vee\otimes x$. We can rewrite Mukai's pairing as
\begin{equation}
\label{eq-Mukai-tr-pairing}
H^i(y)\otimes H^{\dim S-i}(y\otimes\omega_S) \rightarrow H^{\dim S}(S,\omega_S),
\end{equation}
while Serre's duality yields a pairing
\begin{equation}
\label{eq-serre-duality-in-Mukais-setup}
H^i(y^\vee)\otimes H^{\dim S-i}(y\otimes\omega_S) \rightarrow H^{\dim S}(S,\omega_S).
\end{equation}
Mukai interprets the composition
\[
R\SheafHom(R\SheafHom(x,x)^\vee,R\SheafHom(x,x))
\cong R\SheafHom(x,x)\otimes R\SheafHom(x,x) \RightArrowOf{\circ} R\SheafHom(x,x) 
\RightArrowOf{tr} \StructureSheaf{S}
\]
as an isomorphism $\psi:R\SheafHom(x,x)^\vee\rightarrow R\SheafHom(x,x)$, or equivalently,
$\psi:y^\vee\rightarrow y$. Relating the leftmost factors in (\ref{eq-Mukai-tr-pairing})
and (\ref{eq-serre-duality-in-Mukais-setup}) via $\psi$ relates Mukai's trace pairing to Serre's duality.

Let $G:D^b(S)\rightarrow D^b(S)$ be the functor of tensorization by the object $x$.
The right and left adjoints of $G$ are both isomorphic to 
the functor $G^\dagger:D^b(S)\rightarrow D^b(S)$, of tensorization with $x^\vee$.
Let $\Delta_S:S\rightarrow S\times S$ be the diagonal morphism. 
Then $\Delta_{S_*}(\mu_x):\Delta_{S_*}(\StructureSheaf{S})\rightarrow \Delta_{S_*}(x^\vee\otimes x)$
induces the unit natural transformation $\mu_x:id\rightarrow G^\dagger G$
for the adjunction.
The morphism $\Delta_{S_*}(\tr):\Delta_{S_*}(x^\vee\otimes x)\rightarrow \Delta_{S_*}(\StructureSheaf{S})$
induces the counit natural transformation $\tr:G^\dagger G\rightarrow id$.
The morphisms 
\begin{eqnarray}
\label{eq-mu-x-equal-G}
\mu_x:\Hom(\StructureSheaf{S},\StructureSheaf{S}[i])&\rightarrow & 
\Hom(\StructureSheaf{S},x^\vee\otimes x[i]) \cong
\Hom(x,x[i]),
\\
\nonumber
G:\Hom(\StructureSheaf{S},\StructureSheaf{S}[i])&\rightarrow &\Hom(G(\StructureSheaf{S}),G(\StructureSheaf{S})[i])
\end{eqnarray}
are equal under the identification $x=G(\StructureSheaf{S})$, by \cite[Lemma 1.21]{huybrechts-book}.

\begin{rem}
\label{remark-non-vanishing-of-tr}
If $\Hom(x,x)$ is one-dimensional, then 
\begin{equation}
\label{eq-trace-isomorphism-for-simple-x}
\Hom(x,x\otimes\omega_S[\dim S]) \RightArrowOf{tr}
H^{\dim S}(S,\omega_S) 
\end{equation}
is an isomorphism. Indeed, both spaces are one dimensional and the statement reduces to the non-vanishing of
$tr:\Hom(x,x)\otimes\Hom(x,x\otimes\omega_S[\dim S])\rightarrow H^{\dim S}(S,\omega_S)$, which holds 
it being a perfect pairing. Now let $t$ be an element of $H^{\dim S}(S,\omega_S)$. As a consequence of the above isomorphism, we see that 
the element $\mu_x(t)$ of $\Hom(x,x\otimes\omega_S[\dim S])$ vanishes, if $\rank(x)=0$
and $\Hom(x,x)$ is one-dimensional. 
Indeed, $tr(\mu_x(t))=\rank(x)\cdot t=0$ in this case.
\hide{
then the linear functional
$Tr_X$, given in (\ref{eq-Tr-X}), is related to the composition
$
\Hom(x,x\otimes\omega_X[\dim X]) \RightArrowOf{tr}
H^{\dim X}(X,\omega_X) \cong \ComplexNumbers
$
by 
\begin{equation}
\label{equation-relating-two-traces}
tr=\rank(x)\cdot Tr_X. 
\end{equation}
We have the commutative diagram
\[
\xymatrix{
H^{\dim X}(X,\omega_X) \ar[r]^-{\mu_x} \ar[d]_{S.D}^{\cong}
& 
 \Hom(x,x\otimes\omega_X[\dim X]) \ar[r]^-{tr}
& 
H^{\dim X}(X,\omega_X) \ar[d]_{S.D}^{\cong}
\\
H^0(X,\StructureSheaf{X})^* \ar[r]^{tr^*} &
\Hom(x,x)^*\ar[r]^{\mu_x^*}&
H^0(X,\StructureSheaf{X})^*.
}
\]
The composition of the horizontal arrows in both the upper and lower rows is multiplication by $\rank(x)$.
The outer square commutes, but if we identify the middle vector spaces via Serre's duality,
the resulting two smaller squares will not commute.
}
\end{rem}

Let $X$ be a projective $K3$ surface, $M:=M_H(v)$ a moduli space 
of $H$-stable sheaves with Mukai vector $v$ satisfying the hypothesis of Theorem \ref{thm-A-is-direct-sum}
and $\Phi_\U:D^b(X)\rightarrow D^b(M)$ the faithful functor in Theorem
\ref{thm-A-is-direct-sum}. Assume that $(v,v)\geq 2$.
Given objects $x$ and $y$ in $D^b(X)$, set $\Hom^\bullet(x,y):=\oplus_i\Hom(x,y[i])[-i]$,
as an object of the derived category of a point.
Let $Y(x):=\Hom^\bullet(x,x)$ be the Yoneda algebra. 
The morphism $\mu_x$ induces the natural algebra homomorphism 
\[
\mu_x \ : \ Y(\StructureSheaf{X}) \ \ \longrightarrow \ \ Y(x),
\]
for every object $x$ of $D^b(X)$.  Define 
\[
\mu_{\Phi_\U(x)} \ : \ Y(\StructureSheaf{M}) \ \ \longrightarrow \ \ Y(\Phi_\U(x))
\]
similarly. 

Let $t_X\in H^{2}(X,\StructureSheaf{X})$ be a non-zero element.
Let 
\[
\varphi_\U \ : \ H^*(X,\ComplexNumbers) \ \ \ \longrightarrow \ \ \ H^*(M,\ComplexNumbers)
\]
be the homomorphism induced by the correspondence
$\sqrt{td_X}ch(\U)\sqrt{td_M}\in H^*(X\times M,\RationalNumbers)$.
Given $a\in H^*(X,\CC)$, denote by $[\varphi_\U(a)]_2$ the graded summand in $H^2(M,\CC)$.

\begin{prop}
\label{prop-a-topological-formula-for-a-trace}
For every object $x$ of $D^b(X)$, the following equality holds:
\begin{equation}
\label{eq-t-M-does-not-vanish}
\tr\left(\Phi_\U(\mu_x(t_X))\right) \ \ = \ \ 
\rank(x)
[\varphi_\U(t_X)]_2,
\end{equation}
where 
$t_X$ on the right hand side
is considered as an element
of the summand $H^{0,2}(X)$ of $H^2(X,\ComplexNumbers)$, via the Hodge decomposition,
and the left hand side is similarly considered as an element of the subspace 
$H^{0,2}(M)$ of $H^2(M,\ComplexNumbers)$. 
\end{prop}

\begin{proof}
Let $\Phi:D^b(X)\rightarrow D^b(M)$ be the integral functor 
with kernel $\U\otimes \pi_X^*(x)$. So $\Phi(a)=\Phi_\U(x\otimes a)$.
Then $\Phi_\U(\mu_x(t_X))=\Phi(t_X)$, by the equality of the two homomorphisms
displayed in Equation (\ref{eq-mu-x-equal-G}).
Let 
\[
\varphi: H^*(X,\ComplexNumbers) \ \ \longrightarrow \ \ H^*(M,\ComplexNumbers)
\]
be the homomorphism induced by the correspondence
$\pi_X^*\sqrt{td_X}ch(\U\otimes \pi_X^*x)\pi_M^*\sqrt{td_M}$.
Note that $ch(x)t_X = v(x) t_X = \rank(x)t_X$ in  $H^*(X,\ComplexNumbers)$.
So we get the equality
\[
\varphi(t_X)=\varphi_\U(ch(x) t_X)= \rank(x)\varphi_\U(t_X).
\]
It remains to prove the equality
\begin{equation}
\label{eq-a-topological-formula-for-tr-Phi-U-mu-x-t-X}
\tr\left(\Phi(t_X)\right) \ \ = \ \ 
[\varphi(t_X)]_2.
\end{equation}

Let $\eta$ and $\epsilon$ be the unit and counit for the adjunction $\Delta_M^*\dashv\Delta_{M_*}$.
We get the following morphisms
\[
\StructureSheaf{M}=
\Delta_M^*\StructureSheaf{M\times M}
\LongRightArrowOf{\Delta_M^*\left(\eta_{\StructureSheaf{M\times M}}\right)}
\Delta_M^*\Delta_{M_*}\Delta_M^*\StructureSheaf{M\times M}
\LongRightArrowOf{\epsilon_{[\Delta_M^*\StructureSheaf{M\times M}]}}
\Delta_M^*\StructureSheaf{M\times M}=\StructureSheaf{M},
\]
which compose to the identity.
Let $\eta_M:H^2(M,\StructureSheaf{M})\rightarrow HH_{-2}(M)$ be the composition
\begin{eqnarray*}
H^2(M,\StructureSheaf{M})\cong 
\Hom(\StructureSheaf{M},\StructureSheaf{M}[2])
&\LongRightArrowOf{\Delta_M^*\left(\eta_{\StructureSheaf{M\times M}}\right)}&
\Hom(\StructureSheaf{M}[-2],\Delta_M^*\Delta_{M_*}\Delta_M^*\StructureSheaf{M\times M})
\\
& = & \Hom(\StructureSheaf{M}[-2],\Delta_M^*\Delta_{M_*}\StructureSheaf{M}) \  = \ 
HH_{-2}(M).
\end{eqnarray*}
The analogous homomorphism $\eta_X:H^2(X,\StructureSheaf{X})\rightarrow HH_{-2}(X)$, for the $K3$ surfaces $X$,
is an isomorphism. Let $\epsilon_M:HH_{-2}(M)\rightarrow H^2(M,\StructureSheaf{M})$ be the 
morphism induced by $\epsilon_{[\Delta_M^*\StructureSheaf{M\times M}]}$.

We have the following
diagram, where the middle square commutes by
Theorem \ref{HH-H-commute}.
\[
\xymatrix{ 
& HH_{-2}(X) \ar[r]^{\Phi_*} \ar[dd]_{\tilde{I}^X_*} & 
HH_{-2}(M) \ar[dd]^{\tilde{I}^M_*} \ar@/^3pc/[dr]^{\epsilon_M}
\\
t_X\in H^{0,2}(X) \ar[ur]^{\eta_X} \ar[dr]_{=} & & &
\hspace{6Ex}H^{0,2}(M)\ni tr(\Phi(t_X)). \ar[ul]_{\eta_M} 
\\
& H\Omega_{-2}(X) \ar[r]_{\varphi} & 
H\Omega_{-2}(M) \ar[ur]_{\pi^{0,2}}
}
\]
Here $\Phi_*$ is the homomorphism on Hochschild homology 
recalled in Equation (\ref{eq-push-forward-on-Hochschild-homology}).
The triangle on the right (with arrow $\epsilon_M$) commutes as well.
Indeed, the composition 
$$
\gd_M^*\gd_{M,*}\gd_M^*\ko_{M\times M} \stackrel{\widetilde{I}_M}{\longrightarrow} 
\oplus \Omega^i[i] \stackrel{pr^0}{\longrightarrow} \gd^*_M\ko_M,
$$
where $pr^0$ is the projection onto the component in degree $0$
and $\tilde{I}_M$ is given in equation (\ref{eq-normalized-HKR-isomorphisms}),
is nothing but
$\ee_{[\gd^*_M\ko_{M\times M}]}$. Furthermore, $\epsilon_M\circ\eta_M$ is the identity, since
$\ee_{[\gd^*_M\ko_{M\times M}]}\circ \gd^*_M\eta_{\ko_{M\times M}} = id_{\ko_M}$.

Equality (\ref{eq-a-topological-formula-for-tr-Phi-U-mu-x-t-X}) reduces to the equality
\begin{equation}
\label{eq-in-HH-minus-2}
\tr\left(\Phi(t_X)\right) \ \ = \ \ 
\pi^{0,2}\left[\widetilde{I}^M_*\left(\Phi_*(\eta_X(t_X))\right)\right].
\end{equation}
Equivalently, it suffices to prove that 
$\epsilon_M$ maps $\eta_M\left[\tr\left(\Phi(t_X)\right)\right]$ and 
$\Phi_*(\eta_X(t_X))$ to the same element of $H^{2,0}(M)$.

Let $(\Phi_*)^\dagger$ be the adjoint of $\Phi_*$ with respect to the
Serre Duality pairing, given in equation
(\ref{HH-functoriality}).
We will verify equality (\ref{eq-in-HH-minus-2}) by establishing the equality 
\begin{equation}
\label{eq-holds-for-all-nu}
\langle\nu,\eta_M\left[\tr\left(\Phi(t_X)\right)\right]\rangle \ \ = \ \ 
\langle (\Phi_*)^\dagger\nu,\eta_X(t_X)\rangle,
\end{equation}
for every element $\nu \in HH_{-2}(M)^\vee$, which is in the image of the following composition
\[
H^{2n-2}(M,\omega_M)\cong
\Hom(\StructureSheaf{M},\omega_M[2n-2])
\LongRightArrowOf{\Delta_{M_*}}
\Hom(\StructureSheaf{\Delta_{M}},\omega_{\Delta_M}[2n-2])\cong
HH_{-2}(M)^\vee,
\]
where the right isomorphism is Serre's duality.
This suffices because 
$\langle\Delta_{M_*}(\tilde{\nu}),\lambda\rangle=\langle\tilde{\nu},\epsilon_M(\lambda)\rangle$, for every $\lambda\in HH_{-2}(M)$ and
$\tilde{\nu}\in H^{2n-2}(M,\omega_M)$. The following three observations explain the latter equality.
\begin{enumerate}
\item[(i)] The right isomorphism in the displayed composition above is given also by the Mukai pairing under the identification
$HH_2(M)\cong \Hom(\StructureSheaf{\Delta_{M}},\omega_{\Delta_M}[2n-2])$,
by \cite[Subsection 4.11]{cal-mukai1}.
\item[(ii)]
The modified HKR isomorphism $\widetilde{I}^M_*$ is an isometry with respect to the Mukai pairings on 
$HH_*(M)$ and $H\Omega_*(M)$ (see the conjecture in \cite[Sec. 1.8]{cal-mukai2} and its proof in 
\cite[Theorem 0.5]{HN}).
\item[(iii)] The compositon of the map $\widetilde{I}^M_*\circ \gd_{M,*}:H^i(M, \go_M)
\to H\Go_{2n-i}(M)$ with the projection $H\Go_{2n-i}(M) \to H^i(M, \go_M)$ is the identity
(see the proof of  \cite[Prop. 2.1]{HN}). 
\end{enumerate}

Let $\tilde{\nu}$ be an element of $H^{2n-2}(M,\omega_M)$ mapping to $\nu$.
\hide{
Then $\tilde{\nu}\otimes \Phi(t_X)$ is an element of 
$\Hom(\Phi_\U(x),\Phi_\U(x)\otimes \omega_M[2n]).$
The left hand side of (\ref{eq-holds-for-all-nu}) is equal to
\[
\rank(x)\cdot 
\tr\left(\tilde{\nu}\otimes \Phi_\U(\mu_x(t_X)\right)\in H^{2n}(M,\omega_M)\stackrel{S.D.}{\cong} 
H^0(M,\StructureSheaf{M})^*
\cong \ComplexNumbers.
\]
The element above is equal to $\rank(x)\rank(\Phi_\U(x)))$ times the image $\lambda$ 
of $\tilde{\nu}\otimes \Phi_\U(\mu_x(t_X))$
via  the composition
\[
\Hom(\Phi_\U(x),\Phi_\U(x)\otimes \omega_M[2n]) \stackrel{S.D.}{\cong} 
\Hom(\Phi_\U(x),\Phi_\U(x))^\vee\RightArrowOf{(\Phi_\U)^*} 
\Hom(x,x)^\vee\RightArrowOf{H^0(\mu_x)^*} H^0(X,\StructureSheaf{X})^*\cong \ComplexNumbers.
\]
It remains to prove that
$\rank(x)\cdot \lambda$ is equal to the factor
$\langle (\Phi_*)^\dagger\nu,t_X\rangle$ appearing on the right hand side
of Equation (\ref{eq-holds-for-all-nu}). 
The factor $\rank(x)$ appears as follows.
Let $\epsilon_1:\Phi_{\U_L}^\dagger\Phi_\U\rightarrow id$ be the counit of  the adjunction 
$\Phi_{\U_L}^\dagger\dashv\Phi_\U$ and
$\epsilon_2:G^\dagger\Phi_{\U_L}^\dagger\Phi\rightarrow id$ the counit of the adjunction 
$G^\dagger\Phi_{\U_L}^\dagger\dashv\Phi$, where $G(a)=x\otimes a.$
Then $\epsilon_2=\tr(G^\dagger\epsilon_1 G)$.
Recall that 
\[
(\Phi_*)^\dagger(\nu)=\epsilon_2(\Phi^\dagger_R\nu\Phi)\eta_2=
\tr\circ G^\dagger\left(\epsilon_1(\Phi^\dagger_{\U_R}\nu\Phi_\U) \eta_1\right)G\circ\mu_x=
\tr\circ G^\dagger\left((\Phi_{\U_*})^\dagger(\nu)\right)G\circ\mu_x, 
\]
where
$\eta_1$ is the unit for the adjunction $\Phi_\U\dashv\Phi_{\U_R}$ and  
$\eta_2$ is the unit for the adjunction $\Phi\dashv G^\dagger\Phi_{\U_R}^\dagger$.
Hence, 
$(\Phi_*)^\dagger(\nu):\StructureSheaf{\Delta_X}\rightarrow S_{\Delta_X}[-2]$  factors as  the composition:
\[
\id\LongRightArrowOf{\mu_x}G^\dagger G
\LongRightArrowOf{G^\dagger\left((\Phi_{\U_*})^\dagger\nu\right)G}
G^\dagger S_{\Delta_X}[-2]G
\cong
S_{\Delta_X}[-2]G^\dagger G
\LongRightArrowOf{\tr}S_{\Delta_X}[-2].
\]
We claim that the following equality holds
\begin{equation}
\label{eq-Phi-dagger-nu-equal-rank-x-times-Phi-U-bla}
(\Phi_*)^\dagger(\nu) \ \ \ = \ \ \ \rank(x)(\Phi_{\U_*})^\dagger(\nu).
\end{equation}
Indeed, the functor $S_{\Delta_X}[-2]$ has kernel $\Delta_{X_*}\omega_X$, so 
the natural transformation $(\Phi_{\U_*})^\dagger\nu$ is induced by a homomorphism $f$ in
$\Hom(\StructureSheaf{\Delta_X},\Delta_{X_*}\omega_X)$, and it thus commutes with 
the functor $G^\dagger$, i.e., the homomorphism (natural transformation)
$G^\dagger f:\Delta_{X_*}(x^\vee)\rightarrow \Delta_{X_*}(x^\vee\otimes\omega_X)$
is equal to the homomorphism $f G^\dagger$. We conclude Equality
(\ref{eq-Phi-dagger-nu-equal-rank-x-times-Phi-U-bla}) 
from Equality (\ref{eq-tr-composed-with-mu-is-scaling-by-rank}).
}\!\!\!\!
The morphisms $\Phi(t_X):\Phi(\StructureSheaf{X})\rightarrow \Phi(\StructureSheaf{X})[2]$ and 
$\mu_{\Phi(\StructureSheaf{X})[2]}(\tilde{\nu}):\Phi(\StructureSheaf{X})[2]\rightarrow 
\Phi(\StructureSheaf{X})\otimes\omega_M[2n]$
compose to yield the morphism
\[
\mu_{\Phi(\StructureSheaf{X})[2]}(\tilde{\nu})\circ \Phi(t_X):\Phi(\StructureSheaf{X})\rightarrow 
\Phi(\StructureSheaf{X})\otimes\omega_M[2n].
\]
For any two objects $x_1$, $x_2$ of $D^b(X)$, we have the homomorphisms
\[
\Phi \ : \ \Hom(x_1,x_2) \ \ \ \rightarrow \ \ \ \Hom(\Phi(x_1),\Phi(x_2))
\]
and its left adjoint with respect to the Serre Duality pairing
\[
\Phi^\dagger_L \ : \ \Hom(\Phi(x_2),\Phi(x_1)\otimes \omega_M[2n]) \ \ \ \rightarrow \ \ \ 
\Hom(x_2,x_1\otimes \omega_X[2]).
\]
The morphism $\mu_{\Phi(\StructureSheaf{X})[2]}(\tilde{\nu})$ 
belongs to $\Hom(\Phi(x_2),\Phi(x_1)\otimes \omega_M[2n])$,
for $x_1=\StructureSheaf{X}$ and $x_2=\StructureSheaf{X}[2]$.
Hence, $\Phi^\dagger_L\left[\mu_{\Phi(\StructureSheaf{X})[2]}(\tilde{\nu})\right]$ belongs to 
$\Hom(\StructureSheaf{X},\StructureSheaf{X}\otimes\omega_X)$.
The equality
\begin{equation}
\label{eq-prop-3.1-of-caldararu-Mukai-pairing-I}
Tr_M\left(\tr\left[\mu_{\Phi(\StructureSheaf{X})[2]}(\tilde{\nu})\circ \Phi(t_X)\right]\right) \ \ \ = \ \ \ 
Tr_X
\left(\Phi^\dagger_L\left[\mu_{\Phi(\StructureSheaf{X})[2]}(\tilde{\nu})\right]\circ t_X\right)
\end{equation}
is established in \cite[Prop. 3.1]{cal-mukai1}.

We prove next that the left hand sides of equations
(\ref{eq-holds-for-all-nu}) and 
(\ref{eq-prop-3.1-of-caldararu-Mukai-pairing-I}) are equal. The equality
\[
\tr\left[\mu_{\Phi(\StructureSheaf{X})[2]}(\tilde{\nu})\circ \Phi(t_X)\right] \ \ \ = \ \ \ 
\tilde{\nu}\circ \tr\left(\Phi(t_X)\right)
\]
holds, by \cite[Lemma 2.4]{cal-mukai1}. Now
\[
Tr_M\left(\tilde{\nu}\circ \tr\left(\Phi(t_X)\right)\right) \ \ \ = \ \ \ 
Tr_{M\times M}\left(
\nu\circ \eta_M\left[\tr\left(\Phi(t_X)\right)\right]
\right),
\]
by Lemma \ref{lemma-an-equality-of-traces},
applied with $X=M$, $Y=M\times M$, $f=\Delta_M$, $x=\StructureSheaf{M}[-2]$, 
$y=\StructureSheaf{M\times M}$, $t=\tr\left(\Phi(t_X)\right)$, and $\phi=\tilde{\nu}$.
Hence, the left hand sides of equations
(\ref{eq-holds-for-all-nu}) and 
(\ref{eq-prop-3.1-of-caldararu-Mukai-pairing-I}) are equal. 

It remains to prove that the right hand sides 
of equations (\ref{eq-holds-for-all-nu}) and 
(\ref{eq-prop-3.1-of-caldararu-Mukai-pairing-I}) are equal. 
The following relation between $\Phi^\dagger_L$ and $(\Phi_*)^\dagger$ holds, 
for every object $F$ of $D^b(X)$
\[
\Phi^\dagger_L\left(\mu_{\Phi(F)}(\tilde{\nu})\right) \ \ \ = \ \ \ 
\left[(\Phi_*)^\dagger\left(\Delta_{M_*}(\tilde{\nu})\right)\right]_F,
\]
by \cite[Prop. 3.1]{cal-mukai1}. Taking $F=\StructureSheaf{X}$, we get
\begin{equation}
\label{eq-relating-two-adjoint-homomorphisms}
\Phi^\dagger_L\left(\mu_{\Phi(\StructureSheaf{X})}(\tilde{\nu})\right) \ \ \ = \ \ \ 
\left[(\Phi_*)^\dagger\left(\Delta_{M_*}(\tilde{\nu})\right)\right]_{\StructureSheaf{X}} \ \ \ = \ \ \ 
\left[(\Phi_*)^\dagger(\nu)\right]_{\StructureSheaf{X}}.
\end{equation}
In addition, we have the 
following relation between the Serre Duality pairing  for morphisms in $D^b(X)$ and 
the Serre Duality pairing  for morphisms in $D^b(X\times X)$, or for natural transformation.
\begin{equation}
\label{eq-relating-sd-pairings}
\langle(\Phi_*)^\dagger(\nu), \ \eta_X(t_X)\rangle \ \ \ = \ \ \ 
\langle\left[(\Phi_*)^\dagger(\nu)\right]_{\StructureSheaf{X}}, \ t_X\rangle.
\end{equation}
Indeed, the morphism $(\Phi_*)^\dagger(\nu)$ belongs to
$\Hom(\Delta_{X_*}\StructureSheaf{X},\Delta_{X_*}\omega_X)$.
Hence, $(\Phi_*)^\dagger(\nu)=\Delta_{X_*}(\phi)$, for a  
morphism $\phi$ in $\Hom(\StructureSheaf{X},\omega_X)$.
Now apply 
Lemma \ref{lemma-an-equality-of-traces} with 
$Y=X\times X$, $f=\Delta_X$, $x=\StructureSheaf{X}[-2]$, $y=\StructureSheaf{X\times X}$, $t=t_X$,
and $\phi$ as above, to obtain Equation
(\ref{eq-relating-sd-pairings}).

Combining the last two equations above we get:
\[
\langle(\Phi_*)^\dagger(\nu), \ \eta_X(t_X)\rangle \ \ \ \stackrel{(\ref{eq-relating-sd-pairings})}{=} \ \ \ 
\langle\left[(\Phi_*)^\dagger(\nu)\right]_{\StructureSheaf{X}}, \ t_X\rangle \ \ \ 
\stackrel{(\ref{eq-relating-two-adjoint-homomorphisms})}{=} \ \ \ 
\langle \Phi^\dagger_L\left(\mu_{\Phi(\StructureSheaf{X})}(\tilde{\nu})\right), \ t_X\rangle. 
\]
This is precisely the desired equality
of the two right hand sides of 
equations (\ref{eq-holds-for-all-nu}) and 
(\ref{eq-prop-3.1-of-caldararu-Mukai-pairing-I}).
\end{proof}

\hide{
\begin{proof}
The homomorphism $\Phi_\U:\Ext^1(F,F)\rightarrow \Ext^1(\Phi_\U(F),\Phi_\U(F))$
is surjective, by Theorem \ref{thm-faithful}.
The Yoneda pairing $\Ext^1(F,F)\otimes \Ext^1(F,F)\rightarrow \Ext^2(F,F)$ is surjective.
Hence, $\Phi_\U(t_F)$ spans the image of the Yoneda pairing
\[
y:\Ext^1(\Phi_\U(F),\Phi_\U(F))\otimes \Ext^1(\Phi_\U(F),\Phi_\U(F))\rightarrow \Ext^2(\Phi_\U(F),\Phi_\U(F)).
\]
Now the symplectic structure of the moduli space of objects $\Phi_\U(F)$, $F^\vee\in M$, 
as a moduli space of objects on $M$, is induced by 
$t_M^{n-1}\circ \tr\circ y : \Wedge{2}\Ext^1(\Phi_\U(F),\Phi_\U(F))\rightarrow H^{2n}(M,\StructureSheaf{M})$.
(???) Fill in the details along the following lines.
For the analogue for locally free hyper-holomorphic sheaves, see (\cite{verbitsky-1996}, Proposition 9.1).
This construction should generalize for the hyperholomorphic reflexive sheaves $\E_G$, the
restriction of $\E$ to $\{G\}\times M$. Next we have the following commutative diagram.
Denote by $\iota:\E_G[-1]\rightarrow \Phi_\U(F)$ the natural morphism.
\[
\xymatrix{
\Ext^1(\E_G[-1],\E_G[-1])\otimes \Ext^1(\E_G[-1],\E_G[-1])  \ar[r] \ar[d]_{\iota_*\otimes\iota^*}
&
\Ext^2(\E_G[-1],\E_G[-1]) \ar[dr]^{\tr} 
\\
\Ext^1(\E_G[-1],\Phi_\U(F))\otimes \Ext^1(\Phi_\U(F),\E_G[-1])
\ar[ur]^{\circ} \ar[dr]_{\circ} && H^2(M,\StructureSheaf{M}).
\\
\Ext^1(\Phi_\U(F),\Phi_\U(F))\otimes \Ext^1(\Phi_\U(F),\Phi_\U(F)) \ar[r]  \ar[u]^{\iota^*\otimes\iota_*}&
\Ext^2(\Phi_\U(F),\Phi_\U(F)) \ar[ur]_{\pm\tr} 
 }
\]
We need to prove that both left vertical homomorphisms have the same image, e.g., by 
proving that $\iota_*$ and $\iota^*$ induce isomorphisms. A lot of work towards the
proof is included in a deleted section of an old version of this 
paper from February 24, 2010 (Eyal still has the file).
\end{proof}
}

\begin{example}
Let us verify Equation (\ref{eq-t-M-does-not-vanish}) in a simple case. 
Take $n=1$, identify $X$ with $M:=X^{[1]}$, and set 
$\U:=\Ideal{\Delta}$ to be the ideal sheaf of the diagonal in $X\times X^{[1]}$. Choose
a sheaf $F$ on $X$ satisfying $h^i(F)=0$, for $i>0$. Consider the short exact sequence
\[
0\rightarrow \U \rightarrow \StructureSheaf{X\times X}\rightarrow
\StructureSheaf{\Delta_{X}}\rightarrow 0.
\]
Then $\rank(\Phi_\U(F))=\chi(F)-\rank(F)$ and 
$\Phi_\U(F)$ fits in the exact triangle
\[
F[-1]\rightarrow \Phi_\U(F) \rightarrow H^0(F)\otimes_\ComplexNumbers\StructureSheaf{X}
\rightarrow F.
\]
Furthermore, $[\varphi_\U(t_X)]_2=-t_X$, under the identification of $X$ with $M:=X^{[1]}$. 
The Fourier-Mukai transform $\Phi_{\StructureSheaf{X\times X}}$
with respect to the structure sheaf $\StructureSheaf{X\times X}$
sends $\mu_F(t_X):F\rightarrow F[2]$ to zero. Indeed, consider the cartesian diagram
\[
\xymatrix{ X\times X \ar[r]^{\pi_2} \ar[d]_{\pi_1} 
&
X \ar[d]_\kappa
\\
X \ar[r]_{\kappa} & \{\pt\}.
}
\]
Then $R\pi_{2_*}\pi_1^*(\mu_F(t_X))=\kappa^*R\kappa_*(\mu_F(t_X))=\kappa^*\kappa_*(\mu_F(t_X))$
and the morphism
$\kappa_*(\mu_F(t_X)):H^0(F)\rightarrow H^0(F)[2]$ in $D^b(\{\pt\})$ vanishes. 
We get the commutative  diagram
\[
\xymatrix{ 
\Phi_\U(F) \ar[r] \ar[d]_{\Phi_\U(\mu_F(t_X))}
& H^0(F) \otimes_\ComplexNumbers\StructureSheaf{X}\ar[r] \ar[d]_{\Phi_{\StructureSheaf{X\times X}}(\mu_F(t_X))=0}
& F\ar[d]_{\mu_F(t_X)}
\\
\Phi_\U(F)[2] \ar[r] & 
H^0(F) \otimes_\ComplexNumbers\StructureSheaf{X}[2] \ar[r] &
F[2].
}
\]
We see that indeed $\tr(\Phi_\U(\mu_F(t_X)))=-\rank(F)t_X=\rank(F)[\varphi_\U(t_X)]_2.$ 
Note, by the way, that $\Phi_\U:\Hom(F,F[2])\rightarrow \Hom(\Phi_\U(F),\Phi_\U(F)[2])$ is an isomorphism.
If the sheaf $F$ is simple, then $\Hom(\Phi_\U(F),\Phi_\U(F)[2])$ is one-dimensional.
In that case we get the equality, 
\[
\rank(\Phi_\U(F))\Phi_\U(\mu_F(t_X))=\rank(F)\mu_{\Phi_\U(F)}([\varphi_\U(t_X)]_2),
\]
since both sides above have the same trace.
\end{example}

%
\subsection{A relation in the Yoneda algebra of $\Phi_\U(x)$}
\label{sec-a-relation-in-the-yoneda-algebra}

Let $t_X$ be a non-zero element of $H^2(X,\StructureSheaf{X})$ and 
$\varphi_\U:H^*(X,\CC)\rightarrow H^*(M,\CC)$ 
the homomorphism in Proposition \ref{prop-a-topological-formula-for-a-trace}.
Set
\begin{equation}
\label{eq-t-M}
t_M := [\varphi_\U(t_X)]_2.
\end{equation}

\begin{lem}
\label{lemma-t-M-does-not-vanish}
The class $t_M$ spans $H^{2,0}(M)$. Furthermore, the following equality holds, for every object 
$x\in D^b(X)$.
\begin{equation}
\label{eq-trace-is-rank-x-times-t-M}
\tr\left(\Phi_\U(\mu_x(t_X))\right) \ \ \ = \ \ \rank(x) t_M.
\end{equation}
\end{lem}

\begin{proof}
The statement follows immediately from 
Theorem \ref{Yoshioka-main} and Proposition
\ref{prop-a-topological-formula-for-a-trace}. 
\end{proof}

We regard $t_X$ also as an element of the subspace 
$\Ext^2(\StructureSheaf{X},\StructureSheaf{X})$ of $Y(\StructureSheaf{X})$.
Given an object $x$ of $D^b(X)$, 
set 
\begin{eqnarray*}
t_x&:=&\mu_x(t_X)\\
t_{\Phi_\U(x)}&:=&\mu_{\Phi_\U(x)}(t_M). 
\end{eqnarray*}
Let $v(x):=ch(x)\sqrt{td_X}$ be the Mukai vector of $x$.
Then $t_{\Phi_\U(x)}$ is an element of $\Ext^2\left(\Phi_\U(x),\Phi_\U(x)\right)$ satisfying
\begin{equation}
\label{eq-of-traces}
\tr\left[t_{\Phi_\U(x)}\right]=\tr\left[\mu_{\Phi_\U(x)}\left(t_M\right)\right]=
\rank(\Phi_\U(x))\cdot t_M=-\left(v,v(x)^\vee\right)t_M.
\end{equation}

\begin{lem}
\label{lemma-equality-in-Ext-2n}
Let $x$ be an object of $D^b(X)$ satisfying $\Hom(x,x)\cong \ComplexNumbers$. 
The following equation holds in $\Ext^{2n}\left(\Phi_\U(x),\Phi_\U(x)\right)$.
\begin{equation}
\label{eq-relation-in-yoneda-algebra}
-(v,v(x)^\vee)
\left(t_{\Phi_\U(x)}\right)^{n-1} \Phi_\U(t_x)=
\rank(x)
\left(t_{\Phi_\U(x)}\right)^n.
\end{equation}
\end{lem}

\begin{proof}
The vector space 
$\Ext^{2n}\left(\Phi_\U(x),\Phi_\U(x)\right)$ is dual to $\Hom(\Phi_\U(x),\Phi_\U(x))$, which 
is isomorphic to $\Hom(x,x)$, by Theorem \ref{thm-A-is-direct-sum}, and is thus one-dimensional.
Consequently, the two sides of Equation (\ref{eq-relation-in-yoneda-algebra}) are linearly dependent.
Equation (\ref{eq-relation-in-yoneda-algebra}) would thus follow from Equations
(\ref{eq-trace-is-rank-x-times-t-M}) and  (\ref{eq-of-traces}), once we 
prove that the Yoneda product
\[
\left(t_{\Phi_\U(x)}\right)^{n-1} \circ \ : \ 
\Ext^{2}\left(\Phi_\U(x),\Phi_\U(x)\right) \rightarrow \Ext^{2n}\left(\Phi_\U(x),\Phi_\U(x)\right)
\]
factors through $\tr:\Ext^{2}\left(\Phi_\U(x),\Phi_\U(x)\right) \rightarrow H^2(M,\StructureSheaf{M}).$
We prove this factorization next. Note that the morphism
$\mu_{\Phi_\U(x)}:\StructureSheaf{M}\rightarrow \Phi_\U(x)^\vee\otimes \Phi_\U(x)$ is compatible with the 
Yoneda product.
Hence, $\left(t_{\Phi_\U(x)}\right)^{n-1}=\mu_{\Phi_\U(x)}\left(t_M^{n-1}\right)$.
For every object $y$ of $D^b(M)$, and for every integers $i$ and $j$, the outer square of the following
diagram is commutative.
\[
\xymatrix{
\Ext^i(y,y)\otimes \Ext^j(y,y) \ar[r] & \Ext^{i+j}(y,y) \ar[dd]_{\tr}
\\
\Ext^i(y,y)\otimes H^j(\StructureSheaf{M}) \ar[u]^{id\otimes \mu} \ar[d]_{\tr\otimes id} \ar[ur]^{\alpha}
\\
H^i(\StructureSheaf{M})\otimes H^j(\StructureSheaf{M})
\ar[r] & H^{i+j}(\StructureSheaf{M}).
}
\]
The homomorphism $\alpha$, defined to make the diagram commutative,  
factors through the bottom left vertical 
homomorphism $\tr\otimes id$,
whenever the right vertical trace homomorphism
is an isomorphism. Apply it with $y=\Phi_{\U}(x)$, $i=2$, $j=2n-2$, 
the element $t_M^{n-1}$ of $H^{2n-2}(\StructureSheaf{M})$,
and observe that 
the trace homomorphism $\tr:\Ext^{2n}(\Phi_{\U}(x),\Phi_{\U}(x))\rightarrow H^{2n}(\StructureSheaf{M})$
is an isomorphism,  by Remark \ref{remark-non-vanishing-of-tr}, and the fact that
$\Hom(\Phi_{\U}(x),\Phi_{\U}(x))\cong\Hom(x,x)\cong\ComplexNumbers$.
The Equality (\ref{eq-relation-in-yoneda-algebra}) 
now follows, where the coefficient on its left hand side is explained by
the equality $\rank\left(\Phi_\U(x)\right)=-\left(v,v(x)^\vee\right)$.
\end{proof}

Both sides of Equation (\ref{eq-relation-in-yoneda-algebra}) vanish, whenever $\rank(x)=0$
or $\rank(\Phi_\U(x))=0$ (that is $-\left(v,v(x)^\vee\right)=0$). This can be seen directly, without using the above lemma, as follows.
The left hand side vanishes if $\rank(x)=0$, since $t_x:=\mu_x(t_X)$ vanishes, by Remark
\ref{remark-non-vanishing-of-tr}. The right hand side vanishes if $\rank(\Phi_\U(x))=0$, 
since $\left(t_{\Phi_\U(x)}\right)^n:=(\mu_{\Phi_\U(x)}(t_M))^n=\mu_{\Phi_\U(x)}((t_M)^n)$
vanishes, by Remark \ref{remark-non-vanishing-of-tr} again.
Assume next that $\Hom(x,x)$ is one-dimensional.  
Let $tr_x^{-1}$ be the inverse of the isomorphism given in Equation (\ref{eq-trace-isomorphism-for-simple-x}).
If $\rank(x)$ and $\rank(\Phi_\U(x))$ do not vanish, Equation (\ref{eq-relation-in-yoneda-algebra}) 
is equivalent to the following equation:
\begin{equation}
\label{eq-reformulated-relation-in-yoneda-algebra}
\left(t_{\Phi_\U(x)}\right)^{n-1} \Phi_\U(tr_x^{-1}(t_X))=
tr^{-1}_{\Phi_\U(x)}\left(t_M^n\right).
\end{equation}
We will verify Equation (\ref{eq-reformulated-relation-in-yoneda-algebra})
assuming only that $\rank(\Phi_\U(x))$ does not vanish 
(Theorem \ref{thm-the-yoneda-algebra-ofPhi-U-of-a-simple-sheaf}).
We expect the above equation to hold even if  $\rank(\Phi_\U(x))$
vanishes.

\hide{
The functor $Y$, of tensorization over $\ComplexNumbers$ by 
$\oplus_{i=0}^n\Ext^{2i}(\StructureSheaf{M},\StructureSheaf{M})[-2i]$, decomposes as a direct sum. 
The morphism $a\circ q_{T(x)} : YT(x)\rightarrow T(x)$ decomposes accordingly as a
direct sum of
\[
Q_i(a) \ : \ T(x)\otimes_\ComplexNumbers \Ext^{2i}(\StructureSheaf{M},\StructureSheaf{M})[-2i]
\rightarrow T(x).
\]
Let $(t_M^i)^*$ be the basis of $\Ext^{2i}(\StructureSheaf{M},\StructureSheaf{M})^*$
dual to $t_M^i$. Evaluating $Q_i(a)$ at $(t_M^i)^*$ we get the morphism
\[
Q_i\left(a,(t_M^i)^*\right) \ : \ T(x)[-2i] \rightarrow T(x).
\]
The $\widetilde{\YY}$-module axiom implies the equality $Q_i\left(a,(t_M^i)^*\right)=Q_1\left(a,t_M^*\right)^i$.
Lemma \ref{lemma-equality-in-Ext-2n} translates to the following relation in $\Hom(T(x)[-2n],T(x))$:
\begin{equation}
\label{eq-equality-in-Ext-2n-in-terms-of-monads}
(-v,v(x)^\vee) Q_1\left(a,t_M^*\right)^{n-1} T(\mu_x(t_X))
\ \ \ = \ \ \ 
\rank(x) Q_1\left(a,t_M^*\right)^n.
\end{equation}
The natural transformation $q:Y\rightarrow T$ is independent of the object $x$.
The dependence on $x$ of the coefficients $(-v,v(x)^\vee)$ and $\rank(x)$ in 
equation (\ref{eq-equality-in-Ext-2n-in-terms-of-monads})
is due to the dependence of $\mu_x$ on the object $x$, as well as the 
dependence of the action morphism 
$a:T^2(x)\rightarrow T(x)$ of the $\TT$-module $(T(x),a)$ on the object $x$.
In other words, given two {\em different} natural transformations $\tau_1$, $\tau_2$ from 
$T[-2n]$ to $T$, and a collection $Spl$ of objects, such that $\Hom(T(x)[-2n],T(x))$ is one dimensional,
for all $x\in Spl$, 
the ratio between $\tau_1(x)$ and $\tau_2(x)$, $x\in Spl$, depends in general on the object $x$.
}

%
\subsection{The natural transformation $h_2$ is the Mukai vector}
\label{subsection-structure-of-yoneda-algebra}

Set $R(M):=\Ext^{2n}(\StructureSheaf{M},\StructureSheaf{M})[-2n]$, regarded as an object of $D^b(\pt)$.
The objects  of the exact triangle displayed in Equation
(\ref{eq-split-exact-triangle}) correspond to kernels of integral endo-functors of $D^b(X)$.
The object $\R$ corresponds to the functor of tensorization by $R(M)$ over $\ComplexNumbers$.
The object $\Y:=\StructureSheaf{\Delta_X}\otimes_\ComplexNumbers Y(\StructureSheaf{M})$ 
corresponds to the functor of tensorization by $Y(\StructureSheaf{M})$ over $\ComplexNumbers$.
The object $\A$ is the kernel of the functor $\Psi_\U\Phi_\U$. 
The morphisms of the exact triangle (\ref{eq-split-exact-triangle}) correspond to natural transformations
between these endo-functors.

Given objects $F_1$, $F_2$ of $D^b(X)$, we get the short exact sequence
\begin{equation}
\label{eq-degree-zero-short-exact-sequence}
0\rightarrow 
\Hom\left(F_1,F_2\otimes_\ComplexNumbers R(M)\right) \LongRightArrowOf{h_{F_2}}
\Hom\left(F_1,F_2\otimes_\ComplexNumbers Y(\StructureSheaf{M})\right) \LongRightArrowOf{q_{{}_{F_2}}}
\Hom\left(F_1,\Psi_\U\Phi_\U(F_2)\right)\rightarrow 0.
\end{equation}
Exactness of the above sequence follows from the splitting of the exact triangle 
(\ref{eq-split-exact-triangle}). Equivalently, we have the short exact sequence
\begin{eqnarray*}
0&\rightarrow &
\Hom\left(F_1,F_2[-2n]\right)\otimes \Ext^{2n}(\StructureSheaf{M},\StructureSheaf{M}) \RightArrowOf{h_{F_2}}
\oplus_{i=0}^{n}\Hom\left(F_1,F_2[-2i]\right) \otimes  \Ext^{2i}(\StructureSheaf{M},\StructureSheaf{M}) 
\\
&\RightArrowOf{Q}&
\Hom\left(\Phi_\U(F_1),\Phi_\U(F_2)\right)\rightarrow 0.
\end{eqnarray*}

Set $Y^{2k}:=H^{2k}(M,\StructureSheaf{M})$, so that $Y^{2k}[-2k]$ is 
the graded summand of $Y(\StructureSheaf{M})$ of degree $2k$.
Let $t_X\in H^2(X,\StructureSheaf{X})$ be a non-zero class, 
$t_M\in H^2(M,\StructureSheaf{M})$ the class associated to $t_X$ in Equation
(\ref{eq-t-M}), and let $t_F$ be the class $\mu_F(t_X)$ in $\Ext^2(F,F)$.
Write $h=\tilde{h}_0\otimes 1+\tilde{h}_2\otimes t^*_M+ \tilde{h}_4\otimes (t^*_M)^2$,
using the notation of Theorem
\ref{thm-universal-class-h-in-Hochschild-cohomology-of-X}.
Above, $\tilde{h}_{2j}$ is a natural transformation from the identity functor $\id$ of $D^b(X)$ to
$\id[2j]$. 

\begin{thm}
\label{thm-Homs-for-objects-in-the-image-of-Phi-U}
Let $F_1$ and $F_2$ be objects of $D^b(X)$. 
\begin{enumerate}
\item
\label{cor-item-graded-short-exact-seq-of-Y-M-modules}
The following is a short exact sequence 
\[
0\rightarrow 
\Hom^\bullet(F_1,F_2[-2n])\otimes Y^{2n}
\RightArrowOf{h}
\Hom^\bullet(F_1,F_2)\otimes Y(\StructureSheaf{M})
\RightArrowOf{Q} 
\Hom^\bullet\left(\Phi_\U(F_1),\Phi_\U(F_2)\right)
\rightarrow 0,
\]
where $h$ and $Q$ are homomorphisms of degree $0$,
and $Q(g\otimes y)=\mu_{\Phi_\U(F_2)}(y)\Phi_\U(g)$.
\item
\label{cor-item-sheaf-case-has-on-relation-in-degree-2n}
If $F_1$ and $F_2$ are sheaves on $X$, then $\Hom(\Phi_\U(F_1),\Phi_\U(F_2)[k])=0,$
for $k>2n$ and for $k<0$. The homomorphism $Q$ restricts to an isomorphism for degrees
in the range $0\leq k\leq 2n-1$. In degree $2n$ we get the short exact sequence
\[
0\rightarrow \Hom(F_1,F_2)\otimes Y^{2n}\RightArrowOf{h}
\left[
\begin{array}{c}
\Ext^2(F_1,F_2)\otimes Y^{2n-2}
\\
\oplus
\\
\Hom(F_1,F_2)\otimes Y^{2n}
\end{array}
\right]
\RightArrowOf{Q}
\Hom\left(
\Phi_\U(F_1),\Phi_\U(F_2)[2n]
\right)
\rightarrow 0,
\]
where $h$ is given by the equality
\[
h(f\otimes t_M^n) \ \ \ = \ \ \ 
(\tilde{h}_{2_{F_2}}\circ f)\otimes t_M^{n-1} +
(\tilde{h}_{0_{F_2}}\circ f)\otimes t_M^n,
\]
for all $f\in \Hom(F_1,F_2)$.
Consequently, if in addition $\Hom(F_1,F_2)=0$, then $Q$ induces an isomorphism in degree $2n$ as well.
\item
\label{thm-item-yoneda-algebra-of-Phi-F-for-a-simple-sheaf-F}
When $F$ is a simple sheaf on $X$, the kernel of 
\[
Q\ : \ 
\left[
\begin{array}{c}
\Ext^2(F,F)\otimes Y^{2n-2}
\\
\oplus
\\
\Hom(F,F)\otimes Y^{2n}
\end{array}
\right]
\ \ \ \longrightarrow \ \ \ 
\Hom\left(
\Phi_\U(F),\Phi_\U(F)[2n]
\right)
\]
is spanned by the element
\begin{equation}
\label{eq-relation-in-tensor-product-of-Yoneda-algebras}
-(v,v(F^\vee))tr^{-1}_F(t_X)\otimes t_M^{n-1} - 1\otimes t_M^n,
\end{equation}
where $tr_F:\Ext^2(F,F)\rightarrow H^2(X,\StructureSheaf{X})$ is the isomorphism
given in Equation (\ref{eq-trace-isomorphism-for-simple-x}).
\end{enumerate}
\end{thm}

\begin{proof}
(\ref{cor-item-graded-short-exact-seq-of-Y-M-modules})
The short exact sequence in formula (\ref{eq-degree-zero-short-exact-sequence})
establishes the statement in degree zero. For degree $k$, replace $F_2$ by the object $F_2[k]$.
We have the equalities
\[
(q_{F_2})_*(g\otimes y)=\Xi_\U\left(\mu_{\pi_X^*F_2}(\pi_M^*(y))\circ \pi_X^*(g)\right)=
\mu_{\Xi_\U(\pi_X^*F_2)}(y)\circ\Xi_\U( \pi_X^*(g))=
\mu_{\Phi_\U(F_2)}(y)\circ \Phi_\U(g),
\]
where the first equality follows by Lemma \ref{lemma-push-forward-by-q-x-2-is-Xi-U},
the second is due to the fact that the kernel $\Delta_{23_*}(\U)$ in $D^b((X\times M)\times M)$ of the integral functor $\Xi_\U$
is supported on the diagonal, and the third since $\Phi_\U=\Xi_\U\circ \pi_X^*$. The formula for $Q(g\otimes y)$ follows.

(\ref{cor-item-sheaf-case-has-on-relation-in-degree-2n})
Note first that $\Ext^{j}(F_1,F_2)$ vanishes, for $j\not\in\{0,1,2\}$, since $F_1$ and $F_2$
are sheaves. 
When $F_1$ and $F_2$ are sheaves, and $k>2n$ or $k<0$, then the degree $k$ component of 
$h$ is an injective homomorphism 
from a finite dimensional vector space to itself, by part \ref{cor-item-graded-short-exact-seq-of-Y-M-modules}, 
hence an isomorphism. 
If $k$ is in the range $0\leq k\leq 2n-1$, the space
$\Hom(F_1,F_2[k-2n])\otimes Y^{2n}$ vanishes, hence the degree $k$ component of $Q$ is an isomorphism.
The statement in degree $2n$ follows easily from 
part \ref{cor-item-graded-short-exact-seq-of-Y-M-modules}.

(\ref{thm-item-yoneda-algebra-of-Phi-F-for-a-simple-sheaf-F})
If $\rank(F)\neq 0$, then 
the element given in  (\ref{eq-relation-in-tensor-product-of-Yoneda-algebras}) belongs to the kernel of $m_{2n}$, 
by Equations (\ref{eq-tr-composed-with-mu-is-scaling-by-rank}) and 
(\ref{eq-relation-in-yoneda-algebra}). This element must span the kernel, as the kernel is
one-dimensional, by part \ref{cor-item-sheaf-case-has-on-relation-in-degree-2n}. The kernel is spanned by
\[
h(1_F\otimes t_M^n) =  
\tilde{h}_{0_F}\otimes t_M^n+\tilde{h}_{2_F}\otimes t^{n-1}_M
\]
as well.
Hence, $\tilde{h}_{2_F}= \tilde{h}_{0_F}(v,v(F^\vee))tr_F^{-1}(t_X)$, for every simple sheaf $F$ 
satisfying the inequality $\rank(F)\neq 0$. 
We conclude that $\tilde{h}_0$ does not vanish and we normalize the natural transformation $h$
by rescaling it, so that $\tilde{h}_{0_F}=-1.$ Then 
\begin{equation}
\label{eq-topological-formula-for-tilde-h-2-F}
\tilde{h}_{2_F}= -(v,v(F^\vee))tr_F^{-1}(t_X),
\end{equation}
for every simple object $F$ satisfying the inequality $\rank(F)\neq 0$. 
Let $\sigma_X\in H^0(X,\omega_X)$ be the class, such that $Tr_X(t_X\otimes\sigma_X)=1.$
The class $h_2:=\tilde{h}_2\otimes\sigma_X$ belongs to $HH_0(X)$ and the above equation translates to
$Tr_X(h_{2_F})=-(v,v(F^\vee)).$ The latter equality holds for all  simple sheaves $F$ 
satisfying $\rank(F)\neq 0$.
This suffices to determines the algebraic part of the class $\widetilde{I}_*^X(h_2)$ in $H\Omega_0(X)$
as $ch(v)\sqrt{td_X}$, by Lemma \ref{lemma-chern-character-as-a-linear-functional-on-HH}.
Hence, Equality (\ref{eq-topological-formula-for-tilde-h-2-F}) holds regardless of the vanishing of $\rank(F)$.
Part (\ref{thm-item-yoneda-algebra-of-Phi-F-for-a-simple-sheaf-F}) of the Theorem follows.
\end{proof}

\begin{proof}[Proof of part  \ref{thm-item-h-2} of Theorem \ref{thm-universal-class-h-in-Hochschild-cohomology-of-X}]
Write $H^{1,1}(X,\ComplexNumbers)=[H^{1,1}(X,\Integers)\otimes_\Integers\ComplexNumbers]\oplus\Theta(X)$,
where $\Theta(X)$ is the transcendental subspace. 
Both $ch(v)$ and $I^X_*(h_2)$ are elements of
\[
H^0(X,\ComplexNumbers)\oplus H^{1,1}(X,\ComplexNumbers)\oplus H^4(X,\ComplexNumbers).
\]
Now $ch_1(v)$ belongs to $H^{1,1}(X,\Integers)\otimes_\Integers\ComplexNumbers$ and 
Equation (\ref{eq-topological-formula-for-tilde-h-2-F}) in the proof of Theorem
\ref{thm-Homs-for-objects-in-the-image-of-Phi-U} (\ref{thm-item-yoneda-algebra-of-Phi-F-for-a-simple-sheaf-F})
implies that the projection of $I^X_*(h_2)$
to $H^{1,1}(X,\Integers)\otimes_\Integers\ComplexNumbers$ is equal to $ch_1(v)$.
Varying the surface $X$ in a codimension one family in moduli keeping $ch_1(v)$ of Hodge type $(1,1)$, 
the classes $ch(v)$ and $I^X_*(h_2)$
define two continuous sections of the Hodge bundle $\H^{1,1}$ with fiber $H^{1,1}(X,\ComplexNumbers)$, 
the difference of which is purely transcendental. But a purely transcendental continuous section 
of $\H^{1,1}$ over such a family must vanish, by the density of Hodge structures with trivial transcendental subspace.
\end{proof}

Let $F$ be a simple sheaf on $X$ and $Y(F):=\oplus_{i=0}^2\Ext^i(F,F)[-i]$  its Yoneda algebra.
Let $Y(\Phi_\U(F)):=\oplus_{i\in\Integers}\Hom(\Phi_\U(F),\Phi_\U(F)[i])[-i]$ be the Yoneda algebra of 
$\Phi_\U(F)$. Let
\[
f:Y(F)\otimes_\ComplexNumbers \ComplexNumbers[z]\rightarrow Y(\Phi_\U(F))
\]
be the algebra homomorphism restricting to $Y(F)\otimes 1$ as $\Phi_\U$ and sending the indeterminate $z$
to $\mu_{\Phi_\U(F)}(t_M)$ in $\Hom(\Phi_\U(F),\Phi_\U(F)[2])$. The homomorphism $f$ is well defined, since 
$\mu_{\Phi_\U(F)}(t_M)$ belongs to the center of $Y(\Phi_\U(F))$, as $t_M$ is a
natural transformation from $\id_{D^b(M)}$ to $\id_{D^b(M)}[2].$

\begin{thm}
\label{thm-the-yoneda-algebra-ofPhi-U-of-a-simple-sheaf}
The homomorphism $f$ is surjective and its kernel is the ideal generated by 
\begin{equation}
\label{eq-generator-of-principal-ideal}
1\otimes z^n+(v,v(F^\vee))tr_F^{-1}(t_X)\otimes z^{n-1}.
\end{equation}
\end{thm}

Note that the equality $f(1\otimes z^n)=-f\left((v,v(F^\vee))tr_F^{-1}(t_X)\otimes z^{n-1}\right)$ implies the vanishing of 
$f(1\otimes z^{n+1})$, since $(tr_F^{-1}(t_X))^2=0.$ Hence, $f$ factors through $Y(F)\otimes Y(\StructureSheaf{M})$.

\begin{proof}
The algebra $Y(F)\otimes_\ComplexNumbers \ComplexNumbers[z]$ is graded,
where $z$ has degree $2$. 
Let $Y'$ be its quotient 
by the ideal generated by the homogeneous element (\ref{eq-generator-of-principal-ideal}). 
The homomorphism $f$ factors through a homomorphism 
$\bar{f}:Y'\rightarrow Y(\Phi_\U(F))$, by 
Theorem
\ref{thm-Homs-for-objects-in-the-image-of-Phi-U} (\ref{thm-item-yoneda-algebra-of-Phi-F-for-a-simple-sheaf-F}).
The algebra $Y'$ is graded. Denote by $Y'_d$ its graded summand of degree $d$. Then
\[
\dim(Y'_d) = \left\{
\begin{array}{cl}
\dim\Ext^1(F,F) & \mbox{if} \ d \ \mbox{is odd and} \ 0< d< 2n,
\\
2 & \mbox{if} \ d \ \mbox{is even and} \ 0< d< 2n,
\\
1  & \mbox{if} \ d =0 \ \mbox{or} \  d=2n,
\\
0 &  \mbox{otherwise}. 
\end{array}
\right.
\]
These are precisely the dimensions of the graded summands of $Y(\Phi_\U(F))$, by Theorem
\ref{thm-Homs-for-objects-in-the-image-of-Phi-U} (\ref{cor-item-sheaf-case-has-on-relation-in-degree-2n}).
Hence, it suffices to prove that the homomorphism
$\bar{f}$ is surjective. Surjectivity follows from 
Theorem
\ref{thm-Homs-for-objects-in-the-image-of-Phi-U} (\ref{cor-item-sheaf-case-has-on-relation-in-degree-2n}).
\end{proof}

\hide{
\begin{thm}
\label{thm-yoneda}
Let $F$ be a simple sheaf on $X$, set $w:=-v(F^\vee)$, and assume that
$(v,w)\neq 0$.
\begin{enumerate}
\item
\label{thm-item-generaotrs-for-ideal}
Let $m:Y(F)\otimes_{\ComplexNumbers}Y(\StructureSheaf{M})  \rightarrow Y\left(\Phi_\U(F)\right)$ be the  composition
\[
Y(F)\otimes_{\ComplexNumbers} Y(\StructureSheaf{M})
\LongRightArrowOf{\Phi_\U(F)\otimes \mu_{\Phi_\U}}
Y(\Phi_\U(F))\otimes_{\ComplexNumbers} Y(\Phi_\U(F)) 
\LongRightArrowOf{\circ} 
Y\left(\Phi_\U(F)\right).
\]
Then $m$ is a surjective homomorphism of graded algebras.
Its kernel consists of the graded sumands of degree larger than $2n$, as well as
the one dimensional subspace of degree $2n$ spanned by the element
\begin{equation}
\label{eq-relation-in-tensor-product-of-Yoneda-algebras}
t_F\otimes t_M^{n-1} - 1\otimes t_M^n.
\end{equation}
\item
\label{thm-item-m-restricts-to-an-isomorphism}
Let $\iota:\lambda_n\rightarrow Y(\StructureSheaf{M})$ be the morphism from the direct summand of 

degree $<2n$. 
The homomorphism $m$ restricts to an isomorphism of objects in $D^b(\pt)$
\begin{equation}
\label{eq-m-restrict-to-isomorphism}
m : Y(F)\otimes_{\ComplexNumbers}\lambda_n \ \ \longrightarrow \ \ Y(\Phi_\U(F)).
\end{equation}
\item
\label{thm-item-generators-for-submodule}
Let $G_1$ and $G_2$ be two non-isomorphic stable sheaves representing two points of $M$. 
Set $F_i:=G_i^\vee$. 
Denote by $m$
the composition 
\begin{eqnarray*}
\Hom^\bullet(F_1,F_2)\otimes_{\ComplexNumbers}Y(\StructureSheaf{M})
&\LongRightArrowOf{\Phi_\U\otimes\mu_{\Phi_\U(F)}}&
\Hom^\bullet\left(\Phi_\U(F_1),\Phi_\U(F_2)\right)\otimes_{\ComplexNumbers}
\Hom^\bullet\left(\Phi_\U(F_2),\Phi_\U(F_2)\right)  \\
& \LongRightArrowOf{\circ} &
\Hom^\bullet\left(\Phi_\U(F_1),\Phi_\U(F_2)\right).
\end{eqnarray*}
Then $m$ 
is a surjective $Y(F_2)\otimes_{\ComplexNumbers} Y(\StructureSheaf{M})$ bi-module  homomorphism.
Its kernel consists of the graded sumands of degree larger than $2n$. In particular, 
$m$ restricts to an isomorphism 
\begin{equation}
\label{eq-m-restrict-to-isomorphism-case-of-F-1-F-2}
m : \Hom^\bullet(F_1,F_2)\otimes_{\ComplexNumbers}\lambda_n \ \ \longrightarrow \ \ 
\Hom^\bullet\left(\Phi_\U(F_1),\Phi_\U(F_2)\right).
\end{equation}
\end{enumerate}
\end{thm}

\begin{proof}
The relation (\ref{eq-relation-in-tensor-product-of-Yoneda-algebras}) follows immediately from 
Equation (\ref{eq-relation-in-yoneda-algebra}).
The rest follows from Theorem \ref{thm-Homs-for-objects-in-the-image-of-Phi-U}.

\smallskip
(\ref{thm-item-m-restricts-to-an-isomorphism})
Follows immediately from part \ref{thm-item-generaotrs-for-ideal}.

\smallskip
(\ref{thm-item-generators-for-submodule})
The statement follows from Theorem \ref{thm-Homs-for-objects-in-the-image-of-Phi-U}
part (\ref{cor-item-sheaf-case-has-on-relation-in-degree-2n}) and the vanishing of $\Hom(F_1,F_2)$.
\end{proof}
}

%
\subsection{Moduli spaces of sheaves over the moduli space $M$}
\label{sec-moduli-spaces-of-sheaves-over-a-moduli-space}
We continue to assume that $M:=M_H(v)$ and the modular object $\A$ over $X\times X$ is totally split as in 
Definition \ref{def-totally-split-monad}.
Let $S$ be a scheme of finite type over $\ComplexNumbers$.
A coherent $\StructureSheaf{S}$-module $F$ is {\em simple}, 
if $\End_S(F,F)$ is spanned by the identity.
Let $Spl(S)$ be the moduli space of simple sheaves over $S$
\cite{altman-kleiman}.
Denote by $[F]$ the point of $Spl(S)$ corresponding to $F$.

\begin{cor}
\label{cor-isomorphism-of-moduli-spaces-of-simple-sheaves}
Let $F$ be a simple coherent $\StructureSheaf{X}$-module and assume that
$\Phi_\U(F)[i]$ is equivalent to a coherent $\StructureSheaf{M}$-module
$F_M$, for some integer $i$. 
Then $F_M$ is simple.  
The functor $\Phi_\U$ induces an isomorphism, from
the open subset 
\[
U:=\{[F]\ : \ \Phi_\U(F)[i] \ \mbox{is equivalent to a coherent}
\  \StructureSheaf{M}\mbox{-module}\}
\] 
of $Spl(X)$, 
onto an open subset of $Spl(M)$. 
\end{cor}

\begin{proof}
If $F$ is simple, then it is a smooth point of the moduli space
$Spl(X)$ of simple sheaves \cite{mukai-symplectic}.
The functor $\Phi_\U$ induces isomorphisms
$\Phi_\U:\Ext^i_X(F,F)\rightarrow \Ext^i_M(F_M,F_M)$, for $i=0,1$,
by Theorem \ref{thm-Homs-for-objects-in-the-image-of-Phi-U}. 
Hence, $F_M$ is simple, if it is a sheaf.
The functor $\Phi_\U$ defines a morphism $\phi:U\rightarrow Spl(M)$.
This follows 
from a flatness result for Fourier-Mukai functors (\cite{mukai-applications}, 
Theorem 1.6). 
The differential of $\phi$ at $[F]$ is the isomorphism 
of Zariski tangent spaces $\Phi_\U:\Ext^1_X(F,F)\rightarrow\Ext^1_M(F_M,F_M)$. 
Combining the injectivity of the differential with the smoothness
of $Spl(X)$, we conclude that the dimension $D$ of $Spl(M)$
at $[F_M]$ is larger than or equal to the dimension $d$ of 
$Spl(X)$ at $[F]$. The surjectivity of the differential 
implies that $D\leq d$. Thus, $D=d=\dim\left(\Ext^1(F_M,F_M)\right)$,
and $Spl(M)$ is smooth at $[F_M]$. 
The homomorphism
$
\Phi_\U: \Hom(F_1,F_2) \rightarrow \Hom(\Phi(F_1),\Phi(F_2))
$
is an isomorphism, for any two sheaves $F_1$, $F_2$ on $X$, by 
Theorem \ref{thm-Homs-for-objects-in-the-image-of-Phi-U}.
It follows that $\phi$ is injective and hence an isomorphism
onto its image. 
\end{proof}

\begin{example}
Choosing $F$ to be a sky-scraper sheaf of a point of the $K3$ surface $X$ we see that 
$X$ is a connected component of $Spl(M)$. 
The twisted universal sheaf $\U$ over $X\times M$ for the coarse moduli space $M:=M_H(v)$ of 
$H$-stable sheaves over the $K3$ surface $X$ is also a universal sheaf of $X$ as a moduli space 
of sheaves over $M$.
\end{example}

\begin{example}
Let $\Phi_{\U^\vee}:D^b(X)\rightarrow D^b(M)$ 
be the integral functor defined by replacing the kernel $\U$ with $\U^\vee$ 
in equation (\ref{eq-phi-U}). Define $\Psi_{\U^\vee}:D^b(M)\rightarrow D^b(X)$
similarly. 
Note that the kernel of the integral functor $\Psi_{\U^\vee}\circ \Phi_{\U^\vee}$ is also isomorphic to
$\oplus_{i=0}^{n-1}\Delta_*\StructureSheaf{X}[-2i]$,
by Theorem \ref{thm-A-is-direct-sum}, as its kernel is the pullback of the kernel $\A$ of
$\Psi_{\U}\circ \Phi_{\U}$ via the transposition $\tau$ of the two factors of $X\times X$. Now
$\A$ is $\tau$-invariant.
Theorem \ref{thm-Homs-for-objects-in-the-image-of-Phi-U} applies to the functor $\Phi_{\U^\vee}$
as well. 
Let $w\in K(X)$ be another class and
$H'$ a $w$-generic polarization, such that $\M_{H'}(w)$ is smooth and projective. 
Assume that all sheaves parametrized by one of $\M_H(v)$ or $\M_{H'}(w)$ are locally free.\footnote{
If $v=(r,c,s)$ is primitive, $H$ is $v$-generic, $r\geq 2$, $(r,c,s+1)=ku$, where $k$ is an integer and $u$ is a primitive Mukai vector satisfying $(u,u)\leq -4$, 
then all sheaves parametrized by $M_H(v)$ are locally free.
}
Denote by $w_n$ the class of $w\otimes [H'^{\otimes n}]$. Then
$\M_{H'}(w)$ is isomorphic to $\M_{H'}(w_n)$. 
Assume that both $v$ and $w$ have positive rank. 
Fix $n$ sufficiently negative.
Then  
$\Ext^i(F,F')$ is isomorphic to $H^i(F^*\otimes F')$ if $F$ is locally free and to 
$H^{2-i}(F\otimes (F')^*)^*$ if $F'$ is locally free, and it thus 
vanishes, for $i\neq 2$, for all $H$-stable sheaves $F$
of class $v$, and for all $H'$-stable sheaves $F'$ of class $w$. 
Thus, $\Phi_{\U^\vee}(F')[2]$ is equivalent to a coherent sheaf,
for all such $F'$. We get a component of $Spl(\M_H(v))$
isomorphic to $\M_{H'}(w)$, via Corollary 
\ref{cor-isomorphism-of-moduli-spaces-of-simple-sheaves}.
\end{example}

%
\section{The transposition of the factors of $M\times M$}
Let $X$ be a $K3$ surface. 
If $X$ is projective we assume given a primitive Mukai vector $v$ with $c_1(v)$ of Hodge-type $(1,1)$ and 
a $v$-generic polarization $H$ and 
we  set $M:=M_H(v)$. We also consider the case where $X$ is a K\"{a}hler non-projective $K3$ surface
and $M=X^{[n]}$. 
Let $\F$ be the modular object and $\E$ the modular sheaf over $M\times M$ (Definition \ref{def-modular}).
We calculate the extension sheaves 
$\SheafExt^i(\E,\StructureSheaf{M\times M})$ and 
$\SheafExt^i(\E,\E)$ and the torsion sheaves $\SheafTor_i(\E^*,\E)$ in this section.
Let $\tau:M\times M\rightarrow M\times M$ be the transposition of the two factors. 
\begin{lem}\label{lemma-kernel-of-the-adjoint-of-F}
The following two objects of $D^b(M\times M)$ are isomorphic.
\begin{equation}
\label{eq-pull-back-of-F-by-tau}
\F \ \ \cong \ \ \tau^*\F^\vee[2].
\end{equation}
\end{lem}

\begin{proof}
The object $\F$ is defined as $R\pi_{13_*}\left(\pi_{12}^*\U^\vee\otimes \pi_{23}^*\U[2]\right)$.
We have the isomorphisms:
\begin{eqnarray*}
\F^\vee&\cong&R\SheafHom\left(R\pi_{13_*}\left(\pi_{12}^*\U^\vee\otimes \pi_{23}^*\U[2]\right),\StructureSheaf{M\times M}\right)
\\
&\cong & 
R\pi_{13_*}\left\{R\SheafHom\left(\pi_{12}^*\U^\vee\otimes \pi_{23}^*\U[2],\pi_{13}^!\StructureSheaf{M\times M}\right)\right\}
\\
& \cong & 
R\pi_{13_*}\left(\pi_{12}^*\U\otimes \pi_{23}^*\U^\vee\right)
\cong\tau^*\F[-2],
\end{eqnarray*}
where the fist isomorphism is clear, the second is Grothendieck-Verdier Duality, the third uses the triviality of 
$\omega_{\pi_{13}}$, and the last is clear.
\hide{
The statement reflects the following two description of the adjoint $\LL^\dagger$ of $\LL$.
Recall that if $\Phi:D^b(X_1)\rightarrow D^b(X_2)$ is an exact functor between two smooth
projective Calabi-Yau varieties, then a choice of  trivializations of the canonical line-bundles
$\omega_{X_i}$ yields a natural isomorphism $\Phi^\dagger_R\cong\Phi^\dagger_L[d_1-d_2]$
relating the left and right adjoint functors, where $d_i$ is the dimension of $X_i$
(\cite{mukai-duality}, or \cite{huybrechts-book}, Proposition 5.9). 
If $d_1=d_2$ then the two adjoint functors are isomorphic, and we denote both by $\Phi^\dagger$.
Taking $X_1=X$, $X_2=M$, and $\Phi=\Phi_\U$, we get
$\Psi_\U:=\left[\Phi_\U\right]^\dagger_R\cong\left[\Phi_\U\right]^\dagger_L[2-2n]$.
We get
\begin{equation}
\label{eq-perp-versus-perp-adjoint}
\LL^\dagger\cong \left\{\Phi_\U\Psi_\U\right\}^\dagger=
\left\{\Phi_\U\left(\Phi_\U\right)^\dagger_R\right\}^\dagger\cong
\left(\left[\Phi_\U\right]^\dagger_R\right)^\dagger_L\left[\Phi_\U\right]^\dagger_L \cong
\Phi_\U \Psi_\U[2n-2]=\LL[2n-2].
\end{equation}
Hence, the kernel of the endo-functor $\LL^\dagger$ is isomorphic to $\F[2n-2]$.
On the other hand, the kernel of $\LL^\dagger$ is $\tau^*\F^\vee[2n]$,
by \cite{mukai-duality}. This explains the isomorphism in Equation
(\ref{eq-pull-back-of-F-by-tau}). (???) Refer to Caldararu's result relating isomorphisms of integral functors and
their kernels.
}
\end{proof}

Let  $\gd\subset M\times M$ be the diagonal. Pushforward via the inclusion morphism 
embeds the category of coherent sheaves 
on $\Delta$ in that of $M\times M$. We suppress the pushforward notation and regard the former as a subcategory of the latter.
Recall that $\E:=\H^{-1}(\F)$, $\H^0(\F)\cong\StructureSheaf{\Delta}$, and all other sheaf
cohomologies of $\F$ vanish. 
The following is a more precise description of $\E$ due 
to the first author.

\begin{prop}[\cite{markman-hodge}, Proposition 4.5]\label{resofE}
The reflexive sheaf $\E$ is locally free away from $\gd$. Its restriction to $\gd$ 
is untwisted, and is described as follows:
\begin{enumerate}
\item[(i)] $\E\otimes \StructureSheaf{\Delta} \cong \Go^2_{\Delta}/\ko_{\Delta}\cdot \sigma$,
where $\sigma$ is the symplectic form;
\item[(ii)] $\mathcal{T}or_i(\E,\ko_{\Delta})\cong 
\Go^{i+2}_{\Delta}$, for $i>0$.
\end{enumerate}
\end{prop}

Set $\E^\vee:=R\SheafHom(\E,\StructureSheaf{M\times M})$ and 
$\E^*:=\SheafHom(\E,\StructureSheaf{M\times M})$, so that 
$\E^*:=\H^0(\E^\vee)$.

\begin{lem}
\label{lemma-pull-back-of-E-by-transposition}
The following sheaves isomorphisms exist:
\begin{eqnarray*}
\E^*&\cong &\tau^*\E,
\\
\SheafExt^1(\E,\StructureSheaf{M\times M})&\cong &\StructureSheaf{\Delta}, 
\\
\SheafExt^{2n-2}(\E,\StructureSheaf{M\times M})&\cong &\StructureSheaf{\Delta}, 
\end{eqnarray*}
and $\SheafExt^i(\E,\StructureSheaf{M\times M})\cong 0,$ if $i$ does not belong to $\{0, 1, 2n-2\}$.
\end{lem}

\begin{proof}
Dualizing the exact triangle
$
\E[1]\rightarrow\F\rightarrow \StructureSheaf{\Delta}\rightarrow \E[2],
$
we get the exact triangle
\[
\E^\vee[-2]\rightarrow \omega_{\Delta}[-2n]\rightarrow \F^\vee\rightarrow \E^\vee[-1].
\]
Trivializing $\omega_{\Delta}$, shifting by $2$, and applying the 
isomorphism (\ref{eq-pull-back-of-F-by-tau}), we get the following exact triangle.
\[
\E^\vee\rightarrow \StructureSheaf{\Delta}[2-2n]\RightArrowOf{\alpha} \tau^*\F\rightarrow \E^\vee[1].
\]
The sheaf homomorphisms  
$\H^i(\alpha):\H^i\left(\StructureSheaf{\Delta}[2-2n]\right)\rightarrow \H^i(\tau^*\F)$ vanish for all $i$,
since $n\geq 2$.
The long exact sequence of sheaf cohomology breaks into $\H^{-1}(\tau^*\F)\cong \H^{0}(\E^\vee)$, 
$\H^{0}(\tau^*\F)\cong \H^{1}(\E^\vee)$, and $\H^{2n-2}(\E^\vee)\cong \StructureSheaf{\Delta}$,
and $\H^i(\E^\vee)=0$, if $i$ does not belong to $\{0, 1, 2n-2\}$.
\end{proof}

The following proposition is used in the proof of \cite[Theorem 1.11]{torelli}.

\begin{prop}
\label{prop-sheaf-ext-E-E}
The sheaf $\SheafExt^i(\E,\E)$ comes with a decreasing filtration 
\[
F^0\SheafExt^i(\E,\E)\subset F^{-1}\SheafExt^i(\E,\E)\subset \cdots \subset
F^{2-2n}\SheafExt^i(\E,\E)=\SheafExt^i(\E,\E)
\]
with graded pieces $E^{p,q}_\infty:=F^{q}\SheafExt^{p+q}(\E,\E)/F^{q+1}\SheafExt^{p+q}(\E,\E)$ satisfying
\begin{enumerate}
\item
$E^{0,0}_\infty=F^0\SheafExt^0(\E,\E)\cong\E^*\otimes \E/\mbox{torsion},$
\\
$E^{1,-1}_\infty\cong \Omega^3_\Delta,$
\\
$E^{2n-2,2-2n}_\infty \cong  \Omega^{2n}_\Delta,$\\
and all other graded pieces of the filtration of $\SheafExt^0(\E,\E)$ vanish.
\item
We have the short exact sequence
\[
0\rightarrow \Omega^{2}_\Delta/(\sigma)\rightarrow 
\SheafExt^1(\E,\E)\rightarrow \Omega^{2n-1}_\Delta\rightarrow 0,
\]
where the subsheaf is $E^{1,0}_\infty$ and the quotient is $E^{2n-2,3-2n}_\infty$
and all other graded pieces  of the filtration of $\SheafExt^1(\E,\E)$ vanish.
\item
$\SheafExt^i(\E,\E)\cong \Omega^{2n-i}_\Delta$, for $2\leq i \leq 2n-3$.\\
$\SheafExt^{2n-2}(\E,\E)\cong \Omega^{2}_\Delta/(\sigma)$, and \\
$\SheafExt^i(\E,\E)=0$, for $i>2n-2$.
\item
\label{prop-item-Tor-E-E}
The torsion subsheaf of $\E^*\otimes \E$ is isomorphic to $\Omega^{4}_\Delta$.\\
$\SheafTor_j(\E^*,\E)\cong \Omega^{j+4}_\Delta$, for $j\geq 1$.
In particuler, $\SheafTor_j(\E^*,\E)$ vanishes for $j>2n-4$.
\end{enumerate}
\end{prop}

\begin{proof}
Let $W_j\rightarrow W_{j+1}\rightarrow \cdots \rightarrow W_{-1}\rightarrow W_0\rightarrow \E$
be a locally free resolution of $\E$. Denote the locally free complex by $(W_\bullet,d)$
and let $(W^*_\bullet,d^*)$ be the dual complex. We get the double complex
\[
W^{p,q}:=W^*_{-p}\otimes W_q, \ \ \ p\geq 0, \  q\leq 0
\]
and the associated single complex 
$
K^k:=\oplus_{p+q=k}W^{p,q}
$
with differential $D$, which restricts to $W^{p,q}$ as
$(-1)^p(d^*\otimes 1)+(1\otimes d):W^{p,q}\rightarrow W^{p+1,q}\oplus W^{p,q+1}$
\cite[Ch. 4, Sec. 5]{griffiths-harris}.
Note that $(W_\bullet,d)$ is quasi-isomorphic to $\E$, $(W^*_\bullet,d^*)$ represents 
$\E^\vee:=R\SheafHom(\E,\StructureSheaf{M\times M})$, and 
$(K^\bullet,D)$ represents $R\SheafHom(\E,\E)$
in $D^b(M\times M)$. In particular, we have the isomorphism
\[
\H^i(K^\bullet,D) \cong \SheafExt^i(\E,\E).
\]

Consider the decreasing filtration $F^qK^i:=\oplus_{p+q'=i, q'\geq q} W^{p,q'}$ and denote the induced filtration on 
sheaf cohomology by $F^q\SheafExt^i(\E,\E).$ Set
$E^{p,q}_\infty:=F^q\SheafExt^{p+q}(\E,\E)/F^{q+1}\SheafExt^{p+q}(\E,\E)$.
We have a spectral sequence converging to $\H^i(K^\bullet,D)$, with
\begin{eqnarray*}
E_1^{p,q}&:=& \H^p((W_\bullet^*,d^*)\otimes W_q)\cong \SheafExt^p(\E,W_q)\cong 
\SheafExt^p(\E,\StructureSheaf{M\times M})\otimes W_q.
\\
E_2^{p,q}&\cong& \H^q\left(\SheafExt^p(\E,\StructureSheaf{M\times M})\otimes W_\bullet,1\otimes d\right)\cong
\SheafTor_{-q}\left(\SheafExt^p(\E,\StructureSheaf{M\times M}),\E\right).
\end{eqnarray*}

We have seen that $\SheafExt^p(\E,\StructureSheaf{M\times M})$ is isomorphic 
$\StructureSheaf{\Delta}$, for $p=1$ and $p=2n-2$, and it vanishes for all other positive values of $p$,
by Lemma \ref{lemma-pull-back-of-E-by-transposition}. 
Furthermore, we have the isomorphisms 
$\SheafTor_{-q}\left(\StructureSheaf{\Delta},\E\right)\cong \Omega^{2-q}_\Delta$, for $q<0$, and 
$\SheafTor_{0}\left(\StructureSheaf{\Delta},\E\right)\cong \Omega^{2}_\Delta/(\sigma)$,
by Proposition \ref{resofE}. 
We get the following table for the $E_2$ term of the spectral sequence.

\begin{tabular}{|c|c|c|c|c|}  \hline 
\hspace{1ex}
$q\backslash p$& 0 & 1 & $\cdots$ & 2n-2
\\
\hline
0 & $\E^*\otimes \E$ & $\Omega^{2}_\Delta/(\sigma)$ & \hspace{1ex} & $\Omega^{2}_\Delta/(\sigma)$
\\
\hline
-1 & $\SheafTor_1(\E^*,\E)$ & $\Omega^{3}_\Delta$ & \hspace{1ex} & $\Omega^{3}_\Delta$
\\
\hline
-2 & $\SheafTor_2(\E^*,\E)$ & $\Omega^{4}_\Delta$ & \hspace{1ex} & $\Omega^{4}_\Delta$
\\
\hline
\vdots &\hspace{1ex}  &\hspace{1ex}  &\hspace{1ex}  &\hspace{1ex}  
\\
\hline
2-2n & $\SheafTor_{2n-2}(\E^*,\E)$ & $\Omega^{2n}_\Delta$ & \hspace{1ex} & $\Omega^{2n}_\Delta$
\\
\hline
\end{tabular}\\

We show first that $E_\infty^{2n-2,q}=E_2^{2n-2,q}$, for all $q$.
For $j\geq 2$, the differential of the spectral sequence is $d_j:E_j^{2n-2,q}\rightarrow E_j^{2n-j-1,q+j}$.
The vanishing of the columns other than for $p=0,1,2n-2$ implies that 
$E^{2n-2,q}_2=E^{2n-2,q}_{2n-2}$.
The differential $d_{2n-2}: E_{2n-2}^{2n-2,q}\rightarrow E_{2n-2}^{1,q+2n-2}$ vanishes for all $q$. Indeed, for $q=2-2n$
the homomorphism is a section of $H^0(\Omega^{2}_\Delta/(\sigma))$, which vanishes.
For all other values of $q$ the differential vanishes since its target vanishes.

The above vanishing of 
$d_j:E_j^{p,q}\rightarrow E_j^{p-j+1,q+j}$, for $j\geq 2$ and $p\geq 2$, implies also that
$E_\infty^{1,-1}=E_2^{1,-1}$ and $E_\infty^{1,0}=E_2^{1,0}$.

The sheaves $E^{p,q}_\infty$ vanish for $p+q<0$, since $\SheafExt^{p+q}(\E,\E)$
vanishes for these values. Hence,
\[
d_2^{1,q}:E_2^{1,q}\rightarrow E_2^{0,q+2}
\]
is injective for $q\leq -2$ and surjective for $q<-2$. In particular, 
the homomorphism
\[
d_2:E_2^{1,-2}= \ \Omega^{4}_\Delta\rightarrow \E^*\otimes \E
\]
is injective and the sheaves $\SheafTor_{-q}(\E^*,\E)$
are as claimed.
\end{proof}

\begin{rem}
Let $\beta:Y\rightarrow M\times M$ be the blow-up centered along the diagonal.
We claim that the sheaf $\beta^*\E$ has a non-trivial torsion subsheaf.
This is seen as follows. The sheaf $\E^*$ is isomorphic to $\beta_*V$, for a locally free sheaf V over $Y$, by
\cite[Prop. 4.5]{markman-hodge}. 
$\E^*\otimes \E$ has a non-trivial torsion subsheaf, by Proposition \ref{prop-sheaf-ext-E-E} (\ref{prop-item-Tor-E-E}).
On the other hand, $\E^*\otimes \E\cong \beta_*(V\otimes \beta^*E)$, by the projection formula. Hence,
$V\otimes \beta^*E$ has a non-trivial torsion subsheaf. We conclude that so does $\beta^*\E$, as claimed.
\end{rem}
%
\section{A simple and rigid comonad in $D^b(M\times M)$}
\label{sec-a-simple-and-rigid-object}

%
\subsection{$\F$ is simple and rigid}
We keep the notation of section \ref{sec-faithful-functor}, so $(X,H)$ is a 
polarized $K3$ surface,
$M:=M_H(v)$ is a smooth and projective moduli space of $H$-stable sheaves on $X$, of dimension $2n$,
$n\geq 2$, $\U$ is a universal sheaf over $X\times M$, 
$\Phi_\U:D^b(X)\rightarrow D^b(M,\theta)$ is the integral functor with kernel $\U$, and 
$\Psi_\U:D^b(M,\theta)\rightarrow D^b(X)$ is its right adjoint. 
Consider the following object in $D^b(X\times M,\pi_M^*\theta^{-1})$.
\[
\V := \U^\vee\otimes \pi_X^*\omega_X[2].
\]
Let $\pi_{ij}$ be the projection from 
$M\times X\times M$ onto the product of the $i$-th and $j$-th factors.
The kernel of the endo-functor $\LL:=\Phi_\U\Psi_\U:D^b(M,\theta)\rightarrow D^b(M,\theta)$ is the object 
\begin{equation}
\label{eq-F}
\F \  = \ \pi_{13_*}\left(\pi_{12}^*\V\otimes\pi_{23}^*\U\right)
\end{equation}
of $D^b(M\times M,\pi_1^*\theta^{-1}\pi_2^*\theta)$.
We prove in this section the following statement.

\begin{lem}
\label{lemma-F-is-simple-and-rigid}
Assume that $M=X^{[n]}$ and $\U$ is the universal ideal sheaf or, more generally, 
that the morphism $\alpha$ of Theorem \ref{introduction-thm-A-is-direct-sum} is an isomorphism.
\begin{enumerate}
\item
\label{lemma-item-F-is-simple-and-rigid}
The following isomorphisms hold:
\begin{eqnarray*}\
\Hom(\F,\F)&\cong &\CC,
\\
\Hom(\F,\F[k])&=&0,
\end{eqnarray*}
for odd $k$.
\item
\label{lemma-item-higher-homs-from-F-to-itself}
More generally, $\Hom(\F,\F[k])$ is isomorphic to
\[
\oplus_{i=0}^{n-1} \left[HH^{k-2i}(X)\right]^{\oplus i+1} 
\ \ \ \oplus \ \ \ 
\oplus_{j=0}^{n-2} \left[HH^{k+2j+4-4n}(X)\right]^{\oplus j+1},
\]
for all integers $k$.
\end{enumerate}
\end{lem}

\begin{proof}
Part (\ref{lemma-item-F-is-simple-and-rigid}): Note that the Hochschild cohomology
$HH^i(X):=\Hom_{X\times X}(\StructureSheaf{\Delta_X},\StructureSheaf{\Delta_X}[i])$ 
vanishes for $i<0$, for $i>4$, and for odd $i$.
Part (\ref{lemma-item-F-is-simple-and-rigid}) follows from part (\ref{lemma-item-higher-homs-from-F-to-itself})
and the isomorphism
$HH^0(X):=\Hom_{X\times X}(\StructureSheaf{\Delta_X},\StructureSheaf{\Delta_X})\cong\CC$.

Part (\ref{lemma-item-higher-homs-from-F-to-itself}):
Let $\Phi_\V:D^b(X)\rightarrow D^b(M,\theta^{-1})$ be the integral functor with kernel $\V$
and let $\Psi_\V:D^b(M,\theta^{-1})\rightarrow D^b(X)$ be its right adjoint. 
We get  the object 
\[
\V\boxtimes\U :=
\pi_{13}^*\V\otimes \pi_{24}^*\U.
\]
in $D^b(X\times X\times M\times M,\pi_3^*\theta^{-1}\pi_4^*\theta)$. Let
\[
\Gamma \ : \ D^b(X\times X) \ \ \longrightarrow \ \ D^b(M\times M,\pi_1^*\theta^{-1}\pi_2^*\theta)
\]
be the integral functor with kernel $\V\boxtimes\U$ and let
$\Gamma^\dagger_R:D^b(M\times M,\pi_1^*\theta^{-1}\pi_2^*\theta)\rightarrow D^b(X\times X)$ be its right adjoint. 
Then $\Gamma^\dagger_R$ is the integral transform with kernel
\[
(\V\boxtimes \U)^\vee\otimes \pi_1^*\omega_X\otimes \pi_2^*\omega_X[4]
\cong \U\boxtimes\V.
\]
Note that $\Gamma$ is the cartesian product of the functors $\Phi_\W$ and $\Phi_\U$.
Similarly, $\Gamma^\dagger_R$ is the cartesian product of their right adjoints
$\Psi_\V$ and $\Psi_\U$.

Let $\A$ be the object of $D^b(X\times X)$ given in Equation (\ref{eq-kernel-A}).
We have the following isomorphisms.
\begin{eqnarray}
\label{eq-F-is-the-image-of-the-structure-sheaf-of-the-diagonal}
\Gamma\left(\StructureSheaf{\Delta_X}\right) & \cong & \F,
\\
\nonumber
\Gamma^\dagger_R\left(\StructureSheaf{\Delta_M}\right) & \cong & \A.
\end{eqnarray}


The composition $\Psi_\U\Phi_\U:D^b(X)\rightarrow D^b(X)$ has kernel $\A$. The kernel of the
composition $\Psi_\V\Phi_\V:D^b(X)\rightarrow D^b(X)$ is the pull-back of $\A$ via the involution of $X\times X$
interchanging the two factors. $\A$ is invariant under this pull-back, and so $\A$ is the kernel
of $\Psi_\V\Phi_\V$ as well. We conclude that the composite endo-functor
\[
\Gamma^\dagger_R\Gamma \ : \ D^b(X\times X) \ \ \longrightarrow \ \  D^b(X\times X)
\]
has kernel $\pi_{13}^*\A\otimes \pi_{24}^*\A$ in $D^b(X\times X\times X\times X)$. 
We get the following isomorphism.
\begin{equation}
\label{eq-right-adjoint-of-Gamma-composed-with-Gamma-takes-id-to-A-square}
\Gamma^\dagger_R\Gamma(\StructureSheaf{\Delta_X}) \ \ \ \cong \ \ \ \A\circ \A.
\end{equation}
The object 
$\A$ is  isomorphic to $\StructureSheaf{\Delta_X}\otimes_{\ComplexNumbers}\lambda_n$,
where $\lambda_n$ is an object of $D^b(\pt)$, by Theorem \ref{thm-A-is-direct-sum} or by the assumption that $\alpha$ is an isomorphism.
Hence, the endo-functor $\Gamma^\dagger_R\Gamma$ is isomorphic to the functor
of tensorization over $\ComplexNumbers$ with the object $(\lambda_n)^{\otimes 2}$
of $D^b(\pt)$.
Now
\[
(\lambda_n)^{\otimes 2}\ \ \ \cong \ \ \ 
\oplus_{i=0}^{n-1} \left(\StructureSheaf{\pt}[-2i]\right)^{\oplus i+1} 
\ \ \ \oplus \ \ \ 
\oplus_{j=0}^{n-2} \left(\StructureSheaf{\pt}[2j+4-4n]\right)^{\oplus j+1}.
\]
We get the isomorphisms:
\begin{eqnarray*}
\Hom_{M\times M}(\F,\F[k]) & \cong &
\Hom_{M\times M}(\Gamma(\StructureSheaf{\Delta_X}),\Gamma(\StructureSheaf{\Delta_X})[k]) 
\\
& \cong &
\Hom_{X\times X}\left(\StructureSheaf{\Delta_X},\Gamma^\dagger_R\Gamma(\StructureSheaf{\Delta_X})[k]\right) 
\\
& \cong &
\Hom_{X\times X}\left(\StructureSheaf{\Delta_X},\StructureSheaf{\Delta_X}\otimes_\CC (\lambda_n)^{\otimes 2}[k]\right).
 \end{eqnarray*}
Part (\ref{lemma-item-higher-homs-from-F-to-itself}) follows immediately from 
the isomorphisms above.
\end{proof}

%
\subsection{$\E$ is simple and rigid}
Let $\E$ be the sheaf cohomology $\H^{-1}(\F)$, where $\F$ is the object
of $D^b(M\times M)$ given in Equation (\ref{eq-F}). We get the exact triangle
\begin{equation}
\label{eq-exact-triangle-with-E-and-F}
\E[1]\RightArrowOf{\beta} \F \RightArrowOf{\epsilon} \StructureSheaf{\Delta_M}
\RightArrowOf{\gamma} \E[2],
\end{equation}
where $\epsilon$ is the morphism inducing the counit for the adjunction
$\Phi_\U\dashv \Psi_\U$. The following rigidity result is a crucial ingredient in the proof of Theorem 
\ref{thm-torelli}.

\begin{lem}
\label{lemma-Ext-1-E-E-vanishes}
Assume that $M=X^{[n]}$ and $\U$ is the universal ideal sheaf.
Then the sheaf $\E$ is simple and rigid. In other words, $\Hom(\E,\E)$ is one-dimensional
and $\Ext^1(\E,\E)$ vanishes.
\end{lem}

\begin{proof}
\underline{Step 1:}
The integral functor $\Gamma^\dagger_R:D^b(M\times M)\rightarrow D^b(X\times X)$ takes
the exact triangle (\ref{eq-exact-triangle-with-E-and-F})
to the exact triangle
\begin{equation}
\label{eq-Gamma-dagger-takes-exact-triangle-to-a-split-one}
\Gamma^\dagger_R(\E[1]) \LongRightArrowOf{\Gamma^\dagger_R(\beta)}
\A\circ\A \LongRightArrowOf{m} \A 
\LongRightArrowOf{\Gamma^\dagger_R(\gamma)}
\Gamma^\dagger_R(\E[1]),
\end{equation}
where $\A\circ\A$ is the convolution and 
$m:=\Gamma^\dagger_R(\epsilon)$ is the multiplication, by Equations 
(\ref{eq-F-is-the-image-of-the-structure-sheaf-of-the-diagonal}),
(\ref{eq-right-adjoint-of-Gamma-composed-with-Gamma-takes-id-to-A-square}), and 
the proof of Lemma
\ref{lemma-F-is-simple-and-rigid}.
Now $m$ has a right inverse given by
\[
\A\cong \A\circ\StructureSheaf{\Delta_X}\LongRightArrowOf{\eta\circ 1} \A\circ \A,
\]
where $\eta:\StructureSheaf{\Delta_X}\rightarrow \A$ induces the unit for the adjunction $\Phi_\U\dashv \Psi_\U$.
Thus, $\Gamma^\dagger_R(\gamma)=0$ and the exact triangle 
(\ref{eq-Gamma-dagger-takes-exact-triangle-to-a-split-one})
splits. The left adjoint $\Gamma^\dagger_L$ of $\Gamma$ is isomorphic to a shift of the right adjoint
$\Gamma^\dagger_R$, since $X$ and $M$ have trivial canonical line bundles.
We conclude that $\Gamma^\dagger_L(\gamma)=0$. In particular, the homomorphism
\[
\Gamma^\dagger_L(\beta)^* \ : \ 
\Hom(\Gamma^\dagger_L(\F),x)
\ \ \longrightarrow \ \ 
\Hom(\Gamma^\dagger_L(\E[1]),x)
\]
is surjective, for all objects $x$ of $D^b(X\times X)$.
Take $x=\StructureSheaf{\Delta_X}[k]$, apply the adjunction $\Gamma^\dagger_L\dashv \Gamma$,
and use the isomorphism $\Gamma(\StructureSheaf{\Delta_X})\cong \F$, to conclude that the homomorphism
\[
\Hom(\F,\F[k])\rightarrow \Hom(\E[1],\F[k])
\]
is surjective, for all $k$. We get the inequality
\begin{equation}
\label{eq-inequality-of-dimensions-of-Homs}
\dim\Hom(\E[1],\F[k]) \ \ \leq \ \ \dim \Hom(\F,\F[k]),
\end{equation}
for all $k$.

\underline{Step 2:}
Apply the functor $\Hom(\E,\bullet)$
to the exact triangle (\ref{eq-exact-triangle-with-E-and-F}). We get the long exact sequence
\[
\Hom(\E,\StructureSheaf{\Delta_{M}}[k-1])
\rightarrow
\Hom(\E,\E[k+1])
\rightarrow
\Hom(\E,\F[k])
\rightarrow 
\Hom(\E,\StructureSheaf{\Delta_M}[k]).
\]
When $k=-1$, the first and fourth terms vanish and so
$\Hom(\E,\E)\rightarrow \Hom(\E,\F[-1])$ is an isomorphism.
Now $\dim\Hom(\E,\F[-1])\leq \dim\Hom(\F,\F)$, by the inequality 
(\ref{eq-inequality-of-dimensions-of-Homs}), and
$\dim\Hom(\F,\F)=1$, by Lemma \ref{lemma-F-is-simple-and-rigid}.
We conclude that $\dim\Hom(\E,\E)=1$.

When $k=0$, the first term in the long exact sequence above vanishes.
The third term vanishes, by the inequality (\ref{eq-inequality-of-dimensions-of-Homs}) and the
vanishing of $\Hom(\F,\F[1])$ established in Lemma \ref{lemma-F-is-simple-and-rigid}.
We conclude that $\Hom(\E,\E[1])=0$.
\end{proof}

Assume that $M=X^{[n]}$ as above. 
Let $\EE:D^b(M)\rightarrow D^b(M)$ be the endo-functor with kernel $\E$ in $D^b(M\times M)$. 
Given a point $m\in M$, denote by $\CC_m$ the corresponding sky-scraper sheaf.

The following Lemma is used in \cite{torelli}.
\begin{lem}
The homomorphism 
$\kappa:\Hom(\CC_m,\CC_m[1])\rightarrow \Hom(\EE(\CC_m),\EE(\CC_m)[1])$, induced by $\EE$, 
is surjective, for all $m\in M$.
\end{lem}

\begin{proof}
The homomorphism $\tilde{\kappa}:\Hom(\CC_m,\CC_m[1])\rightarrow \Hom(\LL(\CC_m),\LL(\CC_m)[1])$, induced by $\LL$,
is an isomorphism, by Theorem \ref{thm-Homs-for-objects-in-the-image-of-Phi-U}.
Set $\E_m:=\EE(\CC_m)$ and $\F_m:=\LL(\CC_m)$.
The exact triangle 
\[
\E[1]\RightArrowOf{\beta} \F\rightarrow \StructureSheaf{\Delta}\rightarrow \E[2]
\]
gives rise to the exact triangle 
\[
\CC_m[-1]\rightarrow \E_m[1]\RightArrowOf{\beta_m}\F_m\rightarrow \CC_m.
\]
The pullback homomorphism $\beta_m^*:\Hom(\F_m,\F_m[1])\rightarrow \Hom(\E_m[1],\F_m[1])$
is surjective, by an argument analogous to that in step 1 of the proof of Lemma \ref{lemma-Ext-1-E-E-vanishes}.
The kernel of the push-forward homomorphism 
$\beta_{m,*}:\Hom(\E_m[1],\E_m[2])\rightarrow \Hom(\E_m[1],\F_m[1])$
is the image of
$\Hom(\E_m[1],\CC_m)\rightarrow \Hom(\E_m[1],\E_m[2]).$
Now $\Hom(\E_m[1],\CC_m)$ clearly vanishes. Hence, $\beta_{m,*}$ is injective.

Given an element $\xi\in \Hom(\CC_m,\CC_m[1])$, we have the equality
$\beta_m\circ \EE(\xi)=\LL(\xi)\circ\beta_m$, since the homomorphism of kernels induces 
a natural transformation $\beta:\EE[1]\rightarrow \LL$.
We get the equality 
\[
\beta_m^*\circ \tilde{\kappa}=\beta_{m,*}\circ \kappa:\Hom(\CC_m,\CC_m[1])\rightarrow 
\Hom(\E_m[1],\F_m[1]).
\]
The homomorphism $\beta_m^*\circ \tilde{\kappa}$ is surjective, being a composition of such.
Thus, $\beta_{m,*}\circ \kappa$ is surjective. Hence $\beta_{m,*}$ is an isomorphism.
Thus, $\kappa$ is surjective.
\end{proof}

%
\section{A deformation of the derived category $D^b(X)$}
\label{Deformation-of-monad}

%
We prove Theorem \ref{deformability} in this section.
Although we talk of deforming categories, as will be seen presently, we shall only need
to work with deformations of certain fixed objects and morphisms
in a derived category. These are defined as follows.

\begin{defi}
Let $Y \to S$ be a flat family of spaces (varieties or analytic spaces), and $s$ a point 
of $S$. 
\begin{enumerate}
\item A {\em deformation} of a perfect complex $\E \in D^b(Y_s)$ 
is an $S$-perfect complex $\fE \in D^b(Y)$ together with
an isomorphism $\varphi: i^*\fE \stackrel{\cong}{\longrightarrow} \E$, where 
$i: Y_s\to Y$ is the closed immersion.
\item Given a morphism $f: \E_1 \to \E_2$ between perfect complexes and deformations
$(\fE_i, \varphi_i)$, a {\em compatible deformation of} $f$ is a morphism  $\ol{f}: \fE_1 \to \fE_2$
such that $f\circ \varphi_1 =  \varphi_2 \circ i^*\ol{f}$.
\end{enumerate}
\end{defi}

%
\subsection{Deformability of $\F$}
\label{subsec-deformability-of-the-monad}

Denote by ${\fM}_\Lambda$ the moduli space parametrizing marked pairs $(M,\eta)$ consisting
of a holomorphic symplectic manifold $M$, and an isometry $\eta:H^2(M,\Integers)\rightarrow \Lambda$. 
The moduli space ${\fM}_\Lambda$ is constructed in \cite{huybrechts-basic-results}.
Let $\phi: \widetilde{\fM}_\Lambda \to {\fM}_\Lambda$ be the forgetful map given by
$(M, \eta, \Az) \mapsto (M,\eta)$. 
Let $\widetilde{\fM}_\Lambda^0$ be the connected component of $\widetilde{\fM}_\Lambda$ 
in Theorem \ref{deformability}.
Let ${\fM}_\Lambda^0$ be the connected component containing $\phi(\widetilde{\fM}_\Lambda^0)$.
Then $\phi:\widetilde{\fM}_\Lambda^0\rightarrow {\fM}_\Lambda^0$ is surjective,
by \cite[Theorem 4.14]{torelli}.
The following is the key result from \cite{torelli} for the sequel; it is a direct 
consequence of Theorems 1.9(1), and 1.11(1), and Lemma 5.3 in that article.

\begin{thm}\label{uniqueness-of-bundle}
Let $X$ be a K3 surface with trivial Picard group. Then, the
fiber of the morphism $\phi: \widetilde{\fM}_\Lambda^0 \to {\fM}_\Lambda^0$ 
over $(X^{[n]}, \eta)$ consists of the single point
$(X^{[n]}, \eta, \Az)\in \widetilde{\fM}_\Lambda^0$, where $\Az$ or $\Az^*$  is the 
modular Azumaya algebra of the Hilbert scheme $X^{[n]}$ (Def. \ref{def-modular}). 
\end{thm}

The following is an immediate corollary of Theorems \ref{thm-torelli} and \ref{uniqueness-of-bundle}
and the density of Hilbert schemes in $\fM_\Lambda^0$ \cite{density}.

\begin{cor} 
\label{cor-modular-Azumaya-algebras-are-dense}
The subset ${\mathcal Hilb}$ of $\widetilde{\fM}_\Lambda^0$, 
given in Equation
(\ref{dense-subset-of-modular-Azumaya-algebras}), 
consisting of triples $(X^{[n]},\eta,\Az)$
where $X$ is a $K3$ surface with a trivial $\Pic(X)$, is dense in $\widetilde{\fM}_\Lambda^0$.
For each such triple, $\Az$ or $\Az^*$ is the modular Azumaya algebra over $X^{[n]}\times X^{[n]}$.
\end{cor}

Let $\MM^0$ be the universal family\footnote{While a universal family need not exist over the
moduli space of marked pairs of a general class of holomorphic symplectic manifolds, 
such a family does exist over $\fM_\Lambda^0$ as manifolds of $K3^{[n]}$-type have no 
nontrivial automorphisms  which act trivially on cohomology in degree 2.
This follows for Hilbert schemes by a result of Beauville \cite{beauville},
and consequently also for their deformations by 
\cite[Sec. 2]{hassett-tschinkel-lagrangian-planes}.} 
over the moduli space of marked pairs 
$\fM_\Lambda^0$. 
Set 
$\widetilde{\MM}^0:= \MM^0\times_{\fM_\Lambda^0} \widetilde{\fM}_\Lambda$.
There exists a universal Azumaya algebra $\fAz$ over 
$\widetilde{\MM}^0 \times_{\widetilde{\fM}_\Lambda^0} 
\widetilde{\MM}^0$ (see \cite[Section 4]{torelli}). 

\begin{prop}
\label{prop-construction-of-universal-object-F}
There exists a Zariski dense open subset $U\subset \widetilde{\fM}_\Lambda^0$, containing ${\mathcal Hilb}$, 
a universal twisted sheaf $\fE$ over
\[
\MM^2 \ \ := \ \ \left(\widetilde{\MM}^0 \times_{\widetilde{\fM}_\Lambda^0} 
\widetilde{\MM}^0\right)\times_{\widetilde{\fM}_\Lambda^0} U, 
\]
satisfying $\SheafEnd(\fE)\cong\fAz$, and an extension 
\begin{equation}
\label{eq-universal-extension-of-identity-by-E}
\fE[1]\rightarrow \fF\rightarrow \StructureSheaf{\Delta_{\MM}}
\end{equation}
of twisted sheaves over $\MM^2$, where  
$\MM:=\widetilde{\MM}^0\times_{\widetilde{\fM}_\Lambda^0} U$, and 
$\Delta_{\MM}$ is the image of $\MM$ via the diagonal embedding
$\Delta:\MM\rightarrow\MM^2$. 
This extension is non-split along every fiber of the projection $\Pi:\MM^2\rightarrow U$. 
Let $\widetilde{\Pi}:\MM^3\rightarrow U$ be the third fiber product of $\MM$ over $U$ and 
$\Pi_{ij}:\MM^3\rightarrow \MM^2$, 
$1\leq i<j\leq 3$, the natural projections. 
The Brauer class $\Theta$ of $\fE$ satisfies the equality
\begin{equation}
\label{eq-Brauer-class-allows-convolution}
\Pi_{12}^*(\Theta)\Pi_{23}^*(\Theta)=\Pi_{13}^*(\Theta).
\end{equation}
Consequently, both $\fF$ and the convolution $\fF\circ\fF$ are objects of $D^b(\MM^2,\Theta)$.
\end{prop}

\begin{proof}
Step 1: We show first that the Brauer class of $\fAz$ restricts as a trivial class to the diagonal 
$\Delta_{\widetilde{\MM}^0}\subset \widetilde{\MM}^0\times_{\widetilde{\fM}_\Lambda^0}\widetilde{\MM}^0$
over a Zariski dense open subset of $\widetilde{\fM}_\Lambda^0$.
Let $\mu_{2n-2}\subset\ComplexNumbers^*$ be the group of roots of unity of order dividing $2n-2$.
The exponential map 
\[
\exp\left(\frac{2\pi i (\bullet)}{2n-2}\right):\Integers \rightarrow \mu_{2n-2}
\] 
factors through an isomorphism $\Integers/(2n-2)\Integers\cong \mu_{2n-2}$. 
Given a marked pair $(M,\eta)$ in $\fM_\Lambda^0$ we get an isomorphism 
\[
\bar{\eta}:H^2(M\times M,\mu_{2n-2})\IsomRightArrow [\Lambda/(2n-2)\Lambda]^2
\] 
induced by the marking $\eta$.
The component $\widetilde{\fM}_\Lambda^0$ comes with a coset
\begin{equation}
\label{eq-Brauer-classes}
\tilde{\theta}=(\tilde{\theta}_1,\tilde{\theta}_2), 
\end{equation}
with $\tilde{\theta}_1=-\tilde{\theta}_2$ and $\tilde{\theta}_i$
 of order $2n-2$ in $\Lambda/(2n-2)\Lambda$, such that the Brauer class 
 of the Azumaya Algebra $\Az$ of a triple $(M,\eta,\Az)$ in this component is the image of
$\bar{\eta}^{-1}(\tilde{\theta})$ in $H^2(M\times M,\StructureSheaf{}^*)$, by construction 
\cite[Eq. (7.6)]{torelli}. The Brauer class $\Theta$ of $\fAz$ in 
$H^2(\widetilde{\MM}^0 \times_{\widetilde{\fM}_\Lambda^0} \widetilde{\MM}^0,\StructureSheaf{}^*)$, having order $2n-2$,
is the image of a topological class $\tilde{\Theta}$ in 
$H^2(\widetilde{\MM}^0 \times_{\widetilde{\fM}_\Lambda^0} \widetilde{\MM}^0,\mu_{2n-2})$. 
The restriction of $\tilde{\Theta}$ to the fibers $M\times M$ of the family over a marked pair $(M,\eta)$ is 
$\bar{\eta}^{-1}(\tilde{\theta})$.

Given an open analytic subset $U'$ of $\widetilde{\fM}_\Lambda^0$, over which the 
differential fibration $\pi:\widetilde{\MM}^0\rightarrow \widetilde{\fM}_\Lambda^0$ restricts as
a topologically trivial fibration, we get that the pullback $\Delta^*\tilde{\Theta}$ of the Brauer class to 
$\widetilde{\MM}^0$ is trivial over $U'$, by the vanishing of $\tilde{\theta}_1+\tilde{\theta}_2$ mentioned above.

There exists  a $\Theta$-twisted sheaf $\fE$ representing $\fAz$ for some \v{C}ech  $2$-cocycle $\theta$ of 
$\StructureSheaf{\widetilde{\MM}^0\times_{\widetilde{\fM}_\Lambda^0}\widetilde{\MM}^0}^*$
\cite[Section 2]{torelli}. Over the above open subset $U'$,
$\fE$ restricts to the diagonal with a trivial Brauer class. Hence, we may adjust the pull-back $\Delta^*(\Theta)$ by a 
$2$-coboundary, and adjust the gluing of $\Delta^*\fE$ over $U'$ to get an untwisted sheaf $W$. 
For every triple $(X^{[n]},\eta,\Az)$ in $U'$, 
where $\Pic(X)$ is trivial, the restriction of $\Az$ to $X^{[n]}\times X^{[n]}$ is the modular Azumaya algebra or its dual, 
by Theorem \ref{uniqueness-of-bundle}, 
and so 
$W|_{X^{[n]}}$ is isomorphic to $L\otimes [(\wedge^2T^*X^{[n]})/(\StructureSheaf{X^{[n]}}\cdot \sigma)]$, 
for some line-bundle $L$, by Proposition \ref{resofE} (use the isomorphism $\E^*\cong\tau^*\E$ of 
Lemma \ref{lemma-pull-back-of-E-by-transposition} for the dual of the modular Azumaya algebra).
Set $\W:=\wedge^2T^*_\pi/(L_0\pi^*R^0\pi_*\wedge^2T^*_\pi)$.
We conclude the isomorphism 
\begin{equation}
\label{eq-isomorphism-of-Azumaya-algebras-along-diagonal}
\SheafEnd(\W)\cong \SheafEnd(\Delta^*\fE),
\end{equation}
of Azumaya algebras over a Zariski dense open subset of $U'$,
by the density of Hilbert schemes of $K3$ surfaces with trivial Picard group 
(Corollary \ref{cor-modular-Azumaya-algebras-are-dense}) and the upper-semi-continuity theorem. Both sides of the isomorphism (\ref{eq-isomorphism-of-Azumaya-algebras-along-diagonal})
are defined globally over $\widetilde{\MM}^0$. Hence, the isomorphism holds over 
a Zariski dense open subset $U''$ of $\widetilde{\fM}_\Lambda^0$ containing the dense subset
${\mathcal Hilb}$  given in Equation (\ref{dense-subset-of-modular-Azumaya-algebras}).

The sheaf $\Delta^*\fE$ is $\Delta^*\Theta$-twisted. The left hand side of 
(\ref{eq-isomorphism-of-Azumaya-algebras-along-diagonal}) is an Azumaya algebra with a trivial Brauer class.
The isomorphism (\ref{eq-isomorphism-of-Azumaya-algebras-along-diagonal}) 
implies that the cocycle $\Delta^*\Theta$ is a coboundary over $U''$, $\Delta^*\Theta=\delta(\zeta)$.

Step 2: 
Let $\pi_1:\widetilde{\MM}^0\times_{\widetilde{\fM}_\Lambda^0}\widetilde{\MM}^0\rightarrow
\widetilde{\MM}^0$ be the projection to the first factor.
After a refinement of the open covering of 
$\widetilde{\MM}^0\times_{\widetilde{\fM}_\Lambda^0}\widetilde{\MM}^0$
we may change the gluing of $\fE$ over $U''$ via the cochain $\pi_1^*(\zeta)$, 
so that $\fE$ is a twisted sheaf with respect to a new cocycle, denoted again by $\Theta$, which
restricts to the trivial cocycle along the diagonal (with value $1$ along every triple intersection).
Then $\Delta^*\fE$ is an untwisted coherent sheaf, which is isomorphic to $\W\otimes\LB$, for some line bundle
$\LB$ over  the subset $U''$ of $\widetilde{\fM}_\Lambda^0$.

Step 3: Let $\pi:\MM\rightarrow U''$ be the restriction of the universal family from $\widetilde{\fM}_\Lambda^0$
to $U''$. Denote by 
$\Pi:\MM\times_{U''}\MM\rightarrow U''$ the projection from the fiber square.
We show next that the relative homomorphism sheaf 
\begin{equation}
\label{eq-inverse-of-relative-2-extension-line-bundle}
\H_j:=\H om_{\Pi}(\ko_{\gd_\MM}, \fE[j])
\end{equation}
is a line bundle over a Zariski dense open subset $U$ of
$U''$ containing ${\mathcal Hilb}$ for $j=2$, and it vanishes over $U$ for $j=0,1$. 
This $U$ would be  also Zariski dense and open in $\widetilde{\fM}_\Lambda^0$.
Local Grothendieck-Verdier duality yields the isomorphism
\[
R\SheafHom\left(R\gd_{\MM_*}\StructureSheaf{\MM},\fE[j]\right) \cong
R\gd_*\left[R\SheafHom(\StructureSheaf{\MM},\gd^!\fE[j])
\right].
\]
Applying the functor $R(\Pi)_*$ on both sides, using the vanishing of 
$R^i\gd_{\MM_*}\StructureSheaf{\MM}$ for $i>0$, we get the isomorphism
\begin{equation}\label{isom-1}
R\H om_{\Pi}(\ko_{\gd_\MM}, \fE[j]) \cong R\pi_*(\omega_\pi^*\otimes L\gd_\MM^*\fE[j-2n]).
\end{equation}
The relative dualizing sheaf $\omega_\pi$ of the morphism  $\pi$ is trivial along each fiber and is hence the pullback 
of a line bundle over $U''$.
Thus, the sheaf $\H om_{\Pi}(\ko_{\gd_\MM}, \fE[j])$ is isomorphic to the $0$-th direct
image of the complex  $L\gd_\MM^*\fE[j-2n]$, tensored by a line bundle. For $x\in U''$, denote by
$\fE|_x$ the restriction of $\fE$ to the Cartesian square of the fiber 
over $x$. The function $\phi_j: U''\to \ZZ$ given by
\[
x\mapsto \dim(\mathbb{H}^0(\MM_x, L\gd_{\MM_x}^*\fE|_x[j-2n]))
\] 
is upper semi-continuous by
\cite[Prop. 6.4]{Hart}. The value of $\phi_2$ on
triples  in ${\mathcal Hilb}$  is $1$, by  
the calculation of the torsion sheaf $\SheafTor_{2n-2}(\E,\Delta_*\StructureSheaf{X^{[n]}})$ in 
Proposition \ref{resofE}.
The set ${\mathcal Hilb}$ is dense in $\widetilde{\fM}_\Lambda^0$, by Corollary \ref{cor-modular-Azumaya-algebras-are-dense}.
It follows that $\phi\equiv 1$ on a dense open subset $U_2$ of $U''$. 
The sheaf $\H om_{\Pi}(\ko_{\gd_\MM}, \fE[2])$ is a line bundle over $U_2$, since $U''$ is integral. 
The vanishing of $\SheafTor_{2n-j}(\E,\Delta_*\StructureSheaf{X^{[n]}})$ in 
Proposition \ref{resofE}, for $j<2$, implies the claimed vanishing of $\H_j$ over a Zariski dense subset $U_j$ of $U''$
containing ${\mathcal Hilb}$, for $j=0, 1$. Set $U:=U_0\cap U_1\cap U_2$.

Step 4: Set $\H:=\H_2$, where $\H_2$ is  the line-bundle given in 
Equation (\ref{eq-inverse-of-relative-2-extension-line-bundle}). 
The vanishing of $\H_j$, for $j=0,1$, and a standard spectral sequence argument yield an isomorphism
$\Hom(\ko_{\gd_\MM}, \Pi^*(\H^{-1})\otimes \fE[2])\cong H^0\left(\H om_{\Pi}(\ko_{\gd_\MM}, \Pi^*(\H^{-1})\otimes \fE[2])\right)$.
The  right hand space is $H^0(U,\StructureSheaf{U})$, by the projection formula and the definition of $\H$. 
Choosing the constant section $1$ of the latter we get a tautological extension
\[
\Pi^*(\H^{-1})\otimes \fE[1]\rightarrow \fF\rightarrow \StructureSheaf{\Delta_{\MM}}.
\]
Replacing the twisted sheaf $\fE$ by $\pi^*(\H^{-1})\otimes \fE$ we get the desired extension
(\ref{eq-universal-extension-of-identity-by-E}). 

Step 5: We prove in this step the equality (\ref{eq-Brauer-class-allows-convolution}).
Set $r:=2n-2$. Let $F^pH^k(\MM^2,\mu_r)$ be the decreasing 
Leray filtration associated to the morphism $\Pi:\MM^2\rightarrow U$. 
Set $E_\infty^{p,q}:=F^pH^{p+q}(\MM^2,\mu_r)/F^{p+1}H^{p+q}(\MM^2,\mu_r)$.
We have the Leray spectral sequence converging to $E_\infty^{p,q}$ with $E_2$ terms of the form
$E_2^{p,q}:=H^p(U,R^q\Pi_*\mu_r)$ and differential $d_2:E_2^{p,q}\rightarrow E_2^{p+2,q-1}$.

The sheaf $R^0\Pi_*\mu_r$ is the trivial local system $\mu_r$ and 
the sheaf $R^1\Pi_*\mu_r$ vanishes, since the fibers of $\Pi$ are simply connected. 
The sheaf $R^2\Pi_*\mu_r$ is the direct sum $[\Lambda/r\Lambda]^{\oplus 2}$ 
of two copies of the trivial local system $[\Lambda/r\Lambda]$, since the markings provide such a trivialization. 
We conclude the following:
\begin{eqnarray*}
E_\infty^{2,0}&=&E_2^{2,0}\cong H^2(U,\mu_r)
\\
E_\infty^{1,1}&=&0,
\\
E_3^{0,2}&=& E_2^{0,2}\cong [\Lambda/r\Lambda]^{\oplus 2},
\\
E_3^{3,0}&=&E_2^{3,0}\cong H^3(U,\mu_r), 
\\
E_{\infty}^{0,2}&=& \ker\left[
d_3:E_3^{0,2}\rightarrow E_3^{3,0}
\right].
\end{eqnarray*}
The description of $E_\infty^{2,0}$ implies that the homomorphism
$\Pi^*:H^2(U,\mu_r)\rightarrow H^2(\MM^2,\mu_r)$ is injective. The analogous description of the graded summands of the Leray filtrations of $H^2(\MM^n,\mu_r)$ holds for the $n$-th fiber self-products $\MM^n$, $n\geq 1$, where the 
$E_\infty^{0,2}$ is naturally a subgroup of $[\Lambda/r\Lambda]^{\oplus n}$.

We have seen in Step 2 that $\Theta$ belongs to the kernel of 
$\Delta^*:H^2(\MM^2,\StructureSheaf{}^*)\rightarrow H^2(\MM,\StructureSheaf{}^*)$.
We have seen, in addition,  that there exists a lift $\widetilde{\Theta}\in H^2(\MM^2,\mu_r)$ of $\Theta$. 
Next we normalize this lift so that it restricts trivially to the diagonal.
The composition
\[
H^2(\MM,\mu_r)\rightarrow H^2(\MM,\mu_r)/F^1H^2(\MM,\mu_r)\hookrightarrow H^0(U,R^2\pi_*\mu_r)\rightarrow H^0(U,R^2\pi_*\StructureSheaf{\MM}^*)
\]
maps the class $\Delta^*(\widetilde{\Theta})$ to the trivial class. The rightmost homomorphism is injective, since the 
Picard group of a generic fiber of $\pi:\MM\rightarrow U$ is trivial. Consequently, the image of $\Delta^*(\widetilde{\Theta})$ 
in $H^0(U,R^2\pi_*\mu_r)$ is trivial and $\Delta^*(\widetilde{\Theta})$ belongs to $E_\infty^{2,0}$ with respect to the Leray filtration of $H^2(\MM,\mu_r)$.
Hence, there exists a class $\alpha$ in $H^2(U,\mu_r)$, such that
$\Delta^*(\widetilde{\Theta})=\pi^*(\alpha)$. 
The image of $\pi^*(\alpha)$ in $H^2(\MM,\StructureSheaf{}^*)$ is trivial, 
since such is the image of $\Delta^*(\widetilde{\Theta})$. The image of $\Pi^*(\alpha)$ in 
$H^2(\MM^2,\StructureSheaf{}^*)$ is trivial as well, since $\Pi$ factors through $\pi$.
Let $\beta$ be the class $\widetilde{\Theta}\Pi^*(\alpha^{-1})$ in $H^2(\MM^2,\mu_r)$.
Then $\Delta^*(\beta)$ is trivial in $H^2(\MM,\mu_r)$ and $\beta$ maps to the Brauer class $\Theta$ in $H^2(\MM^2,\StructureSheaf{}^*)$. 

The $E_\infty^{2,0}$ graded summands of the Leray filtrations of $H^2(\MM^n,\mu_r)$, $n\geq 1$, are all equal to
$H^2(U,\mu_r)$ and the $E_\infty^{1,1}$ terms all vanish. Hence, 
the kernel of $\Delta^*:H^2(\MM^2,\mu_r)\rightarrow H^2(\MM,\mu_r)$ maps {\em injectively} into the quotient $E_\infty^{0,2}$,
which is naturally a subgroup of $[\Lambda/r\Lambda]^{\oplus 2}$. 
Classes in the kernel map to classes of the form $(-\lambda,\lambda)$. Choose a class $\lambda\in \Lambda/r\Lambda$,
such that $\beta$ maps to $(-\lambda,\lambda)$. Similarly, the kernel of the pullback 
$H^2(\MM^3,\mu_r)\rightarrow H^2(\MM,\mu_r)$, via the diagonal embedding,
maps injectively into $[\Lambda/r\Lambda]^{\oplus 3}$. 
The class $\Pi_{12}^*(\beta)\Pi_{23}^*(\beta)\Pi_{13}^*(\beta^{-1})$ restricts trivially to the diagonal and maps to 
the class
\[
(-\lambda,\lambda,0)+(0,-\lambda,\lambda)+(\lambda,0,-\lambda)=(0,0,0)
\]
in $[\Lambda/r\Lambda]^{\oplus 3}$. Hence, the class $\Pi_{12}^*(\beta)\Pi_{23}^*(\beta)\Pi_{13}^*(\beta^{-1})$
is trivial in $H^2(\MM^3,\mu_r)$. The latter class maps to the class
$\Pi_{12}^*(\Theta)\Pi_{23}^*(\Theta)\Pi_{13}^*(\Theta^{-1})$ in $H^2(\MM^3,\StructureSheaf{})^*$.
Equality (\ref{eq-Brauer-class-allows-convolution}) follows.
This completes the proof of Proposition \ref{prop-construction-of-universal-object-F}.
\end{proof}

\begin{rem}
\label{rem-restriction-of-Brauer-class-over-contractible-V}
The description of the graded pieces of the Leray filtration of $H^2(\MM^2,\mu_r)$ is Step 5 of the proof above 
applies to the restriction $\MM^2_V$ of $\MM^2$ to a contractible open subset $V$ of $U$.
In that case we get that $E_\infty^{2,0}\cong H^2(V,\mu_r)$ vanishes as well, so does $E_\infty^{3,0}$, and
$H^2(\MM^2_V,\mu_r)=E_\infty^{0,2}\cong [\Lambda/r\Lambda]^{\oplus 2}$. 
Similarly, $H^2(\MM_V,\mu_r)\cong \Lambda/r\Lambda$. Consequently, 
the restriction $\widetilde{\Theta}_V$ of $\widetilde{\Theta}$ to $\MM^2_V$ is equal to 
$\pi_1^*(\tilde{\theta}^{-1})\pi_2^*(\tilde{\theta})$ for some class $\tilde{\theta}$ in $H^2(\MM_V,\mu_r)$.
Part \ref{thm-item-restriction-of-Brauer-class-over-contractible-V} of Theorem \ref{deformability} follows.
\end{rem}

Let $X$ be a $K3$ surface, $M:=X^{[n]}$, and $\delta:\F\rightarrow \F\circ\F$ the comultiplication 
of the comonad object (\ref{eq-the-comonad-object-that-deforms}).
Denote by $\pi_{ij}: M\times M \times M \to M\times M$ the projection onto the $ij$-th
factor.
The adjunction $\pi_{13}^*\dashv \pi_{13_*}$ yields the second isomorphism below.
\[
\Hom(\F,\F\circ\F)= \Hom(\F,\pi_{13_*}[\pi_{12}^*\F\otimes\pi_{23}^*\F])
\cong \Hom(\pi_{13}^*\F,\pi_{12}^*\F\otimes\pi_{23}^*\F).
\]
Denote the image of $\delta$ by 
$\tilde{\delta}:\pi_{13}^*\F\rightarrow \pi_{12}^*\F\otimes\pi_{23}^*\F$.
We get the natural morphism
\[
\pi_{13_*}(\tilde{\delta}) \ : \  \pi_{13_*}\pi_{13}^*\F \cong
\F\otimes_\ComplexNumbers Y(\StructureSheaf{M}) \ \ \longrightarrow \ \ 
\F\circ\F,
\]
where $Y(\ko_M)\in D^b(pt)$ is the Yoneda algebra of $M$.
Composing $\pi_{13_*}(\tilde{\delta})$ with the morphism 
$\id_\F\otimes \iota:\F\otimes_\ComplexNumbers \lambda_n\rightarrow \F\otimes_\ComplexNumbers Y(\StructureSheaf{M})$ (see (\ref{eq-composition-of-three-morphisms}))
we get the {\em natural} morphism
\begin{equation}
\label{eq-m-over-M-times-M}
{\mathfrak{m}}:\F\otimes \lambda_n \ \ \rightarrow \ \  \F\circ\F.
\end{equation}

\begin{lem}
\label{thm-m-is-an-isomorphism}
The morphism $\mathfrak{m}$, given in Equation (\ref{eq-m-over-M-times-M}), is an 
isomorphism, in the case where $M$ is the Hilbert scheme $X^{[n]}$, the 
universal sheaf $\U\in D^b(X\times X^{[n]})$ is the ideal sheaf of the 
universal subscheme, and $\F$ is the modular complex $\U \circ \U^\vee [2]$.
\end{lem}

\begin{proof} The morphism $\mathfrak{m}$ is obtained from the sequence of morphisms 
(\ref{eq-composition-of-three-morphisms})
\begin{equation*}
\Delta_*\ko_X\otimes_\ComplexNumbers \lambda_n \LongRightArrowOf{\iota}
\Delta_*\ko_X\otimes_\ComplexNumbers Y(\StructureSheaf{M}) 
\LongRightArrowOf{\eta\otimes id} 
\A\otimes_\ComplexNumbers Y(\StructureSheaf{M}) 
\LongRightArrowOf{m}
\A.
\end{equation*}
by pre-convolution with $\U^\vee [2]$, and post-convolution with $\U$. Here we use the fact
discussed at the end of Construction \ref{cons-yoneda-monad}  that for a morphism 
$f:T\to S$, the monadic action of $f_*f^*$ on objects of the form $f_*(\G)$ is compatible
with the action of the algebra $f_*(\ko_T)$. 
The statement then follows immediately from Theorem
\ref{thm-A-is-direct-sum}, Part (\ref{thm-item-A-is-a-direct-sum-Hilbert-scheme-case}), which 
says that the composition of the morphisms in the display above is an isomorphism.
\end{proof}

Let $\MM$ and $U$ be as in Theorem \ref{deformability} and let
$\MM^n$ denote the $n$-th fiber self-product of $\MM$ over $U$. 
The convolution $\fF\circ \fF$ is defined relative to $U$, i.e., the tensor product is taken over $\MM^3$.
Let $\Pi:\MM^2\rightarrow U$ be the natural morphism.

\begin{lem}\label{ext-line-bundle-2}
\begin{enumerate}
\item\label{item-simple} 
The relative endomorphism sheaf $\H om_{\Pi} (\fF, \fF)$ is canonically isomorphic to the structure
sheaf over a dense open subset of $U$
containing the locus ${\mathcal Hilb}$ of Corollary \ref{cor-modular-Azumaya-algebras-are-dense}.
\item \label{item-line-bundle} The relative homomorphism sheaves 
$$
\left\{  \begin{array}{l} \H om_{\Pi} (\fF, \fF\circ \fF), \mbox{and}  \\
               \H om_{\Pi} (\fF, \fF\circ \fF\circ\fF)  
               \end{array} \right.
$$
are both line bundles on a dense open subset of $U$.
\end{enumerate}
\end{lem}
\begin{proof} We shall consider only the sheaf  $\H om_{\Pi} (\fF, \fF\circ \fF)$ of part 
(\ref{item-line-bundle}).
The other sheaves are shown to be line-bundles in the same way, and the sheaf $\H om_{\Pi} (\fF, \fF)$ clearly has a global non-vanishing identity section establishing its triviality.
For $y:=(M,\eta,\Az)\in U$, denote the restriction  $\fF|_{M\times M}$ more simply as
$\fF_y$. When $y$ belongs to ${\mathcal Hilb}$, i.e., when 
$M=X^{[n]}$ where $X$ is a $K3$ surface with trivial $\Pic(X)$, we have 
\[
\fF_y\circ \fF_y \cong \fF_y\oplus \fF_y[-2] \oplus \cdots \oplus \fF_y[2-2n],
\]
by Theorem \ref{uniqueness-of-bundle}
and Lemma \ref{thm-m-is-an-isomorphism} (note that when $\Az$ is the dual of the modular Azumaya algebra $\fF_y$
is the pullback of the modular complex by the automorphism $\tau$ of $M\times M$ interchanging the two factors, by lemma 
\ref{lemma-pull-back-of-E-by-transposition}). 
Furthermore, $(\fF\circ \fF)_y$ is isomorphic to 
$\fF_y\circ \fF_y$, by the base change theorem \cite[Proposition 6.3]{Hart}.
The locus $U'$ where the fiber-wise homorphisms
$\Hom (\fF_y, \fF_y\circ \fF_y)$ are one-dimensional is locally closed by semi-continuity 
\cite[Proposition 6.4]{Hart}. $U'$ contains every $y\in {\mathcal Hilb}$, since $\Hom(\fF_y,\fF_y[k])$ vanishes for 
the modular $\fF_y$ of a Hilbert scheme and for $k<0$, by 
Lemma \ref{lemma-F-is-simple-and-rigid}. 
$U'$ is a dense open subset of $U$, by Corollary \ref{cor-modular-Azumaya-algebras-are-dense}. 
\end{proof}

We will denote again by $U$ the open subset where the statement of Lemma \ref{ext-line-bundle-2}
holds. Similarly, we will continue to denote 
the universal family by $\pi:\MM\rightarrow U$ and its fiber square by $\Pi:\MM^2\rightarrow U$.
Denote the restriction to
$D^b(\MM^2, \Theta)$ of the extension constructed in Proposition
\ref{prop-construction-of-universal-object-F} by the same symbol $\fF$.
To define the structure of a monad object on $\fF$ extending the modular one on the Hilbert
schemes, we need to produce a counit $\fee$, and a comultiplication $\fdd$.
The former structure map is in fact given by the very definition of $\fF$ in triangle 
(\ref{eq-universal-extension-of-identity-by-E}):
$$
\fE[1] \longrightarrow \fF \stackrel{\fee}{\longrightarrow}  
\ko_{\gd_{\MM}}.
$$
We shall presently produce the other, and verify that the structure maps
satisfy the necessary compatibilities.

Any extension  $\fdd: \fF\to \fF\circ\fF$ of the comultiplication is required to
give a section of the two morphisms 
\begin{equation}\label{morphs-1}
\left.  \begin{array}{l} 
\fF\circ\fee:\fF\circ\fF {\longrightarrow} \fF, \;\mbox{and}, \\
\fee\circ\fF:\fF\circ\fF {\longrightarrow} \fF,
\end{array} \right.
\end{equation}
that is, $(\fF\circ\fee)\circ \fdd = (\fee\circ\fF)\circ \fdd = \id_{\fF}$.
Conversely, the morphisms (\ref{morphs-1}) admit sections over a dense open
subset of $U$ extending the modular comultiplication; we shall call them
$\fdd_1$ and $\fdd_2$, respectively. 
Indeed, consider the following pair of morphisms of sheaves over $U$
obtained from (\ref{morphs-1}):
$$
\xymatrix{
\H om_\Pi(\fF,\fF\circ\fF) \ar@<1ex>[rr]^{\fF\circ\fee} & & 
\ar@<1ex> @{<-}[ll]^{\fee\circ\fF} \H om_\Pi(\fF,\fF) \cong \ko_{U}
}
$$ 
These morphisms are isomorphisms over a dense open subset $U^0$ 
of $U$ containing ${\mathcal Hilb}$, 
by Corollary \ref{cor-modular-Azumaya-algebras-are-dense} and Lemma \ref{lemma-F-is-simple-and-rigid}.
Define $\fdd_1$ and $\fdd_2$ to be the pre-images of $\id_{\fF}$.

\medskip

\noindent {\it Proof of Theorem \ref{deformability} (\ref{deformability-of-monad}).} 
First suppose that there exists a comultiplication $\dd_x$ for $\fF|_x$ at some point
$x\in U^0$. Then, ${\fdd_1}|_x={\fdd_2}|_x=\dd_x$ 
since $\dd_x$ is the unique section of
both morphisms in Equation (\ref{morphs-1}), by the discussion above. 
To complete the proof, first note that the equality $\fdd_1=\fdd_2$ holds over $U^0$
by the fact that we have established it over a dense subset of it. Second, note that 
the coassociativity (\ref{comp-coass}) and counit laws (\ref{comp-counit})
amount to equalities of two sections of the line bundles 
$\H om_\Pi (\fF, \fF\circ \fF)$, respectively 
$\H om_\Pi (\fF, \fF \circ \fF\circ \fF)$, over $U^0$. These equalities
hold, again because they have been established for a dense subset of $U^0$. 
Finally we rename $U^0$ as $U$. 
\qed

\begin{rem}
\label{rem-F-convolved-with-F-splits}
Let $\Pi_{13}:\MM^3\rightarrow \MM^2$ be the projection on the first and third factors. Let $\overline{\lambda}_n$ be the object 
fitting in the exact triangle
\begin{equation}\label{eqn-lambda}
\overline{\lambda}_n\rightarrow R\Pi_{13_*}\StructureSheaf{\MM^3}\rightarrow R^{2n}\Pi_{13_*}\StructureSheaf{\MM^3}
\end{equation}
in $D^b(\MM^2)$, where the right morphism is well defined due to the vanishing of $R^{i}\Pi_{13_*}\StructureSheaf{\MM^3}$
for $i>2n$. 
The construction of the comultiplication $\fdd: \fF\to \fF\circ\fF$ enables us to extend  the
isomorphism $\mathfrak{m}$ in
Equation \ref{eq-m-over-M-times-M} to a morphism
\begin{equation}
\label{eq-splitting-m-of-self-convolution-of-fF}
\overline{\mathfrak{m}}:\fF\otimes_{\StructureSheaf{U}}
\overline{\lambda}_n
\rightarrow 
\fF\circ\fF
\end{equation}
of objects over $\MM^2$, which is an isomorphism over a dense open subset $W$ of $U$ containing 
${\mathcal Hilb}$, by Lemma \ref{thm-m-is-an-isomorphism} and the argument in the proof of Lemma
\ref{lem-A-splits-over-an-open-subset}. Again we rename $W$ as $U$. 
\end{rem}

\begin{lem}
\label{lemma-a-sufficient-condition-to-be-totally-split}
Let $w:=(M,\eta,\Az)$  be a point of the open set $W$ of Remark \ref{rem-F-convolved-with-F-splits}, where 
$M\cong M_H(v)$ is isomorphic to a moduli space of $H$-stable sheaves over some $K3$ surfaces $X$,
$\fE_w$ is the modular sheaf over $M\times M$ 
associated to a twisted universal sheaf $\U$ over $X\times M_H(v)$, and $\Az$ is the
modular Azumaya algebra $\SheafEnd(\fE_w)$.
Let $\A$ be the monad object over $X\times X$ associated to $\U$.
Then the morphism $\alpha:\Delta_*\StructureSheaf{X}\otimes_\ComplexNumbers\lambda_n\rightarrow \A$, given in 
Equation (\ref{eq-composition-of-three-morphisms}), 
is an isomorphism as well. 
\end{lem}

\begin{proof}
This follows from the fact that $\overline{\mathfrak{m}}_w$ is the image of $\alpha$
via the functor $\Gamma$, given in (\ref{eq-F-is-the-image-of-the-structure-sheaf-of-the-diagonal}), and the composition 
$\Gamma^\dagger_R\circ \Gamma$ of $\Gamma$ with its right adjoint admits the identity endo-functor of $D^b(X\times X)$
as a direct summand, by Theorem \ref{thm-A-is-direct-sum} (\ref{thm-item-structure-sheaf-is-a-direct-summand}). 
Denote by $\Sigma$ the endofunctor of $D^b(X\times X)$, such that $\Gamma^\dagger_R\circ \Gamma\cong \id_{D^b(X\times X)}\oplus \Sigma.$ We have the equalities
\[
\alpha\oplus \Sigma(\alpha)=(\Gamma^\dagger_R\circ \Gamma)(\alpha)=\Gamma^\dagger_R(\overline{\mathfrak{m}}_w)
\]
Then $\alpha\oplus \Sigma(\alpha)$ is an isomorphism, since $\overline{\mathfrak{m}}_w$ is. Hence, each of $\alpha$ and $\Sigma(\alpha)$ must be an isomorphism as well.
\end{proof}

\subsection{The triangulated structure on the category of comodules}
\label{subsec-the-triangulated-structure-on-the-category-of-comodules}

We give a proof of part \ref{thm-item-triangulated-structure} of Theorem 
\ref{deformability} here, namely, that the category $D^b(\MM_V,\theta)^{\ol{\LL}}$ constructed above 
carries a 2-triangulated structure. This amounts to verifying Balmer's separability
criterion \cite{bal1} for the comonad $\ol{\LL}$. We start by briefly recalling some
of the necessary background.

Let $\Psi: \catC\to \catD$ be a functor with left adjoint $\Phi$. We say that $\Psi$ is  
{\em separable} if the unit $\eta: \id_\catD \rightarrow \Psi\Phi$ has a 
natural retraction\footnote{Recall that $\xi$ is a {\em retraction} of $\eta$ if it is a left inverse of $\eta$, i.e., 
if $\xi\eta:\id_\catD\rightarrow \id_\catD$ is
the identity natural transformation. In this case $\eta$ is a {\em section} of $\xi$.} 
$\xi: \Psi\Phi\rightarrow \id_\catC$.

An additive category $\catC$ is said to be {\em suspended} if it is equipped 
with an auto-equivalence, called  {\em suspension},
$[1]: \catC \to \catC$ (which, for simplicity, is considered an 
isomorphism, $[1]^{-1}[1]=\id_\catC= [1][1]^{-1}$). For example, any triangulated category
is a suspended category. A functor between suspended
categories is called a {\em suspended functor} if it commutes with suspensions. In general,
a property $P$ of suspended categories or functors is {\em stably} $P$ if its definition 
respects the suspended structures involved. Thus, a suspended functor 
$\Phi: \catD\to \catC$ is {\em stably separable} if it is separable, and the splitting
$\xi$ commutes with suspension.

Let $\catC$ be a category, and $\com=\langle \gL, \ee, \dd\rangle$ a comonad on $\catC$.
Denote by $\catC^\com$ the category of comodules of $\com$. 
The forgetful functor
$F: \catC^\com \to \catC$ has a right adjoint, the {\em free comodule functor},
$G: \catC \to \catC^\com$, which is defined as $G(a):=(\gL a, \dd_a)$.
It is easy to see that if $\catC$ and $\gL$ are suspended, then so is $\catC^\com$, 
and the pair of functors $F$ and $G$ commute with suspension.

\begin{defi}\label{def-separability} 
A comonad $\com$ on a category $\catC$ is said to be a {\em separable comonad} 
if there exists a natural retraction
$\hat{\dd}:\Lambda^2\rightarrow\Lambda$ of the comultiplication 
$\dd:\Lambda\rightarrow \Lambda^2$ such that 
\begin{equation}\label{equivalent-cond-sep}
\hat{\dd}\gL\circ \gL\dd=\dd\circ\hat{\dd}=\gL\hat{\dd}\circ \dd\gL
\end{equation}
\cite[Def. 3.5]{bal1}.
If $\catC$ is suspended, $\com$ is said to be {\em stably separable} if
the various functors and natural morphisms in question respect the suspension.
\end{defi}

The following abridged form of Balmer's Main Theorem \cite{bal1} is all we need 
for our purposes here:

\begin{thm}[\cite{bal1}, Theorem 5.17]\label{balmer-triangulated}
Let $\catC$ be an idempotent-complete category with a triangulation of order $N\geq 2$, and let
$\com$ be a stably separable co-monad on $\catC$, such that $\gL: \catC\to\catC$  is exact
up to order $N$. Then,
the category of $\com$-comodules $\catC^\com$ admits a triangulation of order $N$ such
that, both the free comodule functor $G:\catC\to \catC^\com$ and the forgetful functor
$F:\catC^\com \to \catC$ are exact up to order $N$. In fact, each of these properties characterizes
the triangulation on the category $\catC^\com$.  
\end{thm}

\begin{rem} Rather than go into the specifics of $N$-triangulations, we simply observe  
what is relevant for us (see \cite[Section 5]{bal1}): A 2-triangulated
category satisfies all the axioms of a triangulated category, except the octahedral axiom, while
a 3-triangulated category is also a triangulated category in the sense of Verdier (but not 
{\it vice-versa}). In our intended application of the above result, $\catC$ will be $D^b(\MM_V,\theta)^{\ol{\LL}}$,
the derived category of an abelian category, which, as such, is
in fact canonically $N$-triangulable for all $N\geq 2$ \cite[Corollaire, p. 18]{Malt}. 
While the comonad $\ol{\mathbb{L}}$ is certainly 2-exact, 
it is not clear at all that it is $N$-exact for this structure when $N>2$. 
We expect this is to be true, at least over algebraic fibers $\MM_u$ of $\MM_V$, but are unable to 
prove this at the moment.
\end{rem}

\begin{lem}\label{retraction-criterion}
Let $\com:=(\gL,\epsilon,\delta)$ be a stable comonad on a suspended category
$\catC$ realized by a stable adjunction $\Phi:\catD \rightleftarrows 
\catC :\Psi$. If the functor $\Psi$ is stably separable,
then the comonad $\com$ is stably separable. 
\end{lem}
\begin{proof}
Note first that $\gL:=\Phi\Psi$, $\epsilon:\gL\rightarrow \id_\catC$ is the counit, and $\dd=\Phi\eta \Psi$.
The statement essentially follows from the proof in  \cite[\S 2.9]{hopf}. 
We recall the easy details for the sake of completeness:
Given a retraction $\xi$ of  $\eta: \id_\catD \to \Psi\Phi$, we obtain a retraction of $\dd=\Phi\eta \Psi$
by setting $\hat{\dd}=\Phi\xi \Psi$. We will prove only 
the condition $\hat{\dd}\gL \circ \gL \dd=\dd\circ \hat{\dd}$, as 
the other equality in (\ref{equivalent-cond-sep}) follows in exactly the same way.

The equality $\hat{\dd}\gL \circ \gL \dd=\dd\circ \hat{\dd}$ translates to
$
(\Phi\xi\Psi\Phi\Psi)\circ(\Phi\Psi\Phi\eta\Psi)=(\Phi\eta\Psi)\circ(\Phi\xi\Psi),
$
which in turn would follow from $(\xi\Psi\Phi)\circ(\Psi\Phi\eta)=\eta\circ\xi$. The latter states,
for every object $x$ of $\catD$, 
the commutativity of the diagram:
\[
\xymatrix{
(\Psi\Phi)(x) \ar[r]^-{(\Psi\Phi)\eta_x} 
\ar[d]^{\xi_x}
& 
(\Psi\Phi)(\Psi\Phi)(x)
\ar[d]^{(\xi(\Psi\Phi))_x}
\\
x \ar[r]^-{\eta_x} & (\Psi\Phi)(x),
}
\]
which follows from the naturality of $\xi$ for the arrows $\eta_x$.
\hide{
is a consequence of the interchange law (\cite[\S II.5]{working}: it implies that the
the second diagram below for $\hat{\dd}\gL \circ \gL \dd$  reduces to the first diagram
corresponding to $\dd\circ \hat{\dd}$. 
$$\xymatrix{
\catC \ar[r]^\Psi & \catD \ar@/^2pc/[rr]^{\Psi\Phi}_{\Downarrow \xi} 
\ar[rr]^{\mathbb{1}_\catD}_{\Downarrow \eta} \ar@/_2pc/[rr]_{\Psi\Phi}
& &  \catD \ar[r]^{\Phi} & \catC
}$$

$$\xymatrix{
\catC \ar[r]^\Psi & \catD \ar@/^1pc/[rr]^{\Psi\Phi}  \ar@/_1pc/[rr]_{\mathbb{1}_\catD}
&  \Downarrow \xi&  \catD \ar@/_1pc/[rr]_{\Psi\Phi}  \ar@/^1pc/[rr]^{\mathbb{1}_\catD}
& \Downarrow \eta &\catD \ar[r]^{\Phi}  & \catC
}$$
}
\end{proof}

\bigskip


\medskip

\noindent{\em Proof of Theorem \ref{deformability}(\ref{thm-item-triangulated-structure})}.  
Fix a contractible Stein open subset $V$ of $U$ as in the statement of the Theorem.
We shall work over $V$, but will continue
to use the same notation as above for the restrictions of the various morphisms
appearing in parts (\ref{deformability-of-monad}), and 
(\ref{thm-item-restriction-of-Brauer-class-over-contractible-V}) of this result.

To show that $\fdbp{\MM_V,\theta}$ carries a natural 2-triangulated structure as in the statement,
it suffices to produce a retraction $\ol{\hat{\dd}}: \fF_V\circ\fF_V\to \fF_V$ of the comultiplication
$\fdd:\fF_V \to  \fF_V\circ\fF_V$ by  Theorem \ref{balmer-triangulated}, 
such that the following two diagrams commute:
\begin{equation}\label{separability}
\xymatrix{
 \ar[d]^{\fdd\circ\fF} \fF\circ\fF\ar@{-->}[r]^{\ol{\hat{\dd}}} & \ar[d]^{\fdd} \fF
 & &  \ar[d]^{\fF \circ \fdd} \fF\circ\fF\ar@{-->}[r]^{\ol{\hat{\dd}}} & \ar[d]^{\fdd} \fF\\
  \fF\circ\fF\circ\fF \ar@{-->}[r]_{\fF\circ\ol{\hat{\dd}}} & \fF\circ\fF 
  & & \fF\circ\fF\circ\fF \ar@{-->}[r]_{\ol{\hat{\dd}}\circ \fF} & \fF\circ\fF } 
\end{equation}

Consider the triangle (\ref{eqn-lambda}). The object $R\Pi_{13_*}\StructureSheaf{\MM_V^3}$
is canonically isomorphic to $\Pi^*R\pi_*\ko_{\MM_V}$ by base-change, where $\pi$ and $\Pi$ are the 
structure morphisms $\MM_V\to V$ and $\MM_V^2\to V$, respectively. As all extensions of line bundles vanish over $V$,
this object is canonically split:
$R\Pi_{13_*}\StructureSheaf{\MM_V^3} \cong \ko_{\MM^2_V}\otimes_{\ko_V}Y(\ko_{\MM_V})$, where $Y(\ko_{\MM_V}):=\oplus_{i=0}^{n}R^{2i}\pi_*\StructureSheaf{\MM_V}$.
Define $\ol{\hat{\dd}}$ to be the composition 
$\fF_V\circ\fF_V \stackrel{\ol{\mathfrak{m}}^{-1}}{\to} \fF_V\otimes_{\ko_V}\ol{\lambda}_n \to \fF_V$,
where $\ol{\mathfrak{m}}$ is given in Equation (\ref{eq-splitting-m-of-self-convolution-of-fF}) and 
the second arrow is induced by the splitting above.
We first observe that with this definition of $\ol{\hat{\dd}}$, the diagrams 
(\ref{separability}) commute when restricted to points corresponding to Hilbert schemes with 
their modular complexes. Indeed, in this case, the morphism $\mathfrak{m}$ is nothing
but the morphism on kernels corresponding to the following retraction of the comultiplication
$$
\Phi\Psi\Phi\Psi=\gL^2 \stackrel{\Phi\xi\Psi}{\longrightarrow} \Phi\Psi=\gL.
$$
Here $(\Phi, \Psi)$ is the adjoint pair $\Phi: D^b(X)\rightleftarrows D^b(X^{[n]}) :\Psi$
realizing our comonad $\com$, and 
$\xi:\Psi\Phi\rightarrow \id_{D^b(X)}$ is the retraction of $\eta: \mathbb{1}_{D^b(X)} \to \Psi\Phi$ given by 
Thoerem \ref{thm-A-is-direct-sum}, Part (\ref{thm-item-A-is-a-direct-sum-Hilbert-scheme-case}).
Thus (the restrictions of) the two diagrams (\ref{separability}) 
commute by Lemma \ref{retraction-criterion}.

Consider the sheaf $\H om_{\Pi}(\fF\circ\fF_, \fF\circ\fF)$. It follows from Lemma
\ref{lemma-F-is-simple-and-rigid}, and semi-continuity \cite[Proposition 6.4]{Hart} that
this sheaf is a vector bundle on a dense open subset of $U$ containing ${\mathcal Hilb}$. 
As above, we denote this set by $U$ also. 
The latter  is the open subset of Theorem \ref{deformability}.
Set
$d=\fdd\circ\ol{\hat{\dd}}-({\fF_V}\circ\ol{\hat{\dd}})\circ(\fdd\circ{\fF_V})$;  we note that $d_x=0$ at points
$x\in {\mathcal Hilb}\cap V$. Therefore, since this locus is dense, 
$d=0$ on $V$. The commutativity of the left diagram above follows. The argument establishing the commutativity of the right diagram is similar.
\qed

%
\subsection{Monodromy invariance}
\label{subsec-monodromy-invariance}
We prove part 
\ref{thm-item-monodromy-invariance} of Theorem \ref{deformability} in this section. 
We first recall the monodromy action on the moduli space 
$\widetilde{\fM}_\Lambda^0$ and observe that the open subset $U$ in Theorem \ref{deformability}
may be chosen to be monodromy invariant.
We then use a density theorem of Verbitsky to deduce the stated property of $U$.

The isometry group $O(\Lambda)$ acts on the moduli space of marked pairs $\fM_\Lambda$ as follows.
An element $g\in O(\Lambda)$ acts on a marked pair $(M,\eta)$ by $g(M,\eta)=(M,g\eta)$.
Let $\fM^0_\Lambda$ be a connected component of $\fM_\Lambda$ of marked pairs of $K3^{[n]}$-type.
Denote by $G\subset O(\Lambda)$ the subgroup which send $\fM^0_\Lambda$ to itself. 
The subgroup $G$ is related to the monodromy subgroup $Mon^2(M)$ of the isometry
group of the second integral cohomology a manifold $M$ of $K3^{[n]}$-type
as follows. Given a pair $(M,\eta)$ in $\fM_\Lambda^0$, we have the equality
\[
G=\eta\circ Mon^2(M)\circ \eta^{-1}.
\]
Given an element $u\in\Lambda$ satisfying $(u,u)=\pm 2$, let $R_u$ be the reflection given by 
$R_u(x):=x-\frac{2(u,x)}{(u,u)}u.$
Set 
\[
\rho_u :=\left\{
\begin{array}{rcl}
R_u & \mbox{if} & (u,u)=-2
\\
-R_u& \mbox{if} & (u,u)=2.
\end{array}
\right.
\]
The group $G$ is the subgroup of $O(\Lambda)$ generated by $\{\rho_u \ : \ (u,u)=\pm 2\}$,
by \cite[Theorem 1.2]{markman-constraints}.
There exists a character 
\[
\cov:G\rightarrow \Integers/2\Integers
\] 
satisfying $\cov(\rho_u)=\left\{
\begin{array}{rcl}
0 & \mbox{if} & (u,u)=-2
\\
1 & \mbox{if} & (u,u)=2.
\end{array}
\right.
$
(see \cite[Sec. 4.1]{markman-monodromy-I}).
The unordered pair of cosets
$\{\tilde{\theta}_1,-\tilde{\theta}_1\}$ in $\Lambda/(2n-2)\Lambda$, given in Equation (\ref{eq-Brauer-classes}),
is $G$-invariant and $g(\tilde{\theta}_1)=(-1)^{\cov(g)}\tilde{\theta}_1$, by 
\cite[Lemma 7.3]{markman-hodge}. Let $\widetilde{\fM}^0_\Lambda$
be a connected component of $\widetilde{\fM}_\Lambda$ containing a triple
$(X_0^{[n]},\eta_0,\Az_0)$, where $\Az_0$ is the modular Azumaya algebra over the cartesian square 
$X_0^{[n]}\times X_0^{[n]}$ of the Hilbert scheme of a $K3$ surface $X_0$.

\begin{thm}
\label{thm-monodromy-invariance-of-U}
\begin{enumerate}
\item
\label{thm-item-G-action}
The $G$-action on $\widetilde{\fM}_\Lambda$, given by
\[
g(M,\eta,\Az)= (M,g\eta,\Az^{(*^{\cov(g)})})
\]
maps the connected component $\widetilde{\fM}^0_\Lambda$ to itself, where 
$\Az^{(*^{\cov(g)})}$ is $\Az$, if $cov(g)=0$, and $\Az^*$, if $\cov(g)=1$.
\item
\label{thm-item-U-is-G-invariant}
The open subset $U$ of $\widetilde{\fM}^0_\Lambda$, in Theorem \ref{deformability},
may be chosen to be invariant with respect to the above action of $G$.
\end{enumerate}
\end{thm}

\begin{proof}
\ref{thm-item-G-action}) This statement is a version of \cite[Theorem 1.11]{torelli}
and its proof is included in the proof of that Theorem. 

\ref{thm-item-U-is-G-invariant}) 
All the properties that points of $U$ were required to satisfy depend on the 
isomorphism class of the Azumaya algebra. Hence, $U$ may be 
enlarged replacing it by the union of all translates $g(U)$, for all $g$ in the kernel of $\cov$.
We may thus assume that $U$ is $\ker(\cov)$-invariant. 

Choose an element $g\in G$ with $\cov(g)=1$. 
The dense subset ${\mathcal Hilb}$ is $G$-invariant, by definition.  
Hence, ${\mathcal Hilb}$ is contained in $U\cap g(U)$. The latter is $G$-invariant. 
\end{proof}

\begin{example}
Let $D\subset X^{[n]}$ be the divisor of non-reduced subschemes,  $d\in H^2(X^{[n]},\Integers)$
the class of $D$, and $R_d:H^2(X^{[n]},\Integers)\rightarrow H^2(X^{[n]},\Integers)$
the reflection given by $R_d(x)=x-\frac{2(x,d)}{(d,d)}$. Then $R_d$ is a monodromy operator, by
\cite[Cor. 1.8]{markman-monodromy-I}, and a Hodge isometry. We have $\cov(R_d)=1$, by
\cite[Lemma 4.10(4)]{markman-monodromy-I} ($R_d$
is the image of the duality operator $v\mapsto v^\vee$ via the homomorphism $f$ in that Lemma).
Let $\Az$ be the modular Azumaya algebra
over $X^{[n]}\times X^{[n]}$ and $\eta$ a marking for $X^{[n]}$. Then the triples 
$(X^{[n]},\eta,\Az)$ and $(X^{[n]},\eta R_d,\Az^*)$ belong to the same connected component 
$\widetilde{\fM}^0_\Lambda$ and have the same period.  Hence, the two triples are
inseparable points in moduli.
\end{example}

\begin{proof}[Proof of part \ref{thm-item-monodromy-invariance} of Theorem \ref{deformability}]
$U$ is a non-empty open $G$-invariant subset of $\widetilde{\fM}^0_\Lambda$,
by Theorem \ref{thm-monodromy-invariance-of-U}.
Hence, its image $\phi(U)$ in $\fM^0_\Lambda$ 
is a $G$-invariant non-empty open subset. The main result of \cite[Theorem 4.11]{verbitsky-ergodicity}
states that the $G$-orbit of a marked hyperk\"{a}hler manifold $(M,\eta)$ with second Betti number $b_2(M)\geq 5$ and Picard rank
$\leq b_2(M)-3$ is dense in 
its connected component $\fM^0_\Lambda$.
Hence, $\phi(U)$ contains $(M,\eta)$, for every marking $\eta$, such that $(M,\eta)$ belongs to the image
$\fM^0_\Lambda$ of $\widetilde{\fM}^0_\Lambda$ in the moduli space of marked pairs.
\end{proof}

The following is a sufficient condition for a triple to belong to the open subset $U$ mentioned in Theorem \ref{deformability}.

\begin{lem}
\label{lemma-a-sufficient-condition-to-belong-to-U}
The $G$-orbit of a point $(M,\eta,\Az)$ of \ $\widetilde{\fM}^0_\Lambda$ is dense in $\widetilde{\fM}^0_\Lambda$,
and is thus contained in the open subset $U$ of Theorem \ref{deformability} supporting the deformation  
$\langle \fF, \fee, \fdd \rangle$ of comonad objects,
as well as in the open subset $W$ of Remark \ref{rem-F-convolved-with-F-splits} 
where the the square $\fF\circ\fF$ is isomorphic to a direct sum of shifts of $\fF$, 
provided the following conditions are satisfied.
\begin{enumerate}
\item
\label{lemma-condition-item-non-maximal-rank}
The rank of $\Pic(M)$ is $\leq 20$.
\item
\label{lemma-condition-item-A-is-maximally-twisted}
The order of the Brauer class of $\Az$, which is the image in $H^2(M\times M,\StructureSheaf{}^*)$
of $\eta^{-1}(\tilde{\theta})$,
 is $2n-2$. Here $\tilde{\theta}$ is the class given in (\ref{eq-Brauer-classes}).
\end{enumerate}
\end{lem}

\begin{proof}
Condition (\ref{lemma-condition-item-non-maximal-rank}) implies that the fiber
$\phi^{-1}(\phi(M,\eta,\Az))$ intersects $U$, by 
Theorem \ref{deformability} \nolinebreak (\ref{thm-item-monodromy-invariance}).
Condition (\ref{lemma-condition-item-A-is-maximally-twisted}) implies that the fiber 
$\phi^{-1}(\phi(M,\eta,\Az))$ consists of a single point.
Indeed, the condition implies that the Azumaya algebra $\Az'$ of a triple $(M,\eta,\Az')$ in this fiber 
is slope stable with respect to every K\"{a}hler class on $M\times M$, by
\cite[Prop. 7.8]{markman-hodge}. It follows that 
$\Az$ is isomorphic to $\Az'$, by \cite[Lemma 5.3]{torelli}.
%
\end{proof}

\begin{rem}
\label{rem-modular-Azumaya-algebra-belongs-to-moduli}
Let $X$ be a $K3$ surface, $v$ a primitive algebraic Mukai vector, and $H$ a $v$-generic polarization, such that the dimension of $M:=M_H(v)$ is $\geq 4$. Denote by $\Az$ the modular Azumaya algebra over $M\times M$ (Definition \ref{def-modular}). 
There exists a marking $\eta$ for $M$, such that $(M,\eta,\Az)$ belongs to $\widetilde{\fM}^0_\Lambda$, if and only if $\Az$ is
$\pi_1^*\omega+\pi_2^*\omega$ slope-stable, as an Azumaya algebra, with respect to some K\"{a}hler class $\omega$ on $M$.
The above slope-stability of $\Az$ with respect to every K\"{a}hler class on $M$ is known when the order of the Brauer class of $\Az$ is equal to the rank $2n-2$ of $\Az$, by \cite[Prop. 7.8]{markman-hodge}, as well as when  $\Pic(X)$ is trivial and $v=(1,0,1-n)$ is the Mukai vector 
of the ideal sheaf of a length $n$ subscheme, so that $M=X^{[n]}$, by \cite{markman-stability}. 
Stability of the modular Azumaya algebra over $X^{[n]}\times X^{[n]}$ with respect to some K\"{a}hler class on $X^{[n]}$
is known whenever the rank of $\Pic(X)$ is  less than $20$, by \cite{markman-stability}.
Stability being a Zariski open condition, membership of $(M_H(v),\eta,\Az)$ in $\widetilde{\fM}^0_\Lambda$ follows for the generic member of a family of such moduli spaces over an irreducible base, once known for some fiber.
\end{rem}

\begin{cor}
\label{cor-monad-of-very-twisted-moduli-spaces-are-totally-split}
Let $X$ be a $K3$ surface of Picard rank $\leq 19$, $v$ a primitive algebraic Mukai vector, $H$ a $v$-generic polarization, such that the dimension of $M_H(v)$ is $\geq 4$. Assume further that 
\begin{equation}
\label{eq-gcd-equal-rank}
\gcd\{(u,v) \ : \ u\in H^*(X,\Integers) \ \mbox{and} \ c_1(u)\in H^{1,1}(X)\} \ = \  (v,v).
\end{equation}
Then the morphism $\alpha$ given in Equation (\ref{eq-composition-of-three-morphisms}) is an isomorphism (so the monad object $\A$ is totally split).
\end{cor}

\begin{proof}
The Picard rank of $M_H(v)$ is $\leq 20$. 
The order of the Brauer class of the modular Azumaya algebra $\Az$ over $M_H(v)\times M_H(v)$ is 
equal to the left hand side of Equation (\ref{eq-gcd-equal-rank}), 
by \cite[Lemma 7.3]{markman-hodge}. The rank of the Azumaya algebra $\Az$ is equal to the right hand side of Equation (\ref{eq-gcd-equal-rank}). 
Thus $\Az$ is slope-stable, as an Azumaya algebra, 
with respect to every K\"{a}her class, by \cite[Prop. 7.8]{markman-hodge}. 
Hence, there exists a marking $\eta$, such that the triple $(M_H(v),\eta,\Az)$ corresponds to a point in the open set $U$ of Theorem \ref{deformability}, by Lemma \ref{lemma-a-sufficient-condition-to-belong-to-U}. The assertion now follows from Lemma \ref{lemma-a-sufficient-condition-to-be-totally-split}.
\end{proof}

\hide{
Condition (\ref{lemma-condition-item-unique-birational-model}) of Lemma \ref{lemma-a-sufficient-condition-to-belong-to-U},
requiring uniqueness of the bimeromorphic model, 
is probably not needed (see Remark \ref{rem-M-rather-than-M-prime}). 
When $M$ is projective,  the condition that $M$ has a unique birational model is equivalent to the condition
that the ample cone of $M$ is equal to the interior of its movable cone.
The movable cone of $M_H(v)$ is calculated in \cite[Sec. 9]{markman-hodge}
and the ample cone is calculated in \cite{bayer-macri}.

\begin{example}
Let $r, s$ be even integers with $r$ positive.
Let $X$ be a $K3$ surface with a cyclic Picard group generated by an ample line bundle $H$
of degree $2+2rs$. Set $v=(r,c_1(H),s)$. Then $(v,v)=2$ and $M_H(v)$ 
is $4$-dimensional and satisfies conditions
(\ref{lemma-condition-item-non-maximal-rank}) and (\ref{lemma-condition-item-A-is-maximally-twisted}) above.
If the orthogonal complement $v^\perp$ in the Mukai lattice does not contain any primitive class $\rho$ 
of self-intersection $-10$, such that $(\rho,u)$ is even for all $u\in v^\perp$, 
then Condition (\ref{lemma-condition-item-unique-birational-model}) is satisfied as well 
\cite{hassett-tschinkel}. 
\end{example}
}

%
\subsection{A $K3$ category}
\label{subsection-k3-category}

Let $M$ be an irreducible holomorphic symplectic manifold of $K3^{[n]}$-type admitting
a deformed comonad structure $\LL:=(L,\epsilon,\delta)$ 
constructed above. 

\begin{prop}
The shift by $[2]$ is a Serre functor for the category $D^b(M,\theta)^\LL$
over a dense $G$-invariant open subset, containing ${\mathcal Hilb}$,  of the moduli space $\widetilde{\fM}^0_\Lambda$ of triples. In other words,
given objects  $a$ and $b$ of $D^b(M,\theta)^\LL$,
there exists a natural isomorphism 
\begin{equation}
\label{eq-serre-duality-in-the-comonad-category}
\Hom(a,b) \IsomRightArrow \Hom(b,a[2])^*.
\end{equation}
\end{prop}

\begin{proof}
Let 
$\hat{L}:D^b(M,\theta)\rightarrow D^b(M,\theta)^\LL$ be the natural functor taking an object $x$ of $D^b(M,\theta)$
to $(L(x),\delta_x:L(x)\rightarrow L^2(x))$. The full subcategory of $D^b(M,\theta)^\LL$
with objects of the form $\hat{L}(x)$, for an object $x$ of $D^b(M,\theta)$, 
will be denoted by $D^b(M,\theta)^{\LL}_{fr}$.
We consider first the case $a,b\in D^b(M,\theta)^{\LL}_{fr}$, with
$a=\hat{L}(x)$ and $b=\hat{L}(y)$ for objects $x, y$ of $D^b(M)$.
Denote by $F:D^b(M,\theta)^{\LL}\rightarrow D^b(M,\theta)$ the forgetful functor. 
Then $\hat{L}$ is the right adjoint of $F$.
The right adjoint of the functor $L$ is isomorphic to $L[2n-2]$ over the dense  subset 
of the moduli space of triples consisting of Hilbert schemes, by Lemma 
\ref{lemma-kernel-of-the-adjoint-of-F}. The kernel $\F$ of the functor $L$ has a one dimensional space
$\Hom(\F,\F)$ if $M$ is a Hilbert scheme, and so the
set over which the isomorphism of Lemma \ref{lemma-kernel-of-the-adjoint-of-F} holds is a dense open subset. 
We omit the proof of the latter statement, which  is similar to the proof of Lemma \ref{ext-line-bundle-2}, followed
by that of Lemma \ref{lem-A-splits-over-an-open-subset}.
Over this open set we get 
\[
\Hom(x,F\hat{L}(y)[2n-2])=\Hom(x,L(y)[2n-2]) \cong \Hom(L(x),y) =\Hom(F\hat{L}(x),y)\cong \Hom(\hat{L}(x),\hat{L}(y)),
\]
where the equalities above follow from the equality $F\hat{L}=L$ and the isomorphisms are due to the adjunctions
$L\dashv L[2n-2]$ and $F\dashv \hat{L}$. We conclude that $\hat{L}[2n-2]$ is a left adjoint to the
restriction $F':D^b(M,\theta)^{\LL}_{fr}\rightarrow  D^b(M,\theta)$ of the forgetful functor $F$ to the subcategory 
$D^b(M,\theta)^{\LL}_{fr}$ of $D^b(M,\theta)^{\LL}$. 

We get the bi-functorial isomorphisms
\begin{equation}
\label{eq-bi-functorial-isomorphism}
\Hom(\hat{L}(x),\hat{L}(y))  \cong  \Hom(L(x),y) \cong \Hom(y,L(x)[2n])^*
\cong \Hom(\hat{L}(y),\hat{L}(x)[2])^*.
\end{equation}
The first isomorphism is due to the adjunction $F\dashv \hat{L}$ and the equality $F\hat{L}=L$.
The second is Serre Duality for $D^b(M,\theta)$. 
The last is due to the adjunction $\hat{L}[2n-2]\dashv F'$. Thus,  $D^b(M,\theta)^{\LL}_{fr}$ is a 
$K3$ category.

Denote by $\sigma_{x,y}:\Hom(\hat{L}(x),\hat{L}(y))\rightarrow \Hom(\hat{L}(y),\hat{L}(x)[2])^*$
the isomorphism given in Equation (\ref{eq-bi-functorial-isomorphism}).
Let $e:\hat{L}(w)\rightarrow \hat{L}(x)$ be a morphism. Bi-functoriality yields the commutative diagram
\[
\xymatrix{
\Hom(\hat{L}(x),\hat{L}(y)) \ar[r]^{\sigma_{x,y}} \ar[d]^{e^*} &
\Hom(\hat{L}(y),\hat{L}(x)[2])^* \ar[d]^{(e_*)^*}
\\
\Hom(\hat{L}(w),\hat{L}(y))\ar[r]^{\sigma_{w,y}}  &
\Hom(\hat{L}(y),\hat{L}(w)[2])^*,
}
\]
and the equality 
$
\sigma_{w,y}e^*=(e_*)^*\sigma_{x,y}.
$
Similarly, given a morphism $f:\hat{L}(y)\rightarrow \hat{L}(z)$, we get the analogous equality
\[
\sigma_{x,z} f_*=(f^*)^*\sigma_{x,y}
\]
of homomorphisms from $\Hom(\hat{L}(x),\hat{L}(y))$ to $\Hom(\hat{L}(z),\hat{L}(x)[2])$.
If $w=x$ and $y=z$ the two equalities above yield the equality
\begin{equation}
\label{eq-sigma-commutes-with-idempotents}
\sigma_{x,y}(e^*f_*) = (e_*)^*(f^*)^*\sigma_{x,y}.
\end{equation}

The comonad category $D^b(M,\theta)^\LL$ is the idempotent completion of the category 
of free comodules $D^b(M,\theta)^{\LL}_{fr}$. This follows from \cite[Prop. 2.10]{bal1}
and triangulated Barr-Beck  \cite{bb,elagin} (also see \cite{bal2}). 
Objects of the idempotent completion
are pairs $(\hat{L}(x),e)$, where $e\in \Hom(\hat{L}(x),\hat{L}(x))$ is an idempotent.
A morphism in $\Hom((\hat{L}(x),e),(\hat{L}(y),f))$ is a morphism $g:\hat{L}(x)\rightarrow \hat{L}(y)$
satisfying $fg=g=ge$. In other words, $\Hom((\hat{L}(x),e),(\hat{L}(y),f))$ is the image of 
the idempotent endomorphism $e^*f_*$ of $\Hom(\hat{L}(x),\hat{L}(y))$. 

Set $e_1:=e^*f_*$, $e_2:=e^*(1-f)_*$, $e_3:=(1-e)^*f_*$, and $e_4:=(1-e)^*(1-f)_*.$
These are commuting idempotent endomorphisms of $\Hom(\hat{L}(x),\hat{L}(y))$ satisfying 
$e_ie_j=0$, if $i\neq j$, and $\sum_{i=1}^4e_i=1$. 
Set $\tilde{e}_1:=(e_*)^*(f^*)^*$, $\tilde{e}_2:=(e_*)^*((1-f)^*)^*$, $\tilde{e}_3:=((1-e)_*)^*(f^*)^*$, 
$\tilde{e}_4:=((1-e)_*)^*((1-f)^*)^*$. These are commuting idempotent endomorphisms of 
$\Hom(\hat{L}(y),\hat{L}(x)[2])^*$ satisfying 
$\tilde{e}_i\tilde{e}_j=0$, if $i\neq j$, and $\sum_{i=1}^4\tilde{e}_i=1$. 
We get the decomposition $\sigma_{x,y}=\sum_{i=1}^4\sum_{j=1}^4\tilde{e}_j\sigma_{x,y}e_i$.
Equation (\ref{eq-sigma-commutes-with-idempotents}) implies that $\sigma_{x,y}e_i=\tilde{e}_i\sigma_{x,y}$,
$1\leq i \leq 4$ (note that the equation holds with $e$ replaced by $1-e$ or $f$ replaced by $1-f$). 
Hence, $\tilde{e}_j\sigma_{x,y}e_i=0$, if $i\neq j$. Consequently, 
$\sigma_{x,y}$ maps the image of $e_i$ isomorphically onto the image of $\tilde{e}_i$, for $1\leq i\leq 4$.
Considering the case $i=1$ we conclude that $\sigma_{x,y}$ maps 
$\Hom((\hat{L}(x),e),(\hat{L}(y),f))=Im(e_1)$ isomorphically onto $Im(\tilde{e}_1)$.
Now $Im(\tilde{e}_1)$ maps isomorphically onto 
$\Hom((\hat{L}(y),f),(\hat{L}(x)[2],e))^*$ via
the natural homomorphism 
$\Hom(\hat{L}(y),\hat{L}(x)[2])^*\rightarrow \Hom((\hat{L}(y),f),(\hat{L}(x)[2],e))^*$. 
We thus obtain the desired isomorphism (\ref{eq-serre-duality-in-the-comonad-category}).
\hide{It is known that the idempotent
completion of a triangulated category with a Serre functor carries a Serre functor \cite{idem-serre}.
This suffices to conclude that $D^b(M)^\LL$ is a $K3$ category as follows:
First, if $S: D^b(M)^\LL\to D^b(M)^\LL$ is the Serre functor, then 
for $a,b\in D^b(M)^{\LL}_{fr}$, we have functorial isomorphisms
$$
\Hom (b,a[2])^*\cong \Hom(a,b) \cong \Hom(b,Sa)^*.
$$
The Yoneda Lemma implies that $S|_{D^b(M)^{\LL}_{fr}}=[2]$.
Next, the inclusion 
$D^b(M)^\LL_{fr}\hookrightarrow D^b(M)^\LL$
induces an equivalence between the following functor categories \cite[Prop. 1.3]{bs}:
$$
\Hom_{add}(D^b(M)^\LL, D^b(M)^\LL)\stackrel{\cong}{\longrightarrow}
\Hom_{add}(D^b(M)^\LL_{fr}, D^b(M)^\LL).
$$
As both the endofunctors $S$ and $[2]$ of  $D^b(M)^\LL$ restrict to $[2]$ on 
$D^b(M)^\LL_{fr}$, we have that $S$ and $[2]$ are naturally isomorphic. This 
finishes the proof.
}
\end{proof}

%

\section{Comparison with Toda's Category}
\label{Toda-comp}

The first order deformations of the category of coherent sheaves $Coh(S)$ on a smooth,
projective variety $S$ are parametrized by its degree $2$ Hochschild cohomology 
$HH^2(S)$. This has an interesting interpretation via the HKR-isomorphism,
$$I^*: HT^2(S) = H^0(\wedge^2 T_S)\oplus H^1(T_S) \oplus H^2(\ko_S) \stackrel{\cong}{\to} HH^2(S),$$ 
namely, the general deformation may be understood as composed of
non-commutative, complex and ``gerby'' parts corresponding to the three summands.

Given a class $\eta \in HT^2(S)$, Toda gave an explicit
construction (\cite{toda-deformations}) of the corresponding infinitesimal deformation $Coh(S,\eta)$: Let $\eta_{0,2}$
be the component of $\eta$ in $H^2(\ko_S)$. Then $Coh(S,\eta)$ is the 
$\CC[\eee]/(\eee^2)$-linear abelian category of $\eta_{0,2}$-twisted 
coherent sheaves of modules over a deformed ``structure sheaf'' $\ko_S^\eta$ of
noncommutative $\CC[\eee]/(\eee^2)$-algebras. Further,
he proved that these deformations behave functorially under Fourier-Mukai transformations:

\begin{thm}[\cite{toda-deformations}]\label{toda-main}
Let $S$ and $Y$ be smooth, projective varieties, and suppose that there 
is a Fourier-Mukai equivalence $\Phi: D^b(S) \to D^b(Y)$. If 
$\phi: HH^*(S) \stackrel{\sim}{\to} HH^*(Y)$ is the induced map on cohomology, 
there is an equivalence $\Phi^\dag : D^b(Coh(S,\eta)) \to D^b(Coh(Y, \phi(\eta)))$ such
that the following diagram commutes (up to natural isomorphisms of functors):
$$\xymatrix{
\ar[d]_{\Phi}D^b(S)\ar[r]^{i_*\;\;\;\;\;\;\;} & \ar[d]^{\Phi^\dag} D^b(Coh(S,\eta))\ar[r]^{\;\;\;\;\;\;\;i^*} & \ar[d]^{\Phi^-}D^-(S)\\
D^b(Y) \ar[r]^{i_*\;\;\;\;\;\;\;} & D^b(Coh(Y,\phi(\eta))) \ar[r]^{\;\;\;\;\;\;\;i^*} & D^-(Y)  }
$$
\end{thm}
\noindent The notation $D^-$ refers to the derived category of {\em bounded above} complexes 
of coherent sheaves, while $i_*$ and $i^*$ stand for restriction and extension of scalars
(see \cite[Section 4]{toda-deformations}), respectively.

In this section, $M$ will denote a moduli space $M_H(v)$.
Our goal is to prove that the category of comodules constructed in the previous 
section via deformations of $M$ agrees infinitesimally with Toda's category.
In other words, for every direction 
in the 21-parameter space that $\F$ deforms, there is a 
class $\eta\in HT^2(X)$ for our K3 surface $X$ such that 
the comodule category over the dual numbers is exact equivalent to $D^b(Coh(X,\eta))$. 
This is Theorem \ref{thm-comparison-with-toda}. It is proven under the following assumption, which will remain 
in force sections \S\ref{subsection-map-on-tangent-spaces} and \S\ref{subsection-comparison}.
\begin{assumption}\label{assumption-in-compatibility-section}
The moduli space $M=M_H(v)$ is a {\em fine} moduli space on an algebraic $K3$
surface $X$, supporting a universal
family $\U$, such that the kernel $\U^\vee[2]\circ \U$ is totally-split (as in Definition 
\ref{def-totally-split-monad}). Further, there exists a triple $(M,\eta,\Az)\in \widetilde{\fM}_\Lambda^0$
such that $\Az$ is the modular Azumaya algebra of the moduli space $M_H(v)$ (see Remark 
\ref{rem-modular-Azumaya-algebra-belongs-to-moduli}).
\end{assumption}

\begin{rem} 
The assumption of algebraicity on $X$ is in order to use the results of \cite{toda-deformations},
which are stated in the context of smooth and projective varieties. One expects these results to hold
over proper analytic manifolds also, but this generalization does not appear in the literature.
\end{rem}

\subsection{Hochschild cohomology and deformations}
Let $\Phi: D^b(S) \to D^b(Y)$ be a Fourier-Mukai equivalence; write 
$\P\in D^b(S\times Y)$ for its kernel and $\Q\in D^b(Y\times S)$ for that of its 
inverse. We have an isomorphism $\phi: HH^i(S) \to HH^i(Y)$ defined as
\begin{equation}\label{mapHH}
(\ko_{\gd_S}\stackrel{\eta}{\longrightarrow}\ko_{\gd_S}[i])\mapsto 
(\ko_{\gd_Y}=\P\circ\ko_{\gd_S}\circ\Q
\stackrel{\P\circ\eta\circ\Q}{\longrightarrow}\P\circ\ko_{\gd_S}[i]\circ\Q=\ko_{\gd_Y}[i])
\end{equation}
In particular, for $i=2$, this defines a bijective correspondence between the infinitesimal
deformations of $Coh(S)$ and those of $Coh(Y)$. Note, however, that this makes use
of the fact that $\Phi$ is invertible. In fact, one does not have
functoriality for Hochschild cohomology under general integral transforms.

For later use, we state a criterion of Toda and Lowen for when an object in the
derived category can be deformed to first order. First recall the construction
of the Atiyah class $a(\G)\in \Hom(\G, \G\otimes\Go_S[1])$ for any object $\G\in D^b(S)$:
Regarding the sequence 
$$
0 \to \I_{\gd_S}/\I_{\gd_S}^2 \to \ko_{S\times S}/\I_{\gd_S}^2 \to \ko_{\gd_S}\to 0
$$
as a sequence of Fourier-Mukai kernels, and taking integral 
transforms of $\G$ accordingly, we obtain the triangle 
\begin{equation}\label{atiyah}
\G\otimes \Go_S \to \pi_{2,*}(\pi_1^*(\G)\otimes \ko_{S\times S}/\I_{\gd_S}^2) \to \G. 
\end{equation}
Then, $a(\G)$ is the extension class of (\ref{atiyah}). 

\begin{thm}{\rm (\cite[Prop. 6.1]{toda-deformations}; \cite[Thm. 1.1]{lowen})}
\label{criterion}
Let $\G\in D^b(S)$ and $u\in HT^2(S)$. There exists a perfect object $\fG_1 \in D^b(S,u)$ such that
$i^*\fG_1\cong \G$ if and only if $u\cdot \exp a(\G)=0$ in $\Hom(\G,\G[2])$.
\end{thm}
\noindent Here $\exp a(\G)=\sum a^i(\G)$, the summand $a^i(\G)\in \Hom(\G,\G\otimes\Go^i[i])$
being the $i$-fold composition of $a(\G)$ with itself, followed by anti-symmetrization.

\begin{rem}\label{rem-obstruction-thry}
The sufficiency of the vanishing of  $u\cdot \exp a(\G)$ for the deformability of $\G$ was first proven
by Toda, {\it op. cit.}. The necessity of this condition in this criterion follows from Lowen's work, who proved
that the obstruction to the deformability of $\G$ is precisely the image of $u$ under the {\em characteristic
morphism}:
$$
HH^2(S) \stackrel{\chi_\G}{\longrightarrow} Hom_S(\G,\G[2]).
$$
Regard the elements of 
$HH^2(S)$ as natural transformations between the functors $\mathbb{1}_{D^b(S)}$ and $[2]$.
Evaluating them on $\G$ defines $\chi_\G$. The statement of the criterion can then be deduced from 
the commutativity of the following diagram (see \cite[Proposition 4.5]{{cal-mukai2}}):
$$
\xymatrix{
\ar[d]_{I^*}HT^2(S) \ar[rr]^{\cdot\exp a(\G)\;\;\;\;\;\;\;\;\;} & & \Hom_S(\G,\G[2]). \\
HH^2(S) \ar[urr]_{\chi_\G} & &
}
$$
\end{rem}

\begin{rem}\label{rem-comment-on-toad-main}
Theorem \ref{criterion} is the main component of the proof Theorem \ref{toda-main}. Using it, 
\cite{toda-deformations} proves that that the kernel of the Fourier-Makai 
transform $\Phi: D^b(S)\to D^b(Y)$ deforms to the derived category of the first-order 
deformation of $Coh(S\times Y)$ in the direction $p_S^*(-\check{\eta})+p_Y^*(\phi(\eta))$.
Here $(\check{\_})$ denotes the action of transposition of factors on Hochschild cohomology.  
\end{rem}

\begin{example}\label{example-diagonal}
The identity functor $D^b(S)\to D^b(S)$ is interesting in light of Theorem \ref{toda-main}.
Its kernel $\ko_{\Delta}$ must deform along every direction 
$p_1^*(-\check{\eta})+ p_2^*(\eta)$ by Remark \ref{rem-comment-on-toad-main}.
It is possible to explicitly define a canonical deformation, 
the ``structure sheaf of the deformed diagonal'' $\ko_{\ol{\Delta}}$. In keeping with the following
sections, rather than $\eta$, we prefer to work with its image $u=I^2(\eta)\in HT^2(S)$ under 
the HKR isomorphism. Also, we assume that $u=(\gamma, 0,0)$,  with $\gamma \in H^0(\bigwedge^2TM)$,
which is the only somewhat subtle case.

The theory of quasi-coherent sheaves on first-order noncommutative deformations mirrors that
for schemes \cite[\S3,4]{toda-deformations}. In particular, over any affine $U$, such a sheaf $\M$ is 
determined by its sections $\M(U)$, and its sections over a principal affine $U_f$ are
precisely the localization $\M(U)_f$ \cite[Lemma 3.1, Def. 4.1]{toda-deformations}.

Fix an ample line bundle $\LB$ on $S$; for any section $f\in \Gamma(S,\LB^n)$, $n\in\ZZ$, let $S_f$
be the affine open where $f$ does not vanish. Consider the affine open covering $\C'$ of $S$ given by
$\{S_f: f\in \Gamma(S,\LB^n ), \; n\in \ZZ\}$. Note that $\C'$ is closed under intersection, and that 
given any two elements of $\C'$, their intersection is a principal affine in each of them. Consider the
affine open covering $\C:=\{U\times V: U,V\in\C'\}$ of $S\times S$.

Let $D$ stand for the dual numbers over $\CC$. Denote the class 
$-p_1^*(\g)+p_2^*\g$ by $-\g\boxplus\g$. Note that given $U\times V\in\C$,
$\ko_{S\times S}^{-{\g}\boxplus \g}(U\times V)=
\ko_S^{-{\g}}(U)\otimes_D \ko_S^\g(V)$.
For each $U\times V\in\C$, let $\A_{U\times V}$ be the
$D$-flat coherent $\ko_{U\times V}^{-{\g}\boxplus \g}$ module whose
sections over $(U\times V)$ are $\ko_S^\g(U\cap V)$. The right 
$\ko_S^{-{\g}}(U)\otimes_D \ko_S^\g(V)$-module structure of 
$\ko_S^\g(U\cap V)$ is given by restriction to $U\times V$, followed by left and
right multiplication in the ring  $\ko_S^\g(U\cap V)$. We claim that the coherent sheaves
$\A_{U\times V}$ glue to give a coherent sheaf $\ko_{\ol{\Delta}}$ over $S\times S$.
It suffices to check that for any $U\times V\in \C'$, and principal affine subsets $U_g\subset U$,
$V_h\subset V$, there is a canonical identification between the modules 
$\A_{U\times V}(U_g\times V_h)$ and $\A_{U_g\times V_h}(U_g\times V_h)$. The former is
the localization with respect to the left $\ko_S^\g(U\cap V)$, right $\ko_S^\g(U\cap V)$ bimodule
structure, $\ko_S^\g(U\cap V)_{g\otimes h}$. The latter is $\ko_S^\g(U\cap V)_{gh}$, the
localization being with respect to the right  $\ko_S^\g(U\cap V)$-module structure. Given any two
local sections $a,b$ of $\ko_S\otimes_\CC D$, $a*_\g b=b*_{-\g} a$, so these two localizations are equal.

\end{example}

\subsection{A map on tangent spaces}\label{subsection-map-on-tangent-spaces}
We wish to compare the Hochschild cohomologies
of $X$ and $M_H(v)$, but as the functor $\Phi_\U: D^b(X) \to D^b(M_H(v)$
is not an equivalence, we do not {\it a priori} have a homomorphism
$\phi: HH^*(X) \to HH^*(M)$. Nevertheless, the natural construction 
(\ref {mapHH}) can be modified to give a workable
map between the degree 2 components of these groups:

\medskip 

\begin{cons} Write $\V\in D^b(M\times X)$ for $\U^\vee[2]$, 
the kernel of the adjoint to $\Phi_\U$. As in (\ref{mapHH}), given $\eta\in HH^2(X)$,
one obtains a map:
$$
\F=\U\circ\ko_{\gd_X}\circ\V \stackrel{\U\circ\eta\circ\V}{\longrightarrow} 
\U\circ\ko_{\gd_X}[2]\circ\V=\F[2].
$$ 
Applying the functor $\Hom_{M^2(v)}(\;\_\;,\ko_{\gd_{M}}[2])$ to the triangle
$\E[1]\to\F\stackrel{\ee}{\to}\ko_{\gd_{M}}$, we obtain the sequence
$$
\Hom(\E,\ko_{\gd_{M}}) \to \Hom(\ko_{\gd_{M}},\ko_{\gd_{M}}[2]) \to
\Hom(\F, \ko_{\gd_{M}}[2]) \to \Hom(\E,\ko_{\gd_{M}}[1]).
$$
The first and last groups are $(0)$. Assuming this for the
moment, set $\phi^{HH}(\eta)$ to be the unique lift of 
$\ee[2](\U\circ\eta\circ\V)\in \Hom_{M^2(v)}(\F, \ko_{\gd_{M}}[2])$ in 
$\Hom(\ko_{\gd_{M}},\ko_{\gd_{M}}[2])$. This defines the desired map
\begin{equation}
\label{eq-phi-HH}
\phi^{HH}: HH^2(X) \to HH^2(M).
\end{equation} 

We prove the claim using the spectral sequence
$$
E_2^{p,q}= \Hom^p(\fH^{-q}(\gd_{M}^*\E), \ko_{M})\Longrightarrow
\Hom^{p+q}(\E,\ko_{\gd_{M}})
$$
and Proposition \ref{resofE}. Observe that 
$\H^{-1}(\Delta_M^*(\E),\StructureSheaf{M}) \cong \SheafTor_1(\E,\StructureSheaf{\Delta_M})$ vanishes
by the latter result. Hence, $E_2^{0,1}$ vanishes.
The question reduces to proving the vanishing of the
terms
$$ E_2^{i,0}:=\Hom_{M}^i(\Go^2_{M}/\ko_{M}\cdot\sigma,\ko_{M})$$
for $i=0,1$. This follows from the facts that the sheaf $(\Go^2_{M}/\ko_{M}\cdot\sigma)$ is 
a self-dual direct summand of $\Omega^2_M$, and that $H^i(\Omega^2_M/\ko_{M}\cdot \sigma)$ vanishes for $ i=0,1$.

We can say this slightly differently: There are natural maps
\begin{equation}\label{eq-alt-phi-HH}\xymatrix{
HH^2(X) \ar[r]^{\U\circ\_\hspace{1cm}} & \Hom_{X\times M}(\U,\U[2]) & \ar[l]_{\hspace{1cm}\_\circ \U} HH^2(M),
}
\end{equation} 
the left map is an injection, while the right is an isomoprhism. The map $\phi^{HH}$ is the one obtained by
composing the first with the inverse of the second.
\end{cons}

Hochschild homology is naturally a module over Hochschild cohomology, the action being composition.
This structure gives rise to the homomorphisms
\begin{eqnarray*}
m_X & :& HH^2(X)\rightarrow \Hom(HH_0(X),HH_{-2}(X))
\\
m_M & : & HH^2(M)\rightarrow \Hom(HH_0(M),HH_{-2}(M)).
\end{eqnarray*}
We note that $m_X$ is an isomorphism and $m_M$ is injective.
By functoriality of Hochschild homology, we also have the maps
\begin{eqnarray*}
\Phi_*:=\Phi_{\U_*}&:&HH_{*}(X)\rightarrow HH_{*}(M), \ \mbox{and}
\\
\Psi_*:=\Psi_{\U_*}&:&HH_{*}(M)\rightarrow HH_{*}(X).
\end{eqnarray*}

The map $\phi^{HH}$ intertwines the Hochschild module structures via $\Phi_*$ and $\Psi_*$ :

\begin{lem}
\label{lemma-image-of-phi-HH-commutes-with-Phi-Psi}
The following diagram is commutative for every $\lambda\in HH^2(X)$.
\[
\xymatrix{
HH_i(M)\ar[r]^{\Psi_*}\ar[d]_{m_M(\phi^{HH}(\lambda))}
&
HH_i(X)\ar[r]^{\Phi_*}\ar[d]^{m_X(\lambda)}
&
HH_i(M)
\ar[d]^{m_M(\phi^{HH}(\lambda))}
\\
HH_{i-2}(M)\ar[r]^{\Psi_*}
&
HH_{i-2}(X)\ar[r]^{\Phi_*}
&
HH_{i-2}(M)
}
\]
\end{lem}

\begin{proof}
Let $\lambda_X$ be a class in $HH^2(X)$ and $\lambda_M$ a class in $HH^2(M)$
satisfying the equality $\U\circ \lambda_X=\lambda_M\circ \U$ in $\Hom_{X\times M}(\U,\U[2])$. 
Then the following diagram commutes
\[
\xymatrix{
HH_i(X) \ar[r]^-{\Phi_*} \ar[d]_{m_X(\lambda_X)} &
HH_i(M) 
\ar[d]^{m_M(\lambda_M)}
\\
HH_{i-2}(X) \ar[r]_-{\Phi_*} &
HH_{i-2}(M), 
}
\]
by the proof of  [AT, Prop. 6.1]. The above statement holds and its proof applies without the assumption
that the right map in diagram (\ref{eq-alt-phi-HH}) is an isomorphism. 
The statement thus applies also to the functor $\Psi_\U$. Apply the statement 
with $\lambda_X=\lambda$ and $\lambda_M=\phi^{HH}(\lambda)$
to verify the commutativity of both squares in the statement of the Lemma.
\end{proof}

The endomorphism $\Phi_*\Psi_*$ of $HH_*(M)$ is self-adjoint with respect to the Mukai pairing.
It satisfies $(\Phi_*\Psi_*)^2=n\Phi_*\Psi_*$.
Indeed, the kernel of $\Psi_\U\circ\Phi_\U$ is the direct sum $\oplus_{i=0}^{n-1}\ko_\Delta[-2i]$ by
Assumption \ref{assumption-in-compatibility-section}, and
$\Psi_\U(\Phi_\U(x))=x\oplus x[-2]\oplus \cdots \oplus x[2-2n]$ for any object $x$ in $D^b(X)$. Hence,
$\Psi_*\Phi_*$ is multiplication by $n$. 
The subspace $\mbox{Im}(\Phi_*)$ is the eigenspace of $\Phi_*\Psi_*$ with eigenvalue $n$ and the subspace $\ker(\Psi_*)$
is the eigenspace with eigenvalue $0$. We get the orthogonal direct sum decomposition
\[
HH_*(M) \ = \ \mbox{Im}(\Phi_*) \oplus \ker(\Psi_*).
\]

Let $E$ be an $H$-stable sheaf on $X$ with Mukai vector $v\in HH_0(X)$. 
Let $\alpha\in HH_0(M)$ be the Mukai vector of $\Phi_\U(E^\vee[2])$
and let $\beta\in HH_0(M)$ be the Mukai vector of the sky-scraper sheaf of a point. 
Given a subset $\Sigma$ of $HH_0(M)$, 
denote by 
$
{\rm ann}(\Sigma)
$ 
the subspace of 
$HH^2(M)$ 
consisting of classes $\xi$, such that $m_M(\xi)(c)=0$, for all $c\in \Sigma$. We use the analogous notation 
for subsets of $HH_0(X)$. 

\begin{lem}\label{lemma-image}
\begin{enumerate}
\item
\label{lemma-item-image-commutes}
The image of $\phi^{HH}$ is equal to the subspace of $HH^2(M)$ consisting of classes $\xi$, such that $m_M(\xi)$ 
commutes with $\Phi_*\Psi_*$.
\item
\label{lemma-item-image-of-ann-v-dual}
The following equality of subspaces of $HH^2(M)$ holds:
\begin{equation}
\label{eq-image-of-ann-v-dual}
\phi^{HH}({\rm ann}(v^\vee)) \ = \
{\rm ann}\{\alpha,\beta\}.
\end{equation}
\item
\label{lemma-item-image-of-phi-HH-in-HT}
The normalized HKR isomorphism maps the subspace $\phi^{HH}({\rm ann}(v^\vee))$ of $HH^2(M)$ into the 
direct sum $H^1(TM) \oplus H^2(\StructureSheaf{M})$  in $HT^2(M)$.
The image  is equal to the subspace
\begin{equation}
\label{eq-image-of-phi-HH-in-HT}
\mho:=\{(\xi,\theta) \ : \ \xi\in H^1(TM), \ \theta\in H^2(\StructureSheaf{M}), \ \mbox{and} \  \ 
\xi\cdot c_1(\alpha)+(2-2n)\theta=0\}.
\end{equation}
\end{enumerate}
\end{lem}

\begin{proof}
(\ref{lemma-item-image-commutes}) 
$\Phi_*\Psi_*$
commutes with $m_M(\phi^{HH}(\lambda))$, for all $\lambda\in HH^2(X)$, 
by Lemma \ref{lemma-image-of-phi-HH-commutes-with-Phi-Psi}.
We know that $\phi^{HH}$ is injective and that its image has codimension $1$.
Hence 
it remains to exhibit classes $\xi$ of $HH^2(M)$, which do not commute with $\Phi_*\Psi_*$.

$\Psi_\U$ maps the sky-scraper sheaf 
of the point $[E]$ corresponding to the isomorphism class of the sheaf $E$ to the object
$E^\vee[2]$ in $D^b(X)$. Let $e$ be the class of $E^\vee[2]$ in $HH_0(X)$.
Then $\alpha=\Phi_*(e)$, $\Psi_*(\alpha)=ne$, and $\Psi_*(\beta)=e.$ Hence,
\[
\Psi_*(\alpha-n\beta)=0.
\]

The class $c_1(\alpha)$ does not vanish \cite[Lemma 7.2]{markman-hodge}. Hence, there exists 
a class $\xi$ of $H^1(TM)$, such that $\xi\cdot c_1(\alpha)$ is a non-zero class in 
$H^2(M,\StructureSheaf{M})\subset H\Omega_{-2}(M)$. Considering $\xi$ as a class in $HH^2(M)$, via the HKR isomorphism,
then $\xi$ belongs to ${\rm ann}(\beta)$, but it does not belong to ${\rm ann}(\alpha)$. Hence, $m_M(\xi)(\alpha-n\beta)\neq 0.$
Assume that $m_M(\xi)$ commutes with $\Phi_*\Psi_*$. Then ${\rm Im}(\Phi_*)$ and $\ker(\Psi_*)$ are invariant with respect to 
$m_M(\xi)$. In addition, we have
\[
m_M(\xi)(\alpha-n\beta)=
m_M(\xi)(\alpha)
\]
The left hand side belongs to $\ker(\Psi_*)$, since $\alpha-n\beta$ does, and the right hand side belongs to the image of $\Phi_*$,
since $\alpha$ does.
Hence, $m_M(\xi)(\alpha)$ vanishes and $\xi$ belongs to ${\rm ann}(\alpha)$. A contradiction.
Hence, $m_M(\xi)$ does not commute with $\Phi_*\Psi_*$. 

(\ref{lemma-item-image-of-ann-v-dual})
Let $\lambda\in HH^2(X)$ be a class, such that $\phi^{HH}(\lambda)$ belongs to ${\rm ann}(\beta)$.
Then $\phi^{HH}(\lambda)$ commutes with $\Phi_*\Psi_*$ and so $\phi^{HH}(\lambda)$ belongs also to ${\rm ann}(\alpha)$,
as shown in the proof of part (\ref{lemma-item-image-commutes}). Hence, 
the intersection ${\rm Im}(\phi^{HH})\cap {\rm ann}(\beta)$ is contained ${\rm ann}\{\alpha, \beta\}$.
Note that both ${\rm ann}(\beta)$ and ${\rm Im}(\phi^{HH})$ are hyperplanes in $HH^2(M)$. 
We have seen above that ${\rm ann}(\beta)$ is not contained in ${\rm ann}(\alpha)$.
We conclude the equality
\begin{equation}
\label{eq-ann-alpha-beta-as-intersection-of-image-with-ann-beta}
{\rm Im}(\phi^{HH})\cap {\rm ann}(\beta) = {\rm ann}\{\alpha, \beta\}.
\end{equation}
Furthermore, 
the hyperplanes ${\rm Im}(\phi^{HH})$ and ${\rm ann}(\beta)$ are distinct and 
${\rm ann}\{\alpha, \beta\}$ has codimension $2$ in $HH^2(M)$. 
The equality 
\[
\phi^{HH}({\rm ann}(v^\vee)) \ = \ {\rm Im}(\phi^{HH})\cap {\rm ann}(\alpha), 
\]
follows from the commutativity of the right square in Lemma \ref{lemma-image-of-phi-HH-commutes-with-Phi-Psi}
and the injectivity of $\Phi_*$. 
Hence, if  ${\rm ann}(\alpha)$ is a hyperplane in $HH^2(M)$, then 
the hyperplanes ${\rm ann}(\alpha)$ and ${\rm Im}(\phi^{HH})$ are distinct.
We conclude that either ${\rm ann}(\alpha)$ is contained in ${\rm ann}(\beta)$, or 
${\rm ann}(\alpha)$ is a hyperplane and the three hyperplanes ${\rm ann}(\alpha)$, ${\rm ann}(\beta)$,
and ${\rm Im}(\phi^{HH})$ are distinct. 

If $H_1$, $H_2$, $H_3$
are three distinct hyperplanes  in a vector space, such that $H_1\cap H_2$ and $H_2\cap H_3$
are equal to the same subspace $W$, then $H_1\cap H_3=W$. 
The equality
\[
 {\rm Im}(\phi^{HH})\cap {\rm ann}(\alpha) \ = \ {\rm ann}\{\alpha, \beta\}
\]
thus follows from (\ref{eq-ann-alpha-beta-as-intersection-of-image-with-ann-beta}). 
The above equality is clear if ${\rm ann}(\alpha)$ is contained in ${\rm ann}(\beta)$.
The equality (\ref{eq-image-of-ann-v-dual}) follows from the last two displayed equalities.

(\ref{lemma-item-image-of-phi-HH-in-HT}) 
The HKR isomorphism maps ${\rm ann}(\alpha)$ into 
the hyperplane in $HT^2(M)$, consisting of classes $(\pi,\xi,\theta)$,
$\pi\in H^0(\wedge^2TM), \xi\in H^1(TM), \ \theta\in H^2(\StructureSheaf{M}),$ satisfying
\[
\pi\cdot \alpha_2+\xi\cdot c_1(\alpha)+(2-2n)\theta=0,
\]
where $\alpha_2$ is the graded summand in $H^2(\Omega_M^2)$ of the 
image via the HKR isomorphism of the Mukai vector $\alpha$, 
since the rank of $\alpha$ is $2-2n$. Indeed the latter hyperplane is the kernel of the composition
$HT^2(M)\rightarrow HH_{-2}(M)\rightarrow H^2(\StructureSheaf{M})$
of pairing with $\alpha$ followed by projection on the direct summand $H^2(\StructureSheaf{M})$.
(We are using here the proof of Calduraru's conjecture about the isomorphism of harmonic 
and Hochschild strucures \cite{CRVdB}.)
The HKR isomorphism maps ${\rm ann}(\beta)$ onto the direct sum  $H^1(M,TM)\oplus H^2(M,\StructureSheaf{M})$.
The statement now follows from part (\ref{lemma-item-image-of-ann-v-dual}).
\end{proof}

\subsection{The comparison}\label{subsection-comparison}
Given $u\in HT^2(X)$, it will be convenient to adopt the simpler notation 
$D^b(X, u)$ for the deformed category $D^b(Coh(X,u))$; a similar notation is
used for the corresponding categories on $M$.
Let $\phi^T: HT^2(X) \to HT^2(M)$ denote the conjugate
$(I^2_{M})^{-1}\circ\phi^{HH}\circ I^2_X$. Also let $\check{\_}: HT^*(X)\to HT^*(X)$
be the operation which on a homogeneous $t\in H^p(\wedge^qT_X)$ is defined as
$\check{t}:=(-1)^qt$. For $u\in HT^*(X)$, and $w\in HT^2(M)$, the class 
$\pi_X^*u+\pi^*_{M}w\in HT^2(X\times M)$ will be denoted by $u\boxplus w$; the same
notation will be followed for the other Cartesian products, such as $M\times M$, that appear.

\begin{lem}\label{lem-U-deforms}
Let $u\in HT^2(X)$, and $w=\phi^T(u)\in HT^2(M)$.
There is a perfect object $\fU_1\in D^b(X\times M,-\check{u}\boxplus w)$
whose derived restriction $i^*(\fU_1)$ is isomorphic to $\U$.
\end{lem}
\begin{proof}
By Theorem \ref{criterion}, it suffices to show that the degree $2$ piece of 
$(-\check{u}\boxplus w)\cdot\exp a({\U})$ is $0$. We repeat Toda's calculation
of this obstruction in our setting. According to  
\cite[Lemma 5.8]{toda-deformations} (see Remark \ref{valid} below), the following diagrams commute:
$$\xymatrix{
HT^*(X\times M)\;\;\;\ar[r]^{\times\exp a(\U)} & \;\;\Hom^*_{X\times M}(\U,\U)\\
\ar[u]^{\pi_{X}^*} HT^*(X)\ar[r]_{\tau_*I_X^*} & HH^*(X)\ar[u]_{\U \circ}
}$$
$$\xymatrix{
HT^*(X\times M)\;\;\;\ar[r]^{\times\exp a(\U)} & \;\;\Hom^*_{X\times M}(\U,\U)\\
\ar[u]^{\pi_{M}^*} HT^*(M)\ar[r]_{I_M^*} & HH^*(M)\ar[u]_{\circ\U}
}$$
The symbol $\tau_*$ in the first diagram is the involution of $HH^*(X)$ arising from the
interchange of the factors of $X\times X$. One sees that $\tau_*I_X^*(t)=I_X^*(\check{t})$ (see the 
discussion preceding Proposition 6.1 of \cite{toda-deformations}),
so that
\begin{align}
(-\pi_X^*\check{u}+\pi^*_{M}w)\cdot\exp a(\U) 
                    & = -\U\circ \tau_*I_X^*(\check{u})+I^*_{M}(w)\circ\U\nonumber\\
            	    & = -\U\circ I_X^*({u}) + (\phi^{HH}(I_X^*(u)))\circ \U.
\end{align}
The description (\ref{eq-alt-phi-HH}) of the map $\phi^{HH}$  yields  
that $\phi^{HH}(I_X^*(u))\circ\U= \U\circ I_X^*(u)$. The result follows immediately.
\end{proof}

\begin{rem}\label{valid} Although it is assumed everywhere in \cite{toda-deformations}
that the fuctor $\Phi_\U$ is an equivalence, Lemma 5.8 of that article 
holds without this requirement. Indeed, Toda only works with the {\em compositions}
$\exp a(\U)_X$ and $\exp a(\U)_M$ rather than the morphisms $\exp(a)^+_X$ and
$\exp(a)^+_M$ (see p. 212, {\it op. cit.}), 
and while the latter morphisms may not exist without the assumption that  
$\Phi_\U$ is an equivalence, the former always do.
\end{rem}

Let $\ko_{X\times M}^{-\check{u},w}$ be the deformed structure sheaf of the
product $X\times M_H(v)$, and write
$\fV_1$ for $R\H om_{*}(\fU_1, \ko_{X\times M}^{-\check{u},w})[2]$. Note that derived
dualization is defined on the full subcategory of perfect complexes of
$D^b(X\times M,-\check{u}\boxplus w)$, and sends it into the subcategory of perfect complexes of
$D^b(M\times X,-\check{w} \boxplus u)$:
$$R\H om_{*}(\;\_\;,\ko_{X\times M}^{-\check{u},w}): 
D_{\mbox{perf}}(X\times M,-\check{u}\boxplus w) \to
D_{\mbox{perf}}(M\times X,-\check{w} \boxplus u)$$
The functors corresponding to restriction and extension of scalars between 
various categories will be simply
denoted $i_*$ and $i^*$, without reference to the underlying spaces.

Consider the convolution $\fU_1\circ \fV_1 \in D^b(M\times M, -\check{w}\boxplus w)$. Lemma A.5 of 
\cite{mms} implies that $i^*(\fU_1\circ \fV_1)\cong\F$. We conclude that $\F$ deforms along the direction
$-\check{w}\boxplus w\in HT^2(M\times M)$, $w=\phi^{T}(u)$, for any $u\in HT^2(X)$. Denote 
this infinitesimal deformation $\fU_1\circ \fV_1$ by $\fF'_u$.

The object $\F$ is the restriction to a fiber of the family $\fF$ 
constructed in \S\ref{subsec-deformability-of-the-monad} by Assumtion \ref{assumption-in-compatibility-section}, 
that is, there exists a triple
$(M,\eta,\Az)\in U\subset\widetilde{\fM}_\Lambda^0 $ such that  $\F\cong\fF|_{M\times M}$.
Given a class $\xi\in H^1(TM)$, let $\MM_\xi$ denote the first-order infinitesimal defomation of $M$ 
in the direction of $\xi$. Let $\fF_\xi$ be the restriction of $\fF$ to the fiber square $\mathcal{M}_\xi^2$ of
$\mathcal{M}_\xi$ over the length 2 subscheme of $\widetilde{\fM}^0_\Lambda$ detrmined by $\xi$.

\begin{lem}\label{lem-direction-of-deformtion}
\begin{enumerate}
\item\label{item-class} Let $\xi\in H^1(TM)$. The infinitesimal deformation $\fF_\xi$ of $\F$
is an object of the derived category $D^b(M\times M, -\check{w}\boxplus w)$,
where $w=(0,\xi, \theta)$  is a class in the subspace $\mho\subset HT^2(M)$ defined in 
Lemma \ref{lemma-image} (\ref{lemma-item-image-of-phi-HH-in-HT}). 
\item\label{item-rigidity} Write $u$ for the class $(\phi^T)^{-1}(w)\in HT^2(X)$. 
There is an isomorphism $\fF'_u \cong \fF_\xi$.
\end{enumerate}
\end{lem}
\begin{proof}
(\ref{item-class}) Fix a point $[E]\in M$. The restriction $\F|_{M\times\{[E]\}}$ is isomorphic to 
$\Phi_\U(E^\vee[2])$. Making use of the fact proven in Lemma \ref{lemma-kernel-of-the-adjoint-of-F}
that $\tau^*\F\cong \F^\vee[2]$, we see that the Chern character of $\F$ has the form:
$$
ch(\F)=(2-2n)+(-\pi_1^*(c_1(\alpha))+\pi_2^* (c_1(\alpha)))+ch_2(\F)+\cdots \in \oplus H^i(\Omega_{M\times M}^i). 
$$
The product $(-\check{w}\boxplus w)\cdot ch(\F)$ in 
$\pi_1^*(H^2(\ko_M))\oplus \pi_2^*(H^2(\ko_M))\cong H^2(\ko_{M\times M})$ vanishes if and
only if $w\in \mho$. This product is nothing but the trace of the obstruction class
$(-\check{w}\boxplus w)\cdot\exp a(\F)\in \Hom(\F,\F[2])$ \cite[10.1.6]{HL}. As $\F$ deforms,
Theorem \ref{criterion} implies that this class must indeed vanish, which completes the proof.

(\ref{item-rigidity}) This should follow immediately from the infinitesimal rigidity of $\F$ from known 
results\footnote{One possible approach might be by combining \cite[Theorem 3.1.1]{lieblich} and 
\cite[Prop. 2.2.4.9]{lieblich2}. The former studies deformation theory for complexes, the latter
for twisted sheaves; however, only the algebraic situation is considered.}.
However, we were not able to locate a precise reference, so we sketch a short argument specific to our case. Consider
the object $\fF'_u$. Being a deformation of an object with cohomology sheaves in degrees contained in the interval
[-1,0], it is perfect with cohomology in degrees contained in [-1,0].
So, $\fH^0(\fF'_u)$ is perfect, hence flat over the dual numbers; 
using the observation of the previous sentence again,
we get that $\fH^{-1}(\fF'_u)$ is flat over the dual numbers. Therefore, 
$$i^*(\fH^0(\fF'_u)) \cong (\fH^0(i^*\fF'_u))\cong \ko_{\Delta_M}.$$ 
Similarly, $i^*(\fH^{-1}(\fF'_u)) \cong \E$. The 
(twisted) sheaves $\ko_{\Delta_M}$ and  $\E$  are 
infinitesimally rigid (see Lemma \ref{lemma-Ext-1-E-E-vanishes}). This implies that 
$\fH^0(\fF'_u)\cong \ko_{\Delta_{\mathcal{M}_\xi}}$ and 
$\fH^{-1}(\fF'_u)\cong \fE_\xi$, where $\fE_\xi$ is the
restriction of $\fE$ to $\mathcal{M}_\xi^2$  \cite[Lemma 4.11]{torelli}. 
It follows that the object $\fF'_u$ is an extension of $\ko_{\Delta_{\mathcal{M}_\xi}}$
by $\fE_\xi[1]$. Step 3 of Proposition \ref{prop-construction-of-universal-object-F} 
says that, up to scalars, there is a unique such extension. The claimed isomorphism follows
from this.
\end{proof}

\begin{lem} The functor $\Psi_{\fV_1}: D^b(M,w) \to D^b(X,u)$ is right adjoint to 
the functor $\Phi_{\fU_1}: D^b(X,u) \to D^b(M_H(v),w)$. Moreover, the unit
$\eta: \mathbb{1}_{D^b(X,u)} \rightarrow \Psi_{\fV_1}\Phi_{\fU_1}$
of this adjunction is split, that is, there exits a natural transformation
$\zeta: \Psi_{\fV_1}\Phi_{\fU_1} \rightarrow \mathbb{1}_{D^b(X,u)}$ such that
$\zeta\eta\cong \mathbb{1}_{D^b(X,u)}$.
\end{lem}
\begin{proof}
The exact sequence $0 \to \CC \to \CC[\varepsilon]/(\varepsilon^2) \to \CC\to 0$
yields the triangle
\begin{equation}\label{triang-3}
i_* i^*\fV_1\circ\fU_1 \to \fV_1\circ\fU_1 \to i_* i^*\fV_1\circ\fU_1
\end{equation}
where $\fV_1\circ\fU_1 \in D^b(X\times X, -\check{u}\boxplus u)$ is the 
convolution of $\fV_1$ with $\fU_1$. We have the isomorphism 
$i_* i^*\fV_1\circ\fU_1 \cong i_*\V\circ\U$ by base-change \cite[Lemma A.5]{mms}.

Write $\ko_{{\ol{\gd}}_X}\in D^b(X\times X, -\check{u}\boxplus u)$ for 
the structure sheaf of the diagonal. This is defined above
when the gerby part $\theta\in H^2(\ko_X)$ of $u$ is $0$. In general,
set $\ko_{{\ol{\gd}}_X}=\Delta_*\ko_X^u$,
with the $\ko_{X\times X}^{-\check{u}\boxplus u}$-module structure defined
in Example \ref{example-diagonal}. This makes sense 
as a $(-\theta \boxplus \theta)$-twisted sheaf because 
the class  $(-\theta \boxplus \theta)$ restricts to the trivial class along the diagonal in 
$S\times S$.

Consider the following sequence arising from applying the functor 
$\Hom_{D^b(X\times X,-\check{u}\boxplus u)}(\ko_{{\ol{\gd}}_X},\;\_\;)$ to (\ref{triang-3}):
$$
0\to \Hom_{X^2}(\ko_{\gd_X},\V\circ\U) \to 
 \Hom_{D^b(X^2,-\check{u}\boxplus u)}(\ko_{{\ol{\gd}}_X},\fV_1\circ\fU_1) \to
\Hom_{X^2}(\ko_{\gd_X},\V\circ\U) \to 0
$$
As $\V\circ\U\cong \bigoplus_{i=0}^{n-1}\ko_{\gd_X}[-2i]$, we see that
$$\Hom_{X^2}(\ko_{\gd_X},\V\circ\U)\cong\CC, \mbox{\;\;\;\;} 
\Hom_{X^2}(\ko_{\gd_X},\V\circ\U[1])\cong HH^1(X)=(0),$$
from which it follows that the sequence above is exact. Let 
$\ol{\eta}_1: \ko_{{\ol{\gd}}_X}\to \fV_1\circ\fU_1$ be a lift of the unit $\eta$.
We claim that $\ol{\eta}_1$ is the unit of an adjunction $\Phi_{\fU_1}\dashv \Psi_{\fV_1}$.
Indeed, let $\fA, \fB\in D^b(X\times X, -\check{u}\boxplus u)$, and denote their 
derived restrictions to $D^-(X\times X)$ by $\A$,$\B$, respectively. Consider
the commuting diagram:
$$\xymatrix{
\ar[d] \Hom(\Phi_{\fU_1}\fA,i_*\B)\ar[r] & \ar[d]\Hom(\Phi_{\fU_1}\fA,\fB) \ar[r]&
\ar[d] \Hom(\Phi_{\fU_1}\fA,i_*\B)\\
\ar[d]_{\_\;\circ \ol{\eta_1}} \Hom(\Psi_{\fV_1}\Phi_{\fU_1}\fA,\Psi_{\fV_1}i_*\B)\ar[r] & 
\ar[d]_{\_\;\circ \ol{\eta_1}} \Hom(\Psi_{\fV_1}\Phi_{\fU_1}\fA,\Psi_{\fV_1}\fB) \ar[r]&
\ar[d]^{\_\;\circ \ol{\eta_1}} \Hom(\Psi_{\fV_1}\Phi_{\fU_1}\fA,\Psi_{\fV_1}i_*\B)\\
\Hom(\fA,\Psi_{\fV_1}i_*\B)\ar[r] & 
\Hom(\fA,\Psi_{\fV_1}\fB) \ar[r]&
\Hom(\fA,\Psi_{\fV_1}i_*\B)
}$$
The first and third columns can be identified with the composition
$$
\Hom_{X^2}(\Phi_{\U}\A,\B)\to \Hom_{X^2}(\Psi_{\V}\Phi_{\U}\A,\Psi_{\V}\B)
\stackrel{\_\;\circ\eta}{\to} \Hom_{X^2}(\A,\Psi_{\V}\B)
$$
by flat base-change and the adjunction $i^*\dashv i_*$, which is clearly an
isomorphism. It now follows by Proposition 1.1 of \cite{RD} that the composition
of the arrows in the central column is an isomorphism. As this isomorphism is
bi-functorial in $\fA$ and $\fB$, the claim is proved.

Let $\fH_1$ be the cone of the map $\ol{\eta}_1$:
\begin{equation}\label{triang-4}
\ko_{{\ol{\gd}}_X} \stackrel{\ol{\eta}_1}{\to} \fV_1\circ\fU_1\to\fH_1.
\end{equation}
As above, we can compute the extension group
$\Hom_{D^b(X^2, -\check{u}\boxplus u)}(\fH_1,\ko_{{\ol{\gd}}_X}[1])$
by applying the functor $\Hom_{D^b(X^2, -\check{u}\boxplus u)}(\fH_1,\;\_\;)$ to
the triangle $i_*\ko_{\gd_X} \to \ko_{{\ol{\gd}}_X} \to i_*\ko_{\gd_X}$
and chasing the resulting long exact sequence. Using the fact that 
$\Hom_{X^2}(\fH_1,i_*\ko_{\gd_X}[1])\subset HH^{\mbox{\tiny{odd}}}(X)=(0)$, 
the result is that this group is $(0)$. In particular, triangle (\ref{triang-4})
is split, from which the second statement of the lemma follows immediately.
\end{proof}

Fix $u\in (I_X^2)^{-1}(ann(v^\vee))$, and set $w=\phi^T(u) \in HT^2(M)$; the class
$w$ has the form $(0,\xi, \theta)$ by Lemma \ref{lemma-image}.
As above, let $\fF_\xi$ the restriction of $\fF$ to $\mathcal{M}_\xi^2$.
Note that $\fF_\xi$, together with
the structure maps ${\ee}_1$ and ${\dd}_1$, defines a comonad 
$\langle \LL_1, {\ee}_1, {\dd}_1 \rangle$ on $D^b(M,w)$ by Theorem \ref{deformability}. 

\begin{thm}\label{thm-comparison-with-toda} 
(Theorem \ref{thm-comparison} (\ref{abelian-category}))
There is an exact equivalence of triangulated categories
between $D^b(X,u)$ and the category of comodules
$D^b(M,w)^{\LL_1}$. 
\end{thm}
\begin{proof}
Let $\fU_1$ be the deformation of $\U$ corresponding to the class $-\check{u}\boxplus w$
constructed in Lemma \ref{lem-U-deforms}.
Denote the comonad arising from the adjoint pair $\Phi_{\fU_1}\dashv \Psi_{\fV_1}$
by $\langle L'_1, {\ee'_1}, {\dd'_1} \rangle$. The 
unit $\eta: \mathbb{1}_{D^b(X,u)} \rightarrow \Psi_{\fV_1}\Phi_{\fU_1}$
is split by the previous lemma. Therefore,  the categories $D^b(M,w)^{\LL'_1}$ and $D^b(X,u)$ are equivalent
by the Barr-Beck Theorem for triangulated categories 
\cite{elagin, bb, bal2}, which says that for a split adjunction 
the comparison functor  is an equivalence (see the discussion preceding the 
statement of Proposition \ref{prop-Phi-hat-is-an-equivalence}). 
To prove the result, it only remains to show that there 
is an isomorphism of comonads between $\LL_1$ and $\LL_1'$. 

We have seen above that the kernel $\fF'_u$ of the functor $L_1'$
is isomorphic to $\fF_\xi$ (Lemma \ref{lem-direction-of-deformtion}); 
fix an isomorphism $\mu: \fF'_u \widetilde{\to} \fF_\xi$.
Note that $\Hom(\fF'_u,\fF_\xi)\cong \Hom(\fF_\xi,\fF_\xi) = 
\CC[\varepsilon]/(\varepsilon^2)$ by Lemma \ref{ext-line-bundle-2}.
Let  $\fE$ be the universal twisted sheaf. Step 3 of the proof of  
of Proposition \ref{prop-construction-of-universal-object-F} shows that
$\H om_{\Pi}(\ko_{\gd_\MM}, \fE[1])$ vanishes over $U$. Applying the
functor $R\H om_\Pi(\_, \ko_{\ol{\Delta}})$ to the exact triangle
$$
\fE \to \fF\to \ko_{\ol{\Delta}}
$$
allows one to conclude that
$\Hom(\fF'_u, \ko_{\fgd_1})\cong
\Hom(\fF_\xi, \ko_{\fgd_1})\cong \CC[\varepsilon]/(\varepsilon^2)$. So we 
may modify $\mu$ by a scalar in order that the following diagram commutes:
$$\xymatrix{ 
\ar[d]_{\mu} \fF'_u\ar[r]^{\ee'_1} & 
\ar@{=}[d] \ko_{\fgd_1}\\
\fF_\xi \ar[r]^{\ee_1} & \ko_{\fgd_1}\
}$$
It only remains to check that the isomorphism $\mu$ is compatible with the comultiplication maps 
$\dd'_1$ and $\dd_1$, that is, the left square in the diagram below commutes.
Use the equalities $(\ee'_1\circ\id_{\fF'_u})\circ\dd'_1 = \id_{\fF'_u}$ and
$(\ee_1\circ\id_{\fF_\xi})\circ\dd_1 = \id_{\fF_\xi}$,
and the commutativity of the right square to conclude that the composition
$(\ee_1\circ\id_{\ol{\F}_\xi})\circ(\mu\circ\mu)\circ \dd'_1=\mu$. 
$$\xymatrix{ 
\ar[d]^{\mu} \fF'_u \ar[rr]^{\dd'_1} & & 
\ar[d]_{\mu\circ\mu} \fF'_u\circ\fF'_u\ar[rr]^{\ee'_1\circ\id_{\fF'_u}} & &
\ar[d]^{\mu} \fF'_u\\
\fF_\xi \ar[rr]^{\dd_1} & & \fF_\xi\circ\fF_\xi\ar[rr]^{\ee_1\circ\id_{\fF_\xi}} & & \fF_\xi 
}$$
Thus, $(\mu\circ\mu)\circ\dd'_1\circ\mu^{-1}$ is a section of $\ee_1\circ\id_{\ol{\F}_\xi}$. 
The section $\dd_1$ is unique as composition by  $\ee_1\circ\id_{\ol{\F}_\xi}$
gives an isomorphism $\Hom(\fF_\xi, \fF_\xi\circ\fF_\xi) \widetilde{\to} \Hom(\fF_\xi, \fF_\xi)$
by Assumption \ref{assumption-in-compatibility-section} and Lemma \ref{ext-line-bundle-2}.
We conclude the equality
$\dd_1=(\mu\circ\mu)\circ\dd'_1\circ\mu^{-1}$, proving the commutativity of the left square.
\end{proof}

%
\section{Variations of Hodge structures}

Let $X$ be a $K3$ surface, $M:=M_H(v)$ a moduli space of sheaves on $X$, 
$\Az$ the modular Azumaya algebra over $M\times M$, and assume that 
$(M,\eta,\Az)$ belongs to the open set $U$ of $\widetilde{\fM}^0_\Lambda$ in Theorem
\ref{deformability}, for some marking $\eta$ (see  Remark \ref{rem-modular-Azumaya-algebra-belongs-to-moduli}).
The Hodge structure of the Mukai lattice of $X$ can be deformed, 
as we deform the category $D^b(M,\theta)^\LL$, of comodules for the comonad
$(\LL,\epsilon,\delta)$, 
along deformations of $(M,\LL,\epsilon,\delta)$.
These deformations of the Hodge structure are defined in 
\cite[Theorem 1.10]{markman-constraints}. 
The deformed Mukai lattice should be related to the 
Hochschild homology of the category $D^b(M,\theta)^\LL$,
and the type $(1,1)$ sublattice
should be related to the numerical lattice of the $K$-group of $D^b(M,\theta)^\LL$. 
The Mukai lattice of $M_H(v)$, as defined in \cite[Theorem 1.10]{markman-constraints}, 
is that of $X$, by \cite[Theorem 1.14]{markman-constraints}.
As we deform the Mukai lattice of $M_H(v)$,
the class $v$ remains of Hodge-type $(1,1)$, while its orthogonal complement 
remains isometric to $H^2(M,\Integers)$, by \cite[Theorem 1.10]{markman-constraints}. 
In particular, the class $v$ spans the sublattice of integral classes of Hodge type $(1,1)$, for the Mukai lattice of a generic 
deformation of $M$. This agrees with 
Theorem \ref{thm-comparison} in the current paper, which 
suggests that the family of deformation we get is
the complete family of deformations
of $D^b(X)$, which preserve the Hodge type of the class $v$.
Such deformations include deformations of $D^b(X)$, 
which are equivalent to derived categories of coherent sheaves in
non-commutative and gerby deformations of the $K3$ surface $X$
\cite{toda-deformations,mms}.

%



\begin{thebibliography}{B-N-R}

\bibitem[Ad]{addington} Addington, N.:
{\em New derived symmetries of some Hyperkehler varieties.\/}
Preprint arXiv:1112.0487.

\bibitem[ADM]{adm} Addington, N., Donovan W., Meachan, C.:
{\em On derived categories of moduli spaces of torsion sheaves on $K3$ surfaces.\/}
Preprint 2015.

\bibitem[AT]{AT} Addington, N., Thomas, R.:
{\em Hodge Theory and derived categories of cubic fourfolds.\/} 
Duke Math. J. 163 (2014), no. 10, 1885--1927.

\bibitem[AK]{altman-kleiman} Altman, A., Kleiman, S.:
{\em Compactifying the Picard scheme.\/} Adv. in Math. 35, 50--112 (1980).

\bibitem[Bal1]{bal1} Balmer, P.: 
{\em Separability and triangulated categories.\/}
Adv. in Math. 226, no. 5 (2011), pp. 4352--4372.

\bibitem[Bal2]{bal2} Balmer, P.: 
{\em Descent in triangulated categories.\/}
Math. Ann. 353, no. 1, 109--125.

\bibitem[BaMa]{bayer-macri} Bayer, A., Macri, E.:
{\em MMP for moduli of sheaves on $K3$'s via wall-crossing: nef and movable cones, Lagrangian fibrations\/.}
Invent. Math. 198 (2014), no. 3, 505--590.


\bibitem[BaS]{bs} Balmer, P., Schlichting, M.: 
{\em Idempotent completion in triangulated categories.\/}
J. Algebra 236 (2001), 819--834.

\bibitem[BBW]{hopf} B\"ohm, G., Brzezi\'nski, T., Wisbauer, R.:
{\em Monads and commands on module categories. \/} 
J. Algebra 322 (2009), 1719--1747.

\bibitem[Be2]{beauville-varieties-with-zero-c-1}
Beauville, A.: {\em Varietes K\"ahleriennes dont la premiere classe de Chern 
est nulle.}  J. Diff. Geom. 18, 755--782 (1983).


\bibitem[Be2]{beauville} Beauville, A.: {\em Some remarks on K\"ahler manifolds with $c_1=0$.\/}
Classification of algebraic and analytic manifolds (Katata, 1982), 1--26, Progr. Math., Birkh\"{a}user Boston, Boston MA 1983.

\bibitem[BM]{BM} Bridgeland, T., Maciocia, A.: {\em Fourier-Mukai transforms for $K3$ and elliptic fibrations.\/}
J. Alg. Geom. 11 (2002) 629--657.

\bibitem[BO]{semi-orth} Bondal, A., Orlov, D.:
{\em Semi-orthogonal decompositions for algebraic varieties. \/} Electronic preprint, 
arXiv:alg-geom/9506012v1.

\bibitem[BKR]{BKR} Bridgeland, T., King, A., Reid, M.:
{\em The McKay correspondence as an equivalence of derived categories.\/} J. Amer. Math. Soc. 14 (2001), 535--554.


\bibitem[C1]{Cal-thesis} Caldararu, A.:
{\em Derived categories of twisted sheaves on Calabi-Yau manifolds.\/}
Cornell U. Thesis, 2000.

\bibitem[C2]{cal-mukai1} Caldararu, A.:
{\em The Mukai pairing, I: the Hoschschild structure.\/}
Electronic preprint arXiv:math/0308079v2.

\bibitem[C3]{cal-mukai2} Caldararu, A.:
{\em The Mukai pairing, II: the Hochschild-Kostant-Rosenberg isomorphism.\/} Advances in Math. 194 (2005) 34--66.

\bibitem[CRVdB]{CRVdB} Calaque, D.; Rossi, C.; Van den Bergh, M.:
{\em Caldararu's conjecture and Tsygan's formality.\/} 
Ann. of Math. (2) 176 (2012), no. 2, 865--923. 

\bibitem[CW]{cal-wil-mukai1} Caldararu, A.; Willerton, S.:
{\em The Mukai pairing, I: a categorical approach.\/}
New York J. Math. 16 (2010), 61--98.

\bibitem[E]{elagin} Elagin, A.:
{\em Cohomological descent theory for a morphism of stacks and for equivariant derived 
categories.\/}
Sbornik: Mathematics, 202:4 (2011), 495--526.

\bibitem[Gr-SGA2]{Gr} Grothendieck, A.:
{\em Cohomologie local des faisceaux coh\'{e}rent et th\'{e}r\`{e}me de Lefschetz locaux
et globaux - SGA 2.\/} 
North Holland, Amsterdam, 1968.

\bibitem[GH]{griffiths-harris} Griffiths, P., Harris, J.:
{\em Principles of Algebraic Geometry.\/}
John Wiley \& Sons (1978).

\bibitem[GR]{GR} Grauert, H.; Remmert, R.:
{\em Coherent Analytic Sheaves.\/} 
Springer, Grundlehren der mathematischen Wissenschaften vol. 265 (1984).

\bibitem[Ha1]{Haiman-1} Haiman, M.:
{\em Hilbert schemes, polygraphs and the Macdonald positivity conjecture.\/}
J. Amer. Math. Soc. 14 (2001), no. 4, 941--1006.

\bibitem[Ha2]{Haiman-2} Haiman, M.:
{\em Vanishing theorems and character formulas for the Hilbert scheme of points in the plane.\/}
Invent. Math. 149 (2002), no. 2, 371--407. 


\bibitem[Ha-AG]{AG} Hartshorne, R.:
{\em Algebraic Geometry.\/}
Springer, GTM vol. 52 (1977).

\bibitem[Ha-RD]{RD} Hartshorne, R.:
{\em Residues and Duality.\/}
Springer, LNM vol. 20 (1966).

\bibitem[Hart]{Hart} Hartmann, H.:
{\em Cusps of the K\"{a}hler moduli space and stability conditions on K3 surfaces.\/}
Math. Ann. 354 (2012), no. 1,  1--42.

\bibitem[HaT1]{hassett-tschinkel-lagrangian-planes} Hassett, B., Tschinkel, Y.:
{\em Hodge theory and Lagrangian planes in generalized Kummer fourfolds.\/}
Mosc. Math. J. 13 (2013), no. 1, 33--56.

\bibitem[HaT2]{hassett-tschinkel} Hassett, B., Tschinkel, Y.:
{\em Moving and ample cones of holomorphic symplectic fourfolds.\/}
Geom. Funct. Anal. Vol. 19 (2009) 1065--1080.


\bibitem[Hu1]{huybrechts-book} Huybrechts, D.:
{\em Fourier-Mukai Transforms in Algebraic Geometry.\/}
Oxford University Press, 2006.

\bibitem[Hu2]{huybrechts-basic-results} Huybrechts, D.:
{\em Compact hyperk\"{a}hler manifolds: basic results.\/}
Invent. math. 135, 63--113 (1999).


\bibitem[HN]{HN} Huybrechts, D.; Nieper-Wisskirchen, M.:
{\em Remarks on derived equivalences of Ricci-flat manifolds .\/}  
Math. Z. (2011) 267, 939--963.

\bibitem[HuL]{HL} Huybrechts, D.; Lehn, M.:
{\em The geometry of moduli spaces of sheaves, 2nd ed.\/} 
Cambridge University Press, Cambridge (2010), xviii+325 pp.

\bibitem[HuT]{HT} Huybrechts, D.; Thomas, R.:
{\em Deformation-obstruction theory for complexes via Atiyah and Kodaira-Spencer 
classes.\/}  
Math. Ann.  346  (2010),  no. 3, 545--569.

\bibitem[I]{illu} Illusie, L.: 
{\em Grothendieck's existence theorem in formal geometry.\/} 
Math. Surveys Monogr., 123,  Fundamental algebraic geometry,  179--233, Amer. Math. Soc., Providence, RI, 2005

\bibitem[J]{johnstone} Johnstone, P. T.:
{\em Adjoint lifting for categories of algebras.\/}
Bull. London Math. Soc., 7 (1975), 294--297.


\bibitem[KL]{kleiman-lon} Kleiman, S.; Lonsted, K:
{\em Basics on families of hyperelliptic curves.\/} 
Compositio Math. 38 (1979), no. 1, 83--111.

\bibitem[KLS]{KLS} Kaledin, D.; Lehn, M.; Sorger, Ch.: 
{\em Singular symplectic moduli spaces.\/} 
Invent. Math.  164  (2006),  no. 3, 591--614. 

\bibitem[K]{krug} Krug, A.: {\em Extension groups of tautological sheaves on Hilbert schemes.\/}
J. Alg. Geom. 23 (2014), no. 3, 571--598.

\bibitem[L]{lange}
Lange, H.:
{\em Universal families of extensions.\/}
J. Algebra 83 (1983), no. 1, 101--112. 

\bibitem[Lieb1]{lieblich}
Lieblich, M.:
{\em Moduli of Complexes on a Proper Morphism.\/}
J. Alg. Geom. 15 (2006), no. 1, 175--206.

\bibitem[Lieb2]{lieblich2} Lieblich, M.: 
{\em Moduli of twisted sheaves.} 
Duke Math. J. 138 (2007), no. 1, 23--118. 

\bibitem[Lo]{lowen} Lowen, W.: 
{\em Hochschild cohomology, the characteristic morphism, and derived deformations.\/}
Comp. Math. 144 (2008) 1557--1580.

\bibitem[LN]{Newstead-Lange}
Lange, H.; Newstead, P. E.:
{\em On Poincar\'{e} bundles of vector bundles on curves.\/}
Manuscripta Math. 117 (2005), no. 2, 173--181. 


\bibitem[Mac]{working} MacLane, S.:
{\em Categories for the Working Mathematican.\/} Second Ed., 
Springer-Verlag, Graduate Texts in Math. 5. (1998).

\bibitem[Malt]{Malt} Maltsiniotis, G.: 
{\em Cat\'egories triangul\'ees sup\'erieures.\/} Electronic preprint (2006); available at
http://people.math.jussieu.fr/$\sim$maltsin/ps/triansup.ps.

\bibitem[MaMu]{manes-mulry} Manes, E., Mulry, P.: {\em Monad compositions I:
general constructions and recursive distributive laws.\/}
Theory and applications of ceategories, Vol. 18, No. 7, 2007, 172--208.

\bibitem[Ma1]{markman-diagonal} Markman, E.:
{\em Generators of the cohomology ring of moduli spaces of sheaves on 
symplectic surfaces.\/} J. Reine Angew. Math. 544 (2002), 61--82. 

\bibitem[Ma2]{markman-poisson} Markman, E.:
{\em Integral generators for the cohomology ring of moduli spaces of
sheaves over Poisson surfaces.\/}
Adv. in Math. 208 (2007), 622--646.

\bibitem[Ma3]{markman-monodromy-I} Markman, E.:
{\em On the monodromy of moduli spaces of sheaves on 
K3 surfaces.\/}
J. Algebraic Geom. {\bf 17}  (2008), 29--99. 

\bibitem[Ma4]{markman-constraints} Markman, E.:
{\em Integral constraints on 
the monodromy group of the hyperk\"{a}hler 
resolution of a symmetric product of a $K3$ 
surface.\/}  Internat. J. of Math. 21, (2010), no. 2, 169--223.  

\bibitem[Ma5]{markman-hodge} Markman, E.:
{\em The Beauville-Bogomolov class as a characteristic class.\/}
Electronic preprint, arXiv:1105.3223v1.

\bibitem[Ma6]{markman-survey} Markman, E.:
{\em  A survey of Torelli and monodromy results for hyperkahler manifolds.\/}
In  ``Complex and Differential Geometry'', W. Ebeling et. al. (eds.),
Springer Proceedings in Math. 8, (2011), 257--323.

\bibitem[Ma7]{markman-stability} Markman, E.:
{\em Stability of an Azumaya algebra over the Cartesian square of the Hilbert scheme of
$n$ points of a generic $K3$ surface.}\/ Electronic preprint arXiv:1506.06191.

\bibitem[MM1]{density} Markman, E., Mehrotra, S.:
{\it Hilbert schemes of K3 surfaces are dense in moduli.\/}
Electronic preprint, arXiv:1201.0031.

\bibitem[MM2]{torelli} Markman, E., Mehrotra, S.:
 {\it A global Torelli theorem for rigid hyperholomorphic sheaves.\/}
Electronic preprint, arXiv: 1310.5782.

\bibitem[MS]{bb} Mehrotra, S.:
{\em The Barr-Beck Theorem in the triangulated context.\/}
Preprint 2009.


\bibitem[MSM]{mms} Macri, E.; Stellari, P.:
{\em Infinitesimal Derived Torelli Theorem for K3 surfaces (with an Appendix by Sukhendu Mehrotra).\/} 
Int. Math. Res. Notices 17, 3190--3220 (2009).

\bibitem[Mu1]{mukai-symplectic}  Mukai, S.:
{\em Symplectic structure of the moduli space of
sheaves on an abelian or K3 surface}, 
Invent. math. 77, 101--116 (1984).

\bibitem[Mu2]{mukai-applications} Mukai, S.: 
{\em Fourier functor and its application to the moduli of bundles 
on an Abelian variety.\/}
Adv. Studies in Pure Math. 10, 515--550 (1987).

\bibitem[Mu3]{mukai-tata} Mukai, S.: 
{\em On the moduli space of bundles on $K3$ surfaces. I.\/} in
 Vector bundles on algebraic varieties (Bombay, 1984),  Tata Inst. Fund. Res. Stud. Math., 11,
341--413.

\bibitem[Mu4]{mukai-sugaku} Mukai, S.:
{\em Moduli of vector bundles on $K3$ surfaces and symplectic manifolds. \/} Sugaku Expositions 1 (1988), no. 2, 139--174.

\bibitem[Mu5]{mukai-hodge} Mukai, S.:
{\em On the moduli space of bundles on K3 surfaces I.\/} In the proceedings of 
Vector bundles on algebraic varieties (Bombay, 1984), 341--413,
Tata Inst. Fund. Res. Stud. Math., 11, 1987. 

\bibitem[Mu6]{mukai-duality} Mukai S.: {\em Duality between $D(X)$ and $D(\hat{X})$ with its application to Picard sheaves.\/}  Nagoya Math. J.  81  (1981), 153--175.





\bibitem[OG1]{OG} O'Grady, Kieran G.: 
{\em The weight-two Hodge structure of moduli spaces of sheaves on a $K3$  surface.\/}  
J. Algebraic Geom.  6  (1997),  no. 4, 599--644.

\bibitem[OG2]{OG-ten} O'Grady, Kieran G.: 
{\em Desingularized moduli spaces of sheaves on a K3.  \/}
J. Reine Angew. Math. 512 (1999), 49--117.

\bibitem[Sc]{Scala} Scala, L. :
{\em Cohomology of the Hilbert scheme of points on a surface with values in representations
of tautological bundles.\/}
Duke Math. J. 150 (2009), no. 2, 211--267.

\bibitem[Si]{Simp} Simpson, C.:
{Moduli of representations of the fundamental group of a smooth projective
variety I.\/}
Publ. Math. I.H.E.S. 79 (1994), p. 47--129.

\bibitem[St]{street} Street, R.: {\em The formal theory of monads.\/}
J. of pure and applied algebra 2 (1972) 149--168.

\bibitem[Ti]{idem-serre} Tian, L.:
{\em Idempotent Completions, Serre Functors and Recollements. \/}
Journal of Xiamen University(Natural Science) 2011-04.

\bibitem[To]{toda-deformations} Toda, Y.:
{\em Deformations and Fourier-Mukai transforms.\/} 
Preprint arXiv:math/0502571.
J. Differential Geom. 81 (2009), no. 1, 197--224.

\bibitem[V1]{kaledin-verbitski-book} Verbitsky, M.:
{\em Hyperholomorphic sheaves and new examples of hyperkaehler manifolds.\/} 
In the book: Hyperk\"{a}hler manifolds, by Kaledin, D. and Verbitsky, M., 
Mathematical Physics (Somerville), 12. International Press, 
Somerville, MA, 1999. 

\bibitem[V2]{Ver} Verbitsky, M.:
{\em Coherent sheaves on general K3 surfaces and tori.\/} 
Pure Appl. Math. Q. 4 (2008), no. 3, part 2, 651--714.

\bibitem[V3]{verbitsky-1996} Verbitsky, M.:
{\em Hyperholomorphic bundles over a hyper\"{a}hler manifold.\/}
J. Alg. Geom. 5 (1996), 633--669.

\bibitem[V4]{verbitsky-ergodicity} Verbitsky, M.:
{\em Ergodic complex structures on hyperk\"{a}hler manifolds.\/}
Electronic preprint, arXiv:1306.1498v2.


\bibitem[Y1]{yoshioka-abelian-surface} Yoshioka, K.:
{\em 
Moduli spaces of stable sheaves on abelian surfaces. \/
}
Math. Ann. 321 (2001), no. 4, 817--884.

\bibitem[Y2]{yoshioka-note-on-fourier-mukai}
 Yoshioka, K.:
{\em A Note on Fourier-Mukai transform.\/}
Electronic preprint, arXiv:math.AG/0112267v3.



\end{thebibliography}
\end{document}